\documentclass[12 pt]{article}
\usepackage{amsmath,amsfonts,amsthm,amssymb,mathrsfs,latexsym,paralist,xy}
\usepackage{tikz}
\usepackage{soul}

\newtheorem{thm}{Theorem}[section]
\newtheorem{prop}[thm]{Proposition}
\newtheorem{cor}[thm]{Corollary}
\newtheorem{lemma}[thm]{Lemma}
\theoremstyle{remark}
\newtheorem{rmk}[thm]{Remark}
\newtheorem{example}[thm]{Example}
\theoremstyle{definition}
\newtheorem{defn}[thm]{Definition}

\DeclareMathOperator{\Ad}{Ad}

\newcommand{\bi}{\begin{itemize}}
\newcommand{\ei}{\end{itemize}}
\newcommand{\be}{\begin{enumerate}}
\newcommand{\ee}{\end{enumerate}}

\newcommand{\C}{\mathbb{C}}

\newcommand{\T}{\mathbb{T}}
\renewcommand{\H}{\mathcal{H}}

\newcommand{\R}{\mathbb{R}}
\newcommand{\N}{\mathbb{N}}
\newcommand{\Z}{\mathbb{Z}}

\newcommand{\BH}{\mathcal{B}(\mathcal{H})}

\providecommand{\keywords}[1]{{\textit{Keywords and phrases:}} #1}
\providecommand{\classification}[1]{{\textit{2010 Mathematics Subject Classification:}} #1}

\newcommand{\Ran}{\operatorname{Range}}

\def\IoIIdimdots(#1/#2/#3,#4){\node at (#1,#4) {$.$};\node at (#2,#4) {$.$};\node at (#3,#4) {$.$};}
\def\IIoIIdimdots(#1,#2/#3/#4){\node at (#1,#2) {$.$};\node at (#1,#3) {$.$};\node at (#1,#4) {$.$};}

\def\IoIIIdimdots(#1/#2/#3,#4,#5){\node at (#1,#4,#5) {$.$};\node at (#2,#4,#5) {$.$};\node at (#3,#4,#5) {$.$};}
\def\IIoIIIdimdots(#1,#2/#3/#4,#5){\node at (#1,#2,#5) {$.$};\node at (#1,#3,#5) {$.$};\node at (#1,#4,#5) {$.$};}
\def\IIIoIIIdimdots(#1,#2,#3/#4/#5){\node at (#1,#2,#3) {$.$};\node at (#1,#2,#4) {$.$};\node at (#1,#2,#5) {$.$};}

\usepackage[margin=1in]{geometry}
\AtEndDocument{\bigskip{\small%
  \textsc{Carla Farsi, Judith Packer : Department of Mathematics, University of Colorado at Boulder, Boulder, CO 80309-0395, USA.} \par
  \textit{E-mail address}: \texttt{carla.farsi@colorado.edu, packer@euclid.colorado.edu} \par
  \addvspace{\medskipamount}
  \textsc{Elizabeth Gillaspy : Department of Mathematics, University of Montana, 32 Campus Drive \#0864, Missoula, MT 59812-0864, USA.}\par
  \textit{E-mail address}: \texttt{elizabeth.gillaspy@mso.umt.edu}\par
  \addvspace{\medskipamount}
  \textsc{Palle Jorgensen : Department of Mathematics, 14 MLH, University of Iowa, Iowa City, IA 52242-1419, USA.}\par
  \textit{E-mail address}: \texttt{palle-jorgensen@uiowa.edu}\par
   \addvspace{\medskipamount}
  \textsc{Sooran Kang : College of General Education, Chung-Ang University, 84 Heukseok-ro, Dongjak-gu, Seoul, Republic of Korea.} \par
  \textit{E-mail address}, \texttt{sooran09@cau.ac.kr}
}}

\begin{document}

\title{Separable representations of higher-rank graphs}
\author{Carla Farsi, Elizabeth Gillaspy, Palle Jorgensen,  Sooran Kang, and Judith Packer}
\date{\today}
\maketitle

\begin{abstract} 
In this monograph  we  undertake a comprehensive study of 
separable representations (as well as  their unitary equivalence classes) of  $C^*$-algebras associated to strongly connected  finite $k$-graphs $\Lambda$.   We begin with the representations associated to the $\Lambda$-semibranching function systems introduced by Farsi, Gillaspy, Kang, and Packer in \cite{FGKP}, by giving an alternative characterization of these systems which is more easily verified in examples.  We present a variety of such examples, one of which we use to construct a new faithful separable representation of any row-finite source-free $k$-graph.  Next, we analyze the monic representations of $C^*$-algebras of finite $k$-graphs.  We completely characterize  these representations, generalizing results of Dutkay and Jorgensen \cite{dutkay-jorgensen-monic} and Bezuglyi and Jorgensen \cite{bezuglyi-jorgensen} for Cuntz and Cuntz-Krieger algebras respectively.  We also describe a universal  representation for non-negative monic representations of finite, strongly connected $k$-graphs.
To conclude, we characterize the purely atomic and permutative representations of $k$-graph $C^*$-algebras, and discuss the relationship between these representations and the classes of representations introduced earlier.

\end{abstract}

\classification{46L05, 46L55, 46K10, 16GXX}


\tableofcontents

\section{Introduction}
Commutative harmonic analysis can be thought of as the study of representations of abelian groups.
By extension, non-commutative harmonic analysis includes the study of representations of non-abelian groups.  To build a coherent theory of harmonic analysis, however, one must first develop a version of Pontryagin duality for non-abelian groups.  There are several natural and well-established options for such a duality theory, including Tannaka-Krein duality \cite{krein1, krein2} and Woronowicz' quantum groups \cite{wor1, wor2}.  However, in this monograph we use the $C^*$-algebraic perspective.
 One can associate a $C^*$-algebra to any locally compact topological group, via a construction which reflects Pontryagin duality in the abelian case; therefore, $C^*$-algebras are a natural context for non-commutative harmonic analysis. Indeed, the perspective that representations of $C^*$-algebras should form a non-commutative generalization of harmonic analysis has been a guiding principle throughout the development of $C^*$-algebra theory, since its initial formulation in \cite{connes}.   Its range of applications is wide, and includes (among other fields) quantum physics \cite{kribs2003}, the spectral theory of quasicrystals \cite{bellissard, kellendonk-putnam}, the analysis of fractals and self-similar structures \cite{palle-iterated, dutkay-jorgensen-four, MP}, and symbolic dynamics \cite{GPS, pask-raeburn-weaver, bezuglyi-k-m, skalski-zacharias}. Here, our focus is a non-commutative harmonic analysis of countable 
(discrete) graphs. 
For the present case of higher-rank graphs and their associated semibranching function systems, we show that the representation theoretic approach leads to a fruitful and versatile non-commutative harmonic analysis.  While our focus is non-commutative harmonic analysis, we stress that the literature dealing with analysis on 
 graphs is substantial.  In the paragraphs which follow, we discuss a few of the situations in which analysis on graphs has been profitably employed, and situate the present work in their context.


The research presented in the pages that follow develops a non-commutative harmonic analysis for higher-rank graphs.  More precisely, we  analyze several different types of separable representations for the $C^*$-algebras associated to higher-rank graphs (see Definition \ref{def-higher-rank-gr} below), and also present a variety of examples of such representations.   {The classes of representations that we focus on in this paper are the representations associated to $\Lambda$-semibranching function systems, first introduced in \cite{FGKP}; monic representations, introduced in Section \ref{sec:monic-results}; purely atomic representations, defined and studied in Section \ref{sec:atomic_repn}; and permutative representations, which we introduce in Section \ref{sec:permutative_repn}. Several sections (namely Sections \ref{sec:examples_a}, \ref{sec:examples_gen_measure}, and \ref{sec:monic-examples}) are devoted to illustrating our results via a broad spectrum of examples.}

Building on the work of Robertson and Steger \cite{RS-London, RS}, Kumjian and Pask introduced higher-rank graphs $\Lambda$ -- also known as $k$-graphs -- and their $C^*$-algebras $C^*(\Lambda)$  in \cite{KP} as
generalizations of the Cuntz and Cuntz-Krieger $C^*$-algebras associated to directed graphs (cf.~\cite{Cuntz, Cuntz-Krieger, enomoto-watatani,
kpr}). 
To see the Cuntz and Cuntz-Krieger algebras as graph algebras, note that the  Cuntz algebra $\mathcal O_N$ \cite{Cuntz} (which is perhaps the first example of graph $C^*$-algebras treated in the literature) is the graph $C^*$-algebra associated to the complete graph on $N$ vertices.  Similarly,  the Cuntz-Krieger algebra $\mathcal{O}_A$ \cite{Cuntz-Krieger} is the $C^*$-algebra associated to the directed graph with adjacency matrix $A$. In contrast, a $k$-graph $\Lambda$ (see Definition \ref{def:kgraph-algebra} below) has $k$ adjacency matrices associated to it; one can view $ \Lambda$  as a quotient of an edge-colored directed graph with $k$ colors of edges.

 In addition to their relevance for $C^*$-algebraic classification \cite{prrs, ruiz-sims-sorensen}, the $C^*$-algebras of higher-rank graphs  are closely linked with orbit equivalence for shift spaces \cite{carlsen-ruiz-sims} and with symbolic dynamics more generally \cite{pask-raeburn-weaver,skalski-zacharias-houston, pask-sierakowski-sims}, with fractals and self-similar structures \cite{FGJKP1, FGJKP2}, and with renormalization problems in physics \cite{figueroa-GB}. {More links between higher-rank graphs and symbolic dynamics can be seen via \cite{bezuglyi-k-m, bezuglyi-four-auth, bezuglyi-karpel} and the references cited therein.} A wavelet-type representation for higher-rank graphs was introduced in \cite{FGKP}; indeed, connections with wavelets, which had earlier been identified in certain special cases \cite{BJ-atomic, dutkay-jorgensen-monic, kribs2003,  dutkay-jorgensen-atomic, MP}, were a major source of inspiration for the research we present below.  In the case of the Cuntz algebra $\mathcal O_N$, its wavelet representations have relevance for Brownian motion \cite{AJP} and quantum computing \cite{jorgensen-kribs}, among other applications. 

Higher-rank graph $C^*$-algebras, and their twisted counterparts (developed in 	   \cite{KPS-twisted,KPS-homology}), share many of the important properties of graph $C^*$-algebras, including Cuntz-Krieger uniqueness theorems and realizations as (twisted) groupoid $C^*$-algebras. In addition to the manifold applications of higher-rank graphs already mentioned, many important examples
of $C^*$-algebras (such as noncommutative tori \cite{KPS-homology} and quantum Heegaard spheres \cite{hajac-sims-arxiv}) can be viewed as twisted $k$-graph $C^*$-algebras. 

Both $k$-graph $C^*$-algebras, and the Cuntz and Cuntz-Krieger $C^*$-algebras, often fall in a class of non-type I, and in fact purely infinite $C^*$-algebras. The significance of this for representation theory is that the unitary equivalence classes of irreducible representations of $k$-graph algebras  do not lend themselves to explicit parametrizations. In fact, on account of work by Glimm, Dixmier, and Effros  \cite{Glimm1,Glimm2,Dixmier,Effros1,Effros2}, we know that, for these $C^*$-algebras, the corresponding sets of equivalence classes of irreducible representations do not arise as Borel cross sections. In short, for these $C^*$-algebras, only subfamilies of irreducible representations admit ``reasonable'' parametrizations.

Nonetheless, various specific subclasses of representations of Cuntz and Cuntz-Krieger algebras have been extensively studied by many researchers, who were motivated by their applicability to a wide variety of fields.  In addition to the connections with wavelets which were indicated above  (cf.~also \cite{dutkay-jorgensen-mart,dutkay-jorgensen-four,MP,FGKPexcursions,FGJKP2}), representations of Cuntz-Krieger algebras have been linked to  fractals and  Cantor sets \cite{Str,kawamura-2016,FGJKP1,FGJKP2} and to the endomorphism group  of a Hilbert space \cite{bratteli-jorgensen-price,Laca-Thesis}. Indeed, the astonishing goal of identifying both discrete and continuous series of representations of Cuntz (and to some extent Cuntz-Krieger) algebras,  was accomplished in \cite{dutkay-jorgensen-monic,dutkay-jorgensen-atomic,bezuglyi-jorgensen}, building on the pioneering results of \cite{BJ-atomic}.

In the setting of higher-rank graphs, however, researchers have not yet begun the (admittedly daunting) task of analyzing the representation theory of $k$-graph $C^*$-algebras, save in certain specific cases. Indeed, in the representation of $C^*(\Lambda)$ which appears most commonly in the literature, $C^*(\Lambda)$ acts on $\ell^2(\Lambda^\infty)$.  The space $\Lambda^\infty$ of infinite paths in a $k$-graph $\Lambda$ is usually a Cantor set, making this representation highly nonseparable.
 Although the primitive ideal space of higher-rank graph $C^*$-algebras is well understood \cite{CKSS, Kang-Pask},    representations of these $C^*$-algebras have only been systematically studied in the one-vertex case \cite{dav-pow-yan-atomic,Yang-End,davidson-yang-represent}. 
 This motivated us to undertake 
the  detailed study of separable representations of $k$-graph $C^*$-algebras and their unitary equivalence classes which is contained in these pages.  The representations associated to $\Lambda$-semibranching function systems, which were introduced in \cite{FGKP} for finite higher-rank graphs $\Lambda$, form our jumping-off point.

   Readers familiar with the study of representations of non-type I $C^*$-algebras will appreciate the fact that insight is hard fought, and is gained primarily by a focus on particular cases.  The classes of representations of $k$-graph $C^*$-algebras that we consider in this paper have been chosen with an eye to applications.  As indicated in the preceding paragraphs, particular representations of Cuntz and Cuntz-Krieger algebras have contributed key insights to hundreds of questions in mathematics, physics, and even engineering.  

The present study of representations of higher-rank graphs goes far beyond these recent investigations. By expanding such fertile families of representations from the Cuntz algebra setting to the $k$-graph setting, which encompasses a much broader class of non-type I $C^*$-algebras, we anticipate that the research contained in these pages will facilitate further progress in a diverse and extensive range of fields. The following is a sample: branching laws for endomorphisms of fermions, see e.g., \cite{abe-kawamura}, Markov measures, transfer operators, wavelets and multiresolutions \cite{alpay-jorgensen-Lew}, fractional Brownian motion \cite{AJP}, quasi-crystals, see e.g., \cite{bellissard}, substitution dynamical systems and complexity, structure of invariant measures, Markov and more general path-space measures, measurable partitions \cite{bezuglyi-karpel}, \cite{bezuglyi-four-auth}, \cite{bezuglyi-jorgensen}, and \cite{BJ-atomic}, noncommutative geometry \cite{connes}, topological Markov chains \cite{Cuntz-Krieger}, martingales \cite{dutkay-jorgensen-mart}, fractals and self-similarity \cite{dutkay-jorgensen-four}, spectral triples \cite{FGJKP2} , renormalization theory (physics) \cite{figueroa-GB}, topological orbit equivalence \cite{GPS}, wavelet filters \cite{palle-advances}, continued fraction expansions \cite{kkawa,	
  kawamura-2016}, tilings and tiling space \cite{kellendonk-putnam}, quantum channels  \cite{kribs2003}, group actions on the boundary of triangle buildings  \cite{RS-London}, and entropy   \cite{skalski-zacharias-houston, skalski-zacharias}.

\subsection{Structure of the paper; main results}
We begin in Section \ref{sec-fund-mater} with an introduction to higher-rank graphs and their $C^*$-algebras, followed by a review of the $\Lambda$-semibranching function systems introduced in \cite{FGKP}.
Intuitively, a $\Lambda$-semibranching function system (see Definition \ref{def-lambda-SBFS-1} and Theorem \ref{thm:SBFS-edge-defn} below) encodes the generating relations of $C^*(\Lambda)$ into a measure space $(X, \mu)$, in such a way as to induce a representation of $C^*(\Lambda)$ on $L^2(X, \mu)$.  When $\Lambda$ is a directed graph (equivalently, a Cuntz-Krieger algebra), such semibranching function systems have been  studied  in \cite{gh-Per-Frob,ghr,MP,kawamura,bezuglyi-jorgensen} among others.  Section \ref{sec-fund-mater} concludes with several results related to the Kolmogorov Extension Theorem which we use repeatedly throughout this work.

By Theorem 3.5 of  \cite{FGKP}, when $\Lambda$ is finite,  any $\Lambda$-semibranching function
system on $L^2(X, \mu)$  induces  a  representation of  $C^*(\Lambda)$ on $L^2(X, \mu)$.  However, this representation will not usually be faithful (see Theorem 3.6 of the same paper).  The first main result of this paper is therefore Theorem \ref{pr:sep-faith}
 in Section \ref{sec-a-faithful-sep-repr}, which constructs a faithful separable representation of any row-finite, source-free $k$-graph.  
 The representation of Theorem \ref{pr:sep-faith} is based on a simpler construction (Theorem \ref{prop:sc-faithful}), which exemplifies many of the objects studied in this monograph.
 Although the representation of Theorem \ref{prop:sc-faithful} is initially defined on an inductive limit Hilbert space, rather than $L^2(X, \mu)$, we show in Proposition \ref{prop:sc-faithful-SBFS} that this representation can be viewed as a $\Lambda$-semibranching function system.  It also gives an example of a permutative representation (see Section \ref{sec:permutative_repn} below).

 After rephrasing (in Section \ref{sec:SBFS-revisited}) the definition of a $\Lambda$-semibranching function system into a format which is easier to check in examples, we present a variety of examples of $\Lambda$-semibranching function systems in   Sections \ref{sec:examples_a} (where $X$ is a Lebesgue measure space) and \ref{sec:examples_gen_measure} (where $X = \Lambda^\infty$ is the infinite path space of the higher-rank graph).  Through careful computations of the Radon-Nikodym derivatives associated to these $\Lambda$-semibranching function systems, we analyze the relationship between their associated representations and the $\Lambda$-semibranching representation on $L^2(\Lambda^\infty, M)$ which was introduced in Proposition 3.4 of \cite{FGKP}.\footnote{The measure $M$ was introduced in Definition 8.1 of \cite{aHLRS3}; Theorem 3.14 of \cite{FGJKP2} then established that $M$ is the Hausdorff measure on $\Lambda^\infty$ under mild hypotheses.}

Before turning our attention to a more theoretical and systematic analysis of the separable representations of $C^*(\Lambda)$ in the second half of this monograph, we spend Section \ref{sec-further-foundations} developing the technical tools we will need for this analysis: a refinement of a $\Lambda$-semibranching function system which we call a $\Lambda$-projective system (Definition \ref{def:lambda-proj-system}), and the projection-valued measure $P = P_\pi$ on $\Lambda^\infty$ which is canonically associated to each representation $\pi$ of $C^*(\Lambda)$.  The  projection-valued measure can also be viewed as a  spectral resolution of the infinite path space; the associated decomposition of the representation of the $k$-graph algebras will serve as a base for  analyzing monic and purely atomic representations in Sections \ref{sec:monic-results} and \ref{sec:atomic_repn} respectively.

Section \ref{sec:monic-results} then contains our next main result, Theorem \ref{thm-characterization-monic-repres}: when $\Lambda$ is finite, every monic representation of $C^*(\Lambda)$ arises from a $\Lambda$-projective system on $\Lambda^\infty$, and every such $\Lambda$-projective system gives rise to a monic representation.  A \emph{monic representation} of $C^*(\Lambda)$ is one whose restriction to $C(\Lambda^\infty) \subseteq C^*(\Lambda)$ admits a cyclic vector.  Such representations were first studied in \cite{dutkay-jorgensen-monic} for the Cuntz algebras, and subsequently in \cite{bezuglyi-jorgensen} in the Cuntz-Krieger case.  However, Theorem \ref{thm-characterization-monic-repres} implies  sharper results for the Cuntz and Cuntz-Krieger case than had previously been obtained.  

Having classified the monic representations of finite $k$-graphs in Section \ref{sec:monic-results}, we pause to revisit the examples of Section \ref{sec:examples_a}.  While many of them are monic, Example \ref{exonevthreeed} (see also Example \ref{ex-exonevthreeed-monic}) demonstrates that not all $\Lambda$-semibranching function systems give rise to monic representations.  {Our computations in Section \ref{sec:monic-examples} rely on Theorem \ref{prop-conv-to-8-7}, which gives an alternative characterization of when a $\Lambda$-semibranching function system gives rise to a monic representation.  In addition to streamlining our arguments in this section, we believe Theorem \ref{prop-conv-to-8-7} will be of independent interest.}

Section \ref{sec:univ_repn} relates monic representations to Nelson's universal Hilbert space, which we denote $\H(\Lambda^\infty)$.  In particular, Theorem \ref{prop-universal-name} implies that
every monic representation whose associated $\Lambda$-projective system consists of positive functions is unitarily equivalent to a sub-representation of the so-called ``universal representation'' of $C^*(\Lambda)$ on $\H(\Lambda^\infty)$ which is described in Proposition \ref{proj-valued-measure-universal-representation}.

In the final two sections, we develop analogues for higher-rank graphs of two classes of representations which have been particularly fruitful in the Cuntz algebra setting:  namely, permutative representations, introduced in \cite{BJ-atomic} for Cuntz algebras, and purely atomic representations \cite{dutkay-jorgensen-atomic}.  Permutative representations, in particular, have been linked to continued fractions \cite{kawamura-h} and to fermionic representations \cite{abe-kawamura}. Many of the results which we obtain for purely atomic representations in Section \ref{sec:atomic_repn}, and for permutative representations in Section \ref{sec:permutative_repn}, strongly echo the initial results obtained in the Cuntz algebra setting. Given the fundamental structural differences between Cuntz algebras and $k$-graph algebras (for example, the latter need not be simple, nor is their $K$-theory known in general), the relative ease with which the results for Cuntz algebras extend to the $k$-graph context is another indication of the importance of these classes of representations.

While we were in the process of writing up the results presented below, 
D. Gon\c{c}alves, H.  Li,  and D. Royer posted a manuscript \cite{GLR} on the arXiv in which they introduce a definition of $\Lambda$-semibranching function systems for more general finitely-aligned higher-rank graphs.    While there is some overlap between their work and ours, especially concerning the case of  $k$-graphs with one vertex, our systematic analysis of monic, purely atomic, and permutative  representations of $k$-graph $C^*$-algebras is completely new.  We also hope that our focus in this monograph on concrete examples of representations of finite higher-rank graphs will invite more researchers to join us in studying these fascinating objects.

\subsection{Acknowledgments}   	 
The authors would like to thank Daniel Gon\c calves for his contribution to the proof of Lemma \ref{lem:measure} below, and Janos Englander and Alex Kumjian  for helpful discussions.  E.G.~was partially supported by   the Deutsches Forschungsgemeinshaft via the SFB 878 ``Groups, Geometry, and Actions'' of the Westf\"alische-Wilhelms-Universit\"at M\"unster. 	 
S.K.~ was supported by Basic Science Research Program through the National Research Foundation of Korea (NRF) funded by the Ministry of Education (\#2017R1D2A1B03034697).
J.P.~was  partially supported by a grant from the Simons Foundation (\#316981).
P.J. thanks his colleagues in the Math Department at the University of Colorado, for making a week-long visit there possible, for support, and for kind hospitality.  

  The first three named co-authors (C.F., E.G., and P.J.) thank the Fields Institute, the University of Toronto, and George Elliott in particular, for hosting a Satellite Conference at Fields during the first week of August 2017 where the three of us benefited from discussions, related to the completion of our monograph, with the participants, especially with G. Elliott, Thierry Giordano, Adam Dor On,  Matthew Kennedy, Gilles de Castro, and Jianchao Wu. Moreover, C.F. and E.G. also thank the AMS for funding of travel to the Mathematical Congress of the Americas in Montreal, July  2017, where they worked toward the completion of this monograph.
\section{Foundational material}
\label{sec-fund-mater}

\subsection{Higher-rank graphs}

 We will now describe in detail higher-rank graphs and their $C^*$-algebras. 
 For this purpose we begin by recalling a few central points from the foundational work of Kumjian and Pask \cite{KP} on higher-rank graphs. 
 
Let $\N=\{0,1,2,\dots\}$ denote the monoid of natural numbers under addition, and let $k\in \N$ with $k\ge 1$. We write $e_1,\dots e_k$ for the standard basis vectors of $\N^k$, where $e_i$ is the vector of $\N^k$ with $1$ in the $i$-th position and $0$ everywhere else.

\begin{defn} \cite[Definition 1.1]{KP}
\label{def-higher-rank-gr}
A \emph{higher-rank graph} or \emph{$k$-graph} is a countable small category\footnote{Recall that a small category is one  in which  the collection of arrows is a  set.} $\Lambda$ 
 with a degree functor $d:\Lambda\to \N^k$ satisfying the \emph{factorization property}: for any morphism $\lambda\in\Lambda$ and any $m, n \in \N^k$ such that  $d(\lambda)=m+n \in \N^k$,  there exist unique morphisms $\mu,\nu\in\Lambda$ such that $\lambda=\mu\nu$ and $d(\mu)=m$, $d(\nu)=n$. 
\end{defn}

Readers who are new to the study of higher-rank graphs may wish to review the examples presented in Sections \ref{sec:examples_a} and \ref{sec:examples_gen_measure} below, before reading further.  The diagrams included with these examples, and the factorization rules described there, will give the reader more geometric, analytic, and combinatorial insight into the factorization property and its consequences.


When discussing $k$-graphs, we use the arrows-only picture of category theory; thus, objects in $\Lambda$ are identified with identity morphisms, and the notation $ \lambda \in \Lambda $ means $\lambda$ is a morphism in $\Lambda$.

We often regard $k$-graphs as a generalization of directed graphs, so we call morphisms $\lambda\in\Lambda$ \emph{paths} in $\Lambda$, and the objects (identity morphisms) are often called \emph{vertices}. For $n\in\N^k$ and vertices $v,w$ of $\Lambda$, we write
\begin{equation}
\label{eq:Lambda-n}
\Lambda^n:=\{\lambda\in\Lambda\,:\, d(\lambda)=n\}\
\end{equation}
With this notation, note that $\Lambda^0$ is the set of objects (vertices) of $\Lambda$. Occasionally, we call elements of $\Lambda^{e_i}$ (for any $i$) \emph{edges}.

We write $r,s:\Lambda\to \Lambda^0$ for the range and source maps in $\Lambda$ respectively, and 
\[v\Lambda w:=\{\lambda\in\Lambda\,:\, r(\lambda)=v,\;s(\lambda)=w\}.\]
  Combining this notational convention with that of Equation \eqref{eq:Lambda-n} results in abbreviations such as 
\[ v\Lambda^n:= \{ \lambda \in \Lambda: r(\lambda) = v, \ d(\lambda) = n\}\]
which we will use throughout the paper.

For $m,n\in\N^k$, we write $m\vee n$ for the coordinatewise maximum of $m$ and $n$. Given  $\lambda,\eta\in \Lambda$, we write
\begin{equation}\label{eq:lambda_min}
\Lambda^{\operatorname{min}}(\lambda,\eta):=\{(\alpha,\beta)\in\Lambda\times\Lambda\,:\, \lambda\alpha=\eta\beta,\; d(\lambda\alpha)=d(\lambda)\vee d(\eta)\}.
\end{equation}
Then we denote the set of \emph{minimal common extensions} of $\lambda,\eta \in \Lambda$ by
\[
\operatorname{MCE}(\lambda,\eta):=\{\lambda\alpha\,:\, (\alpha,\beta)\in \Lambda^{\operatorname{min}}(\lambda,\eta)\}
=\{\eta\beta\,:\, (\alpha,\beta)\in \Lambda^{\operatorname{min}}(\lambda,\eta)\}.
\]
If $k=1$, then $\Lambda^{\operatorname{min}}(\lambda, \eta)$ will have at most one element; this need not be true in a $k$-graph if $k > 1$.

We say that a $k$-graph $\Lambda$ is \emph{finite} if $\Lambda^n$ is a finite set for all $n\in\N^k$ and say that $\Lambda$  \emph{has no sources} or \emph{is source-free} if $v\Lambda^n\ne \emptyset$ for all $v\in\Lambda^0$ and $n\in\N^k$. It is well known that this is equivalent to the condition that $v\Lambda^{e_i}\ne \emptyset$ for all $v\in \Lambda$ and all basis vectors $e_i$ of $\N^k$. 
We say that $\Lambda$ is \emph{row-finite} if $|v\Lambda^n| < \infty$ for all $v\in\Lambda^0$ and $n\in\N^k$, and we say that $\Lambda$ is \emph{finitely aligned} if $\Lambda^{\operatorname{min}}(\lambda,\eta)$ is finite (possibly empty) for all $\lambda,\eta\in\Lambda$. We are mostly interested in finite (or row-finite) $k$-graphs; in fact all of our examples are finite $k$-graphs. For those interested in finitely-aligned $k$-graphs, we suggest consulting \cite{RSY2}; semibranching function systems for these $k$-graphs have recently been introduced in \cite{GLR}.  

We often visualize a $k$-graph as a (quotient of a) $k$-colored directed graph via the equivalence relation induced by the factorization rules.
To be precise, for each $1 \leq i \leq k$, we can define the \emph{$i$th vertex matrix} $A_i \in M_{\Lambda^0}(\N)$ by 
\begin{equation} A_i(v,w) = |v\Lambda^{e_i} w|.\label{eq:adj-mx}
\end{equation}
Observe that the factorization rules imply that $A_iA_j=A_jA_i$ for $1\le i,j\le k$.  Indeed, given a pair of composable edges $f_1 \in v\Lambda^{e_i} z, f_2 \in z \Lambda^{e_j} w$, the factorization rule implies that since $d(f_1 f_2) = e_i + e_j = e_j + e_i$, the morphism $f_1 f_2 \in \Lambda$ can also be described uniquely as 
\[ f_1 f_2 = g_2 g_1 \;\;\text{ where } g_2 \in v\Lambda^{e_i} , \ g_1 \in  \Lambda^{e_j} w 
.\] 
Thus, we think of the matrix $A_i$ as the adjacency matrix for a graph of color $i$ on the vertex set $\Lambda^0$.  The factorization rule tells us that we must ``collapse'' commuting bi-colored squares in this graph into a single morphism in $\Lambda$.

{We now describe two fundamental examples of higher-rank graphs which were first mentioned in the foundational paper \cite{KP}.}
\begin{example}
\begin{itemize}
\item[(a)] For any directed graph $E$, let $\Lambda_E$ be the category of its finite paths. Then $\Lambda_E$ is a 1-graph with the degree functor $d:\Lambda_E\to \N$ which takes a finite path $\eta$ to its length $|\eta|$ (the number of edges making up $\eta$).
\item[(b)] For $k\ge 1$, let $\Omega_k$ be the small category with
\[
\operatorname{Obj}(\Omega_k)=\N^k,\quad \text{and}\quad \operatorname{Mor}(\Omega_k)=\{(p,q)\in \N^k\times \N^k\,:\, p\le q\}.
\]
Again, we can also view elements of $\text{Obj}(\Omega_k)$ as identity morphisms, via the map $\text{Obj}(\Omega_k) \ni p \mapsto (p, p) \in \text{Mor}(\Omega_k)$.
The range and source maps $r,s:\operatorname{Mor}(\Omega_k)\to \operatorname{Obj}(\Omega_k)$ are given by $r(p,q)=p$ and $s(p,q)=q$. If we define $d:\Omega_k\to \N^k$ by $d(p,q)=q-p$, then one can check that $\Omega_k$ is a $k$-graph with degree functor $d$.
\end{itemize}

\end{example}

\begin{defn}[\cite{KP} Definitions 2.1]
\label{def:infinite-path}
Let $\Lambda$ be a $k$-graph. An \emph{infinite path} in $\Lambda$ is a $k$-graph morphism (degree-preserving functor) $x:\Omega_k\to \Lambda$, and we write $\Lambda^\infty$ for the set of infinite paths in $\Lambda$. 
Since $\Omega_k$ has a terminal object (namely $0 \in\N^k$) but no initial object, we think of our infinite paths as having a range $r(x) : = x(0)$ but no source.
For each $m\in \N^k$, we have a shift map $\sigma^m:\Lambda^\infty \to \Lambda^\infty$ given by
\begin{equation}\label{eq:shift-map}
\sigma^m(x)(p,q)=x(p+m,q+m)
\end{equation}
for $x\in\Lambda^\infty$ and $(p,q)\in\Omega_k$.

We say that a $k$-graph $\Lambda$ is \emph{aperiodic} if for each $v\in\Lambda^0$, there exits $x\in v\Lambda^\infty$ such that for all $m\ne n\in\N^k$ we have $\sigma^m(x)\ne \sigma^n(x)$.
\end{defn}
\begin{rmk}
The factorization rule implies an important property of infinite paths:  for any $x\in\Lambda^\infty$ and $m\in\N^k$, we have
\[
x=x(0,m)\sigma^m(x).
\]
Thus, every infinite path must contain infinitely many edges of each color.  Moreover, if we take $m = (n, n, \ldots, n) \in \N^k$ for some $n \geq 1$, the factorization rule tells us that $x(0,m)$ can be written uniquely as a ``rainbow sequence'' of edges:
\[ x(0,m) = f_1^1 f_2^1 \cdots f_k^1 f_1^2 \cdots f_k^2 f_1^3 \cdots  f_k^n,\]
where $d(f_i^j) = e_i$.
So if we have a 2-graph $\Lambda$, which can be visualized as a 2-colored graph (red and blue edges) with factorization rules, each infinite path $x\in\Lambda^\infty$ can be uniquely identified with  an infinite string of alternating blue and red edges (setting blue to be ``color 1'' and red to be ``color 2'').  On the other hand, the factorization rule implies that each composable pair of blue-red edges is equivalent to a unique pair of composable red-blue edges, so $x$ is also equivalent to an infinite string of composable edges which alternate in color in the following pattern: red, blue, red, blue, \ldots.
\label{rmk:rainbow}
\end{rmk}
{We stress that even finite $k$-graphs (our main objects of interest in this paper) may have nontrivial infinite paths; in an infinite path, the same edge may occur multiple times and even infinitely many times.}

It is well-known that the collection of cylinder sets 
\[
Z(\lambda)=\{x\in\Lambda^\infty\,:\, x(0,d(\lambda))=\lambda\},
\]
for $\lambda \in \Lambda$, form a compact open basis for a locally compact Hausdorff topology on $\Lambda^\infty$, under  reasonable hypotheses on $\Lambda$ (in particular, when $\Lambda$ is row-finite: see Section 2 of \cite{KP}). If a $k$-graph $\Lambda$ is  finite, then $\Lambda^\infty$ is compact in this topology.

According to Proposition 8.1 of \cite{aHLRS3}, for many finite higher-rank graphs there is a unique Borel probability measure $M$ on $\Lambda^\infty$ satisfying a certain self-similarity condition. To describe the measure $M$, we need more definitions.

\begin{defn}
We say that a $k$-graph is \emph{strongly connected} if, for all $v,w\in\Lambda^0$, $v\Lambda w\ne \emptyset$.
\end{defn}

If a $k$-graph $\Lambda$ is finite and strongly connected with vertex matrices $A_1,\dots A_k\in M_{\Lambda^0}(\N)$, then Proposition~3.1 of \cite{aHLRS3} implies that there is a unique positive vector $\kappa^\Lambda\in (0,\infty)^{\Lambda^0}$ such that $\sum_{v\in\Lambda^0}\kappa_v^{\Lambda}=1$ and for all $1\le i\le k$,
\[
A_i \kappa^{\Lambda}=\rho_i\, \kappa^{\Lambda},
\]
where $\rho_i$ denotes the spectral radius of $A_i$. The vector $\kappa^{\Lambda}$ is called the \emph{(unimodular) Perron-Frobenius eigenvector} of $\Lambda$. Then the measure $M$ on $\Lambda^\infty$ is given by
\begin{equation}
M(Z(\lambda))=(\rho(\Lambda))^{-d(\lambda)}\kappa^{\Lambda}_{s(\lambda)}\quad\text{for}\;\; \lambda\in\Lambda,
\label{eq:M}
\end{equation}
where $\rho(\Lambda)=(\rho_1,\dots \rho_k)$ and $(\rho(\Lambda))^n=\rho_1^{n_1}\dots \rho_k^{n_k}$ for $n=(n_1,\dots n_k)\in \N^k$. We call the measure $M$ the \emph{Perron-Frobenius measure} on $\Lambda^\infty$.  Proposition 8.1 of \cite{aHLRS3} establishes that if $\mu$ is a Borel probability measure on $\Lambda^\infty$ such that 
\[
\mu(Z(\lambda))=\rho(\Lambda)^{-d(\lambda)}\mu (Z(s(\lambda)))\quad\text{for all}\;\; \lambda\in\Lambda,
\]
then $\mu = M$.

Now we introduce the $C^*$-algebra associated to a $k$-graph $\Lambda$. Here we only consider row-finite $k$-graphs with no sources. For $C^*$-algebras associated to more general $k$-graphs, see for example  \cite{RSY2}.
\begin{defn}\label{def:kgraph-algebra}
Let $\Lambda$ be a row-finite $k$-graph with no sources. A \emph{Cuntz--Krieger $\Lambda$-family} is a collection  $\{t_\lambda:\lambda\in\Lambda\}$ of partial isometries in a $C^*$-algebra satisfying
\begin{itemize}
\item[(CK1)] $\{t_v\,:\, v\in\Lambda^0\}$ is a family of mutually orthogonal projections,
\item[(CK2)] $t_\lambda t_\eta=t_{\lambda\eta}$ if $s(\lambda)=r(\eta)$,
\item[(CK3)] $t^*_\lambda t_\lambda=t_{s(\lambda)}$ for all $\lambda\in\Lambda$,
\item[(CK4)] for all $v\in\Lambda$ and $n\in\N^k$, we have
\[
t_v=\sum_{\lambda\in v\Lambda^n} t_\lambda t^*_\lambda.
\]
The Cuntz--Krieger $C^*$-algebra $C^*(\Lambda)$ associated to $\Lambda$ is the universal $C^*$-algebra generated by a Cuntz--Krieger $\Lambda$-family.

\end{itemize}
\end{defn}
One can show that 
\[
C^*(\Lambda)=\overline{\operatorname{span}}\{t_\alpha t^*_\beta\,:\, \alpha,\beta\in\Lambda,\; s(\alpha)=s(\beta)\}.
\]
Also, (CK4) implies that for all $\lambda, \eta\in \Lambda$, we have
\begin{equation}\label{eq:CK4-2}
t_\lambda^* t_\eta=\sum_{(\alpha,\beta)\in \Lambda^{\operatorname{min}}(\lambda,\eta)} t_\alpha t^*_\beta. 
\end{equation}
The universal property implies that the $C^*$-algebra $C^*(\Lambda)$ carries a strongly continuous action  $\gamma$ of the $k$-torus $\T^k$, called the \emph{gauge action}, which is given by
\[
\gamma_z(t_\lambda)=z^{d(\lambda)}t_\lambda,
\]
where $z^n=\prod_{i=1}^k z_i^{n_i}$ for $z=(z_1,\dots, z_k)\in \T^k$ and $n = (n_1, \ldots, n_k) \in\N^k$. Note that we only  discuss the gauge action in Section \ref{sec:faithful-rep}.

\subsection{$\Lambda$-semibranching function systems and their representations}

In \cite{FGKP}, separable representations of $C^*(\Lambda)$ were constructed by using $\Lambda$-semibranching function systems on measure spaces. A $\Lambda$-semibranching function system is a generalization of the semibranching function systems studied by Marcolli and Paolucci in \cite{MP}.  These objects play an important role in our analysis of other classes of representations of $C^*(\Lambda)$, such as the monic representations we introduce in Section \ref{sec:monic-results}, 
{and the  permutative representations of Section \ref{sec:permutative_repn}}. Here we review basic definitions and introduce the standard example of  a $\Lambda$-semibranching function system on $(\Lambda^\infty,M)$ and its associated representation: see Example~\ref{example:SBFS-M}.

\begin{defn}
\label{def-1-brach-system}\cite[Definition~2.1]{MP}\label{defn:sbfs}
Let $(X,\mu)$ be a measure space. Suppose that, for each $1\le i\le N$, we have a measurable map $\sigma_i:D_i\to X$, for some measurable subsets $D_i\subset X$. The family $\{\sigma_i\}_{i=1}^N$ is a \emph{semibranching function system} if the following holds:
\begin{itemize}
\item[(a)] Setting $R_i = \sigma_i(D_i),$ we have 
\[
\mu(X\setminus \cup_i R_i)=0,\quad\quad\mu(R_i\cap R_j)=0\;\;\text{for $i\ne j$}.
\]
\item[(b)] For each $i$, the Radon-Nikodym derivative
\[
\Phi_{\sigma_i}=\frac{d(\mu\circ\sigma_i)}{d\mu}
\]
satisfies $\Phi_{\sigma_i}>0$, $\mu$-almost everywhere on $D_i$.
\end{itemize}
A measurable map $\sigma:X\to X$ is called a \emph{coding map} for the family $\{\sigma_i\}_{i=1}^N$ if $\sigma\circ\sigma_i(x)=x$ for all $x\in D_i$.
\end{defn}

\begin{defn}\cite[Definition~3.2]{FGKP}
\label{def-lambda-SBFS-1}
Let $\Lambda$ be a finite $k$-graph and let $(X, \mu)$ be a measure space.  A \emph{$\Lambda$-semibranching function system} on $(X, \mu)$ is a collection $\{D_\lambda\}_{\lambda \in \Lambda}$ of measurable subsets of $X$, together with a family of prefixing maps $\{\tau_\lambda: D_\lambda \to X\}_{\lambda \in \Lambda}$, and a family of coding maps $\{\tau^m: X \to X\}_{m \in \N^k}$, such that
\begin{itemize}
\item[(a)] For each $m \in \N^k$, the family $\{\tau_\lambda: d(\lambda) = m\}$ is a semibranching function system, with coding map $\tau^m$.
\item[(b)] If $ v \in \Lambda^0$, then  $\tau_v = id$,  and $\mu(D_v) > 0$.
\item[(c)] Let $R_\lambda = \tau_\lambda( D_\lambda)$. For each $\lambda \in \Lambda, \nu \in s(\lambda)\Lambda$, we have $R_\nu \subseteq D_\lambda$ (up to a set of measure 0), and
\[\tau_{\lambda} \tau_\nu = \tau_{\lambda \nu}\text{ a.e.}\]
 (Note that this implies that up to a set of measure 0, $D_{\lambda \nu} = D_\nu$ whenever $s(\lambda) = r(\nu)$).
\item[(d)] The coding maps satisfy $\tau^m \circ \tau^n = \tau^{m+n}$ for any $m, n \in \N^k$.  (Note that this implies that the coding maps pairwise commute.)
\end{itemize}
\end{defn}

\begin{rmk}
\label{rmk:abs-cts-inverse}
We pause to note that  condition (c) of Definition \ref{def-lambda-SBFS-1} above implies that $D_\lambda=D_{s(\lambda)}$ and $R_\lambda\subset R_{r(\lambda)}$ for $\lambda\in\Lambda$. Also, when $\Lambda$ is a finite 1-graph, the definition of a $\Lambda$-semibranching function system is not equivalent to Definition \ref{def-1-brach-system}. In particular, Definition~\ref{def-lambda-SBFS-1}(b) implies that the domain sets $\{D_v:v\in\Lambda^0\}$ must satisfy $\mu(D_v\cap D_w)=\mu(R_v\cap R_w)=0$ for $v\ne w\in\Lambda^0$, but  Definition~\ref{def-1-brach-system} does not require that the domain sets $D_i$ be mutually disjoint $\mu$-a.e. In fact, Definition~\ref{def-lambda-SBFS-1} implies what is called condition (C-K) in Section~2.4 of \cite{bezuglyi-jorgensen}: up to a measure zero set, 
\begin{equation}
\label{eq-partition}
D_v=\cup_{\lambda\in v\Lambda^m} R_\lambda
\end{equation}
for all $v\in\Lambda^0$ and $m\in\N,$ since $R_v=\tau_v(D_v)=id(D_v)=D_v.$ Also notice that in the above decomposition the  intersections $R_\lambda \cap R_{\lambda'},$ $R_\lambda \not= {\lambda'},$ have measure zero. This condition  is crucial to making sense of the representation of $C^*(\Lambda)$ associated to the $\Lambda$-semibranching function system (see Theorem \ref{thm:separable-repn} below). As established in Theorem~2.22 of \cite{bezuglyi-jorgensen}, in order to obtain a representation of a 1-graph algebra $C^*(\Lambda)$ from a semibranching function system, one must also assume that the semibranching function system satisfies condition (C-K).

Finally, we also observe that $(\tau^n)^{-1} (E) = \bigcup_{\lambda \in \Lambda^n} \tau_\lambda(E)$ for any measurable $E\subseteq X$.  Therefore, 
\[ \mu \circ (\tau^n)^{-1} << \mu\]
in any $\Lambda$-semibranching function system.
\end{rmk}

As established in \cite{FGKP}, any $\Lambda$-semibranching function system gives rise to a representation of $C^*(\Lambda)$ { via \lq prefixing' and \lq chopping off' operators that satisfy the Cuntz-Krieger relations}.  {Intuitively, a $\Lambda$-semibranching function system is a way of encoding the Cuntz-Krieger relations at the measure-space level: the prefixing map $\tau_\lambda$ corresponds to the partial isometry $s_\lambda \in C^*(\Lambda)$.  We give a precise formula for the representation in Theorem \ref{thm:separable-repn} below.} For brevity, we will often refer to  representations arising from $\Lambda$-semibranching function systems as \emph{$\Lambda$-semibranching representations.}  Note that  a $\Lambda$-semibranching representation will be separable whenever $L^2(X, \mu)$ is separable; this will be the case for all but one of the representations we consider in this paper.

\begin{thm}\cite[Theorem~3.5]{FGKP}\label{thm:separable-repn}
Let $\Lambda$ be a finite $k$-graph with no sources and suppose that we have a $\Lambda$-semibranching function system on a certain measure space $(X,\mu)$ with prefixing maps $\{\tau_\lambda\}$ and coding maps $\{\tau^m:m\in \N^k\}$. For each $\lambda\in\Lambda$, define an operator $S_\lambda$  on $L^2(X,\mu)$ by
\[
S_\lambda\xi(x)=\chi_{R_\lambda}(x)(\Phi_{\tau_\lambda}(\tau^{d(\lambda)}(x)))^{-1/2} \xi(\tau^{d(\lambda)}(x)).
\]
Then the operators $\{S_\lambda:\lambda\in\Lambda\}$ generate a representation $\pi$ of $C^*(\Lambda)$, and $\pi$ is separable.
\label{thm:SBFS-repn}
\end{thm}

\begin{example}\label{example:SBFS-M}
Here we describe the standard $\Lambda$-semibranching function system on the measure space $(\Lambda^\infty, M)$ for a finite strongly connected $k$-graph $\Lambda$, using the measure $M$ of Equation \eqref{eq:M}. The prefixing maps $\{\sigma_\lambda:Z(s(\lambda))\to Z(\lambda)\}_{\lambda\in\Lambda}$ are given by
\[
\sigma_\lambda(x)=\lambda x,
\]
where $\lambda x \in \Lambda^\infty$ is defined by $\lambda x (0, m) = \lambda(0,m)$ if $d(\lambda ) \geq m$, and $\lambda x(0, m) = \lambda x( 0, m - d(\lambda))$ if $m \geq d(\lambda)$,
and the coding maps $\{\sigma^m:\Lambda^\infty\to \Lambda^\infty\}_{m\in \N^k}$ are given as in Definition \ref{def:infinite-path}:
\[
\sigma^m(x)(p,q)=x(m+p, m+q).
\]
Thus, for $\lambda\in\Lambda$, we let $D_\lambda=Z(s(\lambda))$ and $R_\lambda=\sigma_\lambda(D_\lambda)=Z(\lambda)$.  Proposition 3.4 of \cite{FGKP} establishes that $\{\sigma_\lambda:D_\lambda\to R_\lambda\}$ and $\{\sigma^m\}_{m\in \N^k}$ forms a $\Lambda$-semibranching function system on $(\Lambda^\infty,M)$. In particular, one can show that the Radon-Nikodym derivatives of $\sigma_\lambda$ are positive $M$-a.e. on $Z(s(\lambda))$ and they  are given by
\[
\Phi_{\sigma_\lambda}(x)=\rho(\Lambda)^{-d(\lambda)}.
\]

\begin{rmk}
\label{rmk:defn-standard-repn}
As seen in the above Theorem~\ref{thm:separable-repn}, there is a separable representation $\pi =: \pi_S$ of $C^*(\Lambda)$ associated to the $\Lambda$-semibranching function system on $(\Lambda^\infty, M)$ of Example \ref{example:SBFS-M}. In this case, $S_\lambda = \pi_S(t_\lambda)$ acts on characteristic functions of cylinder sets by
\[\begin{split}
S_\lambda \chi_{Z(\eta)}(x) &=\chi_{Z(\lambda)}(x)\rho(\Lambda)^{d(\lambda)/2}\chi_{Z(\eta)}(\sigma^{d(\lambda)}(x))\\
&=\rho(\Lambda)^{d(\lambda)/2}\chi_{Z(\lambda\eta)}(x).
\end{split}\]
Then the adjoint $S_\lambda^*$ is given by
\[
S_\lambda^*\chi_{Z(\eta)}(x)=\rho(\Lambda)^{-d(\lambda)/2}\sum_{(\alpha,\beta)\in \Lambda^{\operatorname{min}}(\lambda,\eta)}\chi_{Z(\alpha)}(x).
\]
See the detailed calculation in Section~5 of \cite{FGJKP2}. We call the separable representation $\pi_S$ associated to this $\Lambda$-semibranching function system on $(\Lambda^\infty, M)$ the \emph{standard $\Lambda$-semibranching representation} of $C^*(\Lambda)$.
\end{rmk}
\end{example}
The following Lemmas are well-known, and will be the technical tool we will use  in many of the Radon-Nikodym  derivative calculations presented in Section \ref{sec:examples_gen_measure}.	
In particular, we will apply these examples to the case where $X = \Lambda^\infty$ and 
$ \mathcal F_n$ is the $\sigma$-algebra generated by the cylinder sets $Z(\lambda)$ with $d(\lambda) \leq (n, \ldots, n)$.

\begin{lemma}[Kolmogorov Extension Theorem, \cite{kolmogorov, Tum}]
\label{lem-Kolm}
Let $(X, \mathcal{F}_n, \mu_n)_{n\in\N}$ be a sequence of probability measures $(\mu_n)_{n\in \N}$ on the same space $X$, each associated with a $\sigma$-algebra $\mathcal F_n$
; further assume that  $(X, \mathcal{F}_n, \mu_n)_{n\in\N}$  form a projective system, i.e., an inverse limit. Suppose that Kolmogorov's consistency condition holds:
\[ \mu_{n+1}|_{\mathcal F_n} = \mu_n.\]
Then there is a unique extension $\mu$ of the measures $(\mu_n)_{n\in \N}$ to the $\sigma$-algebra $\bigvee_{n\in \N} \mathcal F_n$ generated by $\bigcup_{n\in \N} \mathcal F_n$.
\label{lem:kolmogorov}
\end{lemma}
 In fact, $\mu$ is the unique probability measure which has the given sequence of measures $(\mu_n)_{n\in \N}$ as its marginal distributions with respect to the prescribed filtration $\bigcup_{n\in \N} \mathcal F_n$.

\begin{lemma} (cf.~\cite{bezuglyi-jorgensen}, \cite{SGI} Section 10.2)
\label{lem-RN-der-comp-limit}
	\label{lemma-limit-RN} Let $(X, \mathcal{F}_n, \mu_n)_{n\in \N}$ and $(X, \mathcal F_n, \nu_n)_{n\in \N}$
	be two  sequences of measures on the same space $X$ and same $\sigma$-algebras $(X,{\mathcal F}_n)$.  Suppose that both sequences  form a projective system 
	and satisfy Kolmogorov's consistency condition, so that by Lemma \ref{lem:kolmogorov}, we have induced measures $\mu, \nu$ on 
the $\sigma$-algebra  $\mathcal{F} := \bigvee_n \mathcal{F}_n$  generated by $\cup_n \mathcal{F}_n$.
  
 Suppose moreover that
	\begin{itemize}
		\item $\nu_n<< \mu_n$ for all $n \in \N$;
		\item The Radon-Nikodym derivative $R_n:= d\nu_n/d\mu_n$ exists and is finite for all $n \in \N$;
		\item $R:= \lim_{n\to \infty} R_n$ exists and is finite.
			\end{itemize}
Then 	$\nu << \mu$ if and only if $R >0$, and $R = d\nu/d\mu$. 
	\end{lemma}

Before we finish the section, we 
observe that  Carath\'eodory's theorem implies that finite additivity on ``square'' cylinder sets -- that is, cylinder sets $Z(\lambda)$ with $d(\lambda) = (n, \ldots, n)$ for some $n \in \N$ -- is enough to  obtain a measure on $\Lambda^\infty$. The factorization
rule for $k$-graphs implies that
this hypothesis is equivalent to the consistency hypothesis of Lemma \ref{lem:kolmogorov}, so the following Lemma could also be viewed as a corollary of the Kolmogorov Extension Theorem (Lemma \ref{lem:kolmogorov}).

 \begin{lemma}\label{lem:measure}
 Suppose that $\Lambda$ is a row-finite $k$-graph with no sources.
 If a positive real-valued function $\mu$ defined on the cylinder sets of $\Lambda^\infty$ is finitely additive on square cylinder sets, then $\mu$ extends uniquely to a Borel measure on $\Lambda^\infty$.
  \end{lemma}
 \begin{proof}
 We first observe that for any row-finite $k$-graph, the collection of finite disjoint unions of cylinder sets is a ring: it is closed under unions and relative complements.
To see that $\mu$ extends to a measure on the Borel $\sigma$-algebra $\mathcal{B}_o(\Lambda^\infty)$ of $\Lambda^\infty$, Carath\'eodory's theorem therefore tells us that it suffices to check that $\mu$ is countably additive on disjoint unions of cylinder sets.  In order to define $\mu$ unambiguously on countable disjoint unions, we must check that
 \[ 
 \bigsqcup_{i \in \N} Z(\lambda_i) = \bigsqcup_{j\in \N} Z(\eta_j) \Rightarrow \sum_{i\in \N} \mu(Z(\lambda_i)) = \sum_{j\in \N}\mu(Z(\eta_j)).
 \]
 To this end, note that if $\bigsqcup_{i \in \N} Z(\lambda_i) = \bigsqcup_{j\in \N} Z(\eta_j) $, then for each fixed $i \in \N$ we have 
 \[
 Z(\lambda_i) = \bigsqcup_{j\in \N} Z(\lambda_i)\cap Z(\eta_j),
 \]
 and that $
 Z(\lambda_i) \cap Z(\eta_j) = \bigsqcup_{\zeta \in \text{MCE}(\lambda_i, \eta_j)} Z(\zeta).$ If $n = n(i,j) \in \N$ is such that $(n, \ldots, n) \geq d(\zeta) = d(\lambda_i) \vee d(\eta_j)$, write $I_{\zeta, n} := s(\zeta) \Lambda^{(n, \ldots, n) - d(\zeta)}$; then,
 \[ Z(\lambda_i) \cap Z(\eta_j) = \bigsqcup_{\zeta \in \text{MCE}(\lambda_i, \eta_j)} \bigsqcup_{\nu \in I_{\zeta, n}} Z(\zeta \nu)\]
 is a disjoint union of square cylinder sets.  In fact, for each choice of $i$ and $j$, this union is finite; the row-finiteness of $\Lambda$ implies that $v \Lambda^m$ is finite for all $m\in \N^k$, and that $\Lambda^{\operatorname{min}}(\lambda_i, \eta_j)$ (equivalently,  $\text{MCE}(\lambda_i, \eta_j)$) is finite for all $i,j$.  
 
 Moreover,  each cylinder set $Z(\lambda)$ is compact and open.  Since $\bigsqcup_{j\in \N} Z(\eta_j)$ is a cover for $Z(\lambda_i)$, it follows that there are only finitely many indices $j$ such that $Z(\lambda_i) \cap Z(\eta_j) \not= \emptyset.$ Choose $n= n(i)$ such that $(n, \ldots, n) \geq d(\lambda_i) \vee d(\eta_j)$ for all such $j$; then, the finite additivity of $\mu$ on square cylinder sets implies that
 \[ \mu(Z(\lambda_i)) = \sum_{j \in \N} \sum_{\zeta \in \text{MCE}(\lambda_i, \eta_j)} \sum_{\nu \in I_{\zeta, n}} \mu(Z(\zeta \nu)).\]
 By symmetry, 
 \[ \mu(Z(\eta_j)) = \sum_{i \in \N} \sum_{\zeta \in \text{MCE}(\lambda_i, \eta_j)} \sum_{\nu \in I_{\zeta, m}} \mu(Z(\zeta \nu)),\]
 where $m = m(j)$ is such that $(m, \ldots, m) \geq d(\lambda_i) \vee d(\eta_j)$ for all $i$ such that $Z(\lambda_i) \cap Z(\eta_j) \not= \emptyset$.
 Define $N(i,j) = n(i) \vee m(j)$; then 
 \[ \sum_{i \in \N} \mu(Z(\lambda_i)) = \sum_{i,j \in \N} \sum_{\zeta \in \text{MCE}(\lambda_i, \eta_j)} \sum_{\nu \in I_{\zeta, N(i,j)}} \mu(Z(\zeta \nu)) = \sum_{j \in \N} \mu(Z(\eta_j)),\]
  as desired.
The proof concludes by applying 
the Carath\'eodory extension theorem
 to extend $\mu$ uniquely to give a measure (also denoted $\mu$) on the Borel $\sigma$-algebra of $\Lambda^\infty$.
 \end{proof}

\section{A separable faithful  representation of $C^*(\Lambda)$}
\label{sec:faithful-rep}

In this section, we study two types of separable representations of $C^*(\Lambda)$, which both arise from $\Lambda$-semibranching function systems.  We first revisit the representation $\pi_S$ associated to the standard $\Lambda$-semibranching function system (see Remark \ref{rmk:defn-standard-repn} above); 
Theorem 3.6 of \cite{FGKP} asserts that this representation is faithful iff $\Lambda$ is aperiodic.   The proof given in \cite{FGKP} was flawed, so we offer a corrected proof here.  We next construct a separable representation for $C^*(\Lambda)$ (see Theorem \ref{prop:sc-faithful} and Theorem \ref{pr:sep-faith} below)
which is faithful even when $\Lambda$ is not necessarily aperiodic.  The underlying Hilbert space $\H_x$ of this representation is defined via an inductive limit, but we show in Proposition \ref{prop:sc-faithful-SBFS} that $\H_x \cong \ell^2(X)$ for a discrete measure space $X$ (with counting measure).  This perspective enables us to realize the representations of Theorem \ref{prop:sc-faithful} and Theorem \ref{pr:sep-faith} as  $\Lambda$-semibranching representations (in fact as permutative representations; see Section \ref{sec:permutative_repn} below).  Incidentally, the same arguments used in Proposition \ref{prop:sc-faithful-SBFS} also enable us to show, in Proposition \ref{prop:sc-faithful-SBFS-Lambda} that the standard representation of $C^*(\Lambda)$ on $\ell^2(\Lambda^\infty)$ is a $\Lambda$-semibranching representation, although not a separable one.

\subsection{The separable faithful  representation $\pi_S$ of \cite{FGKP}, {revisited}}
\label{sec-aperiod-faith-repres}
Recall that a higher-rank graph $\Lambda$ is \emph{aperiodic} if for each vertex  $v \in \Lambda^0$, there exists $x \in Z(v)$ such that for all $m \not= n \in \N^k$, we have $\sigma^n(x) \not= \sigma^m(x)$.

\begin{defn}
 If $\Lambda$ is not aperiodic, we write 
\[\text{Per}\,(v) = \{ m-n: \forall \ x \in Z(v), \ \sigma^m(x) = \sigma^n(x)\} \leq \Z^k.\]
If $\Lambda$ is strongly connected, then $\text{Per}\,(v) = \text{Per}\,(w)$ for all $v, w\in \Lambda^0$; we write $\text{Per}\, \Lambda$ for this group.
\end{defn}

The following Proposition is Theorem 3.6 of \cite{FGKP}.
\begin{prop}
\label{prop:pi-S-faithful}
Let $\Lambda$ be a finite, strongly connected $k$-graph.  The standard $\Lambda$-semibranching representation $\pi_S$ of $C^*(\Lambda)$ on $L^2(\Lambda^\infty, M)$, described in Remark \ref{rmk:defn-standard-repn} above, and
Proposition 3.4 and Theorem 3.5 of \cite{FGKP}, is faithful if and only if $\Lambda$ is aperiodic.
\end{prop}

If $\Lambda$ is aperiodic, Theorem 11.1 of \cite{aHLRS3} shows that $C^*(\Lambda)$ is simple, and so, by Theorem 2.1 of \cite{bprz}  (the gauge-invariant uniqueness theorem), the proof of the Proposition reduces to checking that $\pi_S(t_v) \not= 0$ for any vertex $v \in \Lambda^0$.
When $\Lambda$ is not aperiodic, the proof consists of finding a nonzero element $a \in C^*(\Lambda)$ such that $\pi_S(a) = 0$.  The element $a$ chosen in \cite{FGKP} does not, in fact, have $\pi_S(a) = 0$; we rectify that here.

Recall that, for a $k$-graph $\Lambda$ with vertex matrices $A_1, \ldots, A_k$, we define 
\[ \rho(\Lambda) := (\rho(A_1), \rho(A_2), \ldots, \rho(A_k) ) \in \R^k_+,\]
where $\rho(A_i)$ denotes the spectral radius of the matrix $A_i$.  For any $m = (m_1, m_2, \ldots, m_k) \in \Z^k$, we write $\rho(\Lambda)^m := \prod_{i=1}^k \rho(A_i)^{m_i}$.

\begin{prop}
Suppose that $\Lambda$ is a finite, strongly connected $k$-graph and that $0 \not= m-n \in \text{Per } \Lambda$.  Fix $\mu \in \Lambda^m$ and let $\nu \in \Lambda^n$ be such that 
{$\mu z = \nu z \ \forall \ z\in s(\mu)\Lambda^\infty$}.  Then 
\[ b := t_\mu t_\mu^* - \rho(\Lambda)^{(m-n)/2} t_\nu t_\mu^*\]
is a nonzero element of $C^*(\Lambda)$.  Moreover, $\pi_S(b) = 0$.
\end{prop}
\begin{proof}
To see that $b$ is nonzero, we consider the standard representation $\pi_T$ of $C^*(\Lambda)$ on $\ell^2(\Lambda^\infty)$ given in Equation (3.1) on page 342 of  \cite{robertson-sims}.  If $x \in \Lambda^\infty$, let $\xi_x$ denote the associated basis vector of $\ell^2(\Lambda^\infty)$; then 
\begin{equation}
\label{eq:std-rep-on-ell-2}
\pi_T(t_\lambda) x =   \begin{cases} \xi_{\lambda x}, & r(x) = s(\lambda) \\ 
0, & \text{ else.} \end{cases}\end{equation}

Choose an element $x \in \Lambda^\infty$ so that $x = \nu z = \mu z$.  Then, 
\[ \pi_T(b)(\xi_x) = (1-\rho(\Lambda)^{(m-n)/2} )\xi_x,\]
which is nonzero as long as $\rho(\Lambda)^{(m-n)/2} \not= 1.$  

If $\rho(\Lambda)^{(m-n)/2} = 1$, then choose $\omega \in \T^k$ such that $\omega^{n-m} = -1$.  Let $\gamma$ denote the gauge action on $C^*(\Lambda): $ for $z \in \T^k$, $\gamma_z(t_\lambda) = z^{d(\lambda)} t_\lambda$.
Observe that 
\[ b + \gamma_\omega(b) = (t_\mu t_\mu^* - t_\nu t_\mu^*) + t_\mu t_\mu^* - \omega^{n-m} t_\nu t_\mu^* = 2t_\mu t_\mu^*.\]
Now, for any $x \in Z(\mu) \subseteq \Lambda^\infty$,
\[ \pi_T(b + \gamma_\omega(b) ) \xi_x = 2 \xi_x,\]
so $b + \gamma_\omega(b) \not= 0$ and consequently $b \not= 0$.    Thus, regardless of the value of $\rho(\Lambda)^{(m-n)/2}$, we see that $b \not= 0$.

Now, we show that if $\Lambda$ is a strongly connected $k$-graph, $\pi_S(b) = 0$.  Fix $\xi \in L^2(\Lambda^\infty, M)$; then 
\begin{align*}
\pi_S(b)\xi(x) &= S_\mu S_\mu^*(\xi)(x) - \rho(\Lambda)^{(m-n)/2} S_\nu S_\mu^*(\xi)(x) \\
&= \chi_{Z(\mu)}(x) \rho(\Lambda)^{m/2} S_\mu^*\xi (\sigma^m(x)) - \chi_{Z(\nu)}(x)\rho(\Lambda)^{m/2} S_\mu^*\xi(\sigma^n(x)).
\end{align*}
Note that if $x \in Z(\mu)$ then $x = \mu z$ for some infinite path $z$, and hence $x = \nu z$ as well by our choice of $\mu, \nu$.  In this case, 
\[\pi_S(b) \xi(\mu z) = \rho(\Lambda)^{m/2} \left( S_\mu^*\xi(z)- S_\mu^*\xi(z) \right)  = 0.\]
On the other hand, if $x \not\in Z(\mu)$ then $x \not\in Z(\nu)$ also, so $\pi_S(b)\xi(x) = 0$ as well.  Thus, $\pi_S(b)\xi = 0$ for all $\xi \in L^2(\Lambda^\infty, M)$.
\end{proof}

Proposition \ref{prop:pi-S-faithful} led us to look for faithful separable representations of periodic $k$-graphs, and hence to the results in the following section.

\subsection{A new  faithful separable representation}
\label{sec-a-faithful-sep-repr}
Let $\Lambda$ be a strongly connected $k$-graph.  Fix $x \in \Lambda^\infty$  and write $x = x_1 x_2 x_3 \cdots $, where $d(x_i) = (1, 1, \ldots, 1)$ for all $i$.  Let $v_i = r(x_i)$.

For each $i$, write $F_i = \Lambda v_i$ for the set of all morphisms (i.e., finite paths) in $\Lambda$ with source $v_i$.  Then $\ell^2(F_i)$ has basis $\{\xi^i_\lambda: \lambda \in F_i\}$.  Define  $\rho_i \in B(\ell^2(F_i) , \ell^2(F_{i+1}))$ by $\rho_i(\xi^i_\lambda) = \xi^{i+1}_{\lambda x_i} \in \ell^2(F_{i+1})$, and form the inductive limit Hilbert space
\[\H_x := \varinjlim (\ell^2(F_i), \rho_i) = \left(\bigsqcup_{i \in \N} \ell^2(F_i) \right) /\sim,\]
where $\xi^i_\lambda \sim \xi^j_\mu$ (with $i \leq j$) iff $\mu = \lambda x_i x_{i+1} \cdots x_{j-1}$.
For a generator $\xi^i_\lambda$ of $\ell^2(F_i)$, we will denote its equivalence class in $\H_x$ by $[\xi^i_\lambda]$.

Observe that $\H_x$ is separable, because $F_i$ is countable for all $i$. 
Moreover, the same $\lambda \in \Lambda$ may appear in both $F_i$ and $F_j$ without having $[\xi^i_\lambda] = [\xi^j_\lambda]$, if the infinite path $x$ passes through the same vertex multiple times.

For any fixed $\lambda  \in \Lambda,$ we define an operator $T_\lambda \in B(\H_x)$ by 
\begin{equation}
T_\lambda [\xi^i_\mu] = \left\{ \begin{array}{cl}
[\xi^i_{\lambda \mu}], & s(\lambda) = r(\mu) \\
0, & \text{ else.}
\end{array} \right. \label{eq:sc-faithful}
\end{equation}

\begin{thm}\label{prop:faithful-repn}
Let $\Lambda$ be a strongly connected $k$-graph and $x \in \Lambda^\infty$.  The operators $\{T_\lambda\}_{\lambda \in \Lambda}$ of Equation \eqref{eq:sc-faithful} define a faithful separable representation of $C^*(\Lambda)$ on $\H_x$.
\label{prop:sc-faithful}
\end{thm}

\begin{proof}
This proof was inspired by Section 3 of \cite{davidson-power-yang-dilation}.

We first check that the operators $T_\lambda$ are well-defined.  Recall, then,
that if $[\xi^i_\nu] = [\xi^j_\mu]\in \H_x$, then there exists $k \geq i,j$ such that $\mu x_j \cdots x_k = \nu x_i \cdots x_k$. Assuming $i \leq  j$, the factorization property then forces $\mu = \nu x_i \cdots x_{j-1}$ (if $i < j$; if $i=j$, we have $\mu = \nu$). In either case, $[\xi^j_{\lambda \nu}]  = [\xi^i_{\lambda \mu}]$, and hence $T_\lambda$ is well defined.

We now check that the operators $\{T_\lambda\}_{\lambda \in \Lambda}$ define a representation of $C^*(\Lambda)$.  To that end, observe that
\[\begin{split}
\langle T^*_\lambda [\xi_\mu^i] \mid [\xi_\nu^j]\rangle &=\langle [\xi_\mu^i] \mid T_\lambda [\xi_\nu^j]\rangle\\
&=\begin{cases}\langle [\xi_\mu^i]\mid [\xi_{\lambda\nu}^j]\rangle\quad\text{if $s(\lambda)=r(\nu)$}\\0\quad\quad\quad\text{otherwise}\end{cases}\\
&=\begin{cases} 1\quad\text{if $[\xi_\mu^i]=[\xi_{\lambda\nu}^j]$ and $s(\lambda)=r(\nu)$}\\ 0\quad\quad\quad\text{otherwise}.\end{cases}
\end{split}\]
Thus 
\begin{equation}\label{eq:T_adjoint}
T^*_\lambda[\xi_\mu^i]=\begin{cases} [\xi_\nu^j]\quad\quad\text{if $[\xi_\mu^i]=[\xi_{\lambda\nu}^j]$}\\
0\quad\quad\quad\text{otherwise}\end{cases}
\end{equation}

One checks immediately that for any $v, w \in \Lambda^0$, $T_v = T_v^* = T_v^2$  and  $T_v T_w = \delta_{v,w} T_v$. A similarly straightforward check shows that $T_\lambda^* T_\lambda = T_{s(\lambda)}$ and that $T_\lambda T_\mu = \delta_{s(\lambda), r(\mu)} T_{\lambda \mu}$.

It remains to check that for any $n \in \N^k, v \in \Lambda^0$, we have $\sum_{\lambda \in v\Lambda^n}T_\lambda T_\lambda^* = T_v$.  To that end, 
 fix $\lambda$ and $[\xi_\mu^i]$, and  compute 
\[\begin{split}
T_\lambda T^*_\lambda[\xi_\mu^i]&=\begin{cases} T_\lambda[\xi_\nu^j]\quad\quad\text{if $[\xi_\mu^i]=[\xi_{\lambda\nu}^j]$}\\ 0 \quad\quad\quad\text{otherwise}\end{cases}\\
&=\begin{cases} [\xi_{\lambda\nu}^j]\quad\quad\text{if $[\xi_\mu^i]=[\xi_{\lambda\nu}^j]$, $s(\lambda)=r(\nu)$}\\
0\quad\quad\quad\text{otherwise}\end{cases}\\
&=\begin{cases} [\xi_\mu^i]\quad\quad\text{if $[\xi_\mu^i]=[\xi_{\lambda\nu}^j]$ and $s(\lambda)=r(\nu)$}\\ 0\quad\quad\text{otherwise}\end{cases}
\end{split}\]
Now, fix $n \in \N^k$ and $v \in \Lambda^0$.  Observe that 
\[\begin{split}
\Big( \sum_{\lambda\in v \Lambda^n} T_\lambda T_\lambda^* \Big) [\xi^i_\mu]&=\begin{cases} \sum_{\lambda\in v \Lambda^n} [\xi^i_\mu]\quad\quad\text{if $[\xi_\mu^i]=[\xi_{\lambda\nu}^j]$ and $s(\lambda)=r(\nu)$}\\
0\quad\quad\text{otherwise.}\end{cases}\\
\end{split}\]
If $r(\mu) = v$, then choose $k>i$ large enough so that $d(\mu x_i \cdots x_{k-1}) \geq n$.  Then, $[\xi^k_{\mu x_i \cdots x_{k-1}}] = [\xi^i_\mu]$, and the factorization property tells us we can write $\mu x_i \cdots x_{k-1} = \lambda \nu$ for a unique $\lambda \in v\Lambda^n$.  Thus,
\[\begin{split}
\Big( \sum_{\lambda\in v \Lambda^n} T_\lambda T_\lambda^* \Big) [\xi^i_\mu]&=\begin{cases}[\xi^i_\mu]\quad\quad\text{if } r(\mu) = v\\
0\quad\quad\text{otherwise.}\end{cases}\\
\end{split}\]
In other words, $\sum_{\lambda \in v\Lambda^n} T_\lambda T_\lambda^* = T_v$ as claimed.

It now follows that the operators $\{T_\lambda\}_{\lambda \in \Lambda}$ satisfy the Cuntz--Krieger relations, and thus generate a representation $\pi_x$ of $C^*(\Lambda)$ on $\H_x$.

We would like to use the gauge-invariant uniqueness theorem (Theorem 3.4 of \cite{KP}) to show that this representation is faithful.  We begin by checking that $T_v$ is nonzero for each $v \in \Lambda^0$.  To see this, fix $v \in \Lambda^0$.  Since $\Lambda$ is strongly connected, there exists $\lambda \in v\Lambda r(x_1)$.  We have $T_v [\xi^1_\lambda] = [\xi^1_\lambda]$; since $[\xi^1_\lambda]$ is a nontrivial element of $\H_x$, the operator $T_v$ is nonzero, as desired.

In order to apply the gauge-invariant uniqueness theorem, we must establish the existence of a gauge action on $\pi_x(C^*(\Lambda))$.  We do this by defining, for each $z \in \T^k$, a unitary $U_z \in B(\H_x)$:
\[U_z[\xi^i_\mu] = z^{d(\mu) -(i, \ldots, i)} [\xi^i_\mu].\]
Note that $U_z$ is well defined, because if $[\xi^i_\mu] = [\xi^j_\nu]$ with $i < j$, then $\mu x_i \cdots x_{j-1} = \nu$, and
\[z^{d(\nu) - (j, \ldots, j)} = z^{d(\mu) + (j-i)(1, \ldots, 1) - (j, \ldots, j)} = z^{d(\mu) - (i, \ldots, i)}.\]
(Recall that if $i = j$ and $[\xi^i_\mu] = [\xi^i_\nu]$,  the factorization property implies that $\mu=\nu$.)
Since $U_z$ is evidently a unitary, we can define an action of $\T^k$ on $\pi_x(C^*(\Lambda))$ by $z \cdot T_\lambda := \Ad U_z (T_\lambda).$  We calculate:
\begin{align*}
z \cdot T_\lambda [\xi^i_\mu] &= U_z T_\lambda U_z^* [\xi_\mu^i] = z^{(i, \ldots, i) - d(\mu)} U_z T_\lambda [\xi^i_\mu] \\
&= \delta_{s(\lambda), r(\mu)} z^{(i, \ldots, i) - d(\mu)} U_z [\xi^i_{\lambda \mu}] = \delta_{s(\lambda), r(\mu)} z^{(i, \ldots, i) - d(\mu)} z^{d(\lambda \mu) - (i, \ldots, i)} [\xi^i_{\lambda\mu}] \\
&= z^{d(\lambda)} T_\lambda [\xi^i_\mu] = \pi_x( \gamma_z (t_\lambda)) [\xi^i_\mu].
\end{align*}
Thus, $\pi_x \circ \gamma_z = \Ad U_z$ for any $z \in \T^k$, where $\gamma_z$ denotes the gauge action of $\T^k$ on $C^*(\Lambda)$.  The gauge invariant uniqueness theorem now tells us that $\pi_x$ is a faithful representation of $C^*(\Lambda)$ on the separable Hilbert space $\H_x$, as claimed.
\end{proof}

\begin{prop}
\label{prop:sc-faithful-SBFS}
Let $\Lambda$ be a strongly connected $k$-graph and fix $x \in \Lambda^\infty$.  The Hilbert space $\H_x$ is of the form $\ell^2(X)$ {with counting measure; for the definition of $X$ see Equation \eqref{def-space-count-measure}}. 
 Moreover, the faithful separable representation of Theorem \ref{prop:sc-faithful} is a $\Lambda$-semibranching representation.
\end{prop}
\begin{proof}
For each $i \geq 1$, define $G_i = \Lambda v_i \backslash \Lambda x_{i-1} \subseteq F_i.$  Equivalently,
\[G_i = F_i \backslash \left( \bigcup_{j < i} F_j x_j \cdots x_{i-1} \right).\]
 Thus by definition of $\mathcal{H}_x$ we have $\mathcal{H}_x \supseteq \oplus_{i\ge 1} \ell^2(G_i)$. To see that $\mathcal{H}_x\subseteq \oplus _{i\ge 1} \ell^2(G_i)$, first note that any vector of $\mathcal{H}_x$ is of the form $[\xi^i_\mu]$, where $\mu\in F_i$ for some $i\ge 1$ by definition. If $\mu\in F_i \setminus  \left( \bigcup_{j < i} F_j x_j \cdots x_{i-1}\right)$, then $\xi^i_\mu\in \ell^2(G_i)$. If $\mu\in F_i$ and $\mu$ lies in  $\cup_{j<i}F_j x_j \cdots x_{i-1}$, then there exists a unique $\ell \le i$ and $\widetilde{\mu}\in F_\ell \setminus \cup_{j<\ell} F_j x_j \cdots x_{\ell -1}$ such that $\mu=\widetilde{\mu}x_\ell x_{\ell+1}\dots x_{i-1}$. So $\xi^i_\mu \sim \xi^{\ell}_{\widetilde{\mu}}$ and $\xi^{\ell}_{\widetilde{\mu}}\in \ell^2(G_\ell)$, and hence $\mathcal{H}_x\subseteq \oplus_{i\ge 1} \ell^2(G_i)$. 
Consequently,
\[ \H_x = \bigoplus_{i \geq 1} \ell^2(G_i).\]

Set 

\begin{equation}
\label{def-space-count-measure}
X : = \bigsqcup_{i \geq 1} G_i,
\end{equation} 
and let $m$ denote counting measure on $X$: 
\[ m\left( \nu \right) = 1 \ \forall \ \nu \in G_i.\]
Then $\H_x = \bigoplus_{i\geq 1} \ell^2(G_i) = \ell^2(X) = L^2(X, m).$ 

We now describe the $\Lambda$-semibranching function system on $(X, m)$ which gives rise to the representation $\{T_\lambda\}_{\lambda \in \Lambda}$.
For a vertex $v \in \Lambda^0$, define $D_v = \{ \nu \in \bigsqcup_{i \geq 1} G_i: r(\nu) = v\},$ and for $\lambda \in \Lambda$ set
$\tau_\lambda: D_{s(\lambda)} \to D_{r(\lambda)}$ by 
\[ \tau_\lambda(\nu) = \rho, \text{ where } \rho \in G_j \text{ and } \lambda \nu = \rho  x_j \cdots x_{i-1} \text{ if } \nu \in G_i.\]
To see that $\tau_\lambda$ is well-defined, fix $\nu\in G_i$ and suppose that there exist $j_1\ne j_2 \le i$ and $\rho_1\in G_{j_1}$, $\rho_2\in G_{j_2}$ such that
\[
\lambda \nu=\rho_1 x_{j_1} x_{j_1+1}\dots x_{i-1} = \rho_2 x_{j_2} x_{j_2+1}\dots x_{i-1}
\]
Then by the factorization property, assuming without loss of generality that $j_2\ge j_1$, we must have $\rho_1 x_{j_1}\dots x_{j_2-1}=\rho_2$. Thus there exists a unique $j$ and $\rho\in G_{j}$ such that $\tau_\lambda(\nu)=\rho$, and hence $\tau_\lambda$ is well-defined.

It follows that
\[ R_\lambda = Ran(\tau_\lambda) = \{ \rho: \rho \in G_j \text{ for some } j \text{ and } \rho x_j \cdots x_i (0, d(\lambda)) = \lambda \text{ for some } i\}.\]
 If $ d(\lambda)=n$, then for $\rho\in G_j \cap R_\lambda$ find the smallest $i\geq j$ such that
 \[
 d(\rho)+(i-j)(1,\dots,1)\ge n.
 \]
 Then define the coding map $\tau^n$ on $G_j \cap R_\lambda$ by\footnote{If $i=j$ then we take $\tau^n(\rho) = \rho(n, d(\rho))$.}
 \begin{equation}
\label{eq:faithful-tau-n}
\tau^n(\rho) = \rho x_j \cdots x_{i-1}(n, d(\rho) + (i-j)(1, \ldots, 1)) \in G_i.
\end{equation}

Now it is straightforward to see that
\[ \tau^n \circ \tau_\lambda (\nu) = \nu,\]
justifying the name ``coding map.''

We claim that the sets and maps described above satisfy Conditions (a) - (d) of Definition \ref{def-lambda-SBFS-1} and hence define a $\Lambda$-semibranching function system on $X$.

First, we fix $n \in \N^k$ and check that for each $\nu \in \bigsqcup_{i\geq 1} G_i$ we have $\nu \in R_\lambda$ for precisely one $\lambda \in \Lambda^n$, which implies that $X = \bigsqcup_{\lambda \in \Lambda^n} R_\lambda$ for any $n\in \N^k$. Given $\nu \in G_i$, let $j \geq i$ be the smallest integer such that $d(\nu) + (j-i)(1, \ldots, 1) \geq n$.  Set $\lambda = \nu x_i \cdots x_{j-1}(0, n)$; then $\nu \in R_\lambda$. Moreover, for any other $\lambda' \in \Lambda^n$, we have $\nu x_i \cdots x_{j-1}(0, n) \not= \lambda'$, so $\nu \in R_\lambda$ for a unique $\lambda \in \Lambda^n$.  Since we are working in a discrete measure space, the Radon-Nikodym derivatives of the prefixing maps $\tau_\lambda$ are constantly equal to 1 on $D_{s(\lambda)}$.  This completes the check of Condition (a) of Definition \ref{def-lambda-SBFS-1}. 

By our   hypothesis that $\Lambda$ is strongly connected,  if $v \not= v_i$ for any $i$, we have  $\emptyset \not= v \Lambda v_i \subseteq F_i$ for all $i$.  
This implies the existence of at least one $\nu \in v \Lambda \cap \bigsqcup_{j \geq 1} G_j$, so $m(D_v) > 0$.  On the other hand, if $v=v_i$ then $v_i \in G_i$ is an element of $D_{v_i}$. Again, we have $m(D_{v_i}) > 0$, so  Condition (b) is satisfied.

The description in Equation \eqref{eq:faithful-tau-n} of the coding map $\tau^n$ makes it easy to check that, for $\rho \in G_j$ and for any $m, n\in \N^k$,
\[ \tau^{m+n}(\rho) = \rho x_j \cdots x_{\ell-1}(m+n, d(\rho) + (\ell-j)(1, \ldots, 1)) = \tau^m \circ \tau^n(\rho),\]
 where $\ell$ is the smallest such that $d(\rho)+(\ell-j)(1,\dots, 1)\ge m+n$.
In other words, Condition (d) holds.  

Similarly, if
$\tau_\lambda \circ \tau_\nu(\rho) = \alpha \in G_\ell$ for some $\rho \in G_j$, then there exist $j \geq i \geq \ell $ and $\eta \in G_i$ with $\alpha x_\ell \cdots x_{i-1} = \lambda \eta$ and $\eta x_i \cdots x_{j-1} = \nu \rho$.  On the other hand, if $\tau_{\lambda \nu}(\rho) = \beta \in G_n$, then 
\[ \beta x_n \cdots x_{j-1} = \lambda \nu \rho = \lambda \eta x_i \cdots x_{j-1} = \alpha x_\ell \cdots x_{j-1}.\]
Since $\alpha$ and $\beta$ are both in $\bigsqcup_{i\geq 1} G_i$, the factorization rule now implies that $n = \ell $ and $\alpha = \beta$.  
It follows that Condition (c) of Definition \ref{def-lambda-SBFS-1} is also satisfied, so the sets $D_v, R_\lambda$ with the coding and prefixing maps $\tau_\lambda, \tau^n$ determine a $\Lambda$-semibranching function system.

Since $(X, m)$ is a discrete measure space, the representation $\{S_\lambda\}_{\lambda\in \Lambda}$ of $C^*(\Lambda)$ given by this $\Lambda$-semibranching function system, described in Theorem \ref{thm:SBFS-repn} above, has the following formula.  Given $\eta \in G_i$, write $\delta_\eta \in L^2(X, m)$ for the indicator function supported at $\eta$. For $\nu \in G_j$, we have 
\begin{align*}
 S_\lambda (\delta_\eta)(\nu) &= \begin{cases}
  0, & \nu \not \in R_\lambda \\
  \delta_\eta(\tau^{d(\lambda)}(\nu)), & \text{ else.}
 \end{cases} 
\end{align*}
By construction, we have $\tau^{d(\lambda)}(\nu) = \nu x_j \cdots x_{\ell-1}(d(\lambda), d(\nu) + (\ell-j)(1, \ldots, 1)) \in G_\ell$.  Thus, the above formula becomes 
\begin{align*} S_\lambda(\delta_\eta)(\nu) &= \begin{cases} 
0, & \nu \not\in R_\lambda \text{ or } \tau^{d(\lambda)}(\nu) \not\in G_i \\
\delta_\eta(\tau^{d(\lambda)}(\nu))), & \text{ else.}
\end{cases}\\
& =\begin{cases} 
1, & \lambda \eta = \nu x_j \cdots x_{i-1} \\
0, & \text{ else.}
\end{cases}
\end{align*}
Since $\lambda \eta = \nu x_j \cdots x_{i-1}$ iff $\nu = \tau_\lambda(\eta)$, we can rewrite this as  
\begin{equation}
\label{eq:faithful-rep-SBFS-formula}
S_\lambda (\delta_\eta) = \delta_{\tau_\lambda(\eta)}.\end{equation}

To finish the proof, we observe that, under the isomorphism $\H_x \cong L^2(X, m)$, Equation \eqref{eq:faithful-rep-SBFS-formula} agrees with the formula for $T_\lambda$ given in Equation \eqref{eq:sc-faithful}.  This follows from the observation that $\tau_\lambda(\eta) \sim \lambda \eta$ by construction, so $[\xi^j_{\tau_\lambda(\eta)}] = [\xi^i_{\lambda \eta}]$.

\end{proof}

Often, the trickiest part in checking that a family of subsets and coding/prefixing maps constitutes a $\Lambda$-semibranching function system is computing the Radon-Nikodym derivatives.  On a discrete measure space, this computation is rendered trivial, as we saw above.  Thus, in the spirit of the above Proposition, we also have the following:

\begin{prop}\label{prop:sc-faithful-SBFS-Lambda}
Let $\Lambda$ be a finite, strongly connected $k$-graph.  The standard representation of $C^*(\Lambda)$ on $\ell^2(\Lambda^\infty)$, given in Equation \eqref{eq:std-rep-on-ell-2},
is a $\Lambda$-semibranching representation.
\end{prop}
\begin{proof}
We first define subsets $\{D_v\}_{v\in \Lambda^0}$ of $\Lambda^\infty$ and prefixing and coding maps $\tau_\lambda, \tau^n$ which give rise to a $\Lambda$-semibranching function system on $\Lambda^\infty$.  Namely, we have 
\[ D_v = v\Lambda^\infty, \qquad \tau_\lambda(x) = \lambda x, \qquad \tau^n(x) = \sigma^n(x).\]
The fact that Condition (a) of Definition \ref{def-lambda-SBFS-1} holds for these sets follows from the fact that, for fixed $n \in \N^k$, every infinite path $x$ is of the form $\lambda y$ for a unique $\lambda \in \Lambda^n$.  Since $\Lambda^\infty$, in this setting, is a discrete measure space, the Radon-Nikodym derivatives are again  constantly equal to 1, and moreover Condition (b) holds.   Conditions (c) and (d) are immediate consequences of the factorization property.

Thus, the sets $\{D_v\}_{v\in \Lambda^0}$, together with the prefixing and coding maps $\{\tau_\lambda, \tau^n: \lambda \in \Lambda, n \in \N^k\}$, constitute a  $\Lambda$-semibranching function system on $\Lambda^\infty$, viewed as a discrete measure space.  The associated representation $\{S_\lambda\}_{\lambda \in \Lambda}$ is given by (for $x, y \in \Lambda^\infty$)
\[ S_\lambda(\delta_y)(x) = \chi_{Z(\lambda)}(x) \delta_y(\sigma^{d(\lambda)}(x)) = \delta_{\lambda y}(x),\]
so $S_\lambda(\delta_y) = \delta_{\lambda y}$ agrees with the formula for the standard representation \eqref{eq:std-rep-on-ell-2}.
\end{proof}
For a similar result in a more general context, see Theorem 3.3 of  \cite{GLR}.

\begin{prop}
Let $\Lambda$ be a strongly-connected $k$-graph.
 If $x,y \in \Lambda^\infty$ are infinite paths such that $\sigma^m(x) = \sigma^n(y)$ for some $m, n \in \N^k$, then $\pi_x$ is equivalent to $\pi_y$.
\end{prop}
\begin{proof}
Suppose that there are infinite paths $x,y$ such that $\sigma^m(x)=\sigma^n(y)$ for some $m,n\in \N^k$.
We write $x=x_0x_1\cdots$ and $y=y_0y_1\cdots$, where $d(x_i)=d(y_i)=(1,1,\cdots,1)$ for all $i$.
Recall that in this setting, we have $x(\vec{k},\vec{\ell})=x_k x_{k+1}\cdots x_\ell$ for $\vec{k}=(k,k\cdots,k) \leq \vec{\ell}=(\ell,\ell,\cdots,\ell)$

To construct an isomorphism $\phi: \H_x \to \H_y$,  fix $[\xi^i_\mu] \in \H_x$. 
Without loss of generality, assume $\vec{i}\ge m$.  Then $\sigma^n(y)=\sigma^m(x)$ implies that $\sigma^{n-m+\vec{i}}(y)=\sigma^{\vec{i}}(x)$, and hence $y(n-m+\vec{i},\infty)=x(\vec{i},\infty)$. Choose the minimum $j\in \N$ such that $\vec{j}\ge n$ and $\vec{j}-n \ge \vec{i}-m$. Then let 
\[
\lambda_{i,j}=y(n-m+\vec{i},\vec{j}).
\]
Note that, if we write $q = \vec{j} - n - (\vec{i} - m)\in \N$, then $\lambda_{i,j}=x(\vec{i},q)$.
Thus, $\lambda_{i,j}$ is the common segment of $x$ and $y$ that lies between the vertices $r(x_i)$ and $r(y_j) $. It follows that multiplying by $\lambda_{i,j}$ on the right takes  $F_{i,x}$ to $F_{j,y}$. 

To be precise, we define
\begin{equation}
\phi([\xi^i_\mu]_x) : = [\xi^j_{\mu \lambda_{i,j}}]_y.\label{eq:phi-1}
\end{equation}


We first verify that $\phi$ is well defined: suppose that $[\xi^i_\mu]_x = [\xi^k_{\mu \mu'}]_x$, where $\vec{k} > \vec{i} \geq m$. We then have $\mu' = x_i \cdots x_{k-1}$ and 
\[\phi([\xi^i_\mu]_x) = [\xi^j_{\mu \lambda_{i,j}}]_y \text{ and } \phi([\xi^k_{\mu \mu'}]_x) = [\xi^\ell_{\mu \mu' \lambda_{k, \ell}}]_y,\]
where $j$ is the coordinatewise maximum of $n-m + \vec{i}$ and $\ell$ is the coordinatewise maximum of $n-m +\vec{k}$.

 This definition of $j, \ell$ implies that  $\ell - k = j-i$ is the coordinatewise maximum of $n-m$; consequently, 
 \[ d(\lambda_{i,j}) =\vec{j} - n + m - \vec{i} = \vec{\ell} -n+m - \vec{k} = d(\lambda_{k, \ell}).\]
  Also, $\ell - j = k -i > 0$, so we can write 
\[[\xi^j_{\mu \lambda_{i,j}}]_y = [\xi^\ell_{\mu \lambda_{i,j} \eta}]_y\]
where $\eta = y(\vec{j}, \vec{\ell})$.  In other words, $d(\eta) = \vec{\ell} - \vec{j} = \vec{k} - \vec{i} = d(\mu').$

Since $\vec{i} \geq m$, the finite paths $\mu' \lambda_{k,\ell}$ and $\lambda_{i,j} \eta$ lie on both $x$ and $y$.  In fact, 
\[s(\mu' \lambda_{k, \ell}) = r(y_\ell)= s(\lambda_{i,j} \eta) \text{ and } r(\mu' \lambda_{k,\ell}) = r(x_i) = r(\lambda_{i,j} \eta).\]
Moreover, $d(\mu' \lambda_{k,\ell}) = d(\mu') + d(\lambda_{k,\ell}) = d(\eta) + d(\lambda_{i,j}).$  The factorization property then tells us that \[\mu' \lambda_{k, \ell} = \lambda_{i,j} \eta.\]  

It now follows that 
\[\phi([\xi^i_\mu]_x) = [\xi^j_{\mu \lambda_{i,j}}]_y = [\xi^\ell_{\mu \lambda_{i,j}\eta}]_y = [\xi^\ell_{\mu \mu' \lambda_{k, \ell}}]_y = \phi([\xi^k_{\mu\mu'}]_x),\]
so $\phi$ is well defined.

To see that $\phi$ is surjective, fix $\nu \in( F_j)_y$ and consider the associated element $[\xi^j_\nu]_y \in \H_y$.  Pick $t \geq j$ large enough to ensure the existence of $\ell \in \N$ with $m \leq \vec{\ell} \leq  m + \vec{t} - n $: in other words, $\vec{t} - n \geq (\max_m - \min_m) \cdot (1, \ldots, 1)$. Since $\sigma^{\vec{t}}(y) = \sigma^{m+\vec{t} - n}(x)$, our choice of $t$ and $\ell$ ensure that $\lambda_{\ell, t}$ is a sub-path of $\nu y_j \cdots y_{t-1}$.  We can therefore write 
\[ \nu y_j \cdots y_{t-1} = \tilde{\nu} \lambda_{\ell, t}\]
for some $\tilde{\nu} \in (F_\ell)_x$.  It follows that $[\xi^j_\nu]_y= \phi([\xi^\ell_{\tilde{\nu}}]_x)$, so $\phi$ is surjective as claimed.

To see that $\phi$ is injective, suppose that $\phi([\xi^i_\mu]_x) = \phi([\xi^\ell_\nu]_x)$.  Without loss of generality, suppose that $\vec{i}\geq \vec{\ell}  \geq m$, so that 
\[\phi([\xi^i_\mu]_x) = [\xi^j_{\mu \lambda_{i,j}}]_y \quad \text{ and } \quad \phi([\xi^\ell_\nu]_x) = [\xi^{h}_{\nu \lambda_{\ell, h}}]_y \]
where $h$ is the coordinatewise maximum of $n-m+\vec{\ell}$ and $j$ is the coordinatewise maximum of $n-m+\vec{i}$.

Since $i \geq \ell$, we can write $i = \ell + q$ for $q \in \N$. Consequently, the coordinatewise maximum $j$ of $n-m+\vec{i}$ is the same as the sum of the coordinatewise maximum of $n-m+\vec{\ell}$ and $q$. In other words, $j=h+q$.
It follows that 
\[d(\lambda_{i,j}) = \vec{j} - n+m -\vec{i} = \vec{q} + \vec{h} - n +m - \vec{i} = \vec{h} - n +m - \vec{\ell} = d(\lambda_{\ell, h}).\]

Since  $j \geq h$, the equivalence relation on $\H_y$  implies that 
\[\nu \lambda_{\ell, h}y_h \cdots y_{j-1}  = \mu \lambda_{i,j}  \text{ if $j>h$; }\]
if $h=j$ then $i=\ell$ and we must have $\nu = \mu$. 

Observe that $j-h = i-\ell = q$.  Assuming that $j>h$, we can write 
\[\nu \lambda_{\ell, h} y_h \cdots y_{j-1} = \nu x_\ell \cdots x_{\ell + j-h-1} \lambda_{\ell +j-h, j} = \nu x_\ell \cdots x_{i-1} \lambda_{i,j}.\]

Then the factorization property implies that 
\[\nu x_\ell \cdots x_{i-1} = \mu,\]
and consequently $[\xi^\ell_\nu]_x = [\xi^i_\mu]_x$.  In other words, $\phi$ is injective.

Finally it is straightforward to check that $\phi\circ T^x_\lambda=T^y_\lambda\circ\phi$ for $\lambda\in\Lambda$, and hence the representations $\pi_x, \pi_y$ are equivalent, as claimed.
\end{proof}

Observe that the representation $\pi_x$ of Theorem \ref{prop:sc-faithful} is in fact well-defined for any row-finite 
source-free  higher-rank graph $\Lambda$ and any $x \in \Lambda^\infty$, even if $\Lambda$ is not strongly connected. We only required the hypothesis that $\Lambda$ be strongly connected in order to ensure that $T_v$ was nonzero for each $v$.  However, a similar construction will give us a separable faithful representation of $C^*(\Lambda)$ for any row-finite, source-free $k$-graph $\Lambda$.
\begin{thm}
Let $\Lambda$ be a row-finite source-free $k$-graph.  There is a faithful separable representation of $C^*(\Lambda)$.
\label{pr:sep-faith}
\end{thm}
\begin{proof}
For each vertex $v \in \Lambda^0$, choose an infinite path $y_v$ with 
$r(y_v) = v$. (The fact that $\Lambda$ is source-free implies we can always do this.) Since $\Lambda$ is a countable category, there will be at most countably many such paths.  Define
\[ \H := \bigoplus_{v \in \Lambda^0} \H_{y_v}, \quad \pi := \bigoplus_{v\in \Lambda^0} \pi_{y_v}.\]
Then $\H$ is a separable Hilbert space and $\pi$ is a representation of $C^*(\Lambda)$ on $\H$.  We know that $\pi(t_\mu)$ is nonzero for each $\mu \in \Lambda$, because
\[\pi_{y_{s(\mu)}}(t_\mu) [\xi^1_{s(\mu)}] = [\xi^1_\mu]\]
is a nonzero generator of $\H_{y_{s(\mu)}}$ and hence of $\H$.

Moreover, the unitary action {$\gamma$} of $\T^k$ on $\H_{y_v}$ discussed in Theorem \ref{prop:sc-faithful} extends to a unitary action of $\T^k$ on $\H$ via the diagonal action.  Similarly, the fact that each representation $\pi_{y_v}$ intertwines this action with the gauge action on $C^*(\Lambda)$ implies that we again have
\[z \cdot \pi(T) = \pi (\gamma_z(T)),\]
so again  by Theorem 2.1 of \cite{bprz} (the gauge-invariant uniqueness theorem) tells us that $\pi$ is a faithful separable representation of $C^*(\Lambda)$.
\end{proof}

\section{A new perspective on $\Lambda$-semibranching function systems}
\label{sec:SBFS-revisited}

In order  to construct examples of $\Lambda$-semibranching function systems for a finite $k$-graph $\Lambda$ more readily, we will recast the original definition of $\Lambda$-semibranching function system from \cite{FGKP} in a way that  only involves the $k$-colored edges of $\Lambda$. We present in Theorem~\ref{thm:SBFS-edge-defn} a definition of $\Lambda$-semibranching function systems equivalent to the original definition (Definition \ref{def-lambda-SBFS-1} above).

The following theorem shows that checking Conditions (a) and (c)  of Definition~\ref{def-lambda-SBFS-1} for arbitrary  $m\in\N^k$ is equivalent to checking the equivalent conditions for the basis elements $e_1,\dots,e_k$ of $\N^k$.

\begin{thm}\label{thm:SBFS-edge-defn}
Let $\Lambda$ be a finite $k$-graph and let $(X,\mathcal{F}, \mu)$ be a measure space. Let $\{e_1,\dots, e_k\}$ be the standard basis of $\N^k$.
For  $1 \leq i \leq k$, suppose we have a semibranching function system $\{ \tau_\lambda: D_\lambda \to R_\lambda \}_{d(\lambda) = e_i}$ on $X$, with associated coding maps 
 $\tau^{e_i}:X\to X$. 
For $\eta \in \Lambda$, write $\eta = \eta_1 \eta_2 \cdots \eta_\ell$ as a sequence of edges, and define
\begin{equation}
\label{eq:tau-lambda-edge-defn}
\tau_\eta := \tau_{\eta_1} \circ \tau_{\eta_2} \circ \cdots \circ \tau_{\eta_\ell}.
\end{equation}
 Then the semibranching function systems $\{\tau_\lambda:d(\lambda) = e_i\}_{i=1}^k$ and coding maps $\{ \tau^{e_i}\}_{i=1}^k$ satisfy Conditions (i) - (v) below if and only if the operators $\{\tau_\eta: \eta \in \Lambda\}$ form a $\Lambda$-semibranching function system, 
   with coding maps $\tau^m:= (\tau^{e_1})^{m_1} \circ (\tau^{e_2})^{m_2} \circ \cdots \circ (\tau^{e_k})^{m_k}$ for $m = (m_1, \ldots, m_k) \in \N^k$.
\begin{itemize}
\item[(i)] 
 For any edges $\lambda, \nu$ with $s(\lambda) = s(\nu) $, we have $D_\lambda = D_\nu$.  Writing $v= s(\lambda) = s(\nu)$, we set 
 \[D_v := D_\lambda = D_\nu,\]
 and we require  $\mu(D_v) > 0$ for all $v \in \Lambda^0$.
\item[(ii)]  
For $v\ne w\in\Lambda^0$, $\mu(D_v\cap D_w)=0$. 
\item[(iii)] Fix $i,j\in \{1,\dots, k\}$. If $\lambda\alpha=\nu\beta$ for $\lambda,\beta\in \Lambda^{e_i}$ and $\nu,\alpha\in \Lambda^{e_j}$, then $R_\alpha\subset D_\lambda$, $R_\beta\subset D_\nu$, and
    \[
    \tau_\lambda\circ \tau_\alpha=\tau_\nu\circ \tau_\beta.
    \]   
\item[(iv)]
For any $1 \leq i, j \leq k$, we have 
$\tau^{e_i} \circ \tau^{e_j} = \tau^{e_j} \circ \tau^{e_i}$.
\item[(v)] For $v\in \Lambda$ and $1\le i\le k$, we have
\[
\mu(D_v\setminus \cup_{g\in v\Lambda^{e_i}} R_{g})=0.
\]
\end{itemize}
\end{thm}

\begin{proof}
First, suppose we are given a $\Lambda$-semibranching function system as in Definition \ref{def-lambda-SBFS-1}.  Condition (c) of Definition \ref{def-lambda-SBFS-1} guarantees Conditions (i) and (iii) in the statement of this Theorem; Condition (ii) follows from Condition (b) and the fact that the maps $\{\tau_v: v \in \Lambda^0\}$ form a semibranching function system.  Condition (d) of Definition \ref{def-lambda-SBFS-1} implies Condition (iv) above. 
To see (v), fix $i\in \{1,\dots,k\}$ and note that Condition (c) of Definition \ref{def-lambda-SBFS-1}  implies that for $g\in \Lambda^{e_i}$, $R_g\subseteq D_{r(g)}$. Thus, $\cup_{g\in v\Lambda^{e_i}} R_g\subseteq D_v$, and hence $\mu(D_v\setminus \cup_{g\in v\Lambda^{e_i}} R_g)=0$.

For the other direction, suppose that we are given $k$ semibranching function systems $\{\tau_\lambda:  \lambda \in \Lambda, d(\lambda) = e_i\}_{i=1}^k$ with coding maps $\{\tau^{e_i}\}_{i=1}^k$ 
satisfying Conditions (i) - (v) above.  
First fix $\eta\in\Lambda$ and write $\eta=\eta_1\eta_2\dots \eta_{\ell}$ as a sequence of edges. Then Condition (iii) implies that $R_{\eta_j}\subseteq D_{\eta_{j-1}}$ for $2\le j\le \ell$, and hence the formula for $\tau_\eta$ given in \eqref{eq:tau-lambda-edge-defn} is well-defined.
\footnote{Note that if $\lambda = \lambda_1 \lambda_2$ with $d(\lambda_1) = \ell, \ d(\lambda_2) = e_j$, then  $R_{\lambda_2} \subseteq D_{s(\lambda_1)}$ by Conditions (iii) and (i), and hence   the composition $\tau_{\lambda_1} \circ \tau_{\lambda_2}$ is well defined.}  
In fact, 
Condition (iii) and the factorization property of $k$-graphs imply that $\tau_\eta$ is independent of the decomposition of $\eta$ into edges.
Moreover, recall that since each $\{\tau_\lambda: d(\lambda) = e_i\}$ 
  is a semibranching function system, we have $\tau^{e_i} \circ \tau_\lambda = id_{D_\lambda}$ for all $\lambda \in \Lambda^{e_i}$. 
  Consequently, if $\eta \in \Lambda^m$, write $\eta$ as a sequence of edges,  $\eta = \eta_1 \eta_2 \cdots \eta_\ell$ where we list the $m_k$ edges of color $k$ first, then all $m_{k-1}$ edges of color $k-1$, etc.   Also note that $id_{D_\alpha}\circ\tau_\beta$ is well defined for edges $\alpha, \beta$ whenever $s(\alpha)=r(\beta)$, and $id_{D_\alpha}\circ\tau_\beta=\tau_\beta$. Then
\[\tau^m \circ \tau_\eta = (\tau^{e_1})^{m_1} \circ \cdots \circ (\tau^{e_k})^{m_k} \circ \tau_{\eta_1} \circ \cdots \circ \tau_{\eta_\ell} = id_{D_\lambda},\]
since $\tau_k^{m_k} \circ \tau_{\lambda_1} \circ \cdots \circ \tau_{\lambda_{m_k}} = id_{D_{\lambda_{m_k}}}$, and similarly for the other colors.
Hence, $\tau^m$ is a coding map for $\{\tau_\lambda:d(\lambda)=m\}$.

To see that $\{\tau_\lambda: d(\lambda) = m\}$ forms a semibranching function system for each $m \in \N^k$, we proceed by induction.  Note that the case 
$m=e_i$ for $1\le i\le k$ holds by the hypotheses of the Theorem.  For the case  $m = 0$, we begin by defining 
\[\tau_v = id: D_v \to D_v \ \text{for}\;\;  v \in \Lambda^0.\]
Then $\Phi_v(x)  := \frac{d(\mu \circ \tau_v)}{d\mu} (x)= 1$ for all  $x \in D_v$.  By Condition (ii), in order to check that $\{\tau_v: v \in \Lambda^0\}$ is a semibranching function system, it merely remains to check that $\mu(X \backslash \cup_{v\in \Lambda^0} D_v) = 0$.  By Conditions (ii) and (v), and the fact that $\{\tau_{\lambda}: d(\lambda) = e_i\}$ is a semibranching function system,
\begin{align*}
\mu\left(\bigcup_{v\in \Lambda^0} D_v \right) &= \sum_{v\in \Lambda^0} \mu(D_v) = \sum_{v\in \Lambda^0} \sum_{\lambda \in v\Lambda^{e_i}} \mu(R_\lambda) = \mu(X)
\end{align*}
as desired.

Now, suppose that for every $\ell=(\ell_1,\dots,\ell_k) \in \N^k$ with $|\ell |=\ell_1+\ell_2+\dots +\ell_k \leq n$, we have $\{\tau_\lambda: d(\lambda) = \ell\}$ is a semibranching function system with coding map $\tau^{\ell}$.  Let $m = \ell + e_j$.  Given $\lambda \not= \nu \in \Lambda^m$, write $\lambda = \lambda_1 \lambda_2 , \ \nu = \nu_1 \nu_2$, with $d(\lambda_1) = \ell = d(\nu_1)$ and $d(\nu_2) = d(\lambda_2) = e_j$. 
Then  $\tau_\lambda=\tau_{\lambda_1}\circ \tau_{\lambda_2}$ is well-defined and 
\[R_\lambda := \tau_{\lambda_1}(R_{\lambda_2}) \subseteq R_{\lambda_1}.\]  If $\nu_1 \not= \lambda_1$, then $R_{\lambda} \cap R_\nu \subseteq R_{\lambda_1} \cap R_{\nu_1}$ and hence 
\[\mu(R_\lambda \cap R_\nu) \leq \mu( R_{\lambda_1} \cap R_{\nu_1}) = 0.\]
If $\nu_1 = \lambda_1$, then since $\lambda \not= \nu$ we must have that $\lambda_2 \not= \nu_2$.  Thus, since $\Phi_{\lambda_1} = \frac{d(\mu \circ \tau_{\lambda_1})}{d\mu}$ and {$\mu(R_{\lambda_2}\cap R_{\nu_2})=0$}, we have
\begin{align*}
\mu(R_\lambda \cap R_\nu)=   \mu(\tau_{\lambda_1}( R_{\lambda_2} \cap R_{\nu_2})) &  =  \int_{R_{\lambda_2} \cap R_{\nu_2}} 1 \, d(\mu \circ \tau_{\lambda_1}) = \int_{R_{\lambda_2} \cap R_{\nu_2}} 
\Phi_{\lambda_1}\, d\mu = 0.
\end{align*}

To see that $\mu(X \backslash \cup_{\lambda \in \Lambda^m} R_\lambda) = 0$, note that
\[
\bigcup_{\lambda \in \Lambda^m} R_\lambda = \bigcup_{\lambda = \lambda_1 \lambda_2 \in \Lambda^m} \tau_{\lambda_1} (R_{\lambda_2}) = \bigcup_{d(\lambda_1) = \ell} \tau_{\lambda_1}\left( \cup_{\lambda_2 \in s(\lambda_1) \Lambda^{e_j}} R_{\lambda_2} \right) \\
\]
Then Condition (i) and (v) gives 
\[\begin{split}
\bigcup_{d(\lambda_1) = \ell} \tau_{\lambda_1}\left( \cup_{\lambda_2 \in s(\lambda_1) \Lambda^{e_j}} R_{\lambda_2} \right) 
&= \bigcup_{d(\lambda_1) = \ell} \tau_{\lambda_1} (D_{s(\lambda_1)}) \ \text{ almost everywhere} \\
&= \bigcup_{d(\lambda_1) = \ell} R_{\lambda_1} = X \ \text{ almost everywhere.}
\end{split}\]
Thus, $\mu(X \backslash \cup_{\lambda \in \Lambda^m} R_\lambda) = 0$. 

To conclude that $\{\tau_\lambda: d(\lambda) = m\}$ is a semibranching function system,
 we need to show that it satisfies Condition (b) of Definition~\ref{def-1-brach-system}, which states that the Radon-Nikodym derivative $\Phi_\lambda := \Phi_{\tau_{\lambda_1}\circ \tau_{\lambda_2}}$ exists and is positive for all $\lambda = \lambda_1 \lambda_2$ with $d(\lambda_1)=\ell$, $d(\lambda_2)=e_j$. Since $\mu\circ \tau_{\lambda_1} <<\mu$ and $\mu\circ \tau_{\lambda_2}<<\mu$, it is straightforward to see that $\mu\circ \tau_{\lambda_1}\circ \tau_{\lambda_2} << \mu\circ \tau_{\lambda_2}$.
Now we fix a Borel set $E\subset D_{\lambda_2}$, otherwise the following integral is zero, 
and consider 
\[
\int_{X}\chi_E(x)\, d(\mu\circ \tau_{\lambda_1}\circ \tau_{\lambda_2}).
\]
Since $E \subset D_{\lambda_2}$, if $x \in E$ then $\tau_{\lambda_2}(x) =: y \in R_{\lambda_2}$, and so (since $\tau^{e_j} \circ \tau_{\lambda_2} = id_{D_{\lambda_2}}$) we see that we can write every $x \in E$ as  $x = \tau^{e_j}(y)$ for precisely one $y \in R_{\lambda_2}$.  Moreover, the fact that 
$\tau_{\lambda_2} = \tau_{\lambda_2} \circ \tau^{e_j} \circ \tau_{\lambda_2}$ implies that 
 $\tau_{\lambda_2} \circ \tau^{e_j} = id_{R_{\lambda_2}}$, so
\[\begin{split}
\int_{X}\chi_E(x)\, d(\mu\circ \tau_{\lambda_1}\circ \tau_{\lambda_2})(x)
&=\int_X \chi_E(\tau^{e_j}(y))\, d(\mu\circ \tau_{\lambda_1}\circ \tau_{\lambda_2})(\tau^{e_j}(y))\\
&=\int_X(\chi_E\circ \tau^{e_j})(y)\, d(\mu\circ \tau_{\lambda_1})(y).
\end{split}\]
Since $\mu_\circ \tau_{\lambda_1} << \mu$, the above integral becomes
\[
\int_X(\chi_E\circ \tau^{e_j}(y) \Phi_{\tau_{\lambda_1}}(y)\, d\mu(y).
\]
Returning to our original notation, write $y=\tau_{\lambda_2}(x)$ for some $x\in E\subset D_{\lambda_2}$; now we have
\[\begin{split}
\int_X(\chi_E\circ \tau^{e_j})(y) \Phi_{\tau_{\lambda_1}}(y)\, d\mu(y)
&=\int_X (\chi_E\circ \tau^{e_j})(\tau_{\lambda_2}(x)) \Phi_{\tau_{\lambda_1}}(\tau_{\lambda_2}(x))\, d\mu (\tau_{\lambda_2}(x))\\
&=\int_X \chi_E(x) (\Phi_{\tau_{\lambda_1}}\circ \tau_{\lambda_2})(x)\, d(\mu\circ \tau_{\lambda_2})(x).
\end{split}\]
So we have
\[
\int_{X}\chi_E(x)\, d(\mu\circ \tau_{\lambda_1}\circ \tau_{\lambda_2})
=\int_X \chi_E(x) (\Phi_{\tau_{\lambda_1}}\circ \tau_{\lambda_2})(x)\, d(\mu\circ \tau_{\lambda_2})(x).
\]
Thus, by uniqueness of Radon-Nikodym derivatives and the fact that $ \mu \circ \tau_{\lambda_1} \circ \tau_{\lambda_2} << \mu \circ \tau_{\lambda_2}$, we have   
\[
\frac{d(\mu\circ \tau_{\lambda_1}\circ\tau_{\lambda_2})}{d(\mu\circ \tau_{\lambda_2})}=\Phi_{\tau_{\lambda_1}}\circ \tau_{\lambda_2}.
\]
Therefore $\Phi_\lambda := \Phi_{\tau_{\lambda_1}\circ \tau_{\lambda_2}}$ exists and
\[
\Phi_\lambda = \Phi_{\tau_{\lambda_1}\circ \tau_{\lambda_2}}=\frac{d(\mu\circ \tau_{\lambda_1}\circ\tau_{\lambda_2})}{d(\mu\circ \tau_{\lambda_2})}\, \frac{d(\mu\circ \tau_{\lambda_2})}{d\mu}=(\Phi_{\tau_{\lambda_1}}\circ \tau_{\lambda_2})(\Phi_{\tau_{\lambda_1}}),
\]
which is positive since $\Phi_{\tau_{\lambda_1}}$ and $\Phi_{\tau_{\lambda_2}}$ are positive. Hence $\{\tau_{\lambda}:d(\lambda)=\ell + e_j\}$ forms a semibranching function system.
Therefore by induction $\{\tau_\lambda: d(\lambda) =m\}$ forms a semibranching function system for all $m \in \N^k$. This completes the proof that Condition (a) holds.   Note that Condition (b) holds by construction and by Condition (i); Condition (c) holds by construction, Condition (v), and the fact that $\tau_\lambda$ is well defined.  Similarly, Condition (d) holds by construction and by Condition (iv), completing the proof of the Theorem.
\end{proof}

\begin{cor}\label{cor:repn_S}
Let $\Lambda$ be a finite $k$-graph with no sources and let $(X,\mathcal{F}, \mu)$ be a measure space. 
For  each $1 \leq i \leq k$, suppose we have a semibranching function system  $\{\tau_f: D_f\to R_f\}_{d(f)=e_i}$ on $(X,\mu)$ with associated coding map  $\tau^{e_i}:X\to X$ satisfying Conditions (i)--(v) in  Theorem  \ref{thm:SBFS-edge-defn}. For each $1\le i \le k$ and $f\in \Lambda^{e_i}$, define $S_f\in B(L^2(X,\mu))$ by
\begin{equation}\label{eq:standard_S_f}
S_f \xi(x)=\chi_{R_f}(x)(\Phi_{\tau_f}(\tau^{e_i}(x))^{-1/2}\xi(\tau^{e_i}(x)).
\end{equation}
Then the collection of operators $\{S_f: d(f)=e_i\}_{i=1}^k$ generate a representation of $C^*(\Lambda)$.
\end{cor}

\begin{proof}
By Theorem \ref{thm:SBFS-edge-defn}, we obtain a $\Lambda$-semibranching function system on $(X,\mu)$, and thus by Theorem 3.5 of \cite{FGKP}, we have an associated representation of $C^*(\Lambda)$ on $L^2(X, \mu)$.  When we evaluate the formula from \cite{FGKP} Theorem 3.5 on paths $f \in \Lambda$  with $|d(f)| = 1$ we obtain the formula for $S_f$ given in the statement of the Corollary.  Moreover, using (CK2), we can compute $S_\lambda$ for any $\lambda \in \Lambda$ once we know the formulas for $\{S_f: f\in \Lambda, |d(f)| = 1\}$.  The fact that the operators $S_\lambda$ arise from the $\Lambda$-semibranching function system induced by $\{ \{\tau_f: D_f\to R_f\}_{d(f)=e_i}\}_{i=1}^k$ guarantees the necessary commutativity properties to ensure that $S_\lambda$ is well defined.  Namely, suppose $\lambda = f_1 f_2 = g_2 g_1$ for $f_i, g_i$ edges of degree $e_i$ in $\Lambda$.  Then Theorem 3.5 of \cite{FGKP} tells us that   $ S_\lambda = S_{f_1} \circ S_{f_2} = S_{g_2} \circ S_{g_1}$, so writing $S_\lambda$ as a composition of operators $S_f$ for an edge $f$ gives the same formula as in \cite{FGKP}, and moreover is independent of the choice of factorization of $\lambda$ into edges.
\end{proof}

\section{{Examples of $\Lambda$-semibranching function systems on Lebesgue measure spaces}}
\label{sec:examples_a}
{In this section, we describe a number of examples of $\Lambda$-semibranching function systems for finite $k$-graphs $\Lambda$.  In confirming that our examples are indeed $\Lambda$-semibranching function systems, we rely heavily on the characterization given in Theorem \ref{thm:SBFS-edge-defn}.

Our main focus in the present Section \ref{sec:examples_a} is to hint at the flexibility and diversity offered by the $\Lambda$-semibranching function systems, by showcasing a variety of  examples of $\Lambda$-semibranching function systems on familiar measure spaces for 1- and 2-graphs $\Lambda$.  Indeed, the measure spaces we use below are primarily the unit interval or unit square equipped with Lebesgue measure.  We also show that the standard constructions of 2-graphs from 1-graphs (product and double graphs) are compatible with our $\Lambda$-semibranching function systems.  See Definition~\ref{def-double-graph} and Example~\ref{ex-3.3-Kawamura_product} for more details. }

\begin{example} 
\label{exonevthreeed}
Consider the following $1$-graph $\Lambda$ with two vertices $v_1$ and $v_2$ and three edges $f_1,f_2$ and $f_3$.
\[
\begin{tikzpicture}[scale=1.5]
 \node[inner sep=0.5pt, circle] (v) at (0,0) {$v_1$};
    \node[inner sep=0.5pt, circle] (w) at (1.5,0) {$v_2$};
    \draw[-latex, thick] (w) edge [out=50, in=-50, loop, min distance=30, looseness=2.5] (w);
    \draw[-latex, thick] (v) edge [out=130, in=230, loop, min distance=30, looseness=2.5] (v);
\draw[-latex, thick] (w) edge [out=150, in=30] (v);
\node at (-0.75, 0) {\color{black} $f_1$}; 
\node at (0.7, 0.45) {\color{black} $f_2$};
\node at (2.25, 0) {\color{black} $f_3$};
\end{tikzpicture}
\]
To find a $\Lambda$-semibranching function system, we let $X$ be the closed unit interval $[0,1]$ of $\R$ with the usual  Lebesgue $\sigma$-algebra and measure $\mu$. For $v_1$ and $v_2$, let $D_{v_1}=[0, \frac{1}{2}]$ and $D_{v_2}=(\frac{1}{2}, 1]$. Also for each edge $f\in \Lambda$, let $D_{f}=D_{s(f)}$, and hence $D_{f_1}=D_{v_1}=[0,\frac{1}{2}]$,  $D_{f_2}=D_{v_2}=(\frac{1}{2}, 1]$ and $D_{f_3}=D_{v_2}=(\frac{1}{2}, 1]$.
Thus, the set of domains satisfy Conditions (i) and (ii) of Theorem~\ref{thm:SBFS-edge-defn} automatically. 
Now define prefixing maps for $f_1, f_2$ and $f_3$ by
\[\begin{split}
\tau_{f_1}(x)=-\frac{1}{2}x+\frac{1}{2}\quad\quad &\text{for $x\in D_{f_1}=\big[0,\frac{1}{2}\big]$},\\
\tau_{f_2}(x)=-\frac{1}{2}x+\frac{1}{2}\quad\quad&\text{for $x\in D_{f_2}=\big(\frac{1}{2},1\big]$},\\
\tau_{f_3}(x)=x\quad\quad &\text{for $x\in D_{f_3}=\big(\frac{1}{2},1\big]$}.
\end{split}\]
Then $R_{f_1}=\big[\frac{1}{4},\frac{1}{2}\big]$, $R_{f_2}=\big[0,\frac{1}{4}\big)$ and $R_{f_3}=\big(\frac{1}{2},1\big]$. Then the ranges of the prefixing maps are mutually disjoint and $X=R_{f_1}\cup R_{f_2}\cup R_{f_3}$.
For each $f_i$, since Lebesgue measure is regular, the Radon-Nikodym derivative of $\tau_{f_i}$ 
is given by
\begin{equation*}
\Phi_{{f_i}}(x) = \inf_{x \in E \subseteq D_{f_i}} \frac{(\mu\circ \tau_{f_i})(E)}{\mu(E)}\\
= \inf_{x \in E \subseteq D_{f_i}}\left\{ \begin{array}{cl}\frac{\frac{1}{2} \mu(E)}{\mu(E)}, & i=1,2, \\
\frac{\mu(E)}{\mu(E)}, & i=3
\end{array}\right.
\end{equation*}
\begin{equation*}
  = \left\{ \begin{array}{cl} \frac{1}{2}, & i =1,2 \\
1, & i=3.
\end{array}\right.
\end{equation*}
Thus, the Radon-Nikodym derivatives are positive on their respective domains, as desired.  
Now define $\tau^1:X\to X$ by
\[
\tau^1(x)=\begin{cases} \tau^{-1}_{f_1}(x)\quad\text{for $x\in R_{f_1}$}\\
\tau^{-1}_{f_2}(x)\quad\text{for $x\in R_{f_2}$} \\
\tau^{-1}_{f_3}(x)\quad\text{for $x\in R_{f_3}$} \end{cases}
\]
Since the sets $R_{f_i}$ are mutually disjoint, $\tau^1$ is well defined on $X$.
Then $\tau^1$ is the coding map satisfying $\tau^1\circ \tau_{f_i}(x)=x$ for all $x\in D_{f_i}$. This shows that $\{\tau_{f_i}:D_{f_i}\to R_{f_i}, i=1,2,3\}$ is a semibranching function system for $(X,\mu)$. Since $\Lambda$ is $1$-graph, Condition (iv) is trivially satisfied. Also Condition (v) is satisfied since our construction gives 
\[
D_{v_1}=R_{f_1}\cup R_{f_2},\quad\quad D_{v_2}=R_{f_3}.
\]
Since our graph $\Lambda$ is a $1$-graph, Condition (iii) is equivalent to the following:  
\begin{equation}\label{eq:cond_v_edge}
R_g\subseteq D_f \quad\text{whenever $s(f)=r(g)$ for edges $f,g\in \Lambda$}.
\end{equation}
In this example, there are only three composable pairs of edges to check. In particular, the pair $(f_1, f_2)$ has $s(f_1)=r(f_2)$ and $R_{f_2}\subseteq D_{f_1}$, and the pair $(f_3,f_3)$ has $s(f_3)=r(f_3)$ and  $R_{f_3}\subseteq D_{f_3}$; similarly, $(f_1, f_1)$ satisfies $s(f_1) = r(f_1)$ and hence $R_{f_1} \subseteq D_{f_1}$.  This shows \eqref{eq:cond_v_edge}.  Therefore $\{\tau_{f_1},\tau_{f_2}, \tau_{f_3}\}$ with the coding map $\tau^1$ gives a $\Lambda$-semibranching function system on $([0,1],\mu)$.
\end{example}

\begin{example}
\label{exonevtwoe}
Consider the following 2-colored graph (cf. Example 4.1 of \cite{FGKPexcursions}).
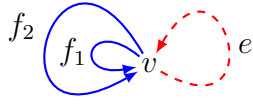
\[
\begin{tikzpicture}[scale=1.7]
 \node[inner sep=0.5pt, circle] (v) at (0,0) {$v$};
\draw[-latex, blue, thick] (v) edge [out=140, in=190, loop, min distance=15, looseness=2.5] (v);
\draw[-latex, blue, thick] (v) edge [out=120, in=210, loop, min distance=40, looseness=2.5] (v);
\draw[-latex, red, thick, dashed] (v) edge [out=-30, in=60, loop, min distance=30, looseness=2.5] (v);
\node at (-0.6, 0.1) {$f_1$}; \node at (-1,0.3) {$f_2$}; \node at (0.75,0.15) {$e$};
\end{tikzpicture}
\]
Then there is a $2$-graph $\Lambda$ with the above skeleton and  factorization rules given by
\begin{equation}\label{eq:ex_facto}
f_1e=ef_2\quad\text{and}\quad ef_1=f_2e.
\end{equation}
Let $X=(0,1)$ be the unit interval with Lebesgue $\sigma$-algebra and measure $\mu$. We construct two semibranching function systems on $(X,\mu)$, namely $\{\tau_g:d(g)=e_i\}_{i=1}^2$ with coding maps $\{\tau^{e_i}:X\to X\}_{i=1}^2$, which satisfy Conditions (i)--(v) of Definition~\ref{thm:SBFS-edge-defn}. 

Let $D_v=(0,1)$. Since $D_\lambda=D_{s(\lambda)}$, $s(f_1)=s(f_2)=s(e)=v$ implies that $D_{f_1}=D_{f_2}=D_{e}=(0,1)$. Since this example has only one vertex $v$, the conditions (i) and (ii) of Theorem~\ref{thm:SBFS-edge-defn} are trivially satisfied. We define $\tau_{f_1}$, $\tau_{f_2}$ and $\tau_{e}$ on $(0,1)$ by
\[
\tau_{f_1}(x)=\frac{1-x}{2}, \quad \tau_{f_2}(x)=\frac{2-x}{2}\;\;\;\text{and}\;\;\; \tau_{e}(x)=-x+1.
\]
Then $R_{f_1}=(0,\frac{1}{2})$, $R_{f_2}=(\frac{1}{2},1)$ and $R_{e}=(0,1)$. 
To compute the Radon-Nikodym derivatives, since the functions $\tau_{f_i}, \tau_{e}$ are linear, we have (for $g=f_1, f_2, e$)
\begin{equation*}
\Phi_{{g}}(x) = \inf_{x \in E \subseteq D_{g}} \frac{(\mu\circ \tau_{g})(E)}{\mu(E)}\\
= \inf_{x \in E \subseteq D_{e}}\left\{ \begin{array}{cl}\frac{\frac{1}{2} \mu(E)}{\mu(E)}, & g = f_1, f_2, \\
\frac{\mu(E)}{\mu(E)}, & g = e
\end{array}\right.
\end{equation*}
\begin{equation*}
  = \left\{ \begin{array}{cl} \frac{1}{2}, & g = f_1, f_2 \\
1, & g= e.
\end{array}\right.
\end{equation*}
Thus, the Radon-Nikodym derivatives are positive on their respective domains, as desired.  

Also it is evident that 
\[
\mu(X\setminus (R_{f_1}\cup R_{f_2}))=0\quad\text{and}\quad \mu(X\setminus R_{e})=0.
\]
Since there is only one vertex $v$, this implies condition (v) of Theorem~\ref{thm:SBFS-edge-defn}.
Define coding maps $\tau^{e_1}$ and $\tau^{e_2}$ by
\[
\tau^{e_1}(x)=\begin{cases} \tau^{-1}_{f_1}(x) \quad\text{if $x\in R_{f_1}$}\\ \tau^{-1}_{f_2}(x) \quad\text{if $x\in R_{f_2}$},\end{cases}\quad\text{and}\quad \tau^{e_2}(x)=\tau^{-1}_{e}(x)\quad\text{for $x\in R_{e}$}.
\]
Since $R_{f_1}$ and $R_{f_2}$ are disjoint, $\tau^{e_1}$ is well-defined on $X$, and similarly $\tau^{e_2}$ is well-defined on $X$.
Then it is straightforward to check that the coding maps satisfy condition (iv) of Theorem~\ref{thm:SBFS-edge-defn},
and that $\{\tau_{f_1},\tau_{f_2}\}$ and $\{\tau_{e}\}$ (together with $\{\tau^{e_1},\tau^{e_2}\}$) form a  semibranching function systems on $(0,1)$ with Lebesgue measure.

Since there are only two factorization rules, to verify condition (iii) of Theorem~\ref{thm:SBFS-edge-defn}, we need to show that
\begin{equation}\label{eq:cond3_1}
R_{e}\subset D_{f_1},\quad R_{f_2}\subset D_{e},\quad R_{f_1}\subset D_{e},\quad R_{e}\subset D_{f_2},
\end{equation}
and 
\begin{equation}\label{eq:cond3_2}
\tau_{f_1}\circ \tau_{e}=\tau_{e}\circ \tau_{f_2}, \quad \tau_{e}\circ \tau_{f_1}=\tau_{f_2}\circ \tau_{e}.
\end{equation}
Since there is only one vertex $v$ and for any edge $g\in \Lambda^{e_i}$, $D_{g}=D_{s(g)}=D_v=(0,1)$, the conditions in \eqref{eq:cond3_1} are trivially satisfied. Also straightforward calculations give \eqref{eq:cond3_2}.  Therefore the semibranching function systems $\{\tau_{f_1},\tau_{f_2}\}$ and $\{\tau_{e}\}$ with coding maps $\tau^{e_1}$, $\tau^{e_2}$ give a $\Lambda$-semibranching function system by Theorem~\ref{thm:SBFS-edge-defn}.
\end{example}

\begin{example}
\label{ex:non-cst-RN}
We present here an example of a $\Lambda$-semibranching function system for the 2-graph of Example \ref{exonevtwoe} for which the Radon-Nikodym derivatives are not constant.

Let $X = [0,1]^2$ and let $D_v = (0,1)^2$.  Define 
\[ \tau_{f_1}(x, y) = (x, x+y-xy), \quad \tau_{f_2}(x,y) = (x, xy), \quad \tau_e(x,y) = (1-x, 1-y).\]
Then $R_{f_1} = \{ (x, y): 0 < x < y\}$ and $R_{f_2} = \{ (x, y): 0 < y < x \}$, and 
\[ \tau^{e_2} = \tau_e, \quad \tau^{e_1}(x, y)  = \begin{cases}
(x, y/x) & \text{ if }0 < y < x \\
\left(x, \frac{y-x}{1-x}\right) & \text{ if } 0 < x < y
\end{cases}\]
To see that these functions satisfy Theorem 3.3, we must check that 
\[ \tau^{e_2} \circ \tau^{e_1} = \tau^{e_1} \circ \tau^{e_2} \quad \text{ and } \quad \tau_{f_i} \circ \tau_{e} = \tau_e \tau_{f_{i+1}}.\]
These equations follow from straightforward calculations.

We now compute the Radon-Nikodym derivatives associated to this $\Lambda$-semibranching function system.  Consider a rectangle $E \subseteq X$ with lower left vertex $(a, b)$ and upper right vertex $(a+\epsilon, b + \delta)$.  Then $\mu(E) = \epsilon \delta$, whereas $ \tau_{f_1}(E) $ is the quadrilateral 
bounded by the lines 
\[x = a, \quad  x = a+\epsilon, \quad y = (1-b-\delta) x + b + \delta, \quad y = (1-b)x + b,\]
so a straightforward calculation tells us that 
\[\mu(\tau_{f_1}(E)) = \delta \epsilon ( 1 - a - \epsilon/2),\] and hence 
$\frac{\mu(\tau_{f_1}(E))}{\mu(E)} = 1-a-\epsilon/2.$
Thus, 
\[\Phi_{f_1}(x,y) = \lim_{E \ni ( x, y)} \frac{\mu_{\tau_{f_1}}(E)}{\mu(E)} = 1-x.\]
Similar calculations to the above show that $\tau_{f_2}(E)$ is the quadrilateral bounded by the lines 
\[x=a, \quad x= a+\epsilon, \quad y =( b+\delta) x , \quad y = b x,\]
and hence 
$\frac{\mu({\tau_{f_2}}(E))}{\mu(E)} = a + \epsilon/2$.  Consequently, 
\[ \Phi_{f_2}(x,y) = x.\]

Since $\tau_e$ is linear, $\Phi_e(x,y) = 1$ for all $(x, y) \in D_v$.

\end{example}

\begin{example}
\label{ex:sbfs-3v8e}
Consider the $2$-graph $\Lambda$ given in Example~7.7 of \cite{LLNSW} with the following skeleton.
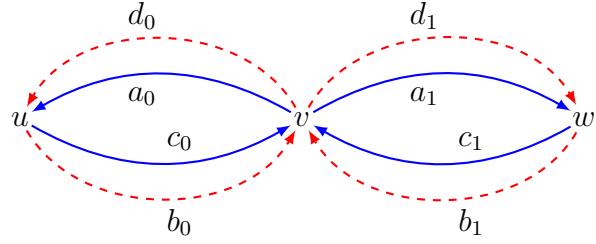
\[
\begin{tikzpicture}[scale=2.5]
\node[inner sep=0.5pt, circle] (u) at (0,0) {$u$};
\node[inner sep=0.5pt, circle] (v) at (1.5,0) {$v$};
\node[inner sep=0.5pt, circle] (w) at (3,0) {$w$};
\draw[-latex, thick, blue] (v) edge [out=150, in=30] (u); 
\draw[-latex, thick, blue] (u) edge [out=-30, in=210] (v); 
\draw[-latex, thick, blue] (v) edge [out=30, in=150] (w); 
\draw[-latex, thick, blue] (w) edge [out=210, in=-30] (v); 
\draw[-latex, thick, red, dashed] (v) edge [out=120, in=60] (u); 
\draw[-latex, thick, red, dashed] (u) edge [out=-60, in=240] (v); 
\draw[-latex, thick, red, dashed] (v) edge [out=60, in=120] (w); 
\draw[-latex, thick, red, dashed] (w) edge [out=240, in=-60] (v); 
\node at (0.65, 0.12) {\color{black} $a_0$}; 
\node at (0.85, -0.12) {\color{black} $c_0$};
\node at (2.15, 0.12) {\color{black} $a_1$};
\node at (2.4, -0.12) {\color{black} $c_1$};
\node at (0.65, 0.55) {\color{black} $d_0$}; 
\node at (0.85, -0.55) {\color{black} $b_0$};
\node at (2.15, 0.55) {\color{black} $d_1$};
\node at (2.4, -0.55) {\color{black} $b_1$};
\end{tikzpicture}
\]
Here the blue and solid edges have degree $e_1$, and the red and dashed edges have degree $e_2$.
The factorization property of $\Lambda$ is given by, for $i=0,1$,
\[
a_ib_i=d_ic_i, \quad a_ib_{1-i}=d_ic_{1-i},\quad\text{and} \quad c_id_i=b_{1-i}a_{1-i}.
\]
In particular, 
\begin{equation}\label{eq:ex3_FP}
\begin{split}
&a_0b_0=d_0c_0, \quad a_1b_1=d_1c_1, \quad a_1b_0=d_1c_0,\\
&a_0b_1=d_0c_1, \quad c_0d_0=b_1a_1,\quad c_1d_1=b_0a_0
\end{split}
\end{equation}
Let $X=(0,1)$ be the unit open interval with  Lebesgue $\sigma$-algebra and measure  $\mu$.

 Let $D_u=(0,\frac{1}{3})$, $D_v=(\frac{1}{3}, \frac{2}{3})$ and $D_w=(\frac{2}{3}, 1)$. Then $\mu(X\setminus (D_u\cup D_v\cup D_w))=0$ and $\mu(D_i\cap D_j)=0$ for $i\ne j$ and $i,j\in \{u,v,w\}$, which gives Condition (i) and (ii) of Theorem~\ref{thm:SBFS-edge-defn}. We first define prefixing maps for blue (solid) edges;
\[\begin{split}
\tau_{a_0}(x)=\frac{3x-1}{3}\quad\;\;&\text{for $x\in D_{a_0}=D_v=\big(\frac{1}{3},\frac{2}{3}\big)$,}\\
\tau_{a_1}(x)=\frac{3x+1}{3}\quad\;\;&\text{for $x\in D_{a_1}=D_v=\big(\frac{1}{3},\frac{2}{3}\big)$,}\\
\tau_{c_0}(x)=\frac{x+1}{2}\quad\;\;&\text{for $x\in D_{c_0}=D_u=\big(0,\frac{1}{3}\big)$,}\\
\tau_{c_1}(x)=\frac{x}{2}\quad\;\;&\text{for $x\in D_{c_1}=D_w=\big(\frac{2}{3},1\big)$.}\\
\end{split}\]
Then the ranges of them are
\[
R_{a_0}=\big(0,\frac{1}{3}\big), \quad R_{a_1}=\big(\frac{2}{3},1\big),\quad R_{c_0}=\big(\frac{1}{2}, \frac{2}{3}\big), \quad\text{and}\quad \ R_{c_1}=\big(\frac{1}{3}, \frac{1}{2}\big).
\]
Thus, up to sets of measure zero, $D_u=R_{a_0}$, $D_v=R_{c_0}\cup R_{c_1}$, and $D_w=R_{a_1}$. So Condition (v) is satisfied for the degree $e_1$. 
 Moreover, for $e\in \{a_0, a_1, c_0, c_1\}$, the Radon-Nikodym derivative of $\tau_e$ on $D_e$ is given by
\begin{equation*}
\Phi_{{e}}(x) = \inf_{x \in E \subseteq D_{e}} \frac{(\mu\circ \tau_{e})(E)}{\mu(E)}\\
= \inf_{x \in E \subseteq D_{e}}\left\{ \begin{array}{cl}\frac{\frac{1}{2} \mu(E)}{\mu(E)}, & e = c_0, c_1 \\
\frac{\mu(E)}{\mu(E)}, & e = a_0, a_1
\end{array}\right.
\end{equation*}
\begin{equation*}
  = \left\{ \begin{array}{cl} \frac{1}{2}, & e = c_0, c_1 \\
1, & e = a_0, a_1
\end{array}\right.
\end{equation*}
since $\tau_e$ is linear for all $e \in \{a_0, a_1, c_0, c_1\}$.
Now define $\tau^{e_1}$ by
\[
\tau^{e_1}(x)=\begin{cases} \tau_{a_0}^{-1}(x)\quad\text{for}\;\; x\in R_{a_0} \\ \tau_{a_1}^{-1}(x)\quad\text{for}\;\; x\in R_{a_1}\\ \tau_{c_0}^{-1}(x)\quad\text{for}\;\; x\in R_{c_0} \\ \tau_{c_1}^{-1}(x)\quad\text{for}\;\; x\in R_{c_1}
\end{cases}
\]
Then $\tau^{e_1}$ is a coding map for $\{\tau_f: d(f)=e_1\}$. Therefore $\{\tau_f:D_f\to R_f, d(f)=e_1\}$ is a semibranching function system on $(X,\mu)$. Similarly, we define a semibranching function system for red (dashed) edges as follows.
\[\begin{split}
\tau_{d_0}(x)=\frac{-3x+2}{3}\quad\;\;&\text{for $x\in D_{d_0}=D_v=\big(\frac{1}{3},\frac{2}{3}\big)$,}\\
\tau_{d_1}(x)=\frac{-3x+4}{3}\quad\;\;&\text{for $x\in D_{d_1}=D_v=\big(\frac{1}{3},\frac{2}{3}\big)$,}\\
\tau_{b_0}(x)=\frac{-x+1}{2}\quad\;\;&\text{for $x\in D_{b_0}=D_u=\big(0,\frac{1}{3}\big)$,}\\
\tau_{b_1}(x)=\frac{-x+2}{2}\quad\;\;&\text{for $x\in D_{b_1}=D_w=\big(\frac{2}{3},1\big)$.}\\
\end{split}\]
Then 
\[
R_{d_0}=\big(0,\frac{1}{3}\big),\quad R_{d_1}=\big(\frac{2}{3},1\big),\quad R_{b_0}=\big(\frac{1}{3},\frac{1}{2}\big),\quad \text{and}\quad R_{b_1}=\big(\frac{1}{2},\frac{2}{3}\big).
\]
Thus, $D_u=R_{d_0}$, $D_v=R_{b_0}\cup R_{b_1}$ and $D_w=R_{d_1}$. So Condition (v) is satisfied. Also we have $\mu(X\setminus (R_{d_0}\cup R_{d_1}\cup R_{b_0}\cup R_{b_1}))=0$ and $\mu(R_i\cap R_j)=0$ for $i\ne j$ and $i,j\in \{d_0,d_1,b_0,b_1\}$. Also, for $e\in \{d_0,d_1,b_0,b_1\},$ the Radon-Nikodym derivative $\Phi_g$ is given by 
\begin{equation*}
\Phi_{{e}}(x) = \inf_{x \in E \subseteq D_{e}} \frac{(\mu\circ \tau_{e})(E)}{\mu(E)}\\
= \inf_{x \in E \subseteq D_{e}}\left\{ \begin{array}{cl}\frac{\frac{1}{2} \mu(E)}{\mu(E)}, & e = b_0,b_1 \\
\frac{\mu(E)}{\mu(E)}, & e = d_0, d_1
\end{array}\right.
\end{equation*}
\begin{equation*}
  = \left\{ \begin{array}{cl} \frac{1}{2}, & e = b_0, b_1 \\
1, & e = d_0, d_1.
\end{array}\right.
\end{equation*}
Now we define $\tau^{e_2}$ by
\[
\tau^{e_2}(x)=\begin{cases} \tau_{d_0}^{-1}(x)\quad\text{for}\;\; x\in R_{d_0} \\ \tau_{d_1}^{-1}(x)\quad\text{for}\;\; x\in R_{d_1} \\ \tau_{b_0}^{-1}(x)\quad\text{for}\;\; x\in R_{b_0} \\ \tau_{b_1}^{-1}(x)\quad\text{for}\;\; x\in R_{b_1}
\end{cases}
\]
Then $\tau^{e_2}$ is a coding map for $\{\tau_g:d(g)=e_2\}$. Thus, $\{\tau_g:D_g\to R_g, d(g)=e_2\}$ is a semibranching function system on $(X,\mu)$. 
To conclude that the above collection of   prefixing maps for blue edges and red edges gives a $\Lambda$-semibranching function system, we need to verify  conditions (iii) and (iv) of Theorem~\ref{thm:SBFS-edge-defn}. For a composable pair of edges $(\lambda,\eta)$,
it is straightforward to see that $R_{\eta}\subseteq D_{\lambda}$. 
Also it is straightforward to check that the prefixing maps satisfy the factorization properties given in \eqref{eq:ex3_FP}:
\[\begin{split}
&\tau_{a_0b_0}=\tau_{d_0c_0},\quad \tau_{a_1b_1}=\tau_{d_1c_1},\quad \tau_{a_1b_0}=\tau_{d_1c_0}\\
&\tau_{a_0b_1}=\tau_{d_0c_1},\quad \tau_{c_0d_0}=\tau_{b_1a_1},\quad \tau_{c_1d_1}=\tau_{b_0a_0}.
\end{split}\]
For example, 
\[
\tau_{a_0b_0}(x)=\tau_{a_0}\circ \tau_{b_0}(x)=\frac{3\tau_{b_0}(x)-1}{3}=(-\frac{x}{2}+\frac{1}{2})-\frac{1}{3}=-\frac{x}{2}+\frac{1}{6}
\]
for $x\in D_{b_0}=D_u=(0,\frac{1}{3})$. Also
\[
\tau_{d_0c_0}(x)=\tau_{d_0}\circ \tau_{c_0}(x)=\frac{-3\tau_{c_0}(x)+2}{3}=-(\frac{x}{2}+\frac{1}{2})+\frac{2}{3}=-\frac{x}{2}+\frac{1}{6}
\]
for $x\in D_{c_0}=D_u=(0,\frac{1}{3})$. 
Therefore the above prefixing maps give a $\Lambda$-semibranching function system on $(0,1)$  with Lebesgue measure by Theorem~\ref{thm:SBFS-edge-defn}.
\end{example}

We now  slightly modify an example from Kawamura (see Example~3.3 of \cite{kawamura}).  {Our main use for this example will be }to construct a $\Lambda$-semibranching function system of a double 2-graph $\Lambda$; {see Example \ref{ex-3.3-Kawamura_product} below.} 
(Note that one can easily modify these examples further in order   to obtain examples with quadratic nonconstant Radon-Nikodym derivative, see \cite{kawamura}.)

\begin{example}
\label{ex-3.3-Kawamura_modified} 
Let $A$ be a vertex matrix given by
\[
{\color{black} A= \begin{pmatrix} 1 & 1 \\
 1 & 0
\end{pmatrix}},
\]
The associated 1-graph $E$ is 
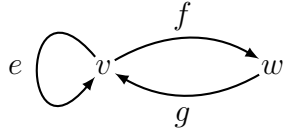
\[
\begin{tikzpicture}[scale=1.5]
 \node[inner sep=0.5pt, circle] (v) at (0,0) {$v$};
    \node[inner sep=0.5pt, circle] (w) at (1.5,0) {$w$};
    \draw[-latex, thick] (v) edge [out=130, in=230, loop, min distance=30, looseness=2.5] (v);
\draw[-latex, thick] (v) edge [out=30, in=150] (w);
\draw[-latex, thick] (w) edge [out=210, in=-30] (v);
\node at (-0.78, 0) {\color{black} $e$}; 
\node at (0.7, 0.45) {\color{black} $f$};
\node at (0.7, -0.45) {\color{black} $g$};
\end{tikzpicture}
\]

Let $X= [0,1]$ be endowed with the  Lebesgue $\sigma$-algebra and measure $\mu$. Fix a number $a\in (0,1)$ and define
\[
{\color{black} D_v = (0,a ),\quad D_w = (a, 1) \subseteq X.}
\]

Define prefixing maps for the edges $e,f$ and $g$ by
\[\begin{split}
&\tau_{e}(x)=\frac{x}{2}\quad \text{for}\;\; x\in D_v,\\
&\tau_{f}(x)=\frac{(1-a)x}{a}+a\quad \text{for}\;\; x\in D_v,\\
&\tau_{g}(x)=-\frac{ax}{2(a-1)}+\frac{a(2a-1)}{2(a-1)} \quad \text{for}\;\; x\in D_w.
\end{split}\]
Then 
\[
R_{e}=(0,\frac{a}{2}),\quad R_f=(a,1), \quad R_g=(\frac{a}{2}, a).
\]
Then the ranges are mutually disjoint and
\[
\mu(D_v\setminus (R_{e}\cup R_{g}))=0\quad \text{and}\quad D_w=R_{f}.
\]
Now define a coding map $\tau^1:X\to X$  by
\[
\tau^1(x)=\begin{cases} \tau^{-1}_{e}\quad \text{for}\;\; x\in R_{e}\\
\tau^{-1}_{f}\quad \text{for}\;\; x\in R_{f}\\
\tau^{-1}_g\quad \text{for}\;\; x\in R_{g}
\end{cases}
\]
Since all of the prefixing maps are linear, the corresponding Radon-Nikodym derivatives are constant and positive. Thus, $\{\tau_{e}, \tau_{f},\tau_{g}\}$ is a semibranching function system on $[0,1]$ with  coding map $\tau^1$.
\end{example}

{\subsection{Representations of double graphs and $\Lambda$-semibranching function systems}
\label{sec:double}
{In keeping with our focus on examples in this section, we indicate how to construct $\Lambda$-semibranching function systems for double graphs (Definition \ref{def-double-graph} below) by showing how to use  Example \ref{ex-3.3-Kawamura_modified}
to build a $\Lambda$-semibranching function system for the associated double graph.  Having clarified the construction via this example, we show in Proposition \ref{prop:double-graph-SBFS} how to construct a $\Lambda$-semibranching function system for a general double graph.}

\begin{defn}
\label{def-double-graph}
Let $E = (E^0, E^1)$ be a 1-graph.  We define the \emph{double graph} of $E$ to be the 2-graph $\Lambda_E$ with $\Lambda_E^0 = E^0$ and $\Lambda^{e_i}_E \cong E^1$ for $i=1, 2$; the factorization rules are trivial.  That is, if $e, f \in E^1$ and $s(e) = r(f)$, denote by $e^i ( f^i)$ the copy of $e$ (respectively $f$) in $\Lambda^{e_i}_E$.  Then we define 
\[ e^1 f^2 = e^2 f^1. \]
Note that all pairs of composable edges in $\Lambda_E^{e_1} \times \Lambda_E^{e_2}$ are of the form $e^1 f^2$ for $e, f$ as above, so the formula above completely defines the factorization rules for a double graph.
\end{defn}

\begin{example}
\label{ex-3.3-Kawamura_product}
{We now describe the double graph associated to the 1-graph $E$ with adjacency matrix
$
{ A= \begin{pmatrix} 1 & 1 \\
 1 & 0
\end{pmatrix}},
$
discussed in Example \ref{ex-3.3-Kawamura_modified} above.}
 The double graph $\Lambda_E$ is a $2$-graph  whose $1$-skeleton is given by
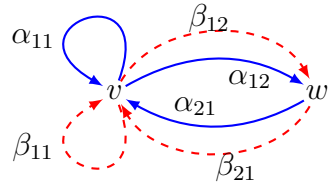
\[
\begin{tikzpicture}[scale=1.8]
 \node[inner sep=0.5pt, circle] (v) at (0,0) {$v$};
    \node[inner sep=0.5pt, circle] (w) at (1.5,0) {$w$};
    \draw[-latex, thick, blue] (v) edge [out=70, in=150, loop, min distance=25, looseness=2.5] (v);
\draw[-latex, thick, blue] (v) edge [out=30, in=150] (w);
\draw[-latex, thick, blue] (w) edge [out=210, in=-30] (v);
\draw[-latex, thick, red, dashed] (v) edge [out=-70, in=-150, loop, min distance=25, looseness=2.5] (v);
\draw[-latex, thick, red, dashed] (v) edge [out=60, in=120] (w);
\draw[-latex, thick, red, dashed] (w) edge [out=240, in=-60] (v);
\node at (-0.6, 0.4) {\color{black} $\alpha_{11}$}; 
\node at (-0.6, -0.4) {\color{black} $\beta_{11}$}; 
\node at (1, 0.1) {\color{black} $\alpha_{12}$};
\node at (0.6, -0.1) {\color{black} $\alpha_{21}$};
\node at (0.7, 0.55) {\color{black} $\beta_{12}$};
\node at (0.9, -0.55) {\color{black} $\beta_{21}$};
\end{tikzpicture}
\]
and the factorization rules  are given by
\[ \alpha_{11} \beta_{11} = \beta_{11} \alpha_{11}, \quad 
\alpha_{11} \beta_{21}=\beta_{11}\alpha_{21}, \quad \alpha_{21}\beta_{12}=\beta_{21}\alpha_{12},\quad \alpha_{12}\beta_{21}=\beta_{12}\alpha_{21}.
\]

As in Example~\ref{ex-3.3-Kawamura_modified}, let $X=(0,1)$ with Lebesgue measure $\mu$. Fix $a\in (0,1)$, then define $D_v, D_w$ and $\tau_{\alpha_{11}}:= \tau_e$, $\tau_{\alpha_{12}} := \tau_f$, $\tau_{\alpha_{21}}:= \tau_g$ as in Example~\ref{ex-3.3-Kawamura_modified}. Then $\{\tau_{\alpha_{11}},\tau_{\alpha_{12}},\tau_{\alpha_{21}}\}$ forms an $E$-semibranching function system on $(X,\mu)$ with the coding map $\tau^{e_1}:=\tau^1$, where $\tau^1$ is the coding map given in Example~\ref{ex-3.3-Kawamura_modified}. Similarly, for red edges $\beta_{11}$, $\beta_{12}$, $\beta_{21}$, define the prefixing maps by
\[
\tau_{\beta_{11}}:=\tau_{\alpha_{11}} = \tau_e, \quad \tau_{\beta_{12}}:=\tau_{\alpha_{12}} = \tau_f, \quad \tau_{\beta_{21}}:=\tau_{\alpha_{21}} = \tau_g.
\]
Then $\{\tau_{\beta_{11}},\tau_{\beta_{12}},\tau_{\beta_{21}}\}$ forms an $E$-semibranching function system on $(X,\mu)$ with the coding map $\tau^{e_2}:=\tau^{e_1} = \tau^1$. Then Conditions (i), (ii), (iv) and (v) of Theorem~\ref{thm:SBFS-edge-defn} are automatically satisfied by construction. To see that Condition (iii) is also satisfied, note that it is straightforward to see 
\[\begin{split}
&R_{\beta_{12}}\subseteq D_w=D_{\alpha_{21}}, \quad R_{\alpha_{21}}\subseteq D_v=D_{\beta_{11}} \quad R_{\beta_{21}}\subseteq D_v=D_{\alpha_{12}}\\
&R_{\alpha_{12}}\subseteq D_w=D_{\beta_{21}}, \quad R_{\beta_{21}}\subseteq D_v=D_{\alpha_{12}}, \quad R_{\alpha_{21}}\subseteq D_v=D_{\beta_{12}}.
\end{split}\]

Also by construction it is straightforward to see the prefixing maps satisfy the factorization rules:
\[
\tau_{\alpha_{11}}\circ \tau_{\beta_{21}}=\tau_{\beta_{11}}\circ \tau_{\alpha_{21}}, \quad 
\tau_{\alpha_{21}}\circ \tau_{\beta_{12}}=\tau_{\beta_{21}}\circ \tau_{\alpha_{12}},\quad 
\tau_{\alpha_{12}}\circ \tau_{\beta_{21}}=\tau_{\beta_{12}}\circ \tau_{\alpha_{21}}.
\]
Therefore, Condition (iii) is satisfied, and hence $\{\tau_{\alpha_{11}},\tau_{\alpha_{12}},\tau_{\alpha_{21}}\}$ and $\{\tau_{\beta_{11}},\tau_{\beta_{12}},\tau_{\beta_{21}}\}$ give a  $\Lambda_E$-semibranching function system on $(X,\mu)$. 

\end{example}

More generally, we have the following:

\begin{prop}\label{prop:double-graph-SBFS}
 Let $E$ be a finite $1$-graph. Suppose there is an $E$-semibranching function system on a measure space $(X,\mathcal{F},\mu)$ with prefixing maps $\{\tau_f: f\in E^1\}$ and coding map $\tau^1$. Let $\Lambda$ be the double graph associated to $E$. Let $\phi_i:(\Lambda^0, \Lambda^{e_i})\to (E^0, E^1)$ be the graph isomorphism for $i=1,2$.
For $f^i\in \Lambda^{e_i}$, define the  prefixing map by
\[\begin{split}
&\tau_{f^i}:=\tau_{\phi_i(f^i)},
\end{split}\]
and define the corresponding coding map by $\tau^{e_i}:=\tau^1$ for $i=1,2$.
Then the collection of prefixing maps $\{\tau_{f^i}: f^i\in \Lambda^{e_i}\}_{i=1,2}$ gives a $\Lambda$-semibranching function system on $(X,\mu)$ with coding maps $\{\tau^{e_i}:X\to X\}_{i=1,2}$. 
\end{prop}
\begin{proof}
Conditions (i) and (iv) of Theorem \ref{thm:SBFS-edge-defn} are satisfied by construction, and Conditions (ii) and (v) hold thanks to the hypothesis that the prefixing and coding maps $\{\tau_f, \tau^1\}$ form an $E$-semibranching function system on $(X, \mu)$. For Condition (iii), we simply observe that if $\lambda \alpha = \nu \beta$ for $\lambda, \beta \in \Lambda^{e_1}, \alpha, \nu \in \Lambda^{e_2}$, then the factorization rule in $\Lambda$ implies that 
\[\phi_1(\lambda) = \phi_2(\nu) \text{ and } \phi_1(\beta) = \phi_2(\alpha).\]
Thus, $\tau_{\lambda} \circ \tau_\alpha = \tau_{\phi_1(\lambda)} \circ \tau_{\phi_2(\alpha)} = \tau_{\phi_2(\nu)} \circ \tau_{\phi_1(\beta)} = \tau_\nu \circ \tau_\beta$ as desired, so Condition (iii) also holds.
\end{proof}

\subsection{$\Lambda$--semibranching function systems for  product graphs and their representations}
\label{sec:product}
{In this section we describe the well-known procedure for constructing products of higher-rank graphs, and show how to construct a $\Lambda$-semibranching function system for the product graph, starting from $\Lambda$-semibranching function systems for the initial graphs.}

\begin{defn}(See \cite[Proposition~1.8]{KP} and \cite[Proposition~5.1]{Kang-Pask})
Let $(\Lambda_1, d_1)$ and $(\Lambda_2, d_2)$ be $k_1$- and $k_2$-graphs respectively. We  define the \emph{product graph}  $(\Lambda_1\times\Lambda_2, d_1\times d_2)$ to consist of the product category $\Lambda_1 \times \Lambda_2$,  with degree map $d_1\times d_2 : \Lambda_1\times \Lambda_2\to \N^{k_1+k_2}$ given by $d_1\times d_2(\lambda_1,\lambda_2)=(d(\lambda_1),d(\lambda_2))\in \N^{k_1}\times \N^{k_2}$ for $\lambda_1\in\Lambda_1$ and $\lambda_2\in \Lambda_2$.
\end{defn}
The product graph $\Lambda_1\times \Lambda_2$ is a $(k_1+k_2)$-graph by Proposition~1.8 of \cite{KP}  and  
\[C^*(\Lambda_1\times\Lambda_2)\cong C^*(\Lambda_1)\otimes C^*(\Lambda_2)\] by Corollary~3.5 of \cite{KP}. 
Also Theorem~5.3 of \cite{Kang-Pask} implies that $\Lambda_1\times\Lambda_2$ is a finite $(k_1+k_2)$-graph with no sources if and only if $\Lambda_i$ is a finite $k_i$-graph with no sources for $i=1,2$.  

Notice that $(\Lambda_1 \times \Lambda_2)^0 = \Lambda_1^0 \times \Lambda_2^0$; moreover, if $e_i$ is a basis vector for $\N^{k_1}$,
\[(v_1, v_2)( \Lambda_1 \times \Lambda_2)^{e_i} (w_1, w_2) = \begin{cases} \emptyset, & w_2 \not= v_2 \\ v_1 \Lambda_1^{e_i} w_1 , & w_2 = v_2. \end{cases} \]
Similarly,  if $e_j$ is a basis vector for $\N^{k_2}$, then  
\[(v_1, v_2)( \Lambda_1 \times \Lambda_2)^{e_j} (w_1, w_2) = \begin{cases} \emptyset, & w_1 \not= v_1 \\ v_2 \Lambda_2^{e_j} w_2 , & w_1= v_1. \end{cases} \]
 Thus, if we choose an ordering of the vertices of $\Lambda_i$ for $i= 1, 2$ and then list the vertices of $\Lambda_1 \times \Lambda_2$ lexicographically, the vertex matrices $A_i$ of $\Lambda_1 \times \Lambda_2$ are given by 
 \[A_i = M_i \otimes I_{k_2}, \ 1 \leq i \leq k_1; \qquad A_{k_1 + j} = I_{k_1} \otimes N_j, \ 1 \leq j \leq k_2,\]
 where $\{M_i\}_{i=1}^{k_1}$ are the vertex matrices for $\Lambda_1$ and $\{N_j\}_{j=1}^{k_2}$ are the vertex matrices for $\Lambda_2$.
 
 To describe the factorization rule in $\Lambda_1 \times \Lambda_2$, suppose that $\lambda \in (v_1, v_2) (\Lambda_1 \times \Lambda_2)^{e_j}(w_1, v_2)$ where $e_j$ is a basis vector for $\N^{k_1}$, and $\nu \in (w_1, v_2) (\Lambda_1 \times \Lambda_2)^{e_\ell}(w_1, w_2)$ where $e_\ell$ is a basis vector for $\N^{k_2}$.  
So $\lambda$ and $\nu$ are composable since $s(\lambda)=(w_1,v_2)=r(\nu)$, and
 then $\lambda$ corresponds to a morphism $\lambda_1 \in v_1 \Lambda_1 w_1$, and $\nu$ corresponds to a morphism $\nu_2 \in v_2 \Lambda_2 w_2$.  
 
 Note that $\nu_2$ also induces $\tilde{\nu}\in (v_1,v_2)(\Lambda_1\times \Lambda_2)^{e_\ell} (v_1,w_2)$, and $\lambda_1$ induces $\tilde \lambda \in (v_1, w_2)(\Lambda_1 \times \Lambda_2)^{e_j}(w_1, w_2)$.  The factorization rule is then given by 
 \[ \lambda \nu = \tilde{\nu} \tilde{\lambda}.\]
 
 For example, let $E$ be the 1-graph of Example \ref{ex-3.3-Kawamura_modified}.  Then the product graph $\Lambda :=E \times E$ has skeleton
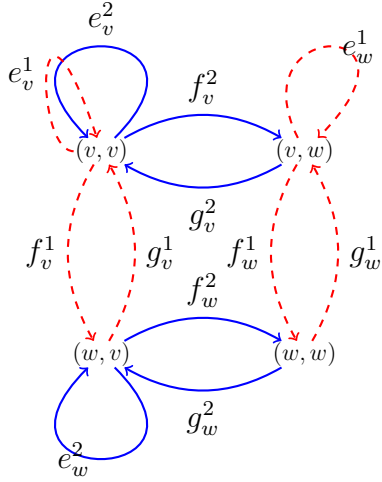
\[
\begin{tikzpicture}[scale=1.8]
 \node[inner sep=0.5pt, scale=0.75] (vv) at (0,0) {$(v, v)$};
    \node[inner sep=0.5pt, scale=0.75] (vw) at (1.5,0) {$(v, w)$};
 \node[inner sep=0.5pt,scale=0.75] (wv) at (0,-1.5) {$(w, v)$};
    \node[inner sep=0.5pt, scale=0.75] (ww) at (1.5,-1.5) {$(w, w)$};
    \draw[->, thick, blue] (vv) ..  node[above]{\color{black}$e^2_{v}$} controls (1,1) and (-1, 1) ..  (vv);
    \draw[dashed, ->, thick, red] (vv) .. node[left]{\color{black}$e^1_{v}$} controls (-0.5, 0) and (-0.5, 1.5) .. (vv);
\draw[->, thick, blue] (vv) to[bend left]  node[above] {\color{black}$f^2_{v}$}(vw);
\draw[->, thick, blue] (vw) to[bend left] node[below]{\color{black}$g^2_{v}$}(vv);
\draw[->, thick, blue] (wv) to[bend left] node[above]{\color{black}$f^2_w$} (ww);
\draw[->, thick, blue] (ww) to[bend left] node[below]{\color{black}$g^2_w$} (wv);
\draw[->, thick, dashed, red] (vw) to[bend right] node[left]{\color{black}$f^1_w$}(ww);
\draw[->, thick, dashed, red] (ww) to[bend right] node[right]{\color{black}$g^1_w$}(vw);
\draw[->, thick, dashed, red] (vv) to[bend right] node[left]{\color{black}$f^1_v$}(wv);
\draw[->, thick, dashed, red] (wv) to[bend right] node[right]{\color{black}$g^1_v$}(vv);
\draw[->, thick, dashed, red] (vw) ..  node[right]{\color{black}$e^1_w$} controls (1, 1) and (2.5, 1) .. (vw);
\draw[->, thick, blue] (wv) .. node[left]{\color{black}$e^2_w$} controls (1, -2.5) and (-1, -2.5) .. (wv);
\end{tikzpicture}
\]
Here the dashed red edges correspond to the first 
copy of $E$ and the solid blue edges correspond to the second. Note that for any $j \in \{ e, f, g\}, \ u \in \{ v, w\}$ we have 
\[ r(j^1_u) = (r(j), u), \ s(j^1_u) = (s(j), u); \qquad r(j^2_u) = (u, r(j)), \ s(j^2_u) = (u, s(j)).\]
 The factorization rule in this case is given by (for $\{i , j\}=\{ 1, 2\}$, $i\ne j$)
\[e^1_v e^2_v = e^2_v e^1_v \quad e^i_v g^j_v = g^j_v e^i_w \quad f^i_v e^j_v = e^j_w f^i_v \quad g^j_w f^i_w = f^i_v g^j_v \quad f^1_w f^2_v = f^2_w f^1_v \quad g^1_v g^2_w= g^2_v g^1_w.\]

Note that there is no factorization when $i=j$ because then the edges are the same color.

Let $\mu$ denote Lebesgue measure on $(0,1) \subseteq \R$. We can define a $\Lambda$-semibranching function system on $((0,1)^2, \mu\times \mu)$ by using the $E$-semibranching function system of Example \ref{ex-3.3-Kawamura_modified}.  Namely, fix $a \in (0,1)$ and define 
\[D_{(v,v)} = (0,a)^2; \quad  D_{(v,w)} = (0,a) \times (a, 1); \quad D_{(w,v)} = (a, 1) \times (0,a) ; \quad D_{(w,w)} = (a, 1)^2,\]
and for any edge $j$ and for $u = v,w$, define  
\[ \tau_{j^1_u}(x,y) = \chi_{D_u}(y) \cdot (\tau_j(x), y) ; \quad \tau_{j^2_u}(x,y) = \chi_{D_u}(x) \cdot (x,\tau_j(y));\]
\[\tau^{e_1}(x, y) = ( \tau^1(x),y), \quad  \tau^{e_2}(x, y) = (x, \tau^1( y)). \]
Thus, $R_{j^1_u} =  R_j \times D_u$ and $R_{j^2_u} = D_u \times R_j$.  

The fact that our original $E$-semibranching function system satisfies Conditions (i) - (v) of Theorem \ref{thm:SBFS-edge-defn} implies immediately that the maps $\{ \tau_{j^i_u}, \tau^{e_i}\}$ satisfy Condition (v) of said Theorem.
It is evident that $\tau^{e_1}$ and $\tau^{e_2}$ commute, so Condition (iv) of Theorem \ref{thm:SBFS-edge-defn} is also satisfied.  Conditions (i) and (ii) follow from our description above of the sets $D_{(p, q)}$ and the observation that the domain of $\tau_{j^1_u}$ is $D_{s(j)} \times D_u = D_{(s(j), u)} = D_{s(j^1_u)}$, and the domain of $\tau_{j^2_u}$ is $D_{(u, s(j))} = D_{s(j^2_u)}$, and Condition (iii) follows from a quick computation.

In more generality, we have:

 \begin{prop}\label{prop:product-graph-SBFS}
For $i=1,2$, let $\Lambda_i $ be a $k_i$-graph with a $\Lambda_i$-semibranching function system $\{\tau^i_{\lambda}: \lambda \in \Lambda_i\}$ on $(X_i, \mu_i)$, with coding maps $\tau^{i, e_j}$ for $1 \leq j \leq k_i$.  For $\lambda \in \Lambda_1$ with $|d(\lambda)|=1$, let $\{\lambda^1_v: v \in \Lambda_2^0\}$ denote the corresponding edges in $\Lambda_1 \times \Lambda_2$, with $s(\lambda_v^1) = (s(\lambda), v)$ and $r(\lambda_v^1) = (r(\lambda), v)$; similarly for $\nu\in \Lambda_2$ but with $s(\nu^2_u)=(u,s(\nu))$, $r(\nu^2_u)=(u,r(\nu))$, where $u\in \Lambda_1^0$.    For $w \in \Lambda_1^0, v \in \Lambda_2^0$, define $D_{w,v} \subseteq X_1 \times X_2$ by 
\[D_{w, v} := D_w \times D_v \subseteq X_1 \times X_2.\]
 Then, define prefixing maps $\tau_{\lambda_v^1}, \tau_{\eta_w^2}$ on $X_1 \times X_2$  by 
\[\tau_{\lambda^1_v} (x, y) := \chi_{D_v}(y ) \cdot (\tau_\lambda^1(x), y), \quad  \tau_{\eta^2_w}(x,y) := \chi_{D_w}(x) \cdot (x, \tau^2_\eta(y) ),\]
and coding maps $\tau^{e_j}(x, y) = (\tau^{1, e_j}(x), y)$ if $1 \leq j \leq k_1$, or $\tau^{e_j}(x,y) = (x, \tau^{2, e_{j -k_1}}(y))$ if $k_1 < j \leq k_1 + k_2$.
The prefixing maps $\{\tau_{\lambda_v^1}, \tau_{\eta_w^2}: v \in \Lambda_1^0, w \in \Lambda_2^0, |d(\lambda)| = |d(\eta)| = 1\}$ and the coding maps $\{ \tau^{e_j}\}$ satisfy Conditions (i) - (v) of Theorem \ref{thm:SBFS-edge-defn} and thus give rise to a $\Lambda_1 \times \Lambda_2$-semibranching function system.
\end{prop}
\begin{proof}
By construction, $D_{\lambda_u^1} = D_{s(\lambda)} \times D_u = D_{s(\lambda), u} =D_{s(\lambda_u^1)}$ and $D_{\lambda_u^2} = D_u \times D_{s(\lambda)} = D_{u, s(\lambda)}=D_{s(\lambda^2_u)}$; since we began with $\Lambda_i$-semibranching function systems on $X_i$, for $i=1,2$, Conditions (i), (ii), (iv), and (v) of Theorem \ref{thm:SBFS-edge-defn} immediately follow.
To see Condition (iii), observe  that any pair $\lambda \in \Lambda_1, \nu \in \Lambda_2$ gives rise to exactly two composable pairs in $\Lambda_1 \times \Lambda_1$, namely $( \lambda_{r(\nu)}^1,\nu_{s(\lambda)}^2)$ and $(\nu_{r(\lambda)}^2 ,\lambda_{s(\nu)}^1)$ since $s(\lambda^1_{r(\nu)})=(s(\lambda),r(\nu))=r(\nu^2_{s(\lambda)})$ and $s(\nu^2_{r(\lambda)})=(r(\lambda),s(\nu))=r(\lambda^2_{s(\nu)})$.
The factorization rule for product graphs implies that
 \[ \lambda_{r(\nu)}^1\nu_{s(\lambda)}^2 = \nu_{r(\lambda)}^2 \lambda_{s(\nu)}^1 \in \Lambda_1 \times \Lambda_2.\]
 Consequently, $\tau_{\lambda_{r(\nu)}^1} \circ \tau_{\nu_{s(\lambda)}^2} = (\tau_\lambda, \tau_\nu) = \tau_{\nu_{r(\lambda)}^2} \circ \tau_{\lambda_{s(\nu)}^1 }$, so Condition (iii) holds.
 
\end{proof}

\section{New classes  of $\Lambda$-semibranching function systems associated to probability measures on $\Lambda^\infty$}
\label{sec:examples_gen_measure}

In this section, {
we change our focus to $\Lambda$-semibranching functions on the infinite path space $\Lambda^\infty$.  We indicate the variety of possible measures on $\Lambda^\infty$ which give rise to $\Lambda$-semibranching function systems, by using Lemmas \ref{lem-Kolm}, \ref{lem-RN-der-comp-limit}, and \ref{lem:measure} to construct many such measures.  These measures, and the associated $\Lambda$-semibranching representations,  will have significance for symbolic dynamics, for probability theory of Markov measures, and quasi-stationary Markov measures, and the associated stochastic processes.

To be precise, in the pages that follow, }
 we describe a variety of examples of $\Lambda$-semibranching function systems on measure spaces of the form $(\Lambda^\infty, \mathcal{B}_\Lambda, \mu)$, using the standard prefixing and coding maps 
\begin{equation}
\label{eq:standard-prefixing-coding} \sigma_\lambda(x) := \lambda x, \qquad \sigma^n(x)(p,q) := x(p+n, q+n),
\end{equation}
 and compare them to the standard $\Lambda$-semibranching function system of Example \ref{example:SBFS-M}.
We begin by describing examples which arise from  Kakutani's product measure construction \cite{Kaku}.  All of the $\Lambda$-semibranching function systems on $(\Lambda^\infty, \mu)$ that we obtain in this way are equivalent to the standard $\Lambda$-semibranching function system, in the sense that the measure $\mu$ is mutually absolutely continuous with respect to the measure $M$ of Equation \eqref{eq:M}.  

Moreover, as Section 3 of \cite{dutkay-jorgensen-monic} shows how to use Markov measures  to construct many inequivalent representations of $\mathcal O_N$,   we also extend these constructions in this section; this is possible since the $2$-graphs we are considering have infinite path spaces homeomorphic to  the infinite path space associated to $\mathcal{O}_N$, or disjoint unions of them. (For the definition of the Markov measures we are using, see  Definition 3.1 of \cite{dutkay-jorgensen-monic}, and also Definition \ref{def-Markov-measure}; for a generalized definition of Markov measures, see \cite{bezuglyi-jorgensen}.)   For a family of higher-rank graphs whose infinite path space can be constructed from that of $\mathcal O_N$, we  use this technique to obtain   $\Lambda$-semibranching function systems on $(\Lambda^\infty, \mu)$ where $\mu$ and $M$ are mutually singular measures.

First, we record {in Theorem \ref{thm-lambda-sbfs-on-the-inf-path-space-via-a-measure}}
a straightforward consequence of the definition (Definition \ref{def-lambda-SBFS-1}) of a $\Lambda$-semibranching function system.  {Theorem \ref{thm-lambda-sbfs-on-the-inf-path-space-via-a-measure} simplifies the work of checking when a probability measure on $\Lambda^\infty$ gives rise to a $\Lambda$-semibranching function system.} In the proof that follows, it will be useful to 
recall that the topology on the infinite path space $\Lambda^\infty$ of a higher-rank graph $\Lambda$ is generated by the cylinder sets $Z(\lambda)$, for $\lambda \in \Lambda$:
\[ Z(\lambda) = \{ x \in \Lambda^\infty: x(0, d(\lambda)) = \lambda\} = \{  \lambda y: y \in \Lambda^\infty\}.\]
In fact, the proof of Lemma 4.1 from \cite{FGKP} establishes that the topology on $\Lambda^\infty$ (and hence the Borel $\sigma$-algebra $\mathcal{B}_o(\Lambda^\infty)$) is generated by the ``square'' cylinder sets
\[ \{ Z(\lambda): d(\lambda) = (n, \ldots, n) \text{ for some } n \in \N\};\]
given any cylinder set $Z(\nu)$ with $d(\nu) \leq (n, \ldots, n)$, let
\[ I = \{ \lambda_i \in \Lambda: d(\nu \lambda_i)  = (n, \ldots, n) \}.\]
Then $Z(\lambda) = \bigsqcup_{\lambda_i \in I} Z(\nu \lambda_i)$ is a disjoint union of square cylinder sets.
\begin{thm} 
\label{thm-lambda-sbfs-on-the-inf-path-space-via-a-measure}	
Let $\Lambda$ be a finite, strongly connected $k$-graph.
Suppose that the infinite path space $\Lambda^\infty$ of $\Lambda$  is endowed with a probability measure $p$ satisfying the following properties:
\begin{itemize}
\item[(a)] The standard  prefixing and coding maps $\{\sigma_\lambda\}_{\lambda\in \Lambda}$, $\{\sigma^m\}_{m \in \N^k}$ on $\Lambda^\infty$ given in Equation \eqref{eq:standard-prefixing-coding} are measurable maps; 
\item[(b)] For all $v\in \Lambda^0$, we have $p(Z(v))>0$.
\item[(c)] Each of the edge prefixing operators $(\sigma_\lambda)_{\lambda\in \Lambda^{e_i}}$ has positive Radon-Nikodym derivative,
\[ \Phi_{\sigma_\lambda}:= \frac{d(p \circ \sigma_\lambda)}{dp} > 0, \text{ p. a.e. on } Z(s(\lambda)).\]
\end{itemize}
Then the maps $\sigma^n, \sigma_\lambda$ endow $(\Lambda^\infty, p)$ with a  $\Lambda$-semibranching function system. 
\end{thm}	

\begin{proof}
The proof is straightforward and completely analogous to the proof of Proposition 3.4 from \cite{FGKP}.  The only argument which differs slightly is to see that all Radon-Nikodym derivatives $\Phi_{\sigma_\lambda}$ are positive for any $\lambda\in\Lambda$.  To that end, let $\eta \in \Lambda$ and write $\eta$ as a concatenation of edges, $\eta = \eta_1 \cdots \eta_n$.  Then, since $\sigma_\eta  = \sigma_{\eta_1} \circ \cdots \circ \sigma_{\eta_n}$, and the fact that $p\circ \sigma_{\eta_j}<<p$ for all $1\le j\le n$ implies that $p\circ \sigma_{\eta_j}\circ \dots \circ \sigma_{\eta_n} << p\circ \sigma_{\eta_{j+1}}\circ\dots \circ \sigma_{\eta_n}$ for $1\le j\le n-1$, standard properties of the Radon-Nikodym derivative allow us to rewrite 
\[ \frac{d(p \circ \sigma_\eta)}{dp} = \frac{d(p \circ \sigma_{\eta_1}\circ \cdots \circ \sigma_{\eta_n})}{d(p \circ \sigma_{\eta_2} \circ \cdots \circ \sigma_{\eta_{n}})} \frac{d(p \circ \sigma_{\eta_2}\circ \cdots \circ \sigma_{\eta_{n}})}{d(p \circ \sigma_{\eta_3} \circ \cdots \circ \sigma_{\eta_{n}})} \cdots\frac{d(p \circ \sigma_{\eta_{n-1}}\circ  \sigma_{\eta_n})}{d(p \circ \sigma_{\eta_n} )}  \frac{d(p \circ \sigma_{\eta_n})}{dp} .\]
Now, observe that  
\[\frac{d(p \circ \sigma_{\eta_{n-1}}\circ  \sigma_{\eta_n})}{d(p \circ \sigma_{\eta_n} )} = \frac{d(p \circ \sigma_{\eta_{n-1}})}{dp} \circ \sigma_{\eta_n}.\]

Similarly, by induction we can rewrite 
\[ \frac{d(p \circ \sigma_\eta)}{dp} =\left( \frac{d(p \circ \sigma_{\eta_1})}{dp} \circ \sigma_{\eta_2 \cdots \eta_n}\right)  \left( \frac{d(p \circ \sigma_{\eta_2})}{dp} \circ \sigma_{\eta_3\cdots \eta_n} \right)   \cdots \left( \frac{d(p \circ \sigma_{\eta_{n-1}})}{dp} \circ \sigma_{\eta_n}\right) \left( \frac{d(p \circ \sigma_{\eta_n})}{dp} \right) .\]
Observe that for any edge $\nu$ and any $\lambda\in \Lambda$ with $s(\nu) = r(\lambda)$, $ \frac{d(p \circ \sigma_\nu)}{dp}$ is positive, by hypothesis, on $Z(s(\nu)) \supseteq Z(\lambda)  = R_{\lambda}$.
In our case, since $\eta=\eta_1\dots \eta_n$, we have $Z(s(\eta_1))\supseteq Z(s(\eta_2))\supseteq \dots Z(s(\eta_n))=Z(s(\eta))$. Also the fact that each Radon-Nikodym derivative $\Phi_{\sigma_{\eta_j}}$ is positive $p$-a.e.~on $Z(s(\eta_j))$ implies that $\Phi_{\sigma_{\eta_j}}\circ \sigma_{\eta_{j+1}\dots \eta_n}$ is positive $p$-a.e.~on $Z(s(\eta))$ for $1\le j\le n-1$. 
Thus, $\frac{d(p \circ \sigma_\eta)}{dp} $ is the product of positive functions on $Z(s(\eta))$, and hence is positive on $Z(s(\eta))$. 
\end{proof}

\subsection{Kakutani-type probability  measures on $\Lambda^\infty$ }
\label{sec:examples_ab-cont-_measure}

We now apply Theorem \ref{thm-lambda-sbfs-on-the-inf-path-space-via-a-measure} to the 2-graphs of Example \ref{exonevtwoe} and Example \ref{ex:sbfs-3v8e}.  To be precise, we use a product measure construction inspired by Kakutani 
to build a Borel measure on the infinite path space $\Lambda^\infty$ which satisfies the hypotheses of Theorem \ref{thm-lambda-sbfs-on-the-inf-path-space-via-a-measure}.  It turns out that any such product measure is equivalent to the canonical measure $M$ on $\Lambda^\infty$, and indeed gives rise to a $\Lambda$-semibranching representation of $C^*(\Lambda)$ which is equivalent to the representation on $L^2(\Lambda^\infty, M)$ described in Proposition 3.4 and Theorem 3.5 of \cite{FGKP}. 
Note that a $\Lambda$-semibranching representation of $C^*(\Lambda)$ is a representation of $C^*(\Lambda)$ on $L^2(X,\nu)$ associated to a $\Lambda$-semibranching function system on a measure space $(X,\nu)$.

Recall from Example \ref{exonevtwoe} and Remark \ref{rmk:rainbow} that for the $2$-graph with one vertex $v$, and two blue edges $f_1$ and $f_2$ and one red edge $e$ satisfying the factorization relations
$$ef_1=f_2e\;\;\;\text{and}\;\;\;ef_2=f_1 e,$$
for any $\xi\in \Lambda^{\infty},$ we can write $\xi$ uniquely as 
$$\xi\equiv\;eg_1eg_2eg_3\cdots eg_n\cdots $$
where $g_i\in \{f_1,f_2\}.$ 
We now fix a sequence of positive numbers $\{p_n=\frac{1}{2}+\gamma_n\}_{n=1}^{\infty}$, where $|\gamma_n|<\frac{1}{2}$ such that  $p_n < 1 $ for all $n$, $\lim_{n\to \infty}p_n=\frac{1}{2},$
and $\sum_{n=1}^{\infty}|\gamma_n|<\infty.$  
Set 
$$q_n=1-p_n= \frac{1}{2}-\gamma_n,$$
 and note that $\{q_n\}_{n=1}^{\infty}$ is also a sequence of positive numbers between $0$ and $1$ that tends to $\frac{1}{2}.$

For each $i \in \N$, define
\begin{equation}\label{eq:seq_alpha}
\alpha_i\;=\begin{cases}p_i=\frac{1}{2}+\gamma_i& \text{if} \;g_i=f_1, \\
          q_i=\frac{1}{2}-\gamma_i\; & \text{if}\;g_i=f_2,\;1\leq i\leq n.
          \end{cases}
\end{equation}
Then we define a function $\mu$ on square cylinder sets $Z(e g_1 e g_2 e \cdots g_n)$ inductively  by 
\[
\mu(Z(eg_1eg_2eg_3\cdots eg_n))=\mu(Z(eg_1eg_2eg_3\cdots eg_{n-1}))\, \alpha_n.
\]
Also we define an empty product by 1, so $\mu(Z(v))=1$.
Then one can check that $\mu$ satisfies the following.
\begin{equation}\label{eq:mu_n}
\mu(Z(eg_1eg_2eg_3\cdots eg_n))=\prod_{i=1}^n\alpha_i.
\end{equation}

\begin{prop}\label{prop:example_exonevtwoe}
Let $\Lambda$ be the 2-graph of Example  \ref{exonevtwoe}. Let $(\alpha_n)_n$ be a sequence given by \eqref{eq:seq_alpha}, and $\mu$ be the function  associated to $(\alpha_n)_n$ as in \eqref{eq:mu_n}.  Then 
\begin{itemize}
\item[(a)] The function $\mu$ extends uniquely to a Borel probability measure on $\Lambda^\infty$, and the standard prefixing and coding maps $(\sigma_\lambda, \sigma^n)$ endow  $(\Lambda^\infty, \mu)$ with a $\Lambda$-semibranching function system. 
\item[(b)] Each such measure $\mu$ is equivalent to the Perron-Frobenius measure $M$ of Equation \eqref{eq:M}.  
\item[(c)] The $\Lambda$-semibranching representation of $C^*(\Lambda)$ on $L^2(\Lambda^\infty, \mu)$ is unitarily equivalent to the standard $\Lambda$-semibranching representation.  In particular, the $\Lambda$-semibranching representations on such measure spaces $L^2(\Lambda^\infty, \mu)$  are all unitarily equivalent.
\end{itemize}
\label{prop:prod-meas-onevtwoe}
\end{prop}

\begin{proof}
To see (a), note that $\mu$ is initially defined on a collection of sets which generates the topology on $\Lambda^\infty$, thus any measure extending $\mu$ is a Borel measure by definition. Moreover, Lemma~\ref{lem:measure} (equivalently, Lemma \ref{lem:kolmogorov})
implies that any function which is finitely additive on (square) cylinder sets gives rise to a measure on $\Lambda^\infty$ as long as $\Lambda$ is row-finite.  Finally, to see that $\mu$ is a probability measure we observe that $\mu(\Lambda^\infty) = \mu(Z(v))$ is the empty product and hence equal to 1 by definition.

Thus, to see that $\mu$ extends to a Borel probability measure on $\Lambda^\infty,$ it only remains to check that $\mu$ is finitely additive on square cylinder sets. If we define $h_i$ to equal $f_1$ when $g_i = f_2$, and vice versa (so that $h_i, g_i \in \{ f_1, f_2\}$ and $h_i \not= g_i$) then we have $Z(e g_1 \cdots e g_n) = Z(e g_1 \cdots eg_ne   g_{n+1}) \sqcup Z(e g_1 \cdots e g_n e h_{n+1})$, a disjoint union of cylinder sets. Therefore, 
\begin{align*} \mu(Z(e g_1 \cdots eg_n e   g_{n+1})) &+ \mu (Z(e g_1 \cdots eg_n e   h_{n+1})) \\
&= \mu(Z(e g_1 \cdots e g_n ))(1/2 + \gamma_{n+1}) + \mu(Z(e  g_1 \cdots e g_n  ))(1/2 - \gamma_{n+1}) \\
&= \mu (Z(e g_1 \cdots e g_n)).
\end{align*}
Arguing inductively, 
we conclude that $\mu$ is finitely additive on square cylinder sets, as claimed. 

We now check that $\mu$ satisfies the hypotheses of Theorem \ref{thm-lambda-sbfs-on-the-inf-path-space-via-a-measure}.  Since $\mu$ is a Borel measure and the maps $(\sigma_\lambda, \sigma^n)$ are continuous, they are measurable; and we observed above that $\mu(Z(v)) = 1$.
It remains to check that each of the edge prefixing operators, $\sigma_{f_1}, \sigma_{f_2}, \sigma_{e}$, has positive Radon-Nikodym derivative. To do so, we will use Lemma 	\ref{lemma-limit-RN}. 

Fix an infinite path $\xi \equiv e g_1 e g_2 e g_3 \cdots$.
Define $\ell_i\in \{0,1\}$ so that $\alpha_i = 1/2 + (-1)^{\ell_i} \gamma_i$, 
and let 
$m_i=1-\ell_i.$ For $N \in \N$, we let $\lambda_N=e g_1 \cdots e g_N$.
Then the factorization rule $e f_i = f_{i+1} e$ implies that 
\[\begin{split}
\sigma_{f_1}(Z(\lambda_{N}))&=\{\zeta =(\zeta_i) \in \Lambda^{\infty}:   \zeta_{2j-1}=e\;\;\text{for}\;\; 1\le j\le N,\;\zeta_2 = f_2, \; \zeta_{2i}=h_i,\;2\;\leq i\leq N\}\\
& = Z(e f_2 e h_1 \cdots e h_N).
\end{split}
\]
Since $g_i\ne h_i\in \{f_1,f_2\}$ as described above,
it follows that 
\[
\mu(\sigma_{f_1}(Z(\lambda_{N}))= (1/2 -\gamma_1)\prod_{i=2}^{N+1}[\frac{1}{2}+(-1)^{m_i}\gamma_i]
\]

Since we also have 
$$\mu(Z(\lambda_{N}))=\prod_{i=1}^N[\frac{1}{2}+(-1)^{\ell_i}\gamma_i],$$
it follows that (multiplying numerator and denominator by $2^N$)
\begin{equation}
\label{eq-Judy-RN-example}
\frac{\mu(\sigma_{f_1}Z(\lambda_{N}))}{\mu(Z(\lambda_{N}))}\;=\;\left( (\frac{1}{2} - \gamma_1)\prod_{i=2}^{N+1}[1+(-1)^{m_i}2\gamma_i] \right) /\left( \prod_{i=1}^N[1+(-1)^{\ell_i}2\gamma_i]\right).
\end{equation}
We then have 
\[ \Phi_{f_1}(\xi):=\frac{d(\mu \circ \sigma_{f_1})}{d\mu}(\xi) = \lim_{N \to \infty} \frac{\mu(\sigma_{f_1}(Z(\lambda_N))}{\mu(Z(\lambda_N))}.\]
To see that the Radon-Nikodym derivative $\Phi_{f_1}$ is positive, note that
standard results on infinite products imply that, since $|\gamma_i|<1/2$ and $\sum_{i\in \N} |\gamma_i |< \infty$ by hypothesis, 
\[ 
\lim_{n\to \infty}\prod_{i=2}^n \left( 1+ (-1)^{m_i} 2 \gamma_i \right) \quad\text{and}\quad \lim_{n\to \infty}\prod_{i=1}^n (1+(-1)^{\ell_i}2\gamma_i)
\]
are both finite, positive and nonzero for any sequences $(m_i)_i , (\ell_i)_i \subseteq \{ 0,1\}^\N$.  
Indeed, if we let $L$ be the sum of the logarithmic series associated to the denominator $P=\prod_{i=1}^\infty [1+(-1)^{\ell_i}2\gamma_i]$, then one can check that $L=\ln P=\sum_{i=1}^\infty \ln ([1+(-1)^{\ell_i}2\gamma_i])$ has the same absolute convergence behavior as the series
\[
\sum_{i=1}^\infty |(-1)^{\ell_i}2\gamma_i|,\;\; i.e.\; \sum_{i=1}^{\infty}|\gamma_i|;
\]
this latter series converges by hypothesis. Thus, the series $\sum_{i=1}^\infty \ln ([1+(-1)^{\ell_i}2\gamma_i])$ converges conditionally to a number $L =\sum_{i=1}^\infty \ln ([1+(-1)^{\ell_i}2\gamma_i])\in \R$. But since $L=\ln P$, it cannot be that $P=0$, because this would mean that $L=\infty$ and not $L\in \R$. Therefore, the Radon-Nikodym derivative 
\[
\Phi_{f_1}(\xi)= \frac{d(\mu \circ \sigma_{f_1})}{d\mu}(\xi) = \lim_{N \to \infty} \frac{\mu(\sigma_{f_1}(Z(\lambda_N))}{\mu(Z(\lambda_N))}
\]
converges and is positive as desired. 

Similar calculations, by  using Lemma 	\ref{lemma-limit-RN},  yield the same conclusion for the Radon-Nikodym derivatives associated to $\sigma_e$ and $\sigma_{f_2}$, showing that all the hypotheses of Theorem~\ref{thm-lambda-sbfs-on-the-inf-path-space-via-a-measure}	 are satisfied in this case.  We conclude that $\mu$ makes $\Lambda^\infty$ into a $\Lambda$-semibranching function system with the standard prefixing and coding maps $(\sigma_\lambda, \sigma^n)$, which proves (a).

To see (b), we now use Kakutani's work on product measures to compare the measures $\mu$ constructed above with the standard (Perron-Frobenius) measure $M$ on $\Lambda^\infty$ (defined in Proposition 8.1 of \cite{aHLRS3}).  Note first that $M$ is a special case of the measure $\mu$ described above, given by taking $\gamma_i = 0$ for all $i$.  

A moment's reflection shows us that  $(\Lambda^{\infty},\mu)$ is measure-theoretically isomorphic to $(\prod_{i=1}^{\infty}[\{0,1\}]_i, \prod_{n=1}^{\infty}\mu_i),$ 
where $\prod_{i=1}^{\infty}[\{0,1\}]_i$ is the set of all sequences consisting of $0$ and $1$ only, and
the measure $\mu_i$ on the $i^{th}$ factor space $\{0,1\}$ is given by 
$$\mu_i(\{0\})=\frac{1}{2}+\gamma_i\quad\text{and}\quad \;\mu_i(\{1\})=\frac{1}{2}-\gamma_i\quad\text{for}\;\; i\in \mathbb N.$$
The isomorphism is given by $\prod_{i\in \N} [\{0,1\}]_i \ni (a_i)_{i\in \N} \mapsto e f_{a_1+1} e f_{a_2+1} e \cdots\in \Lambda^\infty$.
it follows from Corollary 1 of Section 10 of \cite{Kaku} that the measure $\mu$ on $\Lambda^{\infty}$ is equivalent (mutually absolutely continuous)  to the Perron-Frobenius measure $M$ whenever the infinite series
$$\sum_{i=1}^{\infty}\left(\sqrt{\frac{1}{2}}-\sqrt{\frac{1}{2}+\gamma_i}\right)^2+\left (\sqrt{\frac{1}{2}}-\sqrt{\frac{1}{2}-\gamma_i}\right)^2,$$  or equivalently, the infinite series 
$$ \sum_{i=1}^{\infty}\left(1-\frac{\sqrt{1+2\gamma_i}}{2}-\frac{\sqrt{1-2\gamma_i}}{2}\right),
$$
converges.  However, this series converges whenever $\sum_{i \in \N} |\gamma_i| < \infty$.  In other words, all the measures $\mu$ studied in this section are equivalent to $M$.  

To see (c), let $g_\mu \in L^2(\Lambda^\infty, \mu )$ be given by 
\[ g_\mu(x) = \sqrt{\frac{d\mu}{dM}(x)},\]
and define $W_\mu : L^2(\Lambda^\infty, \mu ) \to L^2(\Lambda^\infty, M)$ by $W_\mu (f) = g_\mu  f.$
Then one checks that $W_\mu^*(f) = \frac{f}{g_\mu}$ is given by multiplication by $\sqrt{\frac{dM}{d\mu}(x)}$.

For $\lambda \in \Lambda$, write $S_\lambda^\mu$ for the operator on $L^2(\Lambda^\infty, \mu)$ associated to $\lambda$ via the $\Lambda$-semibranching function system on $(\Lambda^\infty, \mu)$, as in Theorem 3.5 of \cite{FGKP}; that is,  if $d(\lambda) = n$,
\begin{equation}\label{eq:SBFS-rep-formula}
 S_\lambda^\mu(\chi_{Z(\eta)})(x) = \left(\frac{d\mu}{d(\mu \circ \sigma^n)}(x) \right)^{-1/2} \chi_{Z(\lambda \eta)}(x).\end{equation}
Moreover, standard manipulations with Radon-Nikodym derivatives imply that
\begin{align*} W_\mu^* S_\lambda^M W_\mu (\chi_{Z(\eta)} )(x) &= \left( \frac{dM}{d\mu}(x) \right)^{1/2} \chi_{Z(\lambda \eta)}(x) \left( \frac{dM }{d( M\circ \sigma^n)}(x) \right)^{-1/2} \left(\frac{d\mu}{dM}(\sigma^n(x)) \right)^{1/2} \\
&= \chi_{Z(\lambda \eta)}(x)\left( \frac{dM}{d\mu}(x) \right)^{1/2} \left(\frac{d\mu}{dM}(\sigma^n(x)) \right)^{1/2}  \left( \frac{d(M \circ \sigma^n)}{dM}(x) \right)^{1/2} \\
&= \chi_{Z(\lambda \eta)}(x) \left( \frac{ d(\mu \circ \sigma^n)}{d\mu}(x)\right)^{1/2} \\
&= S_\lambda^\mu (\chi_{Z(\eta)})(x).
\end{align*}
Thus, $L^2(\Lambda^\infty, \mu)$ and $L^2(\Lambda^\infty, M)$ are unitarily equivalent, via the unitary $W_\mu$ which intertwines the two $\Lambda$-semibranching representations, $S_\lambda^\mu$ and $S_\lambda^M$.
It follows that
any   $\Lambda$-semibranching function system associated to a measure $\mu$ as described above will give rise to a representation of $C^*(\Lambda)$ which is equivalent to that arising from the $\Lambda$-semibranching function system of Proposition 3.4 of \cite{FGKP}.
\end{proof} 

The equivalence of the $\Lambda$-semibranching representations discussed  above is an instance of a more general phenomenon, as explained below.

We now show how to use a similar product measure construction to obtain a $\Lambda$-semibranching function system for the 2-graph $\Lambda_2$ of Example \ref{ex:sbfs-3v8e}.  Again, any infinite path $\xi \in \Lambda_2^\infty$ can be written uniquely as an alternating sequence of red (dashed) and blue (solid) edges, with the first edge being red.  In fact, such an infinite path is completely determined by the sequence of vertices it passes through: every infinite path $\xi$ with range $v$ is specified uniquely by a string of vertices  
\[ \xi \equiv (v, Q_1, v, Q_3, \ldots ) \text{ where } Q_{2i+1} \in \{ u, w\}.\]
Similarly, if $r(\xi) \in \{ u, w\}$ then $\xi \equiv ( Q_0, v, Q_2, v, \ldots) $ for a unique sequence $(Q_{2i})_i \in \{ u, w\}_{i\in \N}$.

Thus, as in the definition of the product measure $\mu$ above, given any sequence of real numbers $( \delta_n)_{n \in \N}$,  with 
$|\delta_n|<1/2$ for all $n$, and  with $\sum_{n \in \N}|\delta_n| < \infty$, we  define a function $\mu_2$ on square cylinder sets of $\Lambda_2^\infty$ by first setting
\[
\mu_2( Z(u ))=\mu_2( Z(w ))=\mu_2( Z(v))=1.
\]
Given any $N \in \N$ and any $\eta \in \Lambda^{(N,N)}_2$, write $\eta$ as an alternating sequence of red-blue edges and list the vertices through which it passes:
\[ \eta \equiv (v, Q_1, v, \ldots, Q_{2N-1}, v) \quad \text{or} \quad \eta \equiv (Q_0, v, Q_2, \ldots,  v, Q_{2N}).\]
Define $\alpha_n^\eta = \begin{cases}
\delta_n, & Q_n = u \\ - \delta_n, & Q_n = w.
\end{cases}
$, and   set 
\begin{equation} \label{eq:measure-3v8e}
\mu_2( Z(\eta) ): =  \begin{cases} \prod_{0\leq n \leq N} ( 1/2 + \alpha_{2n} ^\eta), & \text{if}\;\; r(\eta) \in \{ u, w\} \\
 \prod_{0\leq n\leq N-1 } ( 1/2 + \alpha_{2n+1} ^\eta), &\text{if}\;\; r(\eta) = v \end{cases}
\end{equation}

\begin{prop}
\label{prop-meas-on-ex-3-vert}
Let $\Lambda_2$ be the $2$-graph given in Example~ \ref{ex:sbfs-3v8e} and $\mu_2$ be the function given by the formula in \eqref{eq:measure-3v8e}. Then we have the followings.
\begin{itemize}
\item[(a)] The function $\mu_2$ defines a Borel measure on $\Lambda_2^\infty$.  
\item[(b)] The standard prefixing and coding maps $(\sigma_\lambda, \sigma^n)$ endow $(\Lambda^\infty_2, \mu_2)$ with a $\Lambda$-semibranching function system.  
\item[(c)] As measure spaces, $(\Lambda^\infty_2, \mu_2)$ is isomorphic to
\[  \left( \prod_{i\in \N} ( \{ 0,1\}, \nu_i) \right) \sqcup \left( \prod_{i\in \N} ( \{ 0,1\}, \nu_i) \right) ,\]
where $\nu_i(j) = 1/2 + (-1)^j \delta_i.$  
\item[(d)] The  measure $\mu_2$ is equivalent to the Perron-Frobenius measure $M$ of \eqref{eq:M}, and the associated $\Lambda$-semibranching representations are also equivalent. 
\end{itemize}
\end{prop}

\begin{proof}
To see (a), we merely need to check that $\mu_2$ is finitely additive on square cylinder sets of $\Lambda^\infty_2$ as in the proof of Proposition \ref{prop:prod-meas-onevtwoe}.
But this follows from the observation that for any $\lambda \in \Lambda_2$ with $d(\lambda) = (N, N), $ 
\begin{align*}
\sum_{\eta \in s(\lambda) \Lambda^{(1,1)}} \mu_2(Z(\lambda \eta)) &= \mu_2(Z(\lambda)) \cdot  \begin{cases} (1/2 - \delta_{2N+1}) + (1/2 + \delta_{2N+1}),  & \text{if}\;\; s(\lambda) = v \\
(1/2 - \delta_{2N+2}) + (1/2 + \delta_{2N +2}), &\text{if}\;\;  s(\lambda) \in \{ u, w\}.
\end{cases} \\
&= \mu_2(Z(\lambda)).
\end{align*}
By induction, it follows that $\mu_2$ is finitely additive on square cylinder sets. Then again Lemma~\ref{lem:measure} implies that $\mu_2$ gives a measure on $\Lambda_2^\infty$.

To see (b), we again use Theorem \ref{thm-lambda-sbfs-on-the-inf-path-space-via-a-measure}.  Thus, it remains to check that the edge prefixing operators all have positive Radon-Nikodym derivatives.

By Lemma 	\ref{lemma-limit-RN}, if $\xi = \bigcap_{n \in \N} Z(\lambda_n)$ with $d(\lambda_n) = (n, n)$, then for an edge $e\in \Lambda_2$, we have
\[ \Phi_e(\xi)=\frac{d(\mu_2 \circ \sigma_e)}{d\mu_2}(\xi) = \lim_{n\to \infty} \frac{\mu_2(\sigma_e(Z(\lambda_n)))}{\mu_2(Z(\lambda_n))}.\]
Similar arguments to those used in the proof of Proposition \ref{prop:prod-meas-onevtwoe} will show that this limit is finite and nonzero for all edges $e$ and all $\xi \in \Lambda^\infty_2$; we detail a few cases here.

First, suppose that $r(\xi) = v$, so that $\lambda_n \equiv (v, Q_1, v, \ldots, Q_{2n-1}, v)$, and that $e\in \{d_0, d_1\}$ is a red edge with range $Q \in \{ u, w\}$.  Since $\xi=\cap_{n\in \N}Z(\lambda_n)$, we have $\xi=(v,Q_1,\dots,Q_{2n-1},v)$, where $(Q_{2i-1})_i\in \{u,w\}$, $1\le i\le n$. Using the factorization rule, rewrite $\xi$ (and each $\lambda_n$) as an alternating blue-red path; with this factorization, $\xi$ passes through the vertices $(v, \tilde Q_1, v, \ldots, \tilde Q_{2n-1}, v, \ldots )$ where $\tilde Q_{2i-1}\in \{u, w\}$ and $Q_{2i-1} \not= \tilde Q_{2i-1}$ for all $i$. Then, prefixing $\lambda_n$ by $e$ results in a red-blue path $e \lambda_n$ with
\[ e \lambda_n \equiv (Q, v, \tilde Q_1, v, \ldots, \tilde Q_{2n-1}, v).\]
Note that for each $n$, $Z(e \lambda_n) = Z (e \lambda_n b_0) \bigsqcup Z( e \lambda_n b_1)$, and 
\[ e \lambda_n b_0 \equiv (Q, v, \tilde Q_1, \ldots, \tilde Q_{2n-1}, v, u); \qquad e \lambda_n b_1 \equiv (Q, v, \tilde Q_1, \ldots, \tilde Q_{2n-1}, v, w).\]
For $1 \leq i \leq n$, write $m_i = 1$ if $\tilde Q_{2i-1} = w$ and $m_i = 0$ if $\tilde Q_{2i-1} = u$. Similarly, write $m_Q = 1$ if $Q = w$ and $m_Q =0$ if $Q=u$.  The definition \eqref{eq:measure-3v8e} of $\mu_2$ then implies
\begin{align*}
\mu_2\circ \sigma_e (Z(\lambda_n)) &= \mu_2 (Z(e \lambda_n b_0)) + \mu_2(Z(e \lambda_n b_1) \\
&= 2^{n+1} (1 + (-1)^{m_Q} 2  \delta_0) \left( \prod_{i=1}^n (1 + (-1)^{m_i} 2 \delta_{2i})\right) \left( 1/2 + \delta_{2n+2} + 1/2 - \delta_{2n+2} \right) \\
&= 2^{n+1} (1 + (-1)^{m_Q} 2  \delta_1) \prod_{i=1}^n (1 + (-1)^{m_i} 2 \delta_{2i}).
\end{align*}

If we write $l_i = m_i +1$, then $Q_{2i-1} = w$ iff $l_i = 1$ and $Q_{2i-1} =u$ iff $l_i = 2$.  Thus, 
\[ \mu_2(Z(\lambda_n)) = 2^n \prod_{i=1}^N (1 + (-1)^{l_i} 2 \delta_{2i-1}).\]
It follows that 
\[\frac{\mu_2(Z(e \lambda_n)) }{ \mu_2(Z(\lambda_n)} = 2 \frac{ (1 + (-1)^{m_Q} 2 \delta_0) \prod_{i=1}^n (1 + (-1)^{m_i} 2 \delta_{2i})}{
 \prod_{i=1}^n (1 + (-1)^{l_i} 2 \delta_{2i-1})}.\]
 
Since we chose the sequence $(\delta_n)_{n\in {\N}}$ such that $|\delta_n| < 1/2$ for all $n$ and $\sum_n |\delta_n |< \infty$, 
\[ \lim_{n\to \infty} \prod_{i=1}^n (1 + (-1)^{m_i} 2 \delta_{2i}) \quad \text{and} \quad \lim_{n \to \infty} \prod_{i=1}^n (1 + (-1)^{l_i} 2 \delta_{2i-1})\]
are both finite, positive, and nonzero.  Using a similar argument to that employed in the proof of Proposition~\ref{prop:example_exonevtwoe}(a), one can show that the Radon-Nikodym derivative 
 \[\Phi_e (\xi) = \lim_{n \to \infty} \frac{\mu_2(Z(e \lambda_n))}{\mu_2(Z(\lambda_n))}\]
  converges and is positive whenever $e = d_0, d_1$ is a red edge with source $v$.
 The other Radon-Nikodym derivatives are similarly computed to be positive on their domains, which proves (b).

To see (c), we use a similar argument to the one used in Proposition \ref{prop:prod-meas-onevtwoe} above to establish that the measures $\mu_2$ on $\Lambda_2^\infty$ can be viewed as product measures on 
\[\left( \prod_{i=1}^\infty (\{ 0,1\}, \nu_i) \right) \sqcup \left( \prod_{i=1}^\infty (\{ 0,1\}, \nu_i) \right).\]
where $\nu_i(j) = 1/2+ (-1)^j \delta_i$  for $j\in \{1,2\}$.  We note that the infinite paths $\xi \in \Lambda_2^\infty$ can be divided into two types; those with range $v$,  and those with a different range vertex.  The fact that an infinite path (written as an alternating sequence of red-blue edges) is completely determined by the sequence of vertices it passes through means that if we identify vertex $u$ with 0 and vertex $w$ with 1, then the infinite paths with range $v$ are in bijection with infinite sequences of 0s and 1s.  However, infinite paths with range $u$ or $w$ are also in bijection with infinite sequences of 0s and 1s.  We therefore need two copies of the infinite product space $ \prod_{i=1}^\infty (\{ 0,1\}, \nu_i) $ to capture the entirety of $\Lambda_2^\infty$. 

We now describe the measure-preserving bijection between $(\Lambda_2^\infty, \mu_2)$ and $\left( \prod_{i=1}^\infty (\{ 0,1\}, \nu_i) \right) \sqcup \left( \prod_{i=1}^\infty (\{ 0,1\}, \nu_i) \right)$.
If we denote by $(a_i)_{i\in \N}$ a sequence in the first product space, and $(b_i)_{i\in \N}$ a sequence in the second, then the isomorphism $\left( \prod_{i=1}^\infty \{ 0,1\} \right) \sqcup \left( \prod_{i=1}^\infty \{ 0,1\} \right) \to \Lambda_2^\infty$ is given by 
\begin{equation}\label{eq:map_infinite_paths}
 (a_i)_i \mapsto ( v, Q_1, v, Q_3, \ldots ); \qquad (b_i)_i \mapsto (Q_0, v, Q_2, v, \ldots),
 \end{equation}
where $Q_{2i+1} = u$ if $a_i = 0$ and $w$ otherwise, whereas $Q_{2i} = u$ if $b_{i} = 0$ and $w$ otherwise.  Our construction of the measure $\mu_2$, and the fact that this isomorphism preserves cylinder sets, implies that this isomorphism is measure-preserving.

To see (d), note that as computed in Example 7.7 of \cite{LLNSW}, the adjacency matrices of $\Lambda_2$ are both given by $\begin{pmatrix}
 0 & 1 & 0 \\ 1 & 0 & 1 \\ 0 & 1 & 0
\end{pmatrix}$.  
Note that here $v$ is the middle vertex of $\Lambda_2$. These matrices both have spectral radius $\sqrt{2}$ and  Perron-Frobenius eigenvector $\frac{1}{2 + \sqrt{2}}
(1, \sqrt{2}, 1)^T$, so the measure $M$ of \cite{aHLRS3} is given on square cylinder sets $Z(\lambda)$ where $d(\lambda) = (n, n)$ by 
\[M(Z(\lambda)) = \begin{cases} \frac{\sqrt{2}}{2^n(2 + \sqrt{2})} , & r(\lambda) = v \\ 
\frac{1}{2^n(2+ \sqrt 2)} , & r(\lambda) \in \{ u, w\}. \end{cases}\]
To compare $M$ with the measure $\mu_2$, we will try to write $M$ in terms of product measures.  If $\lambda \equiv (v, Q_1, v, \cdots, Q_{2n-1}, v)$, then 
\[ M(Z(\lambda)) = \frac{\sqrt{2}}{2 + \sqrt{2}} \prod_{i=1}^n \mu^i\{ Q_{2i-1}\}\]
where $\mu^i(u) = \mu^i(w) = 1/2.$  In other words, on $v\Lambda^\infty$ (the image of the sequences $(a_i)_i$), $M$ is given by $\frac{\sqrt{2}}{2 + \sqrt{2}}$ times a product measure; and on $\Lambda^\infty \backslash v \Lambda^\infty$ (the image of the sequences $(b_i)_i$), $M$ is given by $\frac{1}{2 + \sqrt 2}$ times the same infinite product measure, namely the product measure $\left( \prod_{i=1}^\infty (\{ 0,1\}, \mu^i) \right)$ where $\mu^i(j) = 1/2$.

Just as in Proposition \ref{prop:prod-meas-onevtwoe}, Corollary 1 of Section 10 of \cite{Kaku} then implies that the product measures $M$ and $\mu_2$ are  equivalent measures, thanks to our hypothesis that $\sum_i |\delta_i| < \infty$. At last, the same unitary $W: L^2(\Lambda_2^\infty, \mu_2) \to L^2(\Lambda_2^\infty, M)$ that we used in that Proposition, namely 
\[ W(f)(x) = \sqrt{\frac{d\mu_2}{dM}(x)} f(x),\]
intertwines the $\Lambda$-semibranching representations associated to $\mu_2$ and $M$.
\end{proof}

\label{sec-ex-many-vertex-graph-measure}

In fact, for each $N \in \N$ we can generalize the construction of $\Lambda_{2}$ to construct a $2$-graph $\Lambda_{2N}$ with $2N +1$ vertices labeled $v, u_1, \ldots, u_N, w_1, \ldots w_N$, with red and blue edges connecting $v$ with each of the vertices $u_i, w_i$, in both directions:
 The  skeleton of $\Lambda_{2N}$ is given as below.
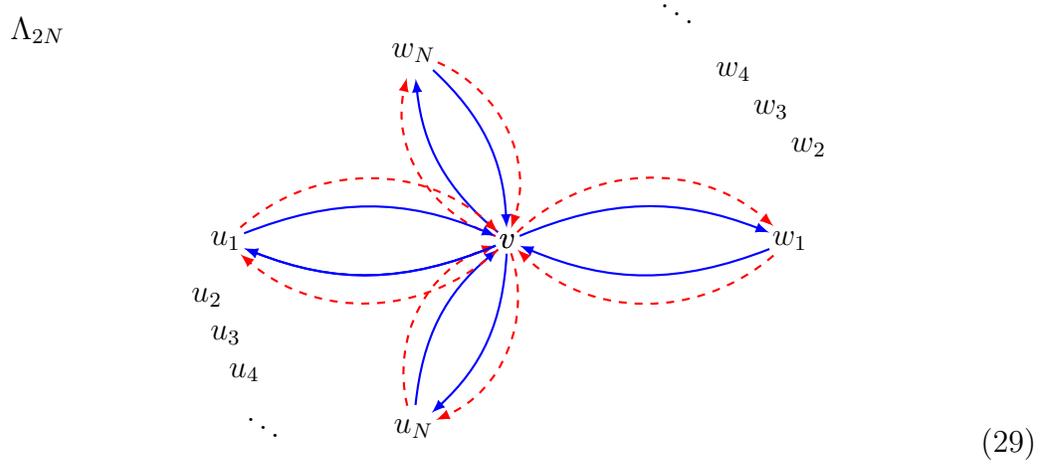
\begin{equation}
\begin{tikzpicture}[scale=2.5]
\node[inner sep=0.5pt, circle] (u_1) at (0,0) {$u_1$};
\node[inner sep=0.5pt, circle] (u_a) at (-.1,-.3) {$u_2$};
\node[inner sep=0.5pt, circle] (u_d) at (0,-.5) {$u_3$};
\node[inner sep=0.5pt, circle] (u_b) at (.1,-.7) {$u_4$};
\node[inner sep=0.5pt, circle] (u_c) at (.2,-.95) {$\ddots$};
\node[inner sep=0.5pt, circle] (v) at (1.5,0) {$v$};
\node[inner sep=0.5pt, circle] (w_1) at (3,0) {$w_1$};
\node[inner sep=0.5pt, circle] (w_a) at (3.1,.5) {$w_2$};
\node[inner sep=0.5pt, circle] (w_b) at (2.9,.7) {$w_3$};
\node[inner sep=0.5pt, circle] (w_d) at (2.7,.9) {$w_4$};
\node[inner sep=0.5pt, circle] (w_c) at (2.4,1.25) {$\ddots$};
\node[inner sep=0.5pt, circle] (v) at (1.5,0) {$v$};
\node[inner sep=0.5pt, circle] (w_2) at (1,1) {$w_N$};
\node[inner sep=0.5pt, circle] (u_2) at (1,-1) {$u_N$};
\draw[-latex, thick, blue] (v) to[bend left=22] (u_2); 
\draw[-latex, thick, blue] (u_2) to[bend left=22] (v); 
\draw[-latex, thick, blue] (v) to[bend left=22](w_2); %
\draw[-latex, thick, blue] (w_2) to[bend left=22] (v); %
\draw[-latex, thick, red, dashed] (v) to[bend left=42] (u_2); %
\draw[-latex, thick, red, dashed] (u_2) to[bend left=42] (v); %
\draw[-latex, thick, red, dashed] (v) to[bend left=42] (w_2); %
\draw[-latex, thick, blue] (v) to[bend left=22] (u_1); 
\draw[-latex, thick, red, dashed] (w_2) to[bend left=42](v); %
\draw[-latex, thick, blue] (v) to[bend left=22] (u_1); 
\draw[-latex, thick, blue] (u_1) to[bend left=22] (v); 
\draw[-latex, thick, blue] (v) to[bend left=22] (w_1); 
\draw[-latex, thick, blue] (w_1) to[bend left=22](v); 
\draw[-latex, thick, red, dashed] (v) to[bend left=42] (u_1); 
\draw[-latex, thick, red, dashed] (u_1)to[bend left=42] (v); 
\draw[-latex, thick, red, dashed] (v) to[bend left=42] (w_1); 
\draw[-latex, thick, red, dashed] (w_1) to[bend left=42](v); 
\node at (-1, 1.12) {\color{black} $\Lambda_{2N}$};
\end{tikzpicture}
\label{eq:Lambda2N-skeleton}
\end{equation}
As in the case of $\Lambda_2$, there are multiple choices of factorization rules that will make $\Lambda_{2N}$ into a 2-graph.  Regardless of the factorization rule we choose, though, every (finite or infinite) path will have a unique representative as an alternating string of blue (solid) and red (dashed) edges, with the first edge being red. In fact, such a path is completely determined by the sequence of vertices it passes through: every infinite path $\xi$ with range $v$ is specified uniquely by a string of vertices  
\begin{equation} \xi \equiv (v, Q_1, v, Q_3, \ldots ) \text{ where } Q_{2i+1} = u_j \text{ or } w_j  \text{ for some } 1 \leq j \leq N.\label{eq:inf-path-Lambda2N}
\end{equation}
Similarly, if $r(\xi) \in \{u_j, w_j\,:\,  1\le j\le N \}$, then $\xi \equiv ( Q_0, v, Q_2, v, \ldots) $ for a unique sequence $(Q_{2i})_i \in \{ u_j, w_j: 1 \leq j \leq N\}$ for $i\in \N$. 
  The adjacency matrices of $\Lambda_{2N}$ are given by 
\[ \begin{split} 
A_1 = A_2; \qquad 1 &= A_j (v, u_i) = A_j(v, w_i) = A_j (u_i, v) = A_j (w_i, v) \;\;\text{for}\;\; 1 \leq i \leq N,; \\ 
0 &= A_j (u_i, w_\ell) = A_j (u_i, u_\ell) = A_j(w_i, w_\ell) \;\;\text{for}\;\; 1 \leq i, \ell \leq N\;,
\end{split}\] 
where $j=1, 2$.
The spectral radius of $A_1 = A_2$ is easily computed to be $\sqrt{2N}$, and the Perron-Frobenius eigenvector is $(x_w)_{w \in \Lambda^0_{2N}}$, where
\[ x_w = \begin{cases} 
\frac{1}{1 +\sqrt{2N} }, & w = v \\ \frac{1}{\sqrt{2N}( 1 + \sqrt{2N})}, & w \not= v.
\end{cases}\]
Thus, the measure $M$ of a square cylinder set $Z(\lambda) \subseteq \Lambda_{2N}^\infty$ with $d(\lambda) = (n, n)\in \N^2$ is given by 
\[ M(Z(\lambda))= \begin{cases}\frac{1}{\sqrt{2N}(1+\sqrt{2N}) (2N)^n}, & r(\lambda) \not= v \\
\frac{1}{( 1 + \sqrt{2N}) (2N)^n}, & r(\lambda) = v. \end{cases}\]

\begin{prop}
\label{prop:inf-path-Lambda-2N}
Let $\Lambda_{2N}$ be any 2-graph with skeleton \eqref{eq:Lambda2N-skeleton}. 
 Then the infinite path space $\Lambda_{2N}^\infty$ is isomorphic to a disjoint union of infinite product spaces:
\[ 
\Lambda_{2N}^\infty \cong \prod_{i\in \N} \Z_{2N}  \sqcup \prod_{i\in \N} \Z_{2N}. 
\]
\end{prop}
\begin{proof}
According to \eqref{eq:inf-path-Lambda2N}, one can identify  $\Lambda^\infty_{2N}$ with the disjoint union of two infinite path spaces, namely,
\[
\Lambda^\infty_{2N}=\{x\in \Lambda^\infty_{2N}\,:\, r(x)=v\} \sqcup \{x\in\Lambda^\infty_{2N}\,:\, r(x)\ne v\}.
\]
Also observe that $\prod_{i\in \N}\Z_{2N}$ is isomorphic to $\{x\in \Lambda^\infty_{2N}\,:\, r(x)=v\}$ via a map similar to the map given in \eqref{eq:map_infinite_paths}, and similarly we see that $\prod_{i\in \N}\Z_{2N}$ is isomorphic to $\{x\in\Lambda^\infty_{2N}\,:\, r(x)\ne v\}$. 
Therefore the result follows by the fact that the square cylinder sets generate the topology on $\Lambda^\infty_{2N}$ with a similar argument to the one given in the proof of Proposition~\ref{prop-meas-on-ex-3-vert}(c).
\end{proof}

Thus, given $N$ sequences $\{ (\delta_i^j)_{i\in \N}\}_{j=1}^N$ with $\sum_i |\delta^j_i| < \infty$ for all $j$, we can define an associated
 product
measure $\mu_{2N}$ on $\Lambda^\infty_{2N}$.  Given  $\eta \in \Lambda_{2N}$ with $d(\eta) = (n,n)$, we start by identifying $\eta$ with the string of vertices it passes through:
\begin{equation}\label{eq:2N-paths}
r(\eta) = v \Rightarrow \eta \equiv (v, Q_1, \ldots, Q_{2n-1}, v) \qquad r(\eta) \not= v \Rightarrow \eta \equiv (Q_0, v, Q_2, \ldots, v, Q_{2n}),
\end{equation}
where $Q_i \not= v$.  Then, we define
\begin{equation}\label{eq:measure_2N}
\alpha_i = \begin{cases} \delta^j_i, & Q_i = u_j \\ -\delta^j_i, & Q_i = w_j \end{cases}  \qquad \text{and } \qquad \mu_{2N}(Z(\eta)) =\begin{cases} \prod_{i=1}^n \frac{1 + \alpha_{2i-1}}{2N}, & r(\eta) = v \\   \prod_{i=0}^n \frac{1 + \alpha_{2i}}{2N}, & r(\eta) \not= v. 
\end{cases}\end{equation}

\begin{prop}
The formula for $\mu_{2N}$ given in \eqref{eq:measure_2N} defines a measure on $\Lambda^\infty_{2N}$.  Moreover, the standard prefixing and coding maps $(\sigma^n, \sigma_\lambda)$ make $(\Lambda^\infty_{2N}, \mu_{2N})$ into a $\Lambda$-semibranching function system.  The resulting $\Lambda$-semibranching representation of $C^*(\Lambda)$ on $L^2(\Lambda^\infty, \mu_{2N})$ is equivalent to that associated to the Perron-Frobenius measure $M$ on $\Lambda^\infty$ of Equation \eqref{eq:M}.
\label{prop:SBFS-on-Lambda2N}
\end{prop}
\begin{proof}
To see that $\mu_{2N}$ defines a measure on $\Lambda^\infty_{2N}$, it suffices to show that $\mu_{2N}$ is additive on square cylinder sets.  So fix $\eta\in \Lambda_{2N}$ with $d(\eta)=(n,n)$. If $r(\eta) \not= v$, then we also have $s(\eta) \not= v$ as in \eqref{eq:2N-paths}.  Consequently, there is a unique red edge $e$ with $s(\eta) = r(e)$, and we will have $s(e) = v$.  Since there are $2N$ blue edges with range $v$,
\[ f_i^u \in v \Lambda^{(0,1)}u_i \quad \text{ and } \quad f_i^w \in v\Lambda^{(0,1)}w_i \quad \text{ for } 1 \leq i \leq N,\]
we have 
\[Z(\eta)= \bigsqcup_{i=1}^N \bigsqcup_{\alpha \in \{u, w\}} Z(\eta e f_i^\alpha).\]
Observe that for each $1 \leq i \leq N$, $\mu_{2N}(Z(\eta e f_i^u)) + \mu_{2N}(Z(\eta e f_i^w)) = \frac{2}{2N} \cdot   \prod_{j=0}^n \frac{1 + \alpha_{2j}}{2N}$.  Thus,
\[ \frac{1}{(2N)^n} \prod_{j=0}^n (1 + \alpha_{2j}) = \mu_{2N}(Z(\eta)) = \sum_{i=1}^N \sum_{\alpha \in \{ u, w\}} \mu_{2N} (Z(\eta e f_i^\alpha)),\]
so $\mu_{2N}$ is additive on square cylinder sets with range $u_i$ or $w_i$ for some $1 \leq i \leq N$.

If $r(\eta) = v = s(\eta)$, then a similar argument, using the fact that there are $2N$ red (dashed) edges $e^\alpha_i$ with range $v$, each with a different source vertex, and each with a unique blue edge $f$ such that $r(f) = s(e^\alpha_i)$, will show that $\mu_{2N}$ is also additive on square cylinder sets with range $v$.

To see that the standard prefixing and coding maps make $(\Lambda^\infty_{2N}, \mu_{2N})$ into a $\Lambda$-semibranching function system, we merely need check that the Radon-Nikodym derivatives $\frac{d(\mu_{2N} \circ \sigma_e)}{d\mu_{2N}}$ are positive for every edge $e$.  As in the proofs of Propositions \ref{prop-meas-on-ex-3-vert} and \ref{prop:prod-meas-onevtwoe}, this follows from the fact that for any sequences $(\ell_i)_i \in \{0,1\}^\N, \ (j_i)_i \in (\Z_N)^\N$, the infinite product 
\[ \prod_{i\in \N} (1 + (-1)^{\ell_i} \delta_i^{j_i})\]
is finite, since $\sum_i |\delta_i^j| < \infty$ for $1 \leq j \leq N$.

To be precise: prefixing will always cause a change in the precise list of vertices $Q_i$.  The precise change is specified by the factorization rules.  Prefixing by an edge $e$ with $r(e) \not= v$ will lengthen the list of vertices $Q_i$ by one.  In the end, $\frac{\mu_{2N}(Z(e \eta))}{\mu_{2N}(Z(\eta))}$ will be a scalar (either $1$ or $(2N)^{-1}$, depending on whether $r(e) =v$ or not) times a quotient of sub-products of convergent infinite products of the form $ \prod_{i\in \N}( 1 + (-1)^{\ell_i} \delta_i^{j_i})$.

To see that $\mu_{2N}$ and $M$ are equivalent, we compute yet another Radon-Nikodym derivative  by using Lemma 	\ref{lemma-limit-RN}, namely $\frac{d\mu_{2N}}{dM}$.  Fix $\xi \in \Lambda_{2N}^\infty$ as in \eqref{eq:inf-path-Lambda2N}.  If we write 
\[\xi_n \equiv \begin{cases} (v, Q_1, v, \ldots, Q_{2n-1}, v) &  \text{ if } r(\xi) = v \\ 
(Q_0, v, \ldots, v, Q_{2n}) & \text{ else},\end{cases} \]
then $\xi = \bigcap_{n\in \N} Z(\xi_n)$.  Moreover,
\[ 
\frac{d\mu_{2N}}{dM}(\xi)=\lim_{n\to \infty}\frac{\mu_{2N}(Z(\xi_n))}{M(Z(\xi_n))} = \begin{cases}
\sqrt{2N}(1 + \sqrt{2N}) \prod_{i=1}^n (1 + \alpha_{2i-1}), & r(\xi) = v \\
\frac{1 + \sqrt{2N}}{2N} \prod_{i=0}^n (1 + \alpha_{2i}), & \text{ else.}
\end{cases}\]
Since both $\lim_{n \to \infty} \prod_{i=1}^n (1 + \alpha_{2i-1})$ and $\lim _{n \to \infty} \prod_{i=0}^n (1 + \alpha_{2i})$ are finite  and nonzero, we see that $\mu_{2N}$ is absolutely continuous with respect to $M$. Similarly, one can show that $M$ is absolutely continuous with respect to $\mu_{2N}$, and hence $\mu_{2N}$ and $M$ are equivalent.
For the final claim, one can check that  the same unitary $W$ used in the proof of Propositions  \ref{prop:prod-meas-onevtwoe} and \ref{prop-meas-on-ex-3-vert} establishes that the $\Lambda$-semibranching representations associated to $\mu_{2N}$ and $M$ are equivalent.
\end{proof}

\subsection{Examples of probability measures on $\Lambda^\infty$ that are mutually singular with  the Perron-Frobenius measure}
\label{sec-Markov-measure-Lambda-semibran-0}

{
In Section 3 of \cite{dutkay-jorgensen-monic}, the authors outline a procedure for constructing Markov measures  on the infinite path space $\mathcal{K}_N$ of the Cuntz algebra $\mathcal{O}_{N}$, such that the resulting measures  are mutually singular.  
In this section, we show how to apply the analysis of \cite{dutkay-jorgensen-monic} first to the 2-graph $\Lambda$ of Example \ref{exonevtwoe}, and then to the 2-graphs $\Lambda_{2N}$ of Proposition \ref{prop:SBFS-on-Lambda2N};  we can do so since the infinite path spaces of these $2$-graphs are either homeomorphic to $\mathcal{K}_N$ or to a disjoint union of them.

Recall that the infinite path space $\mathcal{K}_N$ of $\mathcal{O}_N$ is given by
\[
\mathcal{K}_N=\prod_{i=1}^\infty \Z_N=\{(i_1 i_2\dots)\,:\, i_n\in \Z_N,\;\; n=1,2,\dots\}.
\]

\begin{defn}[Definition 3.1 of  \cite{dutkay-jorgensen-monic}]
		\label{def-Markov-measure} 
A {\em Markov measure} on the infinite path space  $\mathcal K_N$ of the Cuntz algebra  ${\mathcal{O}_N}$ is defined by a vector 
	$\lambda = (\lambda_0, \ldots, \lambda_{N-1}) $ and an $N \times N$ matrix $T$ such that $\lambda_i > 0$, $T_{i,j} > 0$ for all $i,j \in \Z_N,$ and if 
	$e = (1,1, \ldots ,1)^t$ then $\lambda T=\lambda $ and $Te=e$.
The Carath\'eodory/Kolmogorov extension theorem then implies that  there exists a unique   Borel measure $\mu$ on $\mathcal K_N$  extending the measure $\mu_{\mathcal{C}}$ defined on cylinder sets by:
\[
\mu_{\mathcal{C}} (Z(I)) : = \lambda_{i_1} T_{i_1,i_2} \cdots T_{i_{n-1},i_n},\text{ if }I = i_1 \ldots i_n.
\]
The extension  $\mu$ is called a \emph{Markov measure}\footnote{For Markov measure in a more general context, see \cite{bezuglyi-jorgensen}.} on $\mathcal K_N$.
  
	\end{defn}

We now describe a specific example of a Markov measure on the infinite path space $\mathcal{K}_2$ of $\mathcal{O}_2$ which we use extensively in what follows.

Fix a number $ x \in (0,1)$, the unit interval, and define $T_x=(T_{i,j}) = \begin{pmatrix}
 x & (1-x) \\ (1-x)  & x 
\end{pmatrix}$.  
Let $\lambda=(1,1)$ be a row vector with 1 in all entries.
Then it is straightforward to check that the pair $(T_x, \lambda)$ satisfy 
\[
\lambda \,T_x=\lambda ,\quad T_x \,e=e.
\]
Then as in Definition~\ref{def-Markov-measure}, the Markov measure $\mu_x$ on $\mathcal{K}_2$ is given  on the cylinder sets by
 \begin{equation}\label{eq:mu_x}
  \mu_x(Z(i_1 i_2 \cdots i_n)) =  T_{i_1, i_2} T_{i_2, i_3} \cdots T_{i_{n-1}, i_n},
  \end{equation}
 where $i_j\in \Z_2=\{0,1\}$.  
 Using the fact that the infinite path space of the 2-graph $\Lambda$ in Example \ref{exonevtwoe} is homeomorphic to that of $\mathcal{O}_2$ (by mapping $e f_j$ to $i_j$ for $j=1,2$) we can convert $\mu_x$ into a measure on $\Lambda^\infty$, which we will continue to denote by $\mu_x$.
} 
 
Under this correspondence, the measure $\mu_{1/2}$ satisfies 
\[ \mu_{1/2}(E) = 2 M(E)\quad \text{for all Borel sets}\;\; E \subseteq \Lambda^\infty. \]
Moreover, if $x \not= x'$, Theorem 3.9 of \cite{dutkay-jorgensen-monic} guarantees that $\mu_x, \mu_{x'}$ are mutually singular.

\begin{prop} 
\label{prop-RN-of-Markov-2-example}
Let $\Lambda$ be the 2-graph given in Example \ref{exonevtwoe}.
Fix a number $x \in (0,1)$, and let $\mu_x$ be Markov measure given in \eqref{eq:mu_x}.  As operators on $L^2(\Lambda^\infty, \mu_{x})$, the prefixing operators  $\sigma_e,\sigma_{f_1},\sigma_{f_2}$ have positive Radon-Nikodym derivatives at any point $z \in \Lambda^\infty$.
 Consequently, the standard prefixing and coding maps make $(\Lambda^\infty, \mu_x)$ into a $\Lambda$-semibranching function system.
\end{prop}

\begin{proof} 
Since $\mu_x$ is a Borel measure and the standard coding and prefixing maps $\sigma^n, \sigma_\lambda$ are local homeomorphisms, they are $\mu_x$-measurable for any $x \in (0,1)$.  Moreover, we have 
\[ \mu_x(Z(v)) = \sum_{i, j \in \{ 1,2\}} \mu_x(Z(e f_i e f_j)) = 2.\]
Thus, once we show that the  prefixing operators $\sigma_e,\sigma_{f_1},\sigma_{f_2}$ have positive Radon-Nikodym derivatives (for which we use Lemma 	\ref{lemma-limit-RN}), Theorem \ref{thm-lambda-sbfs-on-the-inf-path-space-via-a-measure} tells us that the standard prefixing and coding maps constitute a $\Lambda$-semibranching function system on $(\Lambda^\infty, \mu_x)$.

Thus, fix $z = e f_{i_1} ef_{i_2}ef_{i_3} \ldots \in \Lambda^\infty$,
and a sequence $(z_n)_n$ of finite paths 
\[
z_n:=e f_{i_1}ef_{i_2}ef_{i_3}\ldots f_{i_n}
\]
such that $z = \bigcap_{n\in \N} Z(z_n)$.  By Lemma \ref{lemma-limit-RN}, for any finite path $g \in \Lambda$, we have 

\[
\frac{d(\mu_x \circ \sigma_g)}{d\mu_x}(z) = \lim_{n\to \infty} \frac{ \mu_{x}(   Z( gz_n  ) ) } {\mu_{x} (Z(z_n ))}.
\]
If we take $g=e$, the factorization rules $e f_i = f_{i-1} e$, for $i= 1,2$, imply that
\[
\lim_n \frac{ \mu_{x}(   Z( gz_n)   ) } {\mu_{x} (Z(z_n) )}=\lim_n \frac{ \mu_{x}(  Z(e f_{i_1-1}e f_{i_2-1}\ldots f_{i_n-1} e) ) } {\mu_{x}( Z(  e f_{i_1}e f_{i_2} \ldots f_{i_n})  ) }=\lim_n \frac{T_{i_1-1,i_2-1}\cdots T_{i_{n-1}-1,i_n-1}} {T_{i_1,i_2} \cdots T_{i_{n-1},i_n}} ,
\]
since $Z(e f_{i_1-1}e f_{i_2-1}\ldots f_{i_n-1} e ) = Z(e f_{i_1-1}e f_{i_2-1}\ldots f_{i_n-1} e f_1  ) \sqcup Z(e f_{i_1-1}e f_{i_2-1}\ldots f_{i_n-1} e f_2)$ has 
\[\mu_x (Z(e f_{i_1-1}e f_{i_2-1}\ldots f_{i_n-1} e) ) = \mu_x Z(e f_{i_1-1}e f_{i_2-1}\ldots f_{i_n-1} ) ).\]

Now, observe that for any $i,j \in \Z/2\Z$ we have $T_{i,j} = T_{i-1, j-1}$.  It follows that 
\[ \frac{d(\mu_x \circ \sigma_e)}{d\mu_x}(z) = \lim_{n \to \infty} \frac{T_{i_1-1,i_2-1}\cdots T_{i_{n-1}-1,i_n-1}} {T_{i_1,i_2}\cdots T_{i_{n-1},i_n}} = 1.\]

Similarly, for $j= 1, 2,$  by Lemma 	\ref{lemma-limit-RN},  the Radon-Nikodym derivative
\[
\frac{d(\mu_x \circ \sigma_{f_j})}{d\mu_x}(z) = \lim_{n\to \infty} \frac{ \mu_{x}(   Z( f_j z_n  ) ) } {\mu_{x} (Z(z_n ))} = \lim_{n\to \infty} \frac{T_{j+1, i_1+1} T_{i_1 +1, i_2 + 1} \cdots T_{i_{n-1} + 1, i_n +1}}{T_{i_1, i_2} \cdots T_{i_{n-1}, i_n}} = T_{j+1, i_1 +1}
\]
is positive (indeed, constant on each cylinder set $Z(e f_i)$ for $i = 1, 2$).
\end{proof}

Recall from Proposition \ref{prop:inf-path-Lambda-2N} that, if $\Lambda_{2N}$ denotes a 2-graph with $2N+1$ vertices and skeleton \eqref{eq:Lambda2N-skeleton}, then $\Lambda_{2N}^\infty$ consists of two disjoint copies of the infinite path space $\prod_{i\in \N} \Z_{2N}$ of $\mathcal O_{2N}$. Thus, we can also use the Markov measures of \cite{dutkay-jorgensen-monic}  to construct new measures on $\Lambda_{2N}^\infty$.

The distinction between the vertices $u_i$ and $w_i$ which we relied on to prove Proposition \ref{prop:SBFS-on-Lambda2N} will encumber our notation unnecessarily in the sequel.  Thus, we fix a relabeling the vertices $u_i, w_i$ of $\Lambda_{2N}$ by $\{ Q_i\}_{i=1}^{2N}$.  With this notation, the isomorphism between $\prod_{i\in \N} \Z_{2N} \sqcup \prod_{i \in \N} \Z_{2N}$ and $\Lambda^\infty_{2N}$  is given by mapping a sequence $( a_i)_{i \in \N}$ in the first copy of $\prod_{i\in \N} \Z_{2N}$ to $\xi \equiv (v, Q_{a_1}, v, Q_{a_2}, \ldots)$. Similarly, a sequence $(b_i)_{i\in \N}$ in the second copy of $ \prod_{i\in \N} \Z_{2N}$ maps to the infinite path $\xi \equiv (Q_{b_1}, v, Q_{b_2}, v, \ldots)$.

Observe that a choice of factorization  on $\Lambda_{2N}$ is equivalent to choosing a permutation $\phi$ of $\{1, \ldots, 2N\}$ such that the red-blue path $(v, Q_i, v)$ equals the blue-red path $(v, Q_{\phi(i)}, v)$.
Having specified such a permutation $\phi$, suppose $\phi$ consists of $d$ cycles; write $c_j$ for the smallest entry in the $j$th cycle.

 Fix $d$ vectors $\{x^j\in \R^{2N}: 0<x_i^j<1\;\;\text{for}\;\; 1\le i\le 2N\}_{j=1}^d$ such that $\sum_{i=1}^{2N} x^j_i = 1$ for each $j$, and define $T_x$ to be the  $2 N \times 2N$ matrix with entries from $(0,1)$ such that 
 \[ T_x(i, j) = x^m_{\phi^{n-1}(j)} \ \text{ if } i = \phi^{n-1}(c_m).\]
By construction, we have $T_x(i,j) = T_x(\phi(i), \phi(j))$ for all $1 \leq i, j \leq 2N$.  Moreover,  the fact that all rows of $T$ sum to 1 implies that $(T, ( 1, 1, \ldots, 1)^T)$ satisfies the conditions given in Definition 3.1 of \cite{dutkay-jorgensen-monic}. Therefore, we have Markov measure $\mu_x$ associated to $T$.

\begin{prop}
Let $\Lambda_{2N}$ be a 2-graph with skeleton  \eqref{eq:Lambda2N-skeleton} and factorization rule determined by the permutation $\phi \in S_{2N}$.  For each matrix $T_x$ as above, write $\mu_x$ for the associated measure  on $\Lambda_{2N}^\infty \cong \prod_{i\in \N} \Z_{2N} \sqcup \prod_{i\in \N} \Z_{2N}$, given on a cylinder set in either copy of $\prod_{i\in \N} \Z_{2N} $ by 
\[ \mu_x(Z(a_1 \cdots a_n)) = \prod_{i=1}^{n-1} T_x(a_i, a_{i+1}). \]
Then the standard prefixing and coding maps make $(\Lambda_{2N}^\infty, \mu_x)$ into a $\Lambda$-semibranching function system.  If the vectors $x^m$ are not all constant, $\mu_x$ is mutually singular with respect to the measure $M$ of Equation \eqref{eq:M}.
\label{prop:markov-2N}
\end{prop}
\begin{proof}
As above, we merely need to check the Radon-Nikodym derivatives by using Lemma 	\ref{lemma-limit-RN}.  Fix a red edge $e$ with range $Q_i$, and fix a point $\xi \equiv (v, Q_{b_1}, v, Q_{b_2}, \ldots)$ (with a  red edge listed first).  Then, 
\begin{align*}
\frac{d(\mu_x \circ \sigma_e)}{d\mu_x}(\xi) &= \lim_{n \to \infty} \frac{\mu_x \circ \sigma_e(Z(v, Q_{b_1}, \ldots, Q_{b_n}, v))}{\mu_x(Z(v, Q_{b_1}, \ldots, Q_{b_n}, v))}\\
&= \lim_{n\to \infty} \frac{\mu_x(Z(Q_i, v, Q_{\phi(b_1)}, v, \ldots, Q_{\phi(b_n)}, v))}{\mu_x(Z(v, Q_{b_1}, \ldots, Q_{b_n}, v))}\\
&= \lim_{n \to \infty}\frac{ T_x(i, \phi(b_1)) \prod_{i=1}^{n-1} T_x(\phi(b_i), \phi(b_{i+1}))}{\prod_{i=1}^{n-1} T_x(b_i, b_{i+1})} \\
&= T_x(i, \phi(b_1)).
\end{align*}
Similarly, if we choose a blue edge $f$  with range $Q_i$, the fact that there is a unique red-blue path with range $Q_i$ means that rewriting the path $f \xi$ as an alternating sequence of red-blue edges doesn't change the vertices $Q_j$ that the path passes through.  In other words, we have 
\begin{align*}
\frac{d(\mu_x \circ \sigma_f)}{d\mu_x}(\xi) &= \lim_{n \to \infty} \frac{\mu_x \circ \sigma_f(Z(v, Q_{b_1}, \ldots, Q_{b_n}, v))}{\mu_x(Z(v, Q_{b_1}, \ldots, Q_{b_n}, v))}\\
&= \lim_{n\to \infty} \frac{\mu_x(Z(Q_{i}, v, Q_{b_1}, v, \ldots, Q_{b_n}, v))}{\mu_x(Z(v, Q_{b_1}, \ldots, Q_{b_n}, v))}\\
&= \lim_{n \to \infty}\frac{ T_x(i, b_1 ) \prod_{i=1}^{n-1} T_x(b_i, b_{i+1})}{\prod_{i=1}^{n-1} T_x(b_i, b_{i+1})} \\
&= T_x(i, b_1).
\end{align*}
On the other hand, if $\zeta \equiv (Q_{a_1}, v, Q_{a_2}, v, \ldots)$ is an infinite path and $g$ is a blue edge with source $Q_{a_1}$ and range $v$, prefixing $\zeta$ by $g$ and rewriting the result as a sequence of red-blue edges gives $g \zeta \equiv (v, Q_{\phi(a_1)}, v, Q_{\phi(a_2)}, \ldots)$.  It follows that 
\begin{align*}
\frac{d(\mu_x \circ \sigma_g)}{d\mu_x}(\zeta) &= \lim_{n \to \infty} \frac{\mu_x \circ \sigma_g(Z(Q_{a_1}, \ldots, Q_{a_n}))}{\mu_x(Z(Q_{a_1}, \ldots, Q_{a_n})))}\\
&= \lim_{n\to \infty} \frac{\mu_x(Z( v, Q_{\phi(a_1)}, v, \ldots, Q_{\phi(a_n)}))}{\mu_x(Z( Q_{a_1}, \ldots, Q_{a_n}))}\\
&= \lim_{n \to \infty}\frac{ \prod_{i=1}^{n-1} T_x(\phi(a_i), \phi(a_{i+1}))}{\prod_{i=1}^{n-1} T_x(a_i, a_{i+1})} \\
&= 1.
\end{align*}
Similarly, for any red edge $h$ with source $Q_{a_1}$, the fact that rewriting $\zeta$ as a blue-red path doesn't change the sequence of vertices it passes through means that $\frac{d(\mu_x \circ \sigma_h)}{d\mu_x}(\zeta) = 1$.

Since all the Radon-Nikodym derivatives are positive, we obtain a $\Lambda$-semibranching function system as claimed.

For the final assertion, one simply observes that since the formula for $M(Z(\lambda))$ only depends on the degree (length) of $\lambda$, it is a rescaling of the measure $\mu_x$ corresponding to the choice $x^m = (\frac{1}{2N}, \ldots, \frac{1}{2N})$ for all $m$.  Thus, if any of the vectors $x^m$ are not constant, Theorem 3.9 of \cite{dutkay-jorgensen-monic} implies that $\mu_x$ is mutually singular with respect to $M$.
\end{proof}

\begin{rmk}
The measure $\mu_x$ used in Proposition \ref{prop:markov-2N} above could equally well be defined for any  2-graph $\Lambda_{2N+1}$ with one central vertex and $2N+1$ peripheral vertices, each connected to the center vertex as in \eqref{eq:Lambda2N-skeleton}.  This is because any such 2-graph (equivalently, any factorization rule for this skeleton) is determined by a permutation of the outer $2N+1$ vertices.  The  conclusions of Proposition \ref{prop:markov-2N} above regarding when $\mu_x$ and $M$ are mutually singular also hold in this context.
\end{rmk}

\section{$\Lambda$-projective systems and projection valued measures }
\label{sec-further-foundations}

In Sections \ref{sec:monic-results} and \ref{sec:atomic_repn} below, we turn our attention to two new classes of representations of $k$-graph algebras: monic and purely atomic representations.  For our analysis of these representations, we need certain straightforward but technical results about the projection-valued measure on $C(\Lambda^\infty)$ induced by a $*$-representation of $C^*(\Lambda)$; these form the content of Section \ref{sec:Proj-valued-measures} below.  We also rely on a slight strengthening of the notion of a $\Lambda$-semibranching function system.  These so-called $\Lambda$-projective systems, and their properties, are the content of Section \ref{sec:lambda-proj-systems}.

\subsection{$\Lambda$-projective systems}
\label{sec:lambda-proj-systems}

\begin{defn}\label{def:lambda-proj-system}
Let $\Lambda$ be a finite $k$-graph with no sources. 
A \emph{$\Lambda$-projective system} on a measure space $(X,\mu)$ is a $\Lambda$-semibranching function system on $(X,\mu)$, with prefixing maps $\{\tau_\lambda:D_\lambda\to R_\lambda\}_{\lambda\in\Lambda}$ and coding map $\{\tau^n:n\in\N^k\}$
together with a family of functions $\{ f_\lambda\}_{\lambda\in \Lambda} \subseteq  L^2(X,\mu)$ satisfying the following.
\begin{itemize}
\item[(a)] For any   $\lambda \in \Lambda$,   we have  $ 0\not=\frac{d(\mu \circ (\tau_\lambda)^{-1})}{d\mu} = |f_\lambda |^2$ {on $R_\lambda$}; 
\item[(b)] For any $\lambda, \nu \in \Lambda$, we have 
$f_\lambda \cdot (f_\nu \circ \tau^{d(\lambda)}) = f_{\lambda \nu}.$
\end{itemize}
\end{defn}
{
 Before discussing some of the consequences and implications of this definition, we present an alternative formulation of the important quantity $\left.\frac{d(\mu \circ (\tau_\lambda)^{-1}}{d\mu}\right|_{R_\lambda}$.} 
 \begin{lemma}
  For any $\Lambda$-semibranching function system, \begin{equation}
 \left.\frac{d(\mu \circ (\tau_\lambda)^{-1})}{d\mu}\right|_{R_\lambda} = \frac1{(\Phi_\lambda \circ \tau^{n})|_{R_\lambda}}. 
\label{eq:phi-lambda-eqn}
\end{equation}
  \end{lemma}
  \begin{proof}
 Suppose $d(\lambda) = n \in \N^k$. Since $\tau_\lambda^{-1} = \tau^{n}|_{R_\lambda}$ for $n=d(\lambda)$, Condition (b) in the definition of a semibranching function system (Definition \ref{def-1-brach-system}) implies that $\mu << \mu \circ \tau_\lambda^{-1} $ a.e. on $R_\lambda$.  Moreover,
 \[\frac{d(\mu \circ \tau_\lambda^{-1})}{d\mu} \cdot \frac{d\mu}{d(\mu \circ \tau_\lambda^{-1})} = 1_{R_\lambda}.\]  
 Since $\tau_\lambda\circ \tau^n|_{R_\lambda}=id|_{R_\lambda}$, we have
 \begin{align*}
 \frac{d\mu}{d(\mu \circ \tau_\lambda^{-1})}|_{R_\lambda} &= \frac{d( \mu \circ \tau_\lambda \circ \tau^n)}{d(\mu \circ \tau_\lambda^{-1} \circ \tau_\lambda \circ \tau^n)} |_{R_\lambda} = \left. \frac{d(\mu \circ \tau_\lambda)}{d\mu} \circ \tau^n\right|_{R_\lambda}\\
 &= \Phi_\lambda \circ \tau^{n}|_{R_\lambda},
\end{align*}
and  since $D_{s(\lambda)} = D_\lambda$ is the domain of $\Phi_\lambda$, it follows that 
\begin{equation*}
\frac1{(\Phi_\lambda \circ \tau^{n})|_{R_\lambda}} = \left.\frac{d(\mu \circ (\tau_\lambda)^{-1})}{d\mu}\right|_{R_\lambda}. 
\end{equation*}
  \end{proof}

  \begin{rmk}
 \label{rmk:lambda-proj-comments-1}
 \begin{enumerate}
 
 \item For any $\Lambda$-semibranching function system on $(X, \mu)$, we will have 
 \[\mu \circ (\tau_\lambda)^{-1} << \mu \text{ on }R_\lambda.\]
   To see this, suppose $E \subseteq R_\lambda$ has measure zero.  Setting  $F := \{ x \in D_{s(\lambda)}: \tau_\lambda(x) \in E\}$, we have $E = \tau_\lambda(F)$ and $F = (\tau_\lambda)^{-1}(E)$.  Since the Radon-Nikodym derivative 
 \[ \frac{d(\mu \circ \tau_\lambda)}{d\mu} = \Phi_\lambda\]
 is strictly positive a.e.~on $D_{s(\lambda)}$, if $\mu(E) = \mu \circ \tau_\lambda(F) = 0$ then $\mu(F)$ must also be zero.  Hence $\mu(E) = 0 \Rightarrow \mu \circ (\tau_\lambda)^{-1}(E) = 0$.
 
\item  In fact, $ \frac{d(\mu \circ (\tau_\lambda)^{-1})}{d\mu}$ is always nonzero a.e.~on $R_\lambda$.  To see this, let 
 \[E = \left\{ x \in R_\lambda: \frac{d(\mu \circ (\tau_\lambda)^{-1})}{d\mu}(x) = 0 \right\}\]
  and note that 
 \[ 0=\int_E  \frac{d(\mu \circ (\tau_\lambda)^{-1})}{d\mu} \, d\mu = \mu \circ (\tau_\lambda)^{-1}(E) 
 .\]
 Write $F = (\tau_\lambda)^{-1}(E)$ and observe that 
 \[\mu(F) = 0 \Rightarrow \mu(\tau_\lambda(F)) = \mu(E) = 0 ,\] since $\mu \circ \tau_\lambda << \mu$ in any $\Lambda$-semibranching function system.

 \item A $\Lambda$-projective system on $(X,\mu)$ consists of a $\Lambda$-semibranching function system plus some extra information (encoded in the functions $f_\lambda$).  We have a certain amount of choice for the functions $f_\lambda$; we can take positive or negative (or imaginary!) roots of $\frac{d(\mu \circ (\tau_\lambda)^{-1})}{d\mu} $ for $f_\lambda$, as long as they satisfy the multiplicativity Condition (b) above.
 \item 
 For any $\Lambda$-semibranching function system on $(X, \mu)$
 , there is a natural choice of an associated $\Lambda$-projective system; namely, for $\lambda \in \Lambda^{n}$ we define 
 \[f_\lambda(x) := \Phi_{\lambda}(\tau^{n}(x))^{-1/2}  \chi_{R_\lambda}(x).\]
Since $(\Phi_\lambda \circ \tau^n)^{-1} = \frac{d(\mu \circ \tau_\lambda^{-1})}{d\mu}$, it follows that $f_\lambda$ defined above satisfies Condition (a) of Definition \ref{def:lambda-proj-system}. 

Moreover, since the operators $S_\lambda \in B(L^2(X, \mu))$ of Theorem 3.5 of \cite{FGKP} are given by 
\[S_\lambda(f) = f_\lambda \cdot (f \circ \tau^n),\]
and the aforementioned Theorem 3.5 establishes that $\{S_\lambda\}_{\lambda\in \Lambda}$ is a Cuntz--Krieger family, Proposition \ref{prop:lambda-proj-repn} below shows that $\{f_\lambda\}_{\lambda\in \Lambda}$ is indeed a $\Lambda$-projective system. 

\item Our definition of a $\Lambda$-projective system is based on the definition of a  monic system in \cite{dutkay-jorgensen-monic} (for $\mathcal{O}_N$) and \cite{bezuglyi-jorgensen} (for $\mathcal{O}_A$).  We have decided to change the name because even for $\mathcal{O}_A$, not every monic system gives rise to a monic representation of $\mathcal{O}_A$. 
The word ``projective'' refers to the cocycle-like Condition (c) of Definition \ref{def:lambda-proj-system}.

\item 
 Finally, we observe that Condition (a) of Definition \ref{def:lambda-proj-system}
 forces $f_\lambda (x) = 0 $ a.e.~outside of $R_\lambda$, since $\frac{d(\mu \circ (\tau_\lambda)^{-1})}{d\mu}$ is supported only on $R_\lambda$. 
 \end{enumerate}
\end{rmk}

 Condition (b) of Definition~\ref{def:lambda-proj-system} is needed to associate a representation of $C^*(\Lambda)$ to a $\Lambda$-projective system, as the following Proposition shows. 
 
  \begin{prop}
  \label{prop:lambda-proj-repn}
  Let $\Lambda$ be a finite, source-free $k$-graph. Suppose that a measure space $(X,\mu)$ admits a $\Lambda$-semibranching function system with prefixing maps $\{\tau_\lambda:\lambda\in\Lambda\}$ and coding maps $\{\tau^n:n\in\N^k\}$.  Suppose that $\{f_\lambda\}_{\lambda\in \Lambda}$ is a collection of functions satisfying Condition (a) 
  of Definition \ref{def:lambda-proj-system}. Then $(X,\mu)$ admits a $\Lambda$-projective system with $\{\tau_\lambda\}$, $\{\tau^n\}$ and $\{f_\lambda\}_\lambda$ if and only if the operators $T_\lambda \in B(L^2(X, \mu))$ given by 
  \begin{equation}\label{eq:T-lambda}
  T_\lambda(f) = f_\lambda \cdot (f \circ \tau^{d(\lambda)})
  \end{equation}
 form a Cuntz--Krieger $\Lambda$-family (and hence give a representation of $C^*(\Lambda)$).
  \end{prop}
  \begin{proof}
  Suppose that the operators $\{T_\lambda\}_{\lambda\in \Lambda}$ described in \eqref{eq:T-lambda} form a Cuntz--Krieger $\Lambda$-family.  Then, for any $\lambda,\nu \in \Lambda$ with $s(\lambda)=r(\nu)$ and any $f \in L^2(X,\mu)$ we have 
  \begin{align*}
  T_\lambda T_\nu(f) &= T_\lambda(f_\nu \cdot (f \circ \tau^{d(\nu)})) = f_\lambda \cdot ( f_\nu \circ \tau^{d(\lambda)}) \cdot (f \circ \tau^{d(\lambda) } \circ \tau^{d(\nu)});\\
  T_{\lambda\nu}(f) &= f_{\lambda\nu} \cdot (f \circ \tau^{d(\lambda \nu)}),
  \end{align*}
  
  Since $\{\tau^n, \tau_{\lambda}\}$ is a $\Lambda$-semibranching function system, we know that 
  \[\tau^{d(\lambda \nu)} = \tau^{d(\lambda) + d(\nu)} = \tau^{d(\lambda)} \circ \tau^{d(\nu)};\]
  thus, the fact that $\{T_\lambda\}_{\lambda\in \Lambda}$ is a Cuntz--Krieger $\Lambda$-family, in particular condition (CK2), implies that 
  \[f_\lambda \cdot (f_\nu \circ \tau^{d(\lambda)}) = f_{\lambda \nu}  \ \forall \lambda, \nu \in \Lambda,\]
so Condition (b) holds for $\{f_\lambda\}_{\lambda\in \Lambda}$  as desired.
  
  On the other hand, suppose that $(X,\mu)$ admits a $\Lambda$-projective system with prefixing maps $\{\tau_\lambda\}_{\lambda \in \Lambda}$, coding maps $\{\tau^n\}_{n\in \N^k}$, and functions $\{f_\lambda\}_{\lambda \in \Lambda}$.  We will show that the operators $\{T_\lambda\}$ of Equation \eqref{eq:T-lambda} do indeed form a Cuntz--Krieger family, by showing they satisfy (CK1)-(CK4).
  
First, we observe that for $v \in \Lambda^0$, $T_v(f) = f_v \cdot (f\circ \tau^0)$ is supported on $D_v = R_v$ by Condition (a) of Definition~\ref{def:lambda-proj-system}. Moreover, since  $v= v^2$ for any $v \in \Lambda^0,$ and $\tau_v = id_{D_v} = \tau^0$,  Condition (b) of Definition~\ref{def:lambda-proj-system} implies that 
\[f_v = f_v \cdot (f_v \circ \tau^0) = f_v^2 \Rightarrow f_v = \chi_{D_v}.\]
 Consequently, $T_v(f) = \chi_{D_v} \cdot f$.  Since the sets $\{D_v = R_v\}_{v\in\Lambda^0}$ are disjoint (up to a set of measure zero), it follows that $\{T_v: v\in \Lambda^0\}$ is a set of mutually orthogonal projections; in other words, (CK1) holds. 

For (CK2), fix $\lambda,\nu\in \Lambda$ with $s(\lambda)=r(\nu)$.
 We compute  that 
 \[ f_\nu \circ \tau^{d(\lambda)} (x) = \begin{cases}
 0, & \tau^{d(\lambda)}(x) \not\in R_\nu \\ 
 f_\nu(\tau_\nu(y)), & x = \tau_\eta \circ \tau_\nu(y) \text{ and }d(\eta) = d(\lambda).
 \end{cases}\]
 Thus, using the same notation,
 \begin{align*}
  T_\lambda T_\nu(f)(x) &= T_\lambda( f_\nu \cdot (f \circ \tau^{d(\nu)})) (x)= f_\lambda(x) \cdot (f_\nu \circ \tau^{d(\lambda)})(x) \cdot (f \circ \tau^{d(\lambda) + d(\nu)})(x) \\
  &= f_\lambda(\tau_\eta \circ \tau_\nu(y)) \cdot f_\nu(\tau_\nu(y)) \cdot f(y) \\
  &= \begin{cases}
  0, & \eta \not= \lambda \\
  f_\lambda(\tau_\lambda \circ \tau_\nu(y)) \cdot f_\nu(\tau_\nu(y)) \cdot f(y), & x = \tau_\lambda \tau_\nu(y).
  \end{cases}
 \end{align*}
 On the other hand, since 
 $y = \tau^{d(\lambda) +d(\nu)}(x)$, we can  write $x = \tau_\rho(y) $ for a unique $\rho$ such that $d(\rho) =d(\lambda) + d(\nu) = d(\lambda\nu)$.  Then 
 \begin{align*}
 T_{\lambda\nu}(f)(x)&=  f_{\lambda\nu}(\tau_\rho(y)) \cdot f(y) \\
  &= \begin{cases} 0, & \rho \not= \lambda \nu\\
  f_{\lambda\nu}(x) \cdot f(y), &  \rho = \lambda\nu
  \end{cases}\\
  &= f_\lambda(x) \cdot f_\nu(\tau_\nu(y) ) \cdot f(y)
  \end{align*}
by Condition (b) of Definition \ref{def:lambda-proj-system}.  In other words, $T_\lambda T_\nu = T_{\lambda\nu}$ as claimed.
 
  To check (CK3),
 we first compute that, if $d(\lambda) = n$,
 \begin{align*}
 \langle T_\lambda^* f, g \rangle &= \langle f, T_\lambda g \rangle = \int_X f(x) \overline{f_\lambda(x) g(\tau^n(x))}\, d\mu\\
 &=  \int_{R_\lambda} f(x) \overline{f_\lambda(x) g(\tau^{n}(x))}\,  d\mu = \int_{D_\lambda} f(\tau_\lambda(x)) \overline{f_\lambda(\tau_\lambda(x)) g(x)}  d(\mu \circ \tau_\lambda) \\ 
 &= \int_{D_\lambda} f(\tau_\lambda(x)) \overline{f_\lambda(\tau_\lambda(x)) g(x)} \Phi_\lambda(x) \, d\mu. 
 \end{align*}
 Thus, $T_\lambda^*f = f \circ \tau_\lambda \cdot \overline{f_\lambda \circ \tau_\lambda} \cdot \Phi_\lambda$.
 Alternatively, since $(\Phi_\lambda \circ \tau^{n})^{-1} = |f_\lambda|^2 = f_\lambda \overline{f_\lambda}$, we see that 
 \[
 \overline{f_\lambda(\tau_\lambda(x))} \Phi_\lambda(x) = \big(f_\lambda(\tau_\lambda(x))\big)^{-1}.
 \]  
 Consequently,
 \begin{equation}\label{eq:T-adj}
 T_\lambda^* f= \frac{\chi_{D_\lambda} \cdot (f \circ \tau_\lambda )}{f_\lambda \circ \tau_\lambda}.
 \end{equation}
 Now, we can check (CK3):
 \begin{align*}
 T_\lambda^* T_\lambda(f) &= T_\lambda^* ( f_\lambda \cdot (f \circ \tau^{d(\lambda)})) = \frac{\chi_{D_\lambda} \cdot ( (f_\lambda \cdot (f \circ \tau^{d(\lambda)})) \circ \tau_\lambda)}{f_\lambda \circ \tau_\lambda} \\
 &=\chi_{D_\lambda} f = T_v (f),
 \end{align*}
 since $\tau^{d(\lambda)} \circ \tau_\lambda = id_{D_\lambda}$ in any $\Lambda$-semibranching function system.  Thus, (CK3) holds.  
Now, a straightforward computation shows us that $T_\lambda T_\lambda^* T_\lambda = T_\lambda$ since $\left. (\chi_{D_\lambda} \circ \tau^{d(\lambda)} )\right|_{R_\lambda}= \chi_{R_\lambda}$, and hence $f_\lambda \cdot (\chi_{D_\lambda} \circ \tau^{d(\lambda)}) = f_\lambda$.  Therefore, the operators $\{T_\lambda: \lambda \in \Lambda\}$ are partial isometries as claimed.

For (CK4), we fix $n\in \Lambda^k$ and $v\in \Lambda^0$ and compute 
\begin{equation}
\begin{split}
\sum_{\lambda \in v\Lambda^{n}} T_\lambda T_\lambda^*(f) &= \sum_{\lambda \in v\Lambda^{n}} T_\lambda \left( \frac{\chi_{D_\lambda} \cdot (f \circ \tau_\lambda)}{f_\lambda \circ \tau_\lambda} \right) \\
&= \sum_{\lambda \in v\Lambda^{n}} f_\lambda \cdot \frac{(\chi_{D_\lambda} \circ \tau^{n}) \cdot (f \circ \tau_\lambda \circ \tau^{n})}{f_\lambda \circ \tau_\lambda \circ \tau^{n}} \\
&= \sum_{\lambda \in v\Lambda^{n}} \chi_{R_\lambda} \cdot f,
\end{split}
\label{eq:lambda-proj-range-sets}
\end{equation}
since $\tau_\lambda \circ \tau^{n} = id_{R_\lambda}$.  Condition (a) of Definition \ref{def-1-brach-system} implies that  $\sum_{\lambda \in v\Lambda^{n} }\chi_{R_\lambda} = \chi_{D_v}$. Thus,
\[\sum_{\lambda \in v\Lambda^{e_i}} T_\lambda T_\lambda^* = T_v \]
as desired.
 \end{proof}
 
 \begin{rmk}\label{rmk:standard-sys}
Consider, for a finite strongly connected $k$-graph $\Lambda$, the standard $\Lambda$-semibranching function system on $(\Lambda^\infty, M)$ given in \cite[Proposition~3.4]{FGKP} (also see Example~\ref{example:SBFS-M} above).  The prefixing maps $\{\sigma_\lambda:Z(s(\lambda))\to Z(\lambda)\}_{\lambda \in \Lambda}$ are given by $\sigma_\lambda(x)=\lambda x$. 
One easily checks that the associated Radon-Nikodym derivative satisfies 
\[\Phi_{\sigma_\lambda}=\frac{d(M\circ\sigma_\lambda)}{dM} = \rho(\Lambda)^{-d(\lambda)} >0,\] $M$-almost everywhere on $D_\lambda:=Z(s(\lambda))$. 
The corresponding operators $S_\lambda$ that generate a representation of $C^*(\Lambda)$ (as in Equation \eqref{eq:standard_S_f} above and \cite[Theorem~3.5]{FGKP}) are of the form in \eqref{eq:T-lambda}, with 
\[ f_\lambda=(\Phi_\lambda\circ \sigma^n)^{-1/2} = \rho(\Lambda)^{d(\lambda)/2} \chi_{Z(\lambda)}.\]
Hence Proposition~\ref{prop:lambda-proj-repn} implies that the standard $\Lambda$-semibranching function system on $(\Lambda^\infty, M)$ becomes a $\Lambda$-projective system  with $\{f_\lambda\}_{\lambda \in \Lambda}$ as above.
\end{rmk}

Now suppose that a measure space $(X,\mu)$ admits a $\Lambda$-projective system and there is a measure $\mu'$ on $X$ such that $\mu$ and $\mu'$ are equivalent (i.e. $\mu' << \mu$, and  $\mu << \mu'$).
Then $\frac{d\mu'}{d\mu}>0,\;\mu$-a.e. and $\frac{d\mu}{d\mu'}>0,\;\mu'$-a.e. and 
\[
 \frac{d\mu'}{d\mu}\cdot \frac{d\mu}{d\mu'} \equiv 1,\;\mu \; \text{a.e., which is the same as}\;\;\mu'\;\text{a.e.}.
 \] 
Let $g_1(x)\;=\;\frac{d\mu'}{d\mu}(x)$ and $g_2(x)=\frac{d\mu}{d\mu'}(x)$ so that $g_1(x)\cdot g_2(x)\equiv 1$ a.e. as described above.
It is then easy to check that there is a unitary isomorphism of Hilbert spaces $U:L^2(X,\mu)\to L^2(X,\mu')$ defined by 
\begin{equation}
U(f)(x)\;=\;\sqrt{g_2(x)}\cdot f(x),\; f\in\; L^2(X,\mu),
\label{eq:U}
\end{equation}
and it is evident that $U^{-1}:L^2(X,\mu')\to L^2(X,\mu)$  is given by 
\[
U^{-1}(h)(x)\;=\;\sqrt{g_1(x)}\cdot h(x),\; h\in\; L^2(X,\mu').
\]
We show as follows that one can construct a corresponding $\Lambda$-projective system on $(X,\mu'),$ and that  the two representations of $C^*(\Lambda)$ associated to the  $\Lambda$-projective systems  on $(X,\mu)$ and $(X,\mu')$ are unitarily equivalent.

\begin{prop}
 \label{prop:lambda-proj-repn-un-equiv} Let $\Lambda$ be a finite $k$-graph with no sources. Suppose we are given a $\Lambda$-projective system  $\{\tau_{\lambda}:\lambda\in \Lambda\}$, $\{\tau^n: n \in \N^k\}$ and $\{ f_\lambda : \lambda \in \Lambda\}$ on a measure space $(X,\mu)$. Let $\mu'$ be a measure equivalent to $\mu$ as described above, with $g_1(x)\;=\;\frac{d\mu'}{d\mu}(x)$ and $g_2(x)=\frac{d\mu}{d\mu'}(x)$.  If we define  $\{ \tilde{f}_\lambda \}_{\lambda \in \Lambda}$ by 
$$\tilde{f}_\lambda(x)\;=\;\frac{\sqrt{g_1\circ \tau^{d(\lambda)}(x)}}{\sqrt{g_1(x)}}\cdot f_{\lambda}(x),\;\lambda\in \Lambda,$$
then  $\{\tau_{\lambda}:\lambda\in \Lambda\}$, $\{\tau^n: n\in \N^k\}$ and $\{ \tilde{f}_\lambda \}_{\lambda \in \Lambda}$ give a  $\Lambda$-projective system on $(X,\mu')$. 
Moreover, the associated representations $\{T_{\lambda}:\lambda\in \Lambda\}$ and $\{\tilde{T}_{\lambda}:\lambda\in \Lambda\}$ of $C^*(\Lambda)$ on $L^2(X,\mu)$ and $L^2(X,\mu')$  given by Equation  \eqref{eq:T-lambda} of Proposition \ref{prop:lambda-proj-repn} are unitarily equivalent via the unitary $U$ of Equation \eqref{eq:U}. 
\end{prop}
\begin{proof}
We first check that the functions $\{ \tilde{f}_\lambda \}_{\lambda \in \Lambda}$ satisfy Conditions (a) and (b) of Definition \ref{def:lambda-proj-system}. 
For  (a), recall that 
$\frac{d(\mu \circ (\tau_\lambda)^{-1})}{d\mu} = |f_\lambda |^2.$
To prove that 
$\frac{d(\mu' \circ (\tau_\lambda)^{-1})}{d\mu'} = |\tilde{f}_\lambda |^2,$
we begin with the observation that 
\[ \frac{d(\mu' \circ( \tau_\lambda)^{-1})}{d(\mu \circ (\tau_\lambda)^{-1})} = \frac{d\mu'}{d\mu} \circ (\tau_\lambda)^{-1}.\]
Thus, since $(\tau_\lambda)^{-1} = \tau^{d(\lambda)}$  on $R_\lambda$, and $f_\lambda$ vanishes off $R_\lambda$, 
\[\frac{d(\mu' \circ (\tau_\lambda)^{-1})}{d\mu'}(x) = \frac{d(\mu \circ( \tau_\lambda)^{-1})}{d\mu}(x) \frac{d\mu'}{d\mu}(\tau^{d(\lambda)}(x)) \cdot g_1(x)^{-1} =\frac{d(\mu \circ( \tau_\lambda)^{-1})}{d\mu}(x)  \frac{g_1 \circ \tau^{d(\lambda)}(x)}{g_1(x)}.\]
 But now, checking formulas, this is precisely $|\tilde f_\lambda|^2$, which proves (a).
 

For Condition (b) we note that for $\lambda,\;\nu\in \Lambda$ satisfying $s(\lambda)=r(\nu)$ we have:
$$\tilde{f}_\lambda \cdot (\tilde{f}_\nu \circ \tau^{d(\lambda)})(x)\;=\;\frac{\sqrt{g_1\circ \tau^{d(\lambda)}(x)}}{\sqrt{g_1(x)}} f_{\lambda}(x)\cdot \frac{\sqrt{g_1\circ \tau^{d(\nu)}\circ\tau^{d(\lambda)}(x))}}{\sqrt{g_1(\tau^{d(\lambda)}(x))}}f_{\nu}\circ \tau^{d(\lambda)}(x)$$
$$=\;\frac{\sqrt{g_1\circ \tau^{d(\lambda\nu)}(x)}}{\sqrt{g_1(x)}} f_{\lambda}(x)f_{\nu}\circ \tau^{d(\lambda)}(x)$$
$$=\;\frac{\sqrt{g_1\circ \tau^{d(\lambda\nu)}(x)}}{\sqrt{g_1(x)}}f_{\lambda\nu}(x)\;=\;\tilde{f}_{\lambda\nu}.$$
Thus, the functions $\tilde f_\lambda$ indeed give a $\Lambda$-projective system on $(X, \mu)$.

Finally, we check that our Hilbert space isomorphisms  $U^{-1}:L^2(X,\mu')\to L^2(X,\mu)$ and $U:L^2(X,\mu)\to L^2(X,\mu')$ intertwine the representations $\{T_\lambda\}, \{\tilde T_\lambda\}$ associated to the $\Lambda$-projective systems $\{ f_\lambda\}, \{ \tilde f_\lambda\}$: 
for any  $h\in L^2(X,\mu')$ and $\lambda\in\Lambda$, 
$$U\circ T_{\lambda}\circ U^{-1}(h)(x)\;=\;[\sqrt{g_1(x)}]^{-1}\cdot T_{\lambda}\circ U^{-1}(h)(x)\;=\;[\sqrt{g_1(x)}]^{-1}\cdot T_{\lambda}(\sqrt{g_1}\cdot h)(x)$$
$$\;=\;[\sqrt{g_1(x)}]^{-1}\cdot f_{\lambda}(x)\cdot (\sqrt{g_1}\cdot h)(\tau^{d(\lambda)}(x))\;=\;\frac{\sqrt{g_1\circ \tau^{d(\lambda)}(x)}}{\sqrt{g_1(x)}}\cdot f_{\lambda}(x)\cdot h(\tau^{d(\lambda)}(x))$$
$$\;=\;\tilde{f}_{\lambda}\cdot h(\tau^{d(\lambda)}(x))\;=\;\tilde{T}_{\lambda}(h)(x).$$
Therefore, the two representations of $C^*(\Lambda),\;\{T_{\lambda}:\lambda\in \Lambda\}$ and $\{\tilde{T}_{\lambda}:\lambda\in \Lambda\},$ are unitarily equivalent, as desired.

\end{proof}

  \begin{rmk}
  \label{rmk:lambda-proj-comments}
 
   A straightforward computation reveals that if $f_\lambda$ is the positive square root of $\left(\Phi_\lambda \circ \tau^{d(\lambda)}\right)^{-1}$, then $
 f_\lambda' $ is given by the same formula (using instead the Radon-Nikodym derivative with respect to the measure $\mu'$.)
 
\end{rmk}

 The following Proposition generalizes the familiar Kirchhoff-Ohm rule for adding resistors in parallel on electrical networks, realized on graphs of vertices and edges.  In the graph-theoretic setting, one assigns resistors to each edge of the graph, and describes the current as a function on the space of edges or finite paths.  The voltage is then given by a function on the vertices.  Using this formulation, \cite{dutkay-jorgensen-spectral-cuntz} highlights the relationship between representations of the Cuntz algebras $\mathcal O_N$ and spectral graph analysis. 
 
 Other extensions of the Kirchhoff-Ohm rule to a functional-analytic setting can be found in \cite{jorgensen-pearse, jorgensen-tian} and the papers referenced therein; these treatments rely on analysis of operators in Hilbert space, and on both commutative and non-commutative methods.
 
For a wider study of analysis on infinite graphs, in particular their  relationship with electrical networks,  see also \cite{Peres}.  

\begin{prop}
\label{lem-proposition-2.8} (C.f. Proposition 2.8 of \cite{dutkay-jorgensen-monic})
Let  $\Lambda$ be a finite $k$-graph with no sources and let $\{\tau_{\lambda}:\lambda\in \Lambda\}$, $\{\tau^n: n \in \N^k\}$ and $\{ f_\lambda \}_{\lambda \in \Lambda}$ be a  
$\Lambda$-projective system on $(X,\mu)$. Then for any $n \in \N^k$, the equality
\[
 \quad \sum_{d(\lambda)=n} \frac{1}{|f_\lambda\circ \tau_\lambda|^2} = \frac{d \mu \circ (\tau^{{n}})^{-1} }{d\mu}.
\]
holds $\mu$-a.e.
\end{prop} 
 
\begin{proof} Recall from Remark \ref{rmk:abs-cts-inverse} that for any $\Lambda$-semibranching function system, in particular any $\Lambda$-projective system, we have $\mu \circ (\tau^n)^{-1} << \mu$.  
 Let $\phi$ be a continuous function on $X$, and fix $n\in \N^k$.  Since $ \tau_\lambda \circ \tau^n = id|_{R_\lambda}$ whenever $d(\lambda) = n$, the substitutions $x \mapsto \tau_\lambda(x)$   give
 \[
\int_X \sum_{\lambda: d(\lambda) = n} \phi \cdot \frac{1}{|f_\lambda \circ \tau_\lambda|^2} \, d\mu = \sum_{\lambda:\,d(\lambda)=n}  \int_{D_{s(\lambda)}} \phi \frac{1}{|f_\lambda\circ \tau_\lambda|^2} d \mu = \sum_{\lambda:\,d(\lambda)=n} \int_{R_\lambda}(\phi\circ \tau^n) \frac{1}{|f_\lambda|^2} d (\mu\circ(\tau_\lambda)^{-1}).\]
 Since $\{f_\lambda\}_{\lambda \in \Lambda}$ is a $\Lambda$-projective system, we have $\frac{d(\mu \circ \tau_\lambda^{-1})}{d\mu} = |f_\lambda|^2 = (\Phi_\lambda \circ \tau^n)^{-1}$ is nonzero on $R_\lambda$; thus, $\frac{1}{|f_\lambda|^2} d(\mu \circ (\tau_\lambda)^{-1}) = d\mu$ on $R_\lambda$.  In other words, 
 \begin{align*}
 \sum_{\lambda:\,d(\lambda)=n}\int_{D_{s(\lambda)}} \phi(x) \frac{1}{|f_\lambda\circ \tau_\lambda(x)|^2} d \mu(x) &= \sum_{\lambda: d(\lambda) = n}  \int_{R_\lambda} \phi\circ \tau^n(x)  d \mu (x) = \int_X \phi(x)\   d (\mu\circ(\tau^n)^{-1})(x)\\
 & = \int_X \phi(x)\,  \frac{d (\mu\circ(\tau^n)^{-1})}{d\mu}(x) \, d\mu. \qedhere
 \end{align*}
\end{proof} 

Proposition \ref{prop-2.11} below is the analog of Proposition 2.11 of \cite{dutkay-jorgensen-monic} for $\Lambda$-projective systems. 
\begin{prop}
\label{prop-2.11} Let $\Lambda$ be a finite $k$-graph with no sources. Suppose we are given two $\Lambda$-projective systems on $X$, with the same prefixing and coding maps $\{\tau_{\lambda}:\lambda\in \Lambda\}$, $\{\tau^n: n \in \N^k\}$, but with different measures $\mu, \mu'$ and $\Lambda$-projective functions $\{ f_\lambda \}_{\lambda \in \Lambda}$ for $(X, \mu)$ and $\{ f_\lambda' \}_{\lambda \in \Lambda}$ for $(X,\mu')$.

Let $d\mu' = h^2 d\mu + d\nu 
$
be the Lebesgue-Radon-Nikodym decomposition, with $h \geq 0$ and $\nu$
singular with respect to $\mu$. Then there is a partition of $X$ into Borel sets $
X = A \cup B$
such that:
\begin{enumerate}
\item[(a)]  The function $h$ is supported on $A$, $\nu$ is supported on $B$, and $\mu(B) = 0$, $\nu(A) = 0$.
\item[(b)] The sets $A$, $B$ are invariant under $\tau^n$ for all $n \in \N^k$, i.e., 
\[
(\tau^n)^{-1}(A) = A,\quad\text{and}\quad (\tau^n)^{-1}(B) = B.
\]
\item[(c)]  We have $\nu \circ  \tau_\lambda^{-1} << \nu$ and $k_\lambda := \sqrt{\frac{d(\nu \circ \tau_\lambda^{-1} )}{d\nu}}$ is supported on $B$.
\item[(d)] $| f_\lambda' | \cdot h =| f_\lambda|\cdot (h \circ \tau^{d(\lambda)} )\, {\mu\text{-a.e.~on $A$ and } |f_\lambda'| = k_\lambda\,  \nu\text{-a.e.~on } B.}$ 
\end{enumerate}
\end{prop}

\begin{proof} Recall that the support of a measure is the smallest closed set whose complement has measure zero.  Let $\tilde{B}$ be the support of $\nu$, and observe that $\mu (\tilde{B}) = 0$. We observe that the definitions of $\Lambda$-semibranching function systems and $\Lambda$-projective systems, together with the fact that $(\tau^n)^{-1}(\tilde B) = \bigcup_{\lambda \in \Lambda^n} \tau_\lambda(\tilde B)$, imply that
\[
(\tau^n)^{-1}(\tilde{B})\text{ and } (\tau_\lambda)^{-1}(\tilde{B}) 
\] 
have $\mu$-measure zero. Therefore
we can take the orbit $B$ of $\tilde{B}$ under the functions $\{\tau^n: n \in \N^k\}$ and $\{\tau_\lambda: \lambda \in \Lambda\},$ and $B$ will then have $\mu$-measure zero. Let
$A := X \backslash B$. Then $A$ contains the support of $\mu$, and we can choose $h$ to be supported on $A$.  Moreover, 
$\nu(A) = 0$. By construction, $A$ and $B$ are invariant under $\tau^n$.
To prove $(c)$, let $E$ be a Borel set with $\nu (E) = 0$. Then $\nu (E\cap B) = 0$, so the fact that $\mu$ vanishes on $B$ implies that $\mu' (E\cap B) = 0$. 
We consequently have $\mu'(\tau_\lambda^{-1} (E \cap B)) = 0,$ which means that  $\mu'(\tau_\lambda^{-1} (E) \cap  B) = 0$, so $\nu(\tau_\lambda^{-1} (E)) = 0$.
Since $B$ is invariant under $\tau_\lambda^{-1}$ and $\nu$ and $\nu \circ \tau_\lambda^{-1}$  are supported on $B$, it follows that 
$k_\lambda$ is supported on $B$. 
To see $(d)$, let $f$ be a bounded Borel function supported on $A$. Then we have
\[
\begin{split}
\int_A |f'_\lambda|^2\, f\, h^2\, d\mu &=\int_A |f'_\lambda|^2\, f\,  d\mu'= \int_A f \, \frac{d(\mu' \circ \tau_\lambda^{-1})}{d(\mu')}d\mu'  =\int_A (f\circ \tau_\lambda)\,  d\mu' \\
&= \int_A (f\circ \tau_\lambda)\, h^2\, d\mu = \int_X (f\circ \tau_\lambda)\, (h^2 \circ \tau^{d(\lambda)}\circ \tau_\lambda)\, d\mu= \int_X f\, (h^2 \circ \tau^{d(\lambda)})\, d(\mu\circ \tau_\lambda^{-1})\\
&= \int_X f\, (h^2 \circ \tau^{d(\lambda)})\, | f_\lambda|^2 d\mu,\\
\end{split}
\]
which implies the first relation. The second relation follows from the fact that
$ \mu'|_B = \nu$.
\end{proof}

\begin{defn}
\label{def:disjoint}
Two representations $\pi, \pi'$ of a $C^*$-algebra $A$ are  \emph{disjoint} if  no nonzero
subrepresentation of $\pi$ is unitarily equivalent to a subrepresentation of $\pi'$.
\end{defn}

 \begin{thm} (C.f. Theorem 2.12 of \cite{dutkay-jorgensen-monic})
 \label{thm-disjoint-monic-repres}
Let $\Lambda$ be a finite $k$-graph with no sources. Suppose we are given two $\Lambda$-projective systems on the infinite path space $\Lambda^\infty$ with the standard prefixing and coding maps $\{\sigma_{\lambda}:\lambda\in \Lambda\}$, $\{\sigma^n: n \in \N^k\}$, but associated to different measures $\mu, \mu'$ and different $\Lambda$-projective families of non-negative functions $\{ f_\lambda \}_{\lambda \in \Lambda}$ on $(\Lambda^\infty,\mu)$, and $\{ f_\lambda' \}_{\lambda \in \Lambda}$ on $(\Lambda^\infty,\mu')$. Then the two associated representations $\{T_\lambda: \lambda \in \Lambda\}$ and $\{T_\lambda': \lambda \in \Lambda\}$  of $C^*(\Lambda)$ given by Equation \eqref{eq:T-lambda} of Proposition  \ref{prop:lambda-proj-repn}  are disjoint if and only if the measures $\mu$ and $\mu'$ are mutually singular.
 \end{thm}

 \begin{proof}

If the representations are not disjoint, there exist subspaces $\H_\mu \subseteq L^2(\Lambda^\infty, \mu)$ and $\H_{\mu'} \subseteq L^2(\Lambda^\infty, \mu')$, preserved by their respective representations, and a unitary $W: \H_\mu \to \H_{\mu'}$ such that 
\[W T_\lambda |_{\H_\mu} = T_\lambda' |_{\H_{\mu'}} W, \qquad W T_\lambda^* |_{\H_\mu} = (T_\lambda')^* |_{\H_{\mu'}} W.\]
The fact that each operator $T_\lambda^*$ also preserves $\H_\mu$ implies that 
\begin{align*}
W T_\lambda T_\lambda^* |_{\H_\mu} &= W T_\lambda |_{\H_\mu} T_\lambda^* |_{\H_\mu} = T_\lambda' |_{\H_{\mu'}} W T_\lambda^* |_{\H_\mu} \\
&= T_\lambda' (T_\lambda')^*|_{\H_{\mu'}} W.
\end{align*}
Moreover, Equation \eqref{eq:lambda-proj-range-sets} tells us that 
\[ T_\lambda T_\lambda^* = M_{\chi_{Z(\lambda)}} = T_\lambda' (T_\lambda')^*.\]
In other words,  the representations of $C(\Lambda^\infty)$ given by $\chi_{Z(\lambda)} \mapsto T_\lambda T_\lambda^*$ and $\chi_{Z(\lambda)} \mapsto T_\lambda' (T_\lambda')^*$ (on $L^2(\Lambda^\infty, \mu)$ and $L^2(\Lambda^\infty, \mu')$ respectively) are multiplication representations.
Since $W$ implements a unitary equivalence between their subrepresentations on $\H_\mu$ and $H_{\mu'}$ respectively, Theorem 2.2.2 of \cite{arveson} tells us that  the measures $\mu, \mu'$ cannot be mutually singular.  
 
 For the converse, assume that the representations are disjoint and that the measures
 are not mutually singular. Then, use Proposition \ref{prop-2.11}  and decompose $ d\mu' = h^2 d\mu +d\nu$,
 with the subsets $A$, $B$ as in Proposition \ref{prop-2.11}.
 Define the operator $W$ on $L^2(\Lambda^\infty, \mu')$ by $W(f) = f\cdot h$ if $f \in  L^2(A, \mu')$, and $W(f) = 0$ on the
 orthogonal complement of $L^2(A, \mu') \subseteq L^2(\Lambda^\infty, \mu')$. Since $A$ is invariant under $\tau^n$ for all $ n$, $L^2(A, \mu')$ is an invariant subspace
 for the representation. To  check that $W$ is intertwining, we use part (d) of Proposition \ref{prop-2.11} and the non-negativity condition on $\{f_\lambda\}$ and $\{f'_\lambda\}$ to obtain the almost-everywhere equalities
 \[
 T_\lambda W (f )= f_\lambda(h \circ \tau^{d(\lambda)})(f \circ \tau^{d(\lambda)}) = f_\lambda' \, h\, (f \circ \tau^{d(\lambda)}) = W T_\lambda'(f).
 \]
 Since $W$ intertwines the representations $\{ T_\lambda\}_{\lambda \in \Lambda}, \{ T_{\lambda}'\}_{\lambda \in \Lambda}$ of $C^*(\Lambda)$, we must have $W =0$; hence $h=0$, so $\mu, \mu'$ are mutually singular.
 \end{proof}
 
 \begin{rmk}
 As a Corollary of  Theorem  \ref{thm-disjoint-monic-repres}, 
 we see  that the examples of measures introduced in 
 Section \ref{sec-Markov-measure-Lambda-semibran-0}
 generate representations of $C^*(\Lambda)$ disjoint from the  representation of  \cite[Theorem 3.5]{FGKP}; in fact, these  Markov measures are 
 mutually singular with the Perron-Frobenius measure \cite{dutkay-jorgensen-monic}, \cite{Kaku}.
 \end{rmk}

%

  \begin{thm} (C.f. Theorem 2.13 of \cite{dutkay-jorgensen-monic})\label{thm:ergodic}
    Let $\Lambda$ be a finite $k$-graph with no sources.  
  Suppose that the infinite path space $\Lambda^\infty$ admits a $\Lambda$-projective system on $(\Lambda^\infty,\mu)$ for some measure $\mu$ with standard prefixing maps $\{\sigma_\lambda\, :\, \lambda\in\Lambda\}$, coding maps $\{\sigma^n\, :\, n\in\N^k\}$ and functions $\{f_\lambda\,:\,\lambda\in\Lambda\}$ satisfying Conditions (a)--(c) of Definition~\ref{def:lambda-proj-system}.  Let $\{T_\lambda:\lambda\in\Lambda\}$ be the operators given by Equation \eqref{eq:T-lambda} of Proposition \ref{prop:lambda-proj-repn}.
 Then:
 \begin{itemize}
 \item[(a)] The commutant of the operators $\{T_\lambda: \lambda \in \Lambda\}$ consists of multiplication
      operators by functions $h$ with $h \circ \sigma^n = h$, $\mu$-a.e for all $n\in \N^k$.  
 \item[(b)] The representation given by $\{T_\lambda:\lambda\in\Lambda\}$ is
      irreducible if and only if the coding maps $\sigma^n$ are jointly ergodic with respect to the measure $\mu$, i.e., the only Borel
      sets $A\subset \Lambda^\infty$  with $ (\sigma^{n})^{-1}(A)= A$ for all $n$ are sets of measure zero, or of full measure.       \end{itemize}
       
  \end{thm}

 \begin{proof}
 For (a), let $T$ be an operator in the commutant of the $\Lambda$-projective representation. Then $T$ commutes with the representation $\pi$
 of the abelian subalgebra $C(\Lambda^\infty)$ of $C^*(\Lambda)$, where 
 \[ \pi(\chi_{Z(\lambda)}) = T_\lambda T_\lambda^* = M_{\chi_{Z(\lambda)}}.\]
 {Since $C(\Lambda^\infty) \subseteq L^\infty(\Lambda^\infty, \mu)$, the maximal abelian subalgebra of $B(L^2(\Lambda^\infty, \mu))$, } 
 this implies that $T$ must be a multiplication operator.
 Write $T = M_h$. Since $T$ commutes with $T_\lambda$ we obtain
 \[ 
 h\,f_\lambda\,(f \circ \sigma^n) = f_\lambda (h \circ \sigma^n)(f \circ \sigma^n)\;\;\text{ whenever } d(\lambda)=n,
 \ f \in  L^2(X, \mu).
 \]
 Take $f=1$ and use the definition of $\Lambda$-projective system to conclude that $h \circ \sigma^n = h$ on $Z(\lambda)$.  Since $\Lambda^\infty = \bigsqcup_{\lambda \in \Lambda^n} Z(\lambda)$ for any $n \in \N^k$, it follows that $h \circ \sigma^n = h$ on $\Lambda^\infty.$
 Conversely, it is easy to see that any such multiplication operator $T = M_h$ commutes with each operator $T_\lambda$.
 
 For (b), 
 suppose that the measure is  ergodic with respect to all the coding maps $\sigma^n$. Suppose that $T = M_h$ is in the commutant of the representation $\{T_\lambda\}_\lambda$.  For any constant $\alpha$, and any $n \in \N^k$, we have 
 \[ A_\alpha := \{ x: h(x) = \alpha\} = \{ x: h(\sigma^n(x)) = \alpha \}  = \{ x: \sigma^n(x) \in A_\alpha \} = (\sigma^n)^{-1}(A_\alpha)\]
 (up to sets of measure zero). 
 Since $\mu$ is jointly ergodic for the functions $\sigma^n$, we have $\mu(A_\alpha) \in \{ 0, 1\}$; in other words, $h$ is constant, and $\{T_\lambda\}_\lambda$ determines an irreducible representation.

 On the other hand, if the representation generated by $\{T_\lambda\}_\lambda$ is irreducible, suppose there exists a set $A \subseteq \Lambda^\infty$  such that $(\sigma^{n})^{-1}(A) = A$ for all $n$.  Then 
 \[ \chi_A \circ \sigma^n  = \chi_{(\sigma^{n})^{-1}(A)} = \chi_A \;\;\text{for all $n$}.\]
 
 It now follows from the definition of the operators $T_\lambda$ that multiplication by $\chi_A$ commutes with $T_\lambda$ and $T_\lambda^*,$ for all $\lambda \in \Lambda$.  In other words, 
 \[ \chi_A \in \C \chi_{\Lambda^\infty}\;\;\; \text{if and only if}\;\;\; \mu(A) \in \{0, 1\}.\]
 
 Thus, the operators $\sigma^n$ are jointly ergodic with respect to the measure $\mu$.
 \end{proof}

  \subsection{Projection valued measures}
\label{sec:Proj-valued-measures}

Inspired by Dutkay, Haussermann, and Jorgensen     \cite{dutkay-jorgensen-monic, dutkay-jorgensen-atomic}, here we study   the projection valued measure associated to a  representation of $C^*(\Lambda)$. This material will be employed in Section \ref{sec:monic-results} to analyze when a $\Lambda$-semibranching function system gives rise to a monic representation of $C^*(\Lambda)$ (see Definition \ref{def:lambda-monic-repres}), {as well as underpinning our work on purely atomic representations in Section \ref{sec:atomic_repn}}.

\begin{defn} 
	\label{def-proj-val-measu}
Let $\Lambda$ be a row-finite $k$-graph with no sources.
Given a representation $\{ t_\lambda\}_{\lambda\in\Lambda }$ of a $k$-graph $C^*$-algebra $C^*(\Lambda)$ on a Hilbert space $\mathcal{H}$, we define a projection valued function $P$ on $\Lambda^\infty$ by 
 \[P(Z(\lambda)) = t_\lambda t_\lambda^* \quad\text{for all $\lambda \in\Lambda$}.
 \]
 \end{defn}
 
 For the discussion of this function, in particular for the proof that $P$ is indeed a projection valued measure, we recall the notion of a ``minimal common extension'' in a $k$-graph.  Recall from Equation \eqref{eq:lambda_min} (also see \cite[Definition~2.2]{RSY2}) that given $\lambda, \eta \in \Lambda$ with $r(\lambda) = r(\eta)$, their set of minimal common extensions is 
\[ \Lambda^{\operatorname{min}}(\lambda, \eta) = \{ (\rho, \xi) \in \Lambda \times \Lambda: \lambda \rho = \eta \xi \text{ and } d(\lambda \rho) = d(\lambda) \vee d(\eta)\}.\]
In a directed graph, $\Lambda^{\operatorname{min}}(\lambda, \eta)$ will have at most one element; this need not be true in a $k$-graph if $k > 1$.

As mentioned in Equation \eqref{eq:CK4-2},  for any $n \in \N^k$ the generators $\{t_\lambda\}_{\lambda \in \Lambda}$ of  $C^*(\Lambda)$ satisfy 
 \[ t_\lambda^* t_\eta = \sum_{(\xi, \rho) \in \Lambda^{\operatorname{min}}(\lambda, \eta)} t_\xi t_\rho^*, \quad \text{ and }\quad  t_v = \sum_{\lambda \in v\Lambda^n} t_\lambda t_\lambda^*.\]

 \begin{thm} 
 \label{conj-palle-proj-valued-measure-gen-case}
 Let $\Lambda$ be a row-finite $k$-graph with no sources.
 Given a representation $\{ t_\lambda\}_{\lambda\in\Lambda}$ of a $k$-graph $C^*$-algebra $C^*(\Lambda)$ on a Hilbert space $\mathcal{H}$, then the assignment 
 \[P(Z(\lambda)) = t_\lambda t_\lambda^*  \quad\text{for}\quad  \lambda \in \Lambda
 \]
 extends to  a projection valued measure on the Borel $\sigma$-algebra  $\mathcal{B}_o(\Lambda^\infty)$ of the infinite path space $\Lambda^\infty$. 
 \end{thm}
 
 \begin{proof}
 We first claim that $P$ is finitely additive.  To see this, fix $\lambda\in\Lambda$ and observe that for any $n \in \N^k$ we have 
 \begin{equation} P(Z(\lambda)) = t_\lambda t_\lambda^* = t_\lambda t_\lambda^*\sum_{\zeta \in r(\lambda) \Lambda^n}  t_\zeta t_\zeta^* = \sum_{\zeta \in r(\lambda) \Lambda^n} \sum_{(\rho,\xi) \in \Lambda^{\operatorname{min}}(\lambda, \zeta)} t_{\lambda \rho} t_{\lambda \rho}^*= \sum_{\zeta \in r(\lambda) \Lambda^n} \sum_{(\rho,\xi) \in \Lambda^{\operatorname{min}}(\lambda, \zeta)} P(Z(\lambda \rho)).
\label{eq:P-additivity} 
 \end{equation}
 Now, suppose that $Z(\lambda) = \bigsqcup_{i=1}^p Z(\eta_i)$. Observe that this implies that $\Lambda^{\operatorname{min}}(\lambda, \eta_i)$ is nonempty for each $i$.  Writing $m = d(\lambda) \vee \bigvee_i d(\eta_i)$, we have 
 \[ Z(\lambda) = \bigsqcup_{\eta \in s(\lambda) \Lambda^{m-d(\lambda)}} Z(\lambda \eta) = \bigsqcup_{i=1}^p \bigsqcup_{\alpha_{ij} \in s(\eta_i)\Lambda^{m-d(\eta_i)}} Z(\eta_i \alpha_{ij}).\]
 Since both sides are disjoint unions of cylinder sets of degree $m$, the list of cylinder sets under consideration must be precisely the same on both sides.  That is, each set $Z(\lambda \eta)$ must equal $Z(\eta_i \alpha_{ij})$ for precisely one $i$ and one $\alpha_{ij}$.
 
 Now, observe that if $\zeta \in r(\lambda)\Lambda^m$, the fact that $d(\lambda) \leq m$ implies that $\Lambda^{\operatorname{min}}(\lambda, \zeta) = \{ (\eta, s(\zeta)): \lambda \eta = \zeta\}$ contains precisely one element (unless $\zeta$ is not an extension of $\lambda$, in which case $\Lambda^{\operatorname{min}}(\lambda, \zeta) = \emptyset$).
 Thus, 
 \[P(Z(\lambda)) = \sum_{\zeta \in r(\lambda) \Lambda^m} \sum_{(\rho, \xi) \in \Lambda^{\operatorname{min}}(\lambda, \zeta)} P(Z(\lambda \rho)) = \sum_{\eta \in s(\lambda)\Lambda^{m-d(\lambda)}} P(Z(\lambda \eta)).\]
 Applying the same logic to $P(Z(\eta_i))$, we see that 
 \begin{align*}
 P(Z(\lambda)) &= \sum_{\zeta \in r(\lambda) \Lambda^m} \sum_{(\rho,\xi) \in \Lambda^{\operatorname{min}}(\lambda, \zeta)} P(Z(\lambda \rho)) \\ 
 &= \sum_{\eta \in s(\lambda) \Lambda^{m-d(\lambda)}} P(Z(\lambda \eta)) \\ 
 &= \sum_{i=1}^p \sum_{\alpha_{ij} \in s(\eta_i) \Lambda^{m-d(\eta_i)}} P(Z(\eta_i \alpha_{ij})) \\
 &= \sum_{i=1}^p P(Z(\eta_i)),
 \end{align*}
so $P$ is finitely additive. 
  Countable additivity then follows by Lemma~\ref{lem:measure}; then we apply the { Carath\'eodory/Kolmogorov consistency/extension theorem \cite{kolmogorov}, \cite{Tum},} and we're done.
 \end{proof}

We now establish some properties of the projection valued measure $P$ on $\Lambda^\infty$.
 The equations below are the analogues for $k$-graphs of Equations (2.7) and (2.8) and (2.13) of \cite{dutkay-jorgensen-atomic}.

 \begin{prop} 
 \label{prop-atomic-basic-equns} Let $\Lambda$ be a row-finite, source-free $k$-graph, and fix a representation $\{t_\lambda:\lambda \in \Lambda\}$ of $C^*(\Lambda)$.
 For any  $\eta \in \Lambda$, we have
 \begin{itemize}\label{prop:pvm-properties}
 \item[(a)] For $\lambda, \eta \in\Lambda$ with $s(\lambda)=r(\eta)$, we have $t_\lambda P(Z(\eta)) t_\lambda^*=P(\sigma_\lambda(Z(\eta)))$, where $\sigma_\lambda$ is the {prefixing} 
  map on $\Lambda^\infty$ given in Equation \eqref{eq:shift-map}.
 
 \item[(b)] For any fixed $n \in \N^k$, we have
  \[
  \sum_{\lambda \in  f(\eta) \Lambda^n} t_\lambda P(\sigma_\lambda^{-1}(Z(\eta))) t_\lambda^* = P(Z(\eta)) ;
  \]
  \item[(c)]  For any $\lambda,\eta \in \Lambda$ with $r(\lambda)=r(\eta)$, we have $t_\lambda P(\sigma_\lambda^{-1}(Z(\eta))) = P(Z(\eta)) t_\lambda$;
  \item[(d)] When $\lambda\in\Lambda^n$, we have $t_\lambda P( Z(\eta)) =  P((\sigma^n)^{-1}(Z(\eta))) t_\lambda$.
 \end{itemize}
 \end{prop}

 \begin{proof}
 For (a), first observe that both sides of the equation are zero unless $s(\lambda) = r(\eta)$, because $t_\lambda t_\eta = \delta_{s(\lambda), r(\eta)} t_{\lambda \eta}$. If $s(\lambda) = r(\eta)$, we have
 \[
 LHS= t_\lambda P(\eta) t_\lambda^* =t_\lambda t_\eta  t_\eta^*  t_\lambda^*.
 \]
 But since $\sigma_\lambda(Z(\eta)) = Z(\lambda\eta)$ when $s(\lambda) = r(\eta)$, we also have
 \[
 RHS= t_{\lambda\eta} t_{\lambda\eta}^*= t_{\lambda}t_{\eta} t_{\eta}^*t_{\lambda}^*.
 \]
 For (b), observe that $\sigma_\lambda^{-1}(Z(\eta)) = \bigsqcup_{(\rho, \xi) \in \Lambda^{\operatorname{min}}(\lambda, \eta)} Z(\rho)$.   Thus, Equation \eqref{eq:P-additivity} implies that 
 \begin{align*}
 \sum_{\lambda \in r(\eta)\Lambda^n} t_\lambda P(\sigma_\lambda^{-1}(Z(\eta)) t_\lambda^* & = \sum_{\lambda \in r(\eta)\Lambda^n} \sum_{(\rho,\xi) \in \Lambda^{\operatorname{min}}(\lambda, \eta)} t_\lambda t_\rho t_\rho^* t_\lambda^* =  \sum_{\lambda \in r(\eta)\Lambda^n} \sum_{(\rho,\xi) \in \Lambda^{\operatorname{min}}(\lambda, \eta)} P(Z(\lambda \rho))\\
 & = \sum_{\lambda \in r(\eta)\Lambda^n} \sum_{(\rho,\xi) \in \Lambda^{min}(\lambda, \eta)} P(Z(\eta \xi)) = P(Z(\eta)).
 \end{align*}
 For (c), write $\sigma_\lambda^{-1}(Z(\eta)) = \bigsqcup_{(\rho, \xi) \in \Lambda^{\operatorname{min}}(\lambda, \eta)} Z(\rho)$.  Since $\Lambda$ is row-finite, this union is finite.   Then from (a) and the fact that $t_\lambda^* t_\lambda = t_{s(\lambda)}$, we   have
 \[
 t_\lambda P(\sigma_\lambda^{-1}(Z(\eta)) ) = \sum_{(\rho,\xi) \in \Lambda^{min}(\lambda, \eta)} P(\sigma_\lambda(Z(\rho))) t_\lambda = \sum_{(\rho, \xi)\in \Lambda^{\operatorname{min}}(\lambda, \eta)} t_\lambda t_\rho t_\rho^* .
 \]
 On the other hand, we also have
 \[P(Z(\eta)) t_\lambda = t_\eta t_\eta^* t_\lambda = t_\eta \sum_{(\rho,\xi) \in \Lambda^{\operatorname{min}}(\lambda, \eta)} t_\xi t_\rho^* = \sum_{(\rho, \xi)\in \Lambda^{\operatorname{min}}(\lambda, \eta)} t_\lambda t_\rho t_\rho^* .\]

For (d), notice that 
 for the formula on the LHS not to be zero we need $s(\lambda) = r(\eta)$; in this case the LHS is $t_\lambda t_\eta t_\eta^*$.  
 
 For the right-hand side, 
 notice that $(\sigma^n)^{-1}(Z(\eta)) = \bigsqcup_{\zeta \in \Lambda^n r(\eta)} Z(\zeta \eta)$.  Thus, we have 
 \[ P((\sigma^n)^{-1}(Z(\eta))) t_\lambda = \left( \sum_{\zeta \in \Lambda^n r(\eta)} t_\zeta t_\eta t_\eta^* t_\zeta^* \right) t_\lambda.\]
 Since $d(\zeta) =n = d(\lambda) $, we have  $\Lambda^{\operatorname{min}}(\zeta, \lambda) = \emptyset$ unless $\zeta = \lambda$ -- that is,
 \[ P((\sigma^n)^{-1}(Z(\eta))) t_\lambda = t_\lambda t_\eta t_\eta^*.\]

 \end{proof}

\section{Monic representations of finite $k$-graph algebras}
\label{sec:monic-results} 

The main result of this section, Theorem \ref{thm-characterization-monic-repres}, establishes that every monic  representation 
of a finite, strongly connected $k$-graph algebra $C^*(\Lambda)$ is unitarily equivalent to  a $\Lambda$-projective representation of $C^*(\Lambda)$ on $L^2(\Lambda^\infty, \mu_\pi)$, where the measure $\mu_\pi$ arises from the representation.  (See Definition \ref{def:lambda-monic-repres} and  Equation \eqref{eq:monic_measure} below for details.)  After proving Theorem \ref{thm-characterization-monic-repres}, we return in Section \ref{sec:monic-examples} to the examples introduced in Section \ref{sec:examples_a}  above, and 
use Theorem \ref{thm-characterization-monic-repres} to identify which of the examples are monic.

 
 In the final section of their paper \cite{bezuglyi-jorgensen}, Bezuglyi and Jorgensen studied the relationship between semibranching function systems and monic representations of Cuntz--Krieger algebras (1-graph $C^*$-algebras). A \emph{monic} representation of a Cuntz--Krieger algebra is a faithful representation for which a canonical abelian subalgebra admits a monic vector; see Definition \ref{def:lambda-monic-repres} below.
Theorem 5.6 of \cite{bezuglyi-jorgensen} establishes that within a specific class of semibranching function systems, which the authors term \emph{monic systems}, those for which the underlying space is the infinite path space $\Lambda^\infty$ are precisely the systems which give rise to monic representations of the Cuntz--Krieger algebra.  The $\Lambda$-projective systems studied in Section \ref{sec:lambda-proj-systems} constitute our extension to $k$-graphs of the monic systems for Cuntz--Krieger algebras.
Thus, even in the case of 1-graph algebras (Cuntz--Krieger algebras), our Theorem \ref{thm-characterization-monic-repres} is substantially stronger than Theorem 5.6 of \cite{bezuglyi-jorgensen}.

\begin{defn}\label{def:lambda-monic-repres}
 Let $\Lambda$ be a finite $k$-graph with no sources.
  A representation $\{ t_\lambda\,:\, \lambda\in\Lambda\}$ of a $k$-graph  on a Hilbert space $\H$ is called \emph{monic} if $t_\lambda \not= 0$ for all $\lambda \in \Lambda$, and there exists a vector $\xi\in \mathcal{H}$ such that
 \[
 \overline{\text{span}}_{\lambda \in \Lambda} \{ t_\lambda t_\lambda^* \xi \} = \mathcal{H}.
 \]
 \end{defn}

Recall that for a representation $\{t_\lambda\}_{\lambda\in \Lambda}$ on $\mathcal{H}$ of a  row-finite, source-free $k$-graph $C^*$-algebra $C^*(\Lambda)$,  we have an associated projection valued measure $P$ on the Borel $\sigma$-algebra $\mathcal{B}_o(\Lambda^\infty)$ as in Theorem~\ref{conj-palle-proj-valued-measure-gen-case}. Then we obtain a representation $\pi:C(\Lambda^\infty)\to B( \mathcal{H})$ given by
\[
\pi(f)=\int_{\Lambda^\infty}f(x)dP(x),
\]
which gives, for $\lambda\in\Lambda$,
\begin{equation}\label{eq:repn_pi}
\pi(\chi_{Z(\lambda)})=\int_{\Lambda^\infty}\chi_{Z(\lambda)}\,(x)\,dP(x)=P(Z(\lambda))=t_\lambda t^*_\lambda.
\end{equation}
Since 
we can view $C(\Lambda^\infty)$ as a subalgebra of $C^*(\Lambda)$ via the embedding $\chi_{Z(\lambda)} \mapsto t_\lambda t_\lambda^*$, the representation $\pi$ is often understood as  the restriction of the representation $\{ t_\lambda \}_{\lambda \in \Lambda}$ to the ``diagonal subalgebra'' $\overline{\text{span}}\{ t_\lambda t_\lambda^*\}_{\lambda \in \Lambda}.$

If the representation $\{t_\lambda\}_\lambda$ is monic, then Definition \ref{def:lambda-monic-repres} implies that
there is a cyclic vector $\xi\in \mathcal{H}$ for $\pi$.  This induces  a Borel measure $\mu_\pi$ on $\Lambda^\infty$ given by
\begin{equation}\label{eq:monic_measure}
\mu_\pi(Z(\lambda))=\langle \xi, P(Z(\lambda))\xi\rangle=\langle \xi, t_\lambda t^*_\lambda\xi \rangle.
\end{equation}

\begin{thm}
\label{thm-characterization-monic-repres} 
Let $\Lambda$ be a finite  $k$-graph with no sources.  
If $\{t_\lambda\}_{\lambda \in \Lambda}$ is a monic representation of $C^*(\Lambda)$ on a Hilbert space $\H$,
then $\{t_\lambda\}_{\lambda\in \Lambda}$ is unitarily equivalent to a representation $\{S_\lambda\}_{\lambda\in \Lambda}$ associated to a $\Lambda$-projective system on $(\Lambda^\infty, \mu_\pi)$.

Conversely, if we have a representation of $C^*(\Lambda)$ on $L^2(\Lambda^\infty, \mu)$ which arises from a $\Lambda$-projective system, then the representation is monic. 
\end{thm}
 
 By Remark \ref{rmk:lambda-proj-comments-1}, this implies that a $\Lambda$-semibranching function system on $(\Lambda^\infty, \mu)$, for any Borel measure $\mu$, gives rise to a monic representation of $C^*(\Lambda)$.

\begin{proof}
 Suppose that  the representation $\{t_\lambda\}_{\lambda \in \Lambda}$ of $C^*(\Lambda)$ is monic, and let $\xi\in \H$ be a cyclic vector for $C(\Lambda^\infty)$.
 Note that the map $W : C(\Lambda^\infty) \to  \mathcal{H}$ given by 
\[
  W( f) = \pi(f) \xi
\]
is 
 linear. Moreover, if we think of $C(\Lambda^\infty)$ as a dense subspace of $L^2(\Lambda^\infty, \mu_\pi)$, the operator $W$ is isometric:
\[ 
 \| f\|^2_{L^2} = \int_{\Lambda^\infty} |f|^2 \, d\mu_\pi  = \langle \xi, \pi(|f|^2) \xi \rangle = \| \pi(f) \xi \|^2 = \|W(f)\|^2.
\]
Therefore $W$ extends to an isometry  from $
L^2(\Lambda^\infty, \mu_\pi) $ to $ \mathcal{H}.$
Since $W$ is also onto (because the representation is monic), $W$ is a surjective isometry; that is, $W$ is a unitary.

Moreover, for any $f \in C(\Lambda^\infty)$ and any $\varphi \in L^2(\Lambda^\infty, \mu_\pi)$, we have 
\[ \pi(f) W(\varphi) = \pi(f) \pi(\varphi) \xi = \pi(f \cdot \varphi) \xi = W(f \cdot \varphi).\]
Thus, unitarity of  $W$ implies that $W^* \pi(f) W$ acts on $L^2(\Lambda^\infty, \mu_\pi)$ by multiplication by $f$:
\begin{equation}
W^* \pi(f) W = M_f \text{ and } W M_f W^* = \pi(f).\label{eq:w-intertwines-pi}
\end{equation}

 Now define the operator
\[
S_\lambda = W^* t_\lambda W \quad\quad\text{for $\lambda \in \Lambda$.} 
\]
First note that because $W$ is a unitary operator, and $\{t_\lambda\}_{\lambda \in \Lambda}$ is a representation of $C^*(\Lambda)$, the operators $\{S_\lambda\}_{\lambda \in \Lambda}$ also  give  a representation of $C^*(\Lambda)$. Moreover, since $W$ is a unitary, 
\begin{equation}\begin{split} S_\lambda S_\lambda^* (f) &= W^* t_\lambda t_\lambda^*  W(f) = W^* \pi(\chi_{Z(\lambda)}) \pi(f) \xi  \\ 
&= W^* \pi( \chi_{Z(\lambda)} \cdot f) \xi = W^* W (\chi_{Z(\lambda)} \cdot f) \\
&= \chi_{Z(\lambda)} \cdot f.
\end{split}
\label{eq:range-sets}
\end{equation}

Let $1$ denote the characteristic function of $\Lambda^\infty$, and define a function $f_\lambda \in L^2(\Lambda^\infty, \mu_\pi)$ by 
\[
f_\lambda = S_\lambda 1 = W^* t_\lambda \xi.
\]
 We will now show that the functions $f_\lambda$, combined with the usual coding and prefixing maps on $\Lambda^\infty$, form a $\Lambda$-projective system on $(\Lambda^\infty, \mathcal{B}_o(\Lambda^\infty), \mu_\pi)$. 
  To that end, we will invoke Proposition \ref{prop:pvm-properties}; since 
\[ P(Z(\nu)) = \pi(\chi_{Z(\nu)})\]
for any $\nu \in \Lambda$, and characteristic functions of cylinder sets densely span $L^2(\Lambda^\infty, \mu_\pi)$, the equalities established in Proposition \ref{prop:pvm-properties} still hold if we replace $P(Z(\nu))$ by $\pi(f)$ for any $f \in L^2(\Lambda^\infty, \mu_\pi)$.  In particular, noting that 
\[ \chi_{(\sigma^n)^{-1}(Z(\nu))} = \chi_{Z(\nu)} \circ \sigma^n \quad \text{ and } \quad \chi_{\sigma_\lambda^{-1}(Z(\nu))} = \chi_{Z(\nu)} \circ \sigma_\lambda \]
Part (d) of Proposition \ref{prop:pvm-properties} implies that if $d(\lambda) = n$, 
\begin{equation}
t_\lambda \pi(f) = \pi(f \circ \sigma^n) t_\lambda\label{eq:pvm-useful}
\end{equation} 
and Part (c) implies that 
\begin{equation}
t_\lambda^* \pi(f) = \pi(f \circ \sigma_\lambda) t_\lambda^*.\label{eq:part-c-useful}
\end{equation}

Let $f \in L^2(\Lambda^\infty, \mu_\pi)$ and let $n = d(\lambda)$.  By using Part (d) of Proposition \ref{prop:pvm-properties}, Equation \eqref{eq:w-intertwines-pi}, and the fact that $W$ is a unitary, we obtain 
\begin{align*}
S_\lambda (f) &= W^* t_\lambda W ( f )= W^* t_\lambda \pi(f)\xi \\
&= W^* \pi(f\circ \sigma^n) t_\lambda \xi =  W^* \pi(f\circ \sigma^n)W W^* t_\lambda \xi\\
&=(f\circ \sigma^n) \cdot  f_\lambda.
\end{align*}

In order to show that $\{S_\lambda\}_{\lambda \in \Lambda}$ is a $\Lambda$-projective representation, then, Proposition \ref{prop:lambda-proj-repn} tells us that it remains to check that the standard prefixing and coding maps make $(\Lambda^\infty, \mu_\pi)$ into a $\Lambda$-semibranching function system, and that Condition (a) of Definition \ref{def:lambda-proj-system} holds for the functions $f_\lambda$.

To establish Condition (a), we work indirectly.
Since $W$ is a unitary, we   have (for  any $f \in L^2(\Lambda^\infty, \mu_\pi)$ and any $\lambda \in \Lambda^n$)
\begin{align*}
\int_{\Lambda^\infty} |f_\lambda|^2 \cdot  f \, d\mu_\pi &= \langle S_\lambda(1), S_\lambda(1) \cdot  f \rangle_{ L^2} = \langle W^* t_\lambda (\xi), M_f W^* t_\lambda (\xi)\rangle_{ L^2} \\
&= \langle t_\lambda \xi,  W M_f W^*(t_\lambda \xi )  \rangle_{ \mathcal{H}} = \langle \xi, t_\lambda^* \pi(f) t_\lambda \xi \rangle_{\mathcal{H}}\\
&= \langle \xi, \pi(f \circ \sigma_\lambda) \xi \rangle_{\mathcal{H}} = \int_{\Lambda^\infty} f \circ \sigma_\lambda \, d\mu_\pi\\
&= \int_{\Lambda^\infty} f \, d(\mu_\pi \circ \sigma_\lambda^{-1})
\end{align*}
by using Equations \eqref{eq:w-intertwines-pi} through \eqref{eq:part-c-useful}.  

If $E\subseteq \Lambda^\infty$ is any set for which $\mu_\pi(E) = 0$, then taking $f = \chi_E$ above  shows that $\mu_\pi \circ \sigma_\lambda^{-1}(E) = 0$ also -- in other words, 
\begin{equation}
\mu_\pi \circ \sigma_\lambda^{-1} << \mu_\pi.\label{eq:mu-pi-abs-cts}
\end{equation}
The uniqueness of Radon-Nikodym derivatives then implies that
\[|f_\lambda|^2 =  \frac{d(\mu_\pi \circ \sigma_\lambda^{-1})}{d(\mu_\pi )}= ((\Phi_\pi)_\lambda \circ \sigma^n)^{-1} \]
by Equation \eqref{eq:phi-lambda-eqn}.
In other words, Condition (a) of Definition \ref{def:lambda-proj-system} holds.

Similarly, for any set $F \subseteq Z(s(\lambda))$ such that  $\mu_\pi(F) = 0$, taking $f = \chi_{\sigma_\lambda(F)}$ reveals that 
\[ 0 = \mu_\pi (F) = \mu_\pi \circ \sigma_\lambda^{-1}(\sigma_\lambda(F)) = \int_{\sigma_\lambda(F)} |f_\lambda|^2 \, d\mu_\pi.\]
Since $|f_\lambda|^2 >0$ a.e.~on $Z(\lambda) \supseteq \sigma_\lambda(F)$, we must have 
\[ \mu_\pi \circ \sigma_\lambda(F) = 0\]
and hence $\mu_\pi \circ \sigma_\lambda << \mu_\pi$.

Furthermore, the Radon-Nikodym derivative $\frac{d(\mu_\pi \circ \sigma_\lambda)}{d(\mu_\pi )}$ is nonzero $\mu_\pi$-a.e.~on $Z(s(\lambda))$. 
To see this, we  set  
\[ E = \left\{ x \in Z(s(\lambda)): \frac{d(\mu_\pi \circ \sigma_\lambda)}{d(\mu_\pi )} = 0 \right\}\]
and observe that 
\[ \mu_\pi (\sigma_\lambda(E)) = \int_E \frac{d(\mu_\pi \circ \sigma_\lambda)}{d(\mu_\pi )} \, d\mu_\pi = 0.\]
Equation \eqref{eq:mu-pi-abs-cts} 
therefore  implies that
\[ \mu_\pi (E) = (\mu_\pi \circ \sigma_\lambda^{-1})(\sigma_\lambda(E)) = 0.\]

Theorem \ref{thm-lambda-sbfs-on-the-inf-path-space-via-a-measure} now implies that the standard prefixing and coding maps make $(\Lambda^\infty, \mu)$ into a $\Lambda$-semibranching function system.  Consequently, the functions $f_\lambda$ make $\{S_\lambda\}_{\lambda \in \Lambda}$ into a $\Lambda$-projective representation, which is unitarily equivalent to our initial monic representation by construction. 

For the converse, suppose that $\{t_\lambda\}_{\lambda\in\Lambda}$ is a representation of $C^*(\Lambda)$ on $L^2(\Lambda^\infty, \mu)$ for some Borel measure $\mu$ which arises from a $\Lambda$-projective system $\{f_\lambda\}_{\lambda\in\Lambda}$.   Then (as in \cite{dutkay-jorgensen-monic} Theorem 2.7, or Proposition \ref{prop:lambda-proj-repn} of this paper) the fact that $t_\lambda (f) = f_\lambda \cdot (f\circ \sigma^n)$ for any $f\in L^2(\Lambda^\infty, \mu)$ implies that 
\[ t_\lambda^* f = (\overline{f_\lambda} \circ \sigma_\lambda) \cdot (f \circ \sigma_\lambda) \cdot \Phi_\lambda.\] 
Consequently, 
letting 1 denote the constant function on $\Lambda^\infty$, we have 
\begin{align*}
t_\lambda t_\lambda^* 1 & = t_\lambda ((\overline{f_\lambda} \circ \sigma_\lambda) \cdot (1 \circ \sigma_\lambda) \cdot \Phi_\lambda ) \\
&= f_\lambda \cdot (\overline{f_\lambda} \circ \sigma_\lambda \circ \sigma^n) \cdot (1 \circ \sigma_\lambda \circ \sigma^n) \cdot (\Phi_\lambda \circ \sigma^n),
\end{align*}
which is zero off $Z(\lambda)$.  Moreover, on $Z(\lambda)$, we have $\sigma_\lambda \circ \sigma^n = id$, and Equation \eqref{eq:phi-lambda-eqn} tells us that 
\[(\Phi_\lambda \circ \sigma^n)^{-1}  =\frac{d(\mu \circ \sigma_\lambda^{-1})}{d\mu} = |f_\lambda|^2. \]
Consequently, 
\[t_\lambda t_\lambda^* 1 =  f_\lambda \cdot \overline{f_\lambda} \cdot (\Phi_\lambda \circ \sigma^n) =  \chi_{Z(\lambda)},\]
and so (since the characteristic functions of cylinder sets span a dense subspace of $L^2(\Lambda^\infty, \mu)$) it follows that $1 = \chi_{\Lambda^\infty}$ is a cyclic vector for $C(\Lambda^\infty) \subseteq C^*(\Lambda)$.  Thus, $\{t_\lambda\}_{\lambda \in \Lambda}$ is monic.
\end{proof}


  
\begin{thm} 
\label{thm-theorem-2.9-of-monic}
  Let $\Lambda$ be finite, source-free $k$-graph, and let $\{S_\lambda\}_{\lambda \in \Lambda}$, $\{T_\lambda\}_{\lambda \in \Lambda}$ be two monic representations   of $C^*(\Lambda)$. Let $\mu_S, \mu_T$ be the measures on $\Lambda^\infty$ associated to these representations as in \eqref{eq:monic_measure}.
  The representations $\{S_\lambda\}_{\lambda \in \Lambda}$, $\{T_\lambda\}_{\lambda \in \Lambda}$ are  equivalent if and only if the measures  $\mu_S$ and  $\mu_T$ are equivalent 
  and there exists a function $h$ on $\Lambda^\infty$ such that
  \begin{equation}\label{eq-1-thm-monic}
  \frac{d\mu_S}{d\mu_T} =|h|^2 \quad \text{and}
  \end{equation}
 \begin{equation}\label{eq-2-thm-monic}
  f^S_{\lambda} = \frac{ h \circ \sigma^n}{h }f^T_{\lambda} \quad\text{for all  $\lambda\in\Lambda$ with $d(\lambda)=n$.}
  \end{equation}
 \end{thm}
    
  \begin{proof} 
  Suppose $\{S_\lambda\}_{\lambda \in \Lambda}$, $\{T_\lambda\}_{\lambda \in \Lambda}$ are equivalent representations of $C^*(\Lambda)$.  {From Theorem 
  \ref{thm-disjoint-monic-repres}, it follows that the associated measures $\mu_S, \mu_T$ are equivalent.}
  Let
  \[
  W:L^2(\Lambda^\infty, \mu_S) \to L^2(\Lambda^\infty,\mu_T)
  \]
  be the intertwining 
  unitary for them. Then the two representations are also equivalent when restricted to  the diagonal  subalgebra $C^*(\{ t_\lambda t_\lambda^*: \lambda \in \Lambda \})$.
By linearity, we can extend the formula from Equation \eqref{eq:range-sets} to all of $C(\Lambda^\infty)$.  It follows that $\pi_S, \pi_T$ are both given on $C(\Lambda^\infty)$ by multiplication: 
  \[\pi_S(\phi) = M_\phi \text{ and } \pi_T(\phi) = M_\phi \ \forall \ \phi \in C(\Lambda^\infty).\]
Since $W$ intertwines a dense subalgebra -- namely $\pi_S(C(\Lambda^\infty))$ -- of the maximal abelian subalgebra $L^\infty(\Lambda^\infty, \mu_S) \subseteq B(L^2(\Lambda^\infty, \mu_S))$ which consists of multiplication operators,  with the dense subalgebra $\pi_T(C(\Lambda^\infty)) \subseteq L^\infty(\Lambda^\infty, \mu_T)$, the unitary $W$ must be given by multiplication by some nowhere-vanishing function $h$ on $\Lambda^\infty$: 
   \[W(\phi) = h \phi.\]
   Moreover, since $W$ is a unitary, 
  \[
 \int_{\Lambda^\infty} |W(f)|^2 d\mu_T =  \int_{\Lambda^\infty} |f|^2|h|^2 d\mu_T = \int_{\Lambda^\infty}|f|^2 d\mu_S \quad\text{for all  $f \in L^2(\Lambda^\infty,\mu_S)$},
  \]
  which implies \eqref{eq-1-thm-monic}.  
   
  From the intertwining property   $
  T_\lambda \, W = W \, S_{\lambda} $ 
  we obtain, for any $f \in L^2(\Lambda^\infty, \mu_S)$ and  any $\lambda$ with $d(\lambda)=n$, that 
  \[
   T_\lambda \, W (f) = W \, S_{\lambda}(f) ,\ \text{that is, }
  f^T_\lambda(h \circ \sigma^n) (f \circ \sigma^n) = hf^S_\lambda   (f \circ \sigma^n). 
  \]
  Take $f=1$ and we obtain that 
  \[
  f^T_\lambda \frac{h \circ \sigma^n}{h} = f^S_\lambda\label{eq:10-1}
  \]
  as claimed in \eqref{eq-2-thm-monic}.

  For the converse, suppose that the measures $\mu_S, \mu_T$ are equivalent and there is a function $h$ on $\Lambda^\infty$ satisfying \eqref{eq-1-thm-monic} and \eqref{eq-2-thm-monic}.  Then  define $W: L^2(\Lambda^\infty, \mu_S) \to L^2(\Lambda^\infty, \mu_T)$ by 
  \[
  W f=hf ; \] 
 it is then straightforward to check that $W S_\lambda = T_\lambda W$ and that $W$ is a unitary.\end{proof} 
 

\subsection{$\Lambda$-semibranching function systems and monic representations}
\label{sec:monic-examples}

In this section, we examine the examples of $\Lambda$-semibranching function systems given in Section  \ref{sec:examples_gen_measure},
 and identify which of them give rise to monic representations of $C^*(\Lambda)$ -- or, equivalently, which of these $\Lambda$-semibranching function systems are unitarily equivalent to ones on the infinite path space. Besides, since monic representations are multiplicity free,  it is easy to construct examples of non-monic representations by using direct sums of monic representations, see  \cite{arveson} page 54.

The next theorem, which we use heavily in our analysis of these examples, shows that a $\Lambda$-semibranching representation on $(X, \mu)$ is monic if and only if its associated range sets generate the $\sigma$-algebra of $X$. To  state our result more precisely, we will  denote by
 \[
 ({X,} \mathcal F, \mu)
 \]
the measure space associated to $L^2(X, \mu)$; in particular $\mathcal F$ is the standard $\sigma $-algebra associated to $L^2(X, \mu)$.

\begin{thm}
\label{prop:range-sets-generate-sigma-alg-in-monic}
Let $\Lambda$ be a finite, source-free $k$-graph and 	let  $ \{t_\lambda\}_{\lambda \in \Lambda}$ be a $\Lambda$-semibranching  representation  of $C^*(\Lambda)$ on $L^2(X, \mathcal{F}, \mu)$ 
with $\mu(X)< \infty$. Let   $\mathcal{R} $ be the collection of sets which are modifications of range sets $ R_\lambda$ by sets of measure zero; that is, each element $X \in  \mathcal R$ has the form 
	\[ X = R_\lambda \cup S \quad \text{ or } \quad X = R_\lambda \backslash S\]
	for some set $S$ of measure zero. 
	Let $\sigma(\mathcal{R})$ be the $\sigma $-algebra generated by $\mathcal{R}.$
The representation $ \{t_\lambda\}_{\lambda \in \Lambda}$ 
 is monic, with cyclic vector $\chi_X \in L^2(X, \mathcal F, \mu)$, if and only if $\sigma(\mathcal R) = \mathcal F$. 
		\label{prop-conv-to-8-7}
In particular, for a monic representation $\{t_\lambda\}_{\lambda \in \Lambda}$,  the set
	\[
	{\mathcal{S} }:=  \Big\{ \sum_{i=1}^{n} a_{i}t_{\lambda_i} t_{\lambda_i}^* \chi_X \ | \ n \in \N , \lambda_i \in \Lambda, a_i \in \C\Big\} = \Big\{ \sum_{i=1}^{n} a_{i} \chi_{R_{\lambda_i}}  \ | \ \ n \in \N , \lambda_i \in \Lambda, a_i \in \C\Big\}
	\]
	is dense in  $L^2(X, \mathcal F, \mu)$.

\end{thm}

\begin{proof}

Suppose first that the representation $\{t_\lambda\}_{\lambda\in \Lambda}$ is monic and that $\chi_X$ is a cyclic vector for the representation.  As computed in the proof of Theorem 3.4 of \cite{FGKP}, we have 
\[ t_\lambda t_\lambda^* (\chi_X) = \chi_{R_\lambda}.\]
Therefore, our hypothesis that $\chi_X$ is a cyclic vector implies that for any $f \in L^2(X, \mathcal F, \mu)$, there is a sequence $(f_j)_j$, with $f_j \in \text{span} \{ \chi_{R_\lambda}: \lambda \in \Lambda\}$, such that 
\[\lim_{j\to \infty}  \int_X | f_j - f|^2 \, d\mu =0.\]
In particular, $(f_j) \to f$ in measure.  

For any $\sigma$-algebra $\mathcal T$, standard measure-theoretic results \cite[Proposition 6]{Moore} imply that since $\mu(X) < \infty$, convergence in measure among $\mathcal T$-a.e.~finite measurable functions on $(X, \mathcal T, \mu)$ is metrized by the distance 
\[ d_{\mathcal T} (f, g) :=  \int_{\Omega} \frac{ | f-g| }{  1+ | f-g| } d\mu.\] 
Moreover, $d_{\mathcal T}$ makes the space of $\mathcal S$-a.e.~finite measurable functions into a complete metric space (this can be seen, for example, by combining
 Proposition 1 and Corollary 7 of \cite{Moore}). 

The fact that $(f_j)_j \to f$ in measure in $(X, \mathcal F, \mu)$, and that  $f_j \in L^2(X, \sigma(\mathcal R), \mu)$ for all $j$, implies that $(f_j)_j$ is a Cauchy sequence with respect to both $d_{\mathcal F}$ and $d_{\sigma(\mathcal R)}$.  Consequently, the limit $f$ of $(f_j)_j$ must also be a $\sigma(\mathcal R)$-a.e.~finite measurable function.  In other words, every $f \in L^2(X, \mathcal F, \mu)$ is in fact in $L^2(X, \sigma(\mathcal R), \mu)$.  Since $\mathcal R \subseteq \mathcal F$ by construction we must have $\sigma(\mathcal R) = \mathcal F$, as desired.

We will subdivide the proof of the converse in several steps.   Thus, assume $\sigma(\mathcal R) = \mathcal F$.
\begin{enumerate}
\item[(Step 1)] 
Observe that the set
\[
\tilde{\mathcal{R}} := \{ \text{finite unions of elements in $\mathcal{R}$} \}
\] 
is a subalgebra of $ \mathcal{P}(X)$ -- that is, closed under finite unions and complements.  Closure under finite unions follows from the definition, while the second claim follows from  Equation \eqref{eq-partition}. In addition, 
$\sigma(\mathcal{R})= \sigma(\mathcal{\tilde{R}}) = \mathcal F,$ and 
\[
{\tilde{\mathcal{S}}}:=  \Big\{ \sum_i^{n} a_{i} \chi_{{B_i}}  \ | \ \ n \in \N , B_i \in \sigma(\mathcal{\tilde{R}}), a_i \in \C\Big\}
\]
is dense in  $L^2(X, \mathcal F, \mu)$. 
	\item[(Step 2)] Now we apply the Carath\'eodory/Kolmogorov extension theorem to conclude that the measure $\mu|_{\mathcal{\tilde{R}}}$ restricted to $\tilde{\mathcal{R}}$ induces a unique (extended) measure on $\mathcal{F} = \sigma(\mathcal R)$, which we still call $\mu$. (This is indeed the original measure on $L^2(X,\mathcal{F}, \mu)$ by the uniqueness of the extension.)
	It is a corollary of the Carath\'eodory/Kolmogorov extension theorem\footnote{See \cite{BDurrett} Page 452, Appendix: Measure theory, Exercise 3.1.} that  for any $\tilde{B} \in \sigma(\mathcal{ \tilde{\mathcal{R}}} ) = \mathcal F$ and for any $\epsilon >0,$ there exists $ A_\epsilon^{\tilde{B}}\in \mathcal{ \tilde{\mathcal{R}}} $ with
	\[
	\mu\Big( \tilde{B} \Delta  A_\epsilon^{\tilde{B}} \Big) < \epsilon,
	\]
	where $\Delta$ denotes symmetric difference.
	\item[(Step 3)] Recall  the following fundamental result about metric spaces: if $(Q,d_Q)$ is a metric space, and if $\tilde{\Sigma} \subseteq Q$  is a dense subset of $(Q,d_Q)$, then any other subset ${\Sigma} \subseteq Q$ having the property
	\[
	\forall \epsilon >0, \ \forall \tilde{x}\in \tilde{\Sigma} ,\ \exists \  x_\epsilon \in {\Sigma} \text{ with } d_Q(\tilde{x}, x_\epsilon ) <\epsilon
	\]
	is also dense in $(Q,d_Q)$.
	\item[(Step 4)] 
	To show that the vector 
	 $\chi_X$ is monic, equivalently that the set $\mathcal{S}$ is dense in $(L^2(X, \mathcal{F}, \mu), d_{L^2(X, \mathcal{F}, \mu) })$,\footnote{Here $d_{L^2(X, \mathcal{F}, \mu)}$ denotes the standard metric coming from the $L^2$-norm on $L^2(X, \mathcal{F}, \mu)$.}
	we will apply Step 4 to the metric space
	\[
	(Q,d_Q) = (L^2(X, \mathcal{F}, \mu), d_{L^2(X, \mathcal{F}, \mu) }),\text{ with } \tilde{\Sigma} = \tilde{\mathcal{S}},\ {\Sigma} = {\mathcal{S}}.
	\]
	Indeed we know from Step 2 that $\tilde{\Sigma} = \tilde{\mathcal{S}}$ is dense in $(L^2(X, \mathcal{F}, \mu), d_{L^2(X, \mathcal{F}, \mu) }).$
	To end the proof, it will then be enough to show that
	\[
	\forall \epsilon >0, \ \forall \tilde{s}\in \tilde{\mathcal{S}}, \ \exists \  s_\epsilon \in {\mathcal{S}} \text{ with } d_{L^2(X, \mathcal{F}, \mu) }(\tilde{s}, s_\epsilon ) <\epsilon.
	\]
	To do so, we will use the approximation results in Step 3. Indeed, without loss of generality we can assume 
	\[
\tilde{s}=  \sum_i^{n} a_{i} \chi_{{B_i}} , \text{ for some }  n \in \N , B_i \in \sigma(\mathcal{\tilde{R}}), a_i \not=0 .
	\]
	Fix $\epsilon >0$, and  define
	$
	A:=\sum_i^{n} |a_{i} |\in \C.
	$
	By Step 3, given  $i$, there exists  $A_i^\epsilon \in \tilde{\mathcal{R}} $ such that
	\[
	\mu(B_i \Delta A_i^\epsilon   )<\frac{\epsilon^2}{A^2},\text{ or, equivalently,  }\Big(\mu(B_i \Delta A_i^\epsilon   )\Big)^{1/2}<\frac{\epsilon}{A}.
	\]
	So if we now define $
	s_\epsilon := \sum_i^{n} a_{i} \chi_{A_i^\epsilon} , $
	we see, by using the triangle inequality, that
	\[
	 d_{L^2(X, \mathcal{F}, \mu) }(\tilde{s}, 	s_\epsilon    ) <\epsilon,
	\]
	as desired.
\end{enumerate}
\end{proof}

\begin{example} 
\label{ex-exonevthreeed-monic}
The semibranching function system given in Example  \ref{exonevthreeed}  is not monic.  To see this, we argue by contradiction.  First, observe that the only finite paths with range $v_2$ are of the form $f_3 f_3 \cdots f_3$; and since $\tau_{f_3}(x)=x$ on $D_3=(1/2, 1]$, we have
\[R_{f_3} = R_{f_3 f_3 \cdots f_3} = (1/2, 1].\]
Every other finite path $\lambda$, having range $v_1$, will satisfy $R_\lambda \subseteq D_{v_1} = [0, 1/2]$.

Consequently, $\mathcal R = \{R_\lambda\}_{\lambda\in \Lambda}$ does not generate the usual Lebesgue $\sigma$-algebra on $[0,1]$, even after modification by sets of measure zero, since the restriction of $\mathcal R$  to $(1/2, 1]$ contains no nontrivial measurable sets.  Theorem \ref{prop:range-sets-generate-sigma-alg-in-monic} therefore
implies that the representation of $C^*(\Lambda)$ associated to this semibranching function system is not monic, and {hence is not equivalent to any representation on $L^2(\Lambda^\infty,\mu)$ arising from a $\Lambda$-projective system.}

\end{example}  

\begin{example}
The $\Lambda$-semibranching function system of the $2$-graph $\Lambda$ in Example \ref{exonevtwoe}, however, does give rise to a monic representation of $C^*(\Lambda)$.  To see this, let $\lambda \in\Lambda^{(n,n)}$ and use the factorization rules to write $\lambda = \lambda_1 \lambda_2 \cdots \lambda_n$, where $\lambda_i = f_{j_i} e$.  Then, one can compute that $R_\lambda$ is an open interval of length $2^{-n}$, whose left endpoint is 
\[\sum_{i \leq n : j_i = 2} 2^{i-n-1}.\]
In other words, every interval of the form $\left( \frac{k}{2^j}, \frac{k+1}{2^j} \right)$ is $R_\lambda$ for some $\lambda \in \Lambda$.  Since these intervals generate the standard topology on $(0,1)$ up to  measure-zero sets (cf.~\cite{hutchinson} or \cite[Section 4.1]{palle-analysis-text}), {it follows from Theorem \ref{prop-conv-to-8-7}   that the representation of $C^*(\Lambda)$ associated to this $\Lambda$-semibranching function system is monic,} and hence equivalent to a representation on $\Lambda^\infty$.

\end{example}

\begin{example}
From the $\Lambda$-semibranching function system of Example \ref{ex:sbfs-3v8e}, we again obtain a monic representation of $C^*(\Lambda)$.  To see this, we first notice that whenever $d(\lambda) = (1,1)$, $R_\lambda$ is an interval of the form $(k/6, (k+1)/6)$;  there are six such paths $\lambda$, so each such interval is realized as $R_\lambda$ for some $\lambda$.\footnote{ This last assertion follows from the fact that 
\[\mu \left( X \backslash \bigcup_{\lambda \in \Lambda^{(1,1)}} R_\lambda \right) = 0 \]
 for any $\Lambda$-semibranching function system.}   Indeed, for any such $\lambda$, one can calculate that $\tau_\lambda$ is a linear function on a connected domain with slope $-1/2$.  
 
Furthermore, for any vertex $v\in \Lambda^0$, we have 
\[|\{\eta: d(\eta) = (1,1), \ s(\eta) = v\}| = 2.\]
It follows that there are 12 paths $\lambda$ with degree $(2,2)$, and for each such $\lambda$ we know that $\tau_\lambda$ will be a linear function on a connected domain with slope $1/4$.

  Proceeding inductively, we see that for any $n\in \N$, we have $| \Lambda^{(n,n)} | = 2^n \cdot 3$ and that for each $\lambda \in \Lambda^{(n,n)}$, $\tau_\lambda$ is a linear function on a connected domain with slope $(-2)^{-n}$.  It follows that 
  \[ \{ R_\lambda: \lambda \in \Lambda^{(n,n)}\} = \{ \left( \frac{k}{2^n \cdot 3}, \frac{k+1}{2^n \cdot 3} \right) : 0 \leq k \leq 2^n \cdot 3-1\}\]
  for any $n \in \N$.
 Again (cf.~\cite[Section 4.1]{palle-analysis-text} or \cite{hutchinson})
 these sets 
 $\{R_\lambda: \lambda \in \Lambda^{(n,n)} \text{ for some } n \in \N\}$
 generate the standard Borel $\sigma$-algebra on $(0,1)$; and thus, together with sets of measure zero, they generate the entire Lebesgue $\sigma$-algebra on $(0,1)$. {Consequently, Theorem \ref{prop-conv-to-8-7} implies that
 the constant function $1= \chi_{[0,1]}$ is a cyclic vector for the representation associated to this $\Lambda$-semibranching function system, as claimed.} 
\end{example}

\begin{prop}
Let $E$ be a finite directed graph with no sources and let $\Lambda$ be its double graph.
If an $E$-semibranching function system gives rise to a monic representation of $C^*(E)$, then the associated $\Lambda$-semibranching function system also give rise to a  monic representation of $C^*(\Lambda)$.
\end{prop}
\begin{proof}
Since every edge in $E$ is repeated twice in $\Lambda$, once per color, it follows that all the finite composable paths in $E$ are also finite composable paths in $\Lambda$.  Moreover, if $\lambda$ is a finite path in $E$, then $R_\lambda$ is the same in the $E$-semibranching function system as in the $\Lambda$-semibranching function system by Proposition~\ref{prop:double-graph-SBFS}.  Thus, if 
\[\H = \overline{\text{span}} \{ t_\lambda t_\lambda^* \xi: \lambda \text{ a finite path in } E\},\] 
where $\xi\in \mathcal{H}$ is a monic vector of $C^*(E)$,
then we also have $\H = \overline{\text{span}} \{ t_\lambda t_\lambda^* \xi: \lambda \in \Lambda\}$, since $t_\lambda t_\lambda^* \xi = \chi_{R_\lambda} \xi.$
\end{proof}

\begin{prop}
Suppose that $\Lambda_1, \Lambda_2$ are finite, source-free $k_1$- and $k_2$-graphs respectively and that (for $i=1,2$) we have $\Lambda_i$-semibranching function systems on $(X_i, \mu_i)$.  If both of these $\Lambda_i$-semibranching function systems give rise to monic representations of $C^*(\Lambda_i)$, then the associated $\Lambda_1 \times \Lambda_2$-semibranching function system on $(X_1 \times X_2, \mu_1 \times \mu_2)$ is also a monic representation.
\end{prop}
\begin{proof}
If $\{R_\lambda: \lambda \in \Lambda_i\}$ generates the $\mu_i-\sigma$-algebra for $i = 1,2$, then $\{R_\lambda \times R_\nu: \lambda \in \Lambda_1, \nu \in \Lambda_2\}$ generates the $\mu_1 \times \mu_2 - \sigma$-algebra of $X_1 \times X_2$.  Moreover, given any path $\eta \in \Lambda_1 \times \Lambda_2$, the definition of the product $k$-graph ensures that we can write $\eta =( \lambda \times r(\nu)) (s(\lambda) \times \nu)$ for some finite paths $\lambda\in \Lambda_1, \nu \in \Lambda_2$.  Then, since 
\[\tau_\eta = \tau_{(\lambda \times r(\nu))} \circ \tau_{(s(\nu) \times \nu)}\]
has range $R_\lambda \times R_\nu$, we see that $\{R_\eta: \eta \in \Lambda_1 \times \Lambda_2\}$ generates the $\mu_1 \times \mu_2 - \sigma$-algebra on $X_1 \times X_2$ whenever $\{R_\lambda: \lambda \in \Lambda_i\}$ generates the $\mu_i- \sigma$-algebra for $i=1,2$.
\end{proof}

\begin{example}
The two previous Propositions tell us that, in order to check that the $\Lambda$-semibranching function systems of  Sections \ref{sec:double} and \ref{sec:product} give rise to monic representations of $C^*(\Lambda)$, it suffices to show that, for the 1-graph $E$ in Example \ref{ex-3.3-Kawamura_modified}, $\{R_\lambda: \lambda \text{ a finite path in }E\}$ generates the standard Lebesgue $\sigma$-algebra on $(0,1)$, up to sets of measure zero.

To that end, note that if $\lambda$ is a finite path in $E$ with both source and range $v$, then $\lambda$ is of the form $\lambda _1 \lambda_2 \cdots \lambda_n$, where each $\lambda_i$ is either equal to $e$ or to $g f$.  Furthermore, we note that both $\tau_e$ and $\tau_g \circ \tau_f$ are linear functions on a connected domain (namely $D_v = (0,a)$) with slope $1/2$.  Consequently, $R_\lambda$ is an interval of length $a \cdot 2^{-n}$; since we can choose either $e$ or $gf$ for each $\lambda_i$, there are $2^n$ such sets $R_\lambda$, and so 
\[\{R_\lambda: \lambda \in v E v\} = \{ \left( \frac{j \cdot a}{ 2^k}, \frac{(j+1) \cdot a}{2^k} \right): k \in \N, \ 0 \leq j \leq 2^k - 1\}\]
generates the Lebesgue $\sigma$-algebra for $(0,a)$, after modification by sets of measure zero.

Moreover, if $\lambda \in w E v$ is a finite path in $E$ with range $w$ and source $v$, it must be of the form $\lambda = f \lambda_1 \cdots \lambda_n$, with $\lambda_i  \in \{e, gf\}$ as above.  For each $n$ there are $2^n$ such paths.
Consequently, since $\tau_f$ has slope $(1-a)/a$, we see that 
\[\{R_\lambda: \lambda \in w E v \} = \{ \left( \frac{j (1-a)}{2^k}, \frac{(j+1)(1-a)}{2^k} \right): k \in \N, 0 \leq j \leq 2^k -1\}.\]
Again, these sets generate the Lebesgue $\sigma$-algebra for the interval $(a, 1)$, which has length $1-a$.

In other words, the $E$-semibranching function system given in Example \ref{ex-3.3-Kawamura_modified} gives rise to a monic representation of $C^*(E)$, and hence to monic representations of $C^*(\Lambda)$ and of $C^*(E \times E)$, where $\Lambda$ is the double graph of $E$. 
\end{example}

The following Corollary shows that the measures constructed in Section \ref{sec:examples_gen_measure} give rise to monic representations of $C^*(\Lambda)$. 

\begin{cor} \label{cor-repres-via-a-measure-on-the-inf-path-space}	
Suppose that $\Lambda$ is a $k$-graph and $p$ is a Borel measure on its infinite path space $\Lambda^\infty$ satisfying the hypotheses of Theorem \ref{thm-lambda-sbfs-on-the-inf-path-space-via-a-measure}.
Then the representation of  $C^*(\Lambda)$ on $L^2(\Lambda^\infty, p)$ arising from  the $\Lambda$-semibranching function system of Theorem  \ref{thm-lambda-sbfs-on-the-inf-path-space-via-a-measure} is monic.
\end{cor}

\begin{proof}
Recall from Theorem 3.4 of \cite{FGKP} that the representation $\{ S_\lambda\}_{\lambda \in \Lambda}$ associated to such $\Lambda$-semibranching function systems satisfies
\[
S_\lambda S_\lambda^*  f = \chi_{Z(\lambda)} \cdot f \quad  \text{for $f\in L^2(\Lambda^\infty, p).$} 
\]
 Since the cylinder sets $\{ Z(\lambda)\}_{\lambda \in \Lambda}$ generate the topology on $\Lambda^\infty$, and $p$ is a Borel measure by hypothesis, the characteristic function of $\Lambda^\infty$ is a monic vector.
\end{proof}

We conclude this section with a remark on the existence of monic representations of $\mathcal O_N$ which do not contain any monic sub-representations.

    As we mentioned in the Introduction, the lack of Borel cross-sections for the irreducible representations of $\mathcal O_N$ implies that  there are not many theorems which apply to all representations of  the Cuntz algebras, and even fewer which apply  to all representations of higher-rank graph  $C^*$-algebras. Nonetheless, a structure theory of a different kind can be found in Theorems 1.2 and 1.4 of the memoir \cite{Bra-Jor-Ostro} by Bratteli, Jorgensen, and Ostrovskyi.  These Theorems describe direct integral decompositions for any non-degenerate representation of  $\mathcal O_N$. Combining these results with the Remarks 1.3 and 1.5 in \cite{Bra-Jor-Ostro}, the reader will be able to identify many representations of $\mathcal O_N$ which do not contain any monic sub-representations.

\section{A universal representation for nonnegative monic  $\Lambda$-projective  systems}
\label{sec:univ_repn}

Throughout this section, as in Section \ref{sec:monic-results}, we assume that $\Lambda$ is a finite $k$-graph with no sources.  The focus of this section is the construction of a \lq universal representation' of $C^*(\Lambda)$, generalizing the work in Section 4 of \cite{dutkay-jorgensen-monic} for Cuntz algebras, such that  all of the non-negative monic representations of $C^*(\Lambda)$ are a sub-representation of the universal representation.  
The Hilbert space $\mathcal{H}(\Lambda^\infty)$ on which our universal representation is defined is
the \lq universal space' for representations of $C^*(\Lambda)$, see \cite{nelson}, and also  \cite{dutkay-jorgensen-monic,bezuglyi-jorgensen-infinite,alpay-jorgensen-Lew,palle-iterated}.
For the case of $\mathcal{O}_N,$ this space was also shown to be the  \lq universal representation space'
for monic representations in \cite{dutkay-jorgensen-monic}.  We recall the construction of $\mathcal{H}(\Lambda^\infty)$ below.

\begin{defn} 
\label{def-universal-Hilbert-space}
Let $\Lambda$ be a finite 
 $k$-graph with no sources, and let $\Lambda^\infty$ be the infinite path space of $\Lambda,$ endowed with the topology generated by the cylinder sets and the Borel $\sigma$-algebra associated to it. Consider the collection of pairs $(f, \mu)$, where $\mu$ is a Borel measure on $\Lambda^\infty$, and  $f\in L^2(\Lambda^\infty,\mu)$. 
 
 We say that two pairs $(f,\mu)$ and $(g,\nu)$ are equivalent, denoted by $(f,\mu)\sim (g,\nu)$, if there exists a finite Borel measure $m$ on $\Lambda^\infty$ such that
\[
\mu << m,\ \nu << m,\ \hbox{ and }\;\; f \sqrt{\frac{d\mu}{dm}} = g \sqrt{\frac{d\nu}{dm}} \;\; \text{ in } \, L^2(\Lambda^\infty, m).
\]

Let $\mathcal{H}(\Lambda^\infty)$ be the set of equivalence classes of pairs $(f,\mu)$,~i.e.
\[
\mathcal{H}(\Lambda^\infty)=\{(f,\mu)\mid \mu\;\;\text{is a finite Borel measure on $\Lambda^\infty$},\; f\in L^2(\Lambda^\infty,\mu)\}\big/\sim
.\]
Denote equivalence classes with respect to $\sim$ by 
\[
[(f,\mu )] := f\,\sqrt{d\mu}.
\]

\end{defn}

Proposition~8.3 of \cite{bezuglyi-jorgensen-infinite} establishes that $\mathcal{H}(\Lambda^\infty)$ is a Hilbert space, with the vector space structure given by scalar multiplication and 
\[  f\,\sqrt{d\mu} +  g\,\sqrt{d\nu} := \Big( f \sqrt{\frac{d\mu}{d(\mu + \nu)}}+g \sqrt{\frac{d\nu}{d( \mu + \nu)}}\Big) \sqrt{{d(\mu + \nu)}},
\]
and the inner product given by 
\begin{equation}
\label{eq:Hilbert-sp}
\langle f\,\sqrt{d\mu} ,  g\,\sqrt{d\nu}\rangle := \int_{\Lambda^\infty} \overline{f} g\, 
\Big( \sqrt{\frac{d\mu}{d(\mu + \nu)}}\ \sqrt{\frac{d\nu}{d(\mu + \nu) }}\Big) {{d(\mu + \nu)}}.
\end{equation}
Moreover,  $\mathcal{H}(\Lambda^\infty)$ admits the following universal property: every {non-negative} monic representation of $C(\Lambda^\infty)$ -- where the topological space $\Lambda^\infty$ is endowed with the topology generated by the cylinder sets -- is unitarily equivalent to a representation of $C(\Lambda^\infty)$ on a subspace of $\mathcal{H}(\Lambda^\infty)$. 
Because of this, we will call $\mathcal{H}(\Lambda^\infty)$ the \emph{universal Hilbert space}  for $\Lambda^\infty$.


Recall from \cite{dutkay-jorgensen-monic}, \cite{alpay-jorgensen-Lew}, \cite{palle-iterated} the following fundamental property of $\mathcal{H}(\Lambda^\infty)$:

\begin{prop}(\cite[Theorem~3.1]{palle-iterated}, \cite{dutkay-jorgensen-monic}, \cite{alpay-jorgensen-Lew})
\label{prop-dutkay-jorgensen-lemma-4.6}
Let $\Lambda$ be a finite $k$-graph with no sources.
For every finite Borel measure $\mu$ on $\Lambda^\infty$, define the operator $W_\mu$
from $L^2(\Lambda^\infty, \mu)$ to $\mathcal{H}(\Lambda^\infty)$ by
\[
W_\mu(f) = f \sqrt{d\mu}. 
\]
Then $W_\mu$ is an isometry of $L^2(\Lambda^\infty, \mu)$ onto a subspace of $\mathcal{H}(\Lambda^\infty)$, which we will denote by $\mathcal{L}^2(\mu)$. 
\end{prop}

We are now ready to present the universal representation $\pi_{univ}$  of $C^*(\Lambda)$ on $\mathcal{H}(\Lambda^\infty)$.

\begin{prop}
\label{prop-univ-repres}
\label{univ-repres-definition-result}
Let $\Lambda$ be a finite  $k$-graph with no sources. Fix $(f,\mu)\in \mathcal{H}(\Lambda^\infty)$.  For each $\lambda \in \Lambda^n$, define  
\[
 S_\lambda^{univ} (f\,\sqrt{d\mu}) : = (f \circ \sigma^n) \,\sqrt{d(\mu \circ \sigma_\lambda^{-1})},
\]
where $\sigma_\lambda$ and $\sigma^n$ are the standard prefixing  and coding maps for $\Lambda^\infty$.
Then:
\begin{itemize}
\item[(a)] The adjoint of $S_\lambda^{univ}$ is given by
\[
 (S_\lambda^{univ})^* (f\,\sqrt{d\mu}) : = (f \circ \sigma_\lambda) \,\sqrt{d(\mu \circ \sigma_\lambda)}.
\]
\item[(b)] The operators $\{S_\lambda^{univ}: \lambda \in \Lambda\}$ generate a representation $\pi_{univ}$ of $C^*(\Lambda)$ on $\H(\Lambda^\infty)$, which we call the \lq universal  representation'. 
\item[(c)] The projection valued measure $P$ on $\Lambda^\infty$ given in Definition \ref{def-proj-val-measu} associated to  the universal representation $\pi_{univ}$ is
given by:
\begin{equation}
\label{proj-valued-measure-universal-representation}
P(A)(f\,\sqrt{d\mu}) = (\chi_A \cdot f)\,\sqrt{d\mu} ,
\end{equation}
where $A$ is a Borel set of the Borel $\sigma$-algebra generated by the cylinder sets.
\end{itemize}
\end{prop}

\begin{proof}
The proof is similar to that of  Proposition 4.2 of   \cite{dutkay-jorgensen-monic}, although the details are more involved because of the more complicated $k$-graph structure. 
To simplify the notation, in this proof we will also drop the superscript $univ$ form $S^{univ}_\lambda$, since we will consider no other  representations here, and so there is no danger of confusion.
To explicitly check that   we obtain a representation of $C^*(\Lambda)$,
first we observe that the operators $S_\lambda$ are well defined; in other words, if $f \sqrt{d\mu} = g \sqrt{d\nu}$, we must have $(f \circ \sigma^n) \sqrt{d(\mu \circ \sigma_\lambda^{-1})} = (g \circ \sigma^n)\sqrt{d(\nu \circ \sigma_\lambda^{-1})}$ for all $\lambda \in \Lambda$.  Recall that $f \sqrt{d\mu} = g \sqrt{d\nu} \in \H(\Lambda^\infty)$ iff
\[f \sqrt{\frac{d\mu}{dm}} = g \sqrt{\frac{d\nu}{dm}} \in L^2(\Lambda^\infty, m)\]
for some Borel measure $m$ on $\Lambda^\infty$.

Observe that $\mu \circ \sigma_\lambda^{-1}$ is zero off $Z(\lambda)$, and $\mu \circ \sigma_\lambda^{-1} =  \left(\mu|_{Z(s(\lambda))} \circ \sigma_\lambda^{-1}\right)$ on $Z(\lambda)$; similarly for $\nu$. Therefore, since $ f \sqrt{\frac{d\mu}{dm}} = g \sqrt{\frac{d\nu}{dm}}$ on all of $\Lambda^\infty$, these functions agree in particular on $Z(s(\lambda))$.  The fact that $\sigma_\lambda^{-1} = \sigma^n$ on $Z(\lambda)$, where $n=d(\lambda)$, now implies that
\begin{align*}
(f \circ \sigma^n) \sqrt{d(\mu \circ \sigma_\lambda^{-1})} &= \left( (f \circ \sigma^n) \sqrt{d(\mu \circ \sigma_\lambda^{-1})}\right)|_{Z(\lambda)} = \left(f \sqrt{d\mu}\right) |_{Z(s(\lambda))} \circ \sigma_\lambda^{-1} \\
&= \left(g \sqrt{d\nu}\right) |_{Z(s(\lambda))} \circ \sigma_\lambda^{-1}= \left( (g\circ \sigma^n) \sqrt{d(\nu \circ \sigma_\lambda^{-1})}\right)|_{Z(\lambda)} \\
&= (g\circ \sigma^n) \sqrt{d(\nu \circ \sigma_\lambda^{-1})}.
\end{align*}
It follows that $S_\lambda$ is well defined.

To check the formula for $S_\lambda^*$ given in the statement of the proposition, we compute:
\begin{align*}
\langle S_\lambda^* (f \sqrt{d\mu}),  g \sqrt{d\nu}\rangle &= \langle f \sqrt{d\mu}, S_\lambda g \sqrt{d\nu} \rangle = \langle f \sqrt{d\mu}, (g \circ \sigma^n) \sqrt{d(\nu \circ \sigma_\lambda^{-1})} \rangle \\
&= \int_{\Lambda^\infty} \overline{f(x)} ( g \circ \sigma^n)(x) \sqrt{\frac{d\mu}{d(\mu + (\nu \circ \sigma_\lambda^{-1}))}} \sqrt{\frac{d(\nu \circ \sigma_\lambda^{-1})}{d(\mu + (\nu \circ \sigma_\lambda^{-1}))}}\, d(\mu + (\nu \circ \sigma_\lambda^{-1})).
\end{align*}
This integral vanishes off $Z(\lambda)$, since $\sigma_\lambda^{-1}$ (and consequently $d(\nu \circ \sigma_\lambda^{-1})$) do.  We thus use the fact that $(\nu \circ \sigma_\lambda^{-1})|_{Z(\lambda)} = \nu|_{Z(s(\lambda))} \circ \sigma_\lambda^{-1}$ to rewrite 
\begin{align*}
\langle S_\lambda^* (f \sqrt{d\mu}), g \sqrt{d\nu}\rangle &= \int_{Z(s(\lambda))} \overline{f \circ \sigma_\lambda(x)} g(x) \sqrt{\frac{d(\mu \circ \sigma_\lambda)}{d((\mu \circ \sigma_\lambda) + \nu)}} \sqrt{\frac{d\nu}{d((\mu \circ \sigma_\lambda)+ \nu)}} d((\mu \circ \sigma_\lambda) + \nu).
\end{align*}
Hence $S_\lambda^*(f\sqrt{d\mu}) = (f \circ \sigma_\lambda) \sqrt{d(\mu \circ \sigma_\lambda)}$, which proves (a).

To see (b), we first note that if $v \in \Lambda^0$ then $S_v$ is given by multiplication by $\chi_{Z(v)}$, as is $S_v^*$; in other words, $\{S_v: v \in \Lambda^0\}$ are a collection of mutually orthogonal projections, so (CK1) holds.

For (CK2), choose $\lambda \in \Lambda^\ell, \nu \in s(\lambda)\Lambda^n$ and compute:
\begin{align*}
S_\lambda S_\nu (f \sqrt{d\mu})&= S_\lambda ( f \circ \sigma^n \sqrt{d(\mu \circ \sigma_\nu^{-1})}) = f \circ \sigma^n \circ \sigma^\ell \sqrt{d(\mu \circ \sigma_\nu^{-1} \circ \sigma_\lambda^{-1})} \\
&= f \circ \sigma^{n+\ell} \sqrt{d(\mu \circ (\sigma_\lambda\sigma_\nu)^{-1})} = f \circ \sigma^{d(\lambda\nu)} \sqrt{d(\mu \circ \sigma_{\lambda \nu}^{-1})} \\
&= S_{\lambda \nu} (f \sqrt{d\mu}).
\end{align*}

For (CK3), fix $\lambda \in \Lambda^\ell$ and compute:
\begin{align*}
S_\lambda^* S_\lambda(f\sqrt{d\mu}) &= S_\lambda^*( f\circ \sigma^\ell \sqrt{d(\mu \circ \sigma_\lambda^{-1})}) = f \circ \sigma^\ell \circ \sigma_\lambda \sqrt{d(\mu \circ \sigma_\lambda^{-1} \circ \sigma_\lambda)} \\
&= (\chi_{Z(s(\lambda))} \cdot f) \sqrt{d\mu} \\
&= S_{s(\lambda)} (f \sqrt{d\mu}),
\end{align*}
since $\sigma^\ell \circ \sigma_\lambda = id_{Z(s(\lambda))}$. In particular, it now follows that $S_\lambda S_\lambda^* S_\lambda = S_\lambda$.

For (CK4), we calculate that for any fixed $n \in \N^k$,
\begin{align*}
\sum_{\nu \in v\Lambda^n} S_\nu S_\nu^*(f \sqrt{d\mu}) &= \sum_{\nu\in v\Lambda^n} S_\nu( f \circ \sigma_\nu) \sqrt{d(\mu \circ \sigma_\nu)}\\
&= \sum_{\nu \in v\Lambda^n} f \circ \sigma_\nu \circ \sigma^n \sqrt{d(\mu \circ \sigma_\nu \circ \sigma_\nu^{-1})} \\
&= \sum_{\nu \in v\Lambda^n} \chi_{Z(\nu)} \cdot f \sqrt{d\mu},
\end{align*}
since $\mu \circ \sigma_\nu \circ \sigma_\nu^{-1}$ vanishes off $Z(\nu)$, and $\sigma_\nu \circ \sigma_\nu^{-1} = id$ on $Z(\nu)$. Consequently, since $Z(v) = \sqcup_{\nu \in v\Lambda^n} Z(\nu)$, we have 
\[\sum_{\nu \in v\Lambda^n} S_\nu S_\nu^* (f\sqrt{d\mu}) = \chi_{Z(v)} f \sqrt{d\mu} = S_v (f \sqrt{d\mu}).\]
Thus, the operators $\{S_\lambda\}$ give a representation of $C^*(\Lambda)$.

To see (c), note that Equation \eqref{proj-valued-measure-universal-representation} follows from the observation that $S_\nu S_\nu^*$ acts by multiplication by $\chi_{Z(\nu)}$; the fact that disjoint unions of cylinder sets $Z(\nu)$ generate the $\sigma$-algebra up to sets of measure zero therefore enables us  to compute $P(A)$ by linearity.
\end{proof}

The following two Propositions, which detail additional technical properties of the projection valued measure associated to $\pi_{univ}$,  will be used later  in the proof of the main results of this Section.

\begin{prop} 
\label{prop-commutant-dutkay-jorgensen-lemma-4.4}
Let $\Lambda$ be a finite $k$-graph with no sources and let $\mathcal{H}(\Lambda^\infty)$ be the Hilbert space described in Definition~\ref{def-universal-Hilbert-space}, and let $\pi_{univ}=  \{S_\lambda^{univ}: \lambda \in \Lambda\}$ be the universal representation of $C^*(\Lambda)$ on $\mathcal{H}(\Lambda^\infty)$ given in Proposition~\ref{prop-univ-repres}.

\begin{itemize}
\item[(a)] For $y\in \mathcal{H}(\Lambda^\infty)$, define a function $\nu_y$ on $\Lambda^\infty$ by
\[
\nu_{y} (Z(\lambda) ) : =  \langle y, (S^{univ}_\lambda (S^{univ}_\lambda)^*) y \rangle,
\]
where $\langle \cdot,\cdot\rangle$ is the inner product given on $\mathcal{H}(\Lambda^\infty)$ in Equation~\eqref{eq:Hilbert-sp}. Then $\nu_y$ gives a measure on $\Lambda^\infty$.

\item[(b)]  Let $T$ be a bounded operator on $\mathcal{H}(\Lambda^\infty)$.  If  $T$ commutes with $\pi_{univ}|_{C(\Lambda^\infty)}$,  then for any  $x\in \mathcal{H}(\Lambda^\infty)$ we have
\[
\nu_{T(x)} << \nu_{x}.
\]
\end{itemize}
\end{prop}

\begin{proof} As in Equation \eqref{eq:monic_measure}, it is straightforward to see (a). For (b), fix $x\in \mathcal{H}(\Lambda^\infty)$. Then since $T$ commutes with $S^{univ}_\lambda$, we have
 \[
\nu_{T(x)}(Z(\lambda) ) =  \langle (S^{univ}_\lambda (S^{univ}_\lambda)^*)T(x), T(x) \rangle =  \langle (S^{univ}_\lambda (S^{univ}_\lambda)^*)x, T^* T (x) \rangle.
 \]
Since each $S^{univ}_\lambda$ is a partial isometry, the Cauchy-Schwarz inequality then gives
  \[
 \nu_{T(x)}(Z(\lambda) )^2 \leq  \Vert(S^{univ}_\lambda (S^{univ}_\lambda)^*)x\Vert^2 \ \Vert T^* T x \Vert^2 = \nu_{x}(Z(\lambda) )^2  \ \Vert T^* T x \Vert^2 ,
  \]
which gives that $\nu_T(x) << \nu_x$.
\end{proof}

\begin{prop} 
\label{prop-dutkay-jorgensen-lemma-4.5}  Let $\Lambda$, $\mathcal{H}(\Lambda^\infty)$ and $\pi_{univ}$ be as in Proposition~\ref{prop-commutant-dutkay-jorgensen-lemma-4.4}. Then for every   vector
$f \sqrt{d \mu}\in \mathcal{H}(\Lambda^\infty)$, we have
\[
 \nu_{f \sqrt{d \mu}} = |f|^2  \mu.
\]
\end{prop}

\begin{proof} 
By definition, for any cylinder set $Z(\eta)$ we have
\[
\nu_{f \sqrt{d \mu}} (Z(\eta) ) : =  \langle (S^{univ}_\eta (S^{univ}_\eta)^*)f \sqrt{d \mu}, f \sqrt{d \mu} \rangle.
\]
By using Equation (\ref{proj-valued-measure-universal-representation}) this becomes
\[
\int_{Z(\eta)} \nu_{f \sqrt{d\mu}} \, d\mu= 
\nu_{f \sqrt{d \mu}} (Z(\eta) )  =  \langle( {\chi_{Z(\eta)}} f)\,\sqrt{d\mu} , f \sqrt{d \mu} \rangle =\int \chi_{Z(\eta)} \cdot  |f|^2 d\mu \;  = \int_{Z(\eta)} |f|^2 \, d\mu.
\]
This
 gives the desired result 
 (see also Carath\'eodory/Kolmogorov's consistency/extension theorems \cite{kolmogorov}, \cite{Tum}): since the cylinder sets generate the Borel $\sigma$-algebra   $\mathcal{B}_\Lambda$ on $\Lambda^\infty$, two measures which agree on all cylinder sets must be the same measure.
\end{proof}

We now present  an important result which will allow us to derive, in Theorem  \ref{prop-universal-name}, the desired universal property of the representation.

\begin{thm} 
\label{them-universal-representation-commutant} 
Let $\Lambda$ be a finite $k$-graph with no sources. Let $\mathcal{H}(\Lambda^\infty)$ be the universal Hilbert space  for $\Lambda^\infty$ and $\pi_{univ}$ be the universal representation of $C^*(\Lambda)$ on $\mathcal{H}(\Lambda^\infty)$. Then:
\begin{enumerate}
\item [(a)] An operator $T\in \mathcal{B}(\mathcal{H}(\Lambda^\infty))$ commutes with  $\pi_{univ}|_{C(\Lambda^\infty)}$ 
if and only if for each finite Borel  measure $ \mu $ on $\Lambda^\infty$ which arises from a monic representation of $C^*(\Lambda)$ as in Equation \eqref{eq:monic_measure}, 
there exists a function $F_\mu$ in ${L}^\infty(\Lambda^\infty, \mu)$ such that: 
\begin{enumerate}
\item[(i)]  $sup \{ \Vert F_\mu\Vert : \mu \text{ arises from a monic representation }\} < \infty$.
\item[(ii)]  If $\mu << \lambda$ then $F_\mu = F_\lambda$, $\mu$-a.e.
\item[(iii)]  $T(f\sqrt{d\mu}) = F_\mu f\sqrt{d\mu}$ for all $f\sqrt{d\mu} \in \mathcal{H}(\Lambda^\infty)$ 
\end{enumerate}
\item[(b)]  Let $\mathcal H$ denote the subspace of $\mathcal H(\Lambda^\infty)$ spanned by vectors of the form $f \sqrt{d\mu}$ 
where $\mu$ arises from a monic representation.  An operator $T\in \mathcal{B}(\mathcal{H}(\Lambda^\infty))$  commutes with the  restriction $\pi_{univ}|_\mathcal H$ of the universal representation  $\pi_{univ}$   if and only if  for every finite Borel measure $\mu$ on $\Lambda^\infty$ arising from a monic representation of $C^*(\Lambda)$,  
and for each $\lambda \in \Lambda$, we have
\[
F_\mu =F_{\mu\circ \sigma_\lambda^{-1}} \circ \sigma_\lambda,\  \mu-a.e.
\]
\end{enumerate}
\end{thm}

\begin{proof} 
{Recall from Proposition}~\ref{prop-dutkay-jorgensen-lemma-4.6} {the isometry $W_\mu$ of $L^2(\Lambda^\infty, \mu)$ onto $\mathcal{L}^2(\mu)$.  Throughout the proof, we will assume that the finite Borel measure $\mu$ arises from a monic representation.}
We first claim that if   $T$ commutes with  $\pi_{univ}|_{C(\Lambda^\infty)}$, then $T$ maps $\mathcal{L}^2(\mu)$
into itself. 
To prove this, let  $x = f\sqrt{d\mu}$ be in $\mathcal{L}^2(\mu)$,  {where $(f,\mu)\in \mathcal{H}(\Lambda^\infty)$, meaning that $\mu$ is a finite Borel measure on $\Lambda^\infty$ arising from a monic representation, and $f\in L^2(\Lambda^\infty, \mu)$.} Also let
$T(x) = g\sqrt{d \zeta}$ {for $(g,\zeta)\in \mathcal{H}(\Lambda^\infty)$.}
 Then Proposition \ref{prop-commutant-dutkay-jorgensen-lemma-4.4} implies that
\[
\nu_{T(x)} << \nu_{x}.
\]
{where $\nu_x$ is a measure on $\Lambda^\infty$.}
But by Proposition  \ref{prop-dutkay-jorgensen-lemma-4.5}, we have
\[ 
\nu_{x} = |f|^2\mu, \;\;\text{ and }\;\; \nu_{T(x)} = |g|^2 \zeta.
\]
Therefore $|g|^2 \zeta << \mu $,  so by the Radon--Nikodym theorem there exists $h \geq 0$ in  $L^1(\Lambda^\infty, \mu)$
such that $|g|^2\, d\zeta = h\, d\mu$. Then
\[
|g| \sqrt{d\zeta} =\sqrt{ h} \sqrt{d\mu},\;\; \text{ and } \;\;  |g|\, g \sqrt{d\zeta} =g \, \sqrt{ h}\, \sqrt{d\mu}.
\]
If $g =0$ on some Borel set $A$, then $\sqrt{h} \sqrt{d\mu}(A) = 0$ also. Therefore, 
\[
g\sqrt{d\zeta} =\begin{cases}
\frac{ g \sqrt{h} }{|g|} \sqrt{d\mu} \; \in \mathcal{L}^2(\mu), & g \not= 0 \\ 
 0, & g=0
 \end{cases}
\]
which shows that $T$ maps $\mathcal{L}^2(\mu)$
into itself. 

{
We now make some computations regarding the relationship between an arbitrary monic representation $\pi$ and the universal representation $\pi_{univ}$.  Observe that for any $\lambda \in \Lambda$ and any $f \in L^2(\Lambda^\infty, \mu_\pi)$, we have 
\[ \pi_{univ}(\chi_{Z(\lambda)})(f \sqrt{d\mu_\pi}) = S^{univ}_\lambda (S_\lambda^{univ})^* (f \sqrt{d\mu_\pi})= \left( \chi_{Z(\lambda)} \cdot f \right) \sqrt{d\mu_\pi}.\]
Extending by linearity, we see that $\pi_{univ}(\psi) (f \sqrt{d\mu_\pi}) = \left( \psi \cdot f\right) \sqrt{d\mu_\pi}$ for any $f \in L^2(\Lambda^\infty, \mu_\pi)$.

On the other hand, since $\pi$ is a monic representation, $\pi(\chi_{Z(\lambda)}) f = \chi_{Z(\lambda)} \cdot f \in L^2(\Lambda^\infty, \mu_\pi)$ by Equation \eqref{eq:range-sets}.  Therefore, 
\[ \pi_{univ}(\psi)(f \sqrt{d\mu_\pi}) = \left[ \pi(\psi)(f) \right] \sqrt{d\mu_\pi}.\]

By hypothesis, $T$ commutes with $\pi_{univ}|_{C(\Lambda^\infty)}$. 
Since $T$ preserves $\mathcal L^2(\mu)$ for each measure $\mu$ arising from a monic representation, there must exist  $g \in \mathcal L^2(\mu_\pi)$ such that  $T(f \sqrt{d\mu_\pi}) = g \sqrt{d\mu_\pi}$.  Consequently, 
\begin{align*}
T[ \pi(\psi) f] \sqrt{d\mu_\pi} & = T \pi_{univ}(\psi) ( f \sqrt{d\mu_\pi}) = \pi_{univ}(\psi) T(f \sqrt{d\mu_\pi})\\
 &= \pi_{univ}(\psi)( g \sqrt{d\mu_\pi})
 = [\pi(\psi)(g)] \sqrt{d\mu_\pi}  = \pi(\psi) T \left( f \sqrt{d\mu_\pi}\right),
\end{align*}
so (identifying $\mathcal L^2(\mu_\pi)\subseteq \H(\Lambda^\infty)$ with $L^2(\Lambda^\infty, \mu_\pi)$) we see that $T$ commutes with $\pi(\psi)$ for all $\psi \in C(\Lambda^\infty)$.} 
%
%

 Therefore, we can pull-back  $T$ to an operator $\widetilde{T}$ on $L^2(\Lambda^\infty, \mu)$  that commutes with all of the multiplication operators $\{ M_f: f \in C(\Lambda^\infty)\}$.  The fact (cf.~\cite{vJones}) that the  maximal abelian subalgebra of $\mathcal{B}(L^2(\Lambda^\infty, \mu)) $, for any finite Borel measure $\mu$, is the sub-algebra $L^\infty(\Lambda^\infty, \mu)$ consisting of multiplication operators now implies that 
$\widetilde{T}$ must be a multiplication operator too.

So there exists a function $F_\mu$ in $L^\infty(\Lambda^\infty, \mu)$ such that 
\begin{equation}\label{eq:T}
T(f\sqrt{d\mu}) =F_\mu\, f\, \sqrt{d\mu}
\end{equation}
for all $f\in L^2(\Lambda^\infty,\mu)$, establishing $\it{(iii)}$.
It remains to check the properties of the functions $F_\mu$. One immediately observes that 
\[
\Vert F_\mu \Vert_{L^\infty(\mu)} \leq  \Vert  T \Vert 
\]
and this implies $\it{(i)}$. To check $\it{(ii)}$, suppose $\mu << \lambda$.  Then, for all $f \in L^2(\Lambda^\infty, \mu)$, 
we have $f\sqrt{d\mu} = f \sqrt{d \mu/ d\lambda} \sqrt{d\lambda}$, and hence 
\[
\begin{split}
&T(f\sqrt{d\mu}) = T(f \sqrt{d \mu/ d\lambda} \sqrt{d\lambda}) \\
\Longrightarrow \quad & F_\mu f\sqrt{d\mu} = F_\lambda f \sqrt{d \mu/ d\lambda} \sqrt{d\lambda}\end{split} \]
Writing $d\mu = \frac{d\mu}{d\lambda} d\lambda$ reveals that, as elements of $\mathcal L^2(\lambda)$,
\[ F_\mu f \sqrt{d \mu/ d\lambda} = F_\lambda f \sqrt{d \mu/ d\lambda}
\]
for any $f \in L^2(\Lambda^\infty, \mu)$, which implies $F_\mu {d \mu/ d\lambda} = F_\lambda  {d \mu/ d\lambda} \;\; (\lambda-a.e.).$  Thus, for any Borel set $A$,
\[ \int_A(  F_\mu - F_\lambda ) \, d\mu = \int_A (F_\mu - F_\lambda) \frac{d\mu}{d\lambda} d\lambda = 0,\]
so $F_\mu = F_\lambda$, $\mu$-a.e. 
This proves $\it{(ii)}$.

For the converse, assume that  $T$ is given by a function $F_\mu\in L^\infty(\Lambda^\infty,\mu)$ satisfying  $\it{(i),(ii),(iii)}$, { i.e.~$T(f\sqrt{d\mu})=F_\mu f \sqrt{d\mu}$ for all $f \sqrt{d\mu} \in\mathcal{H}(\Lambda^\infty)$ such that $\mu$ arises from a monic representation.}
First, we check that $T$ is well defined. 
Take 
$f\sqrt{d\mu} = g\sqrt{d\nu}$  and let
$\lambda$ be a measure such that $\mu, \nu << \lambda$. Then: 
\[
f\sqrt{d\mu / d\lambda }= g \sqrt{d\nu / d\lambda} ,\ \ (\lambda-a.e.).
\]
Now, from  {\it{(ii)}} we know that $F_\mu = F_\lambda\, (\mu-a.e.)$ and {$F_\nu=F_\lambda\, (\nu-a.e)$.} Thus     
\[
F_\mu \sqrt{d\mu / d\lambda } = F_\lambda \sqrt{d\mu / d\lambda } ,\ \ (\lambda-a.e.), \text {and } 
F_\nu \sqrt{d\nu / d\lambda } = F_\lambda \sqrt{d\nu / d\lambda } ,\ \ (\lambda-a.e.).
\]
Therefore 
$
 f F_\mu \sqrt{d\mu / d\lambda }  = g F_\nu \sqrt{d\nu / d\lambda } ,\  \ (\lambda-a.e.), 
$
 and hence
  \[
  f F_\mu \sqrt{d\mu}  = g F_\nu \sqrt{d\nu  } ,
  \]
  so $T$ is well defined. 
 
 Now {\it{(i)}} implies that T is bounded with 
 \[\Vert T \Vert \leq sup_\mu\{  \Vert F_\mu \Vert_{L^\infty(\mu)} \}.\]
Since $T$ acts as a multiplication operator on each $\mathcal L^2(\mu)$, Part (c) of Proposition \ref{univ-repres-definition-result} implies that  $T$ commutes with $P(A)$ for all Borel subsets $A$ and therefore $T$ commutes with the {restricted} universal representation, {$\pi_{univ}|_{C(\Lambda^\infty)}$, which proves (a).}
 
 For (b), note that if an operator $T \in \mathcal B(\H(\Lambda^\infty))$  commutes with the universal representation $\pi_{univ}$ of $C^*(\Lambda)$ on $\mathcal H$, then in particular $T$ commutes with $\pi_{univ}|_{C(\Lambda^\infty)}$ on $\mathcal H$, and hence $T(f \sqrt{d\mu}) = F_\mu f \sqrt{d\mu}$ is a multiplication operator on each $\mathcal L^2(\mu)$ when the  measure $\mu$ arises from a monic representation.  In particular, $T$ is normal (when restricted to $\mathcal H$).  Therefore, by the Fuglede-Putnam theorem,  $T|_{\mathcal H}$  commutes with $\pi_{univ}$ iff $T S_\lambda^{univ}|_{\mathcal H} = S_\lambda^{univ} T|_{\mathcal H}$ for all $\lambda\in \Lambda$. 
 Using the formulas for $S_\lambda^{univ}$ from Theorem \ref{them-universal-representation-commutant}, we see that $T|_{\mathcal H}$ commutes with $\pi_{univ}|_{\mathcal H}$ if and only if, 
for each $f \sqrt{d\mu}\in \mathcal{H}$ and $\lambda\in\Lambda^n$,
\[ 
 F_{\mu \circ \sigma_\lambda^{-1}}(f \circ \sigma^n)\sqrt{d\mu \circ \sigma_\lambda^{-1}} =  (F_\mu \circ \sigma^n)(f \circ \sigma^n)\sqrt{d\mu \circ \sigma_\lambda^{-1}} ,
\]
or equivalently,
\[ 
 F_{\mu \circ \sigma_\lambda^{-1}}(f \circ \sigma^n)=  (F_\mu \circ \sigma^n)(f \circ \sigma^n)\quad \text{for all } \;\; \lambda\in\Lambda^n,\; (\mu \circ \sigma_\lambda^{-1})-a.e.
\]
for all measures $\mu$ arising from monic representations of $C^*(\Lambda)$.  That is, for such $\mu$,
\[ 
 F_{\mu \circ \sigma_\lambda^{-1}}=  (F_\mu \circ \sigma^n) \ \text{ for } \lambda\in\Lambda^n, \; (\mu \circ \sigma_\lambda^{-1}) - a.e .
\]
  Composing with $\sigma_\lambda$ gives the desired result of (b).
\end{proof}
 
%
%

 The following proposition, consequence of Theorem \ref{them-universal-representation-commutant}, will justify the name `universal representation' for non-negative monic representations.

 \begin{thm}
 \label{prop-universal-name} 
 Let $\Lambda$ be a finite $k$-graph with no sources.
 Let $\pi = \{t_\lambda\}_{\lambda \in \Lambda}$ be a nonnegative  monic representation of $C^*(\Lambda)$ on $L^2(\Lambda^\infty, \mu_\pi)$. 
 Let $W$ be the isometry from $L^2(\Lambda^\infty, \mu_\pi)$ onto $\mathcal{L}^2(\mu_\pi )$ given in  Proposition \ref{prop-dutkay-jorgensen-lemma-4.6}:
 \[
  W f = f\sqrt{d\mu_\pi} .
 \]
  Then $W$ intertwines the monic representation $\{t_\lambda\}_{\lambda\in\Lambda}$  with  the sub-representation $\{S_\lambda^{univ}|_{\mathcal L^2(\mu_\pi)}\}_{\lambda\in\Lambda}$ of the universal representation $\{S_\lambda^{univ}\}_\lambda$.
 \end{thm} 
 
 \begin{proof} 
 By Theorem \ref{thm-characterization-monic-repres} {and Proposition~\ref{prop:lambda-proj-repn}}, we can assume that $t_\lambda$ is of the form 
 \[ t_\lambda(f) = f_\lambda \cdot (f \circ \sigma^{d(\lambda)}),\]
 {where, since $\pi$ is assumed nonnegative, we may assume $f_\lambda=\sqrt{\frac{d(\mu_\pi\circ (\sigma_\lambda)^{-1})}{d\mu_\pi}}$.}
 By Theorem \ref{them-universal-representation-commutant} and our hypothesis that $\pi$ be nonnegative, we then have
 \[
 \begin{split}
 W (t_\lambda f) &= W( f_\lambda ( f\circ \sigma^{d(\lambda)} ))=  f_\lambda ( f\circ \sigma^{d(\lambda)} )\sqrt{d\mu_\pi}= ( f\circ \sigma^{d(\lambda)} )\sqrt{| f_\lambda|^2d\mu_\pi}\\
 &= ( f\circ \sigma^{d(\lambda)} )\sqrt{d [\mu_\pi \circ (\sigma_\lambda)^{-1}] }= S_\lambda^{univ}(f\sqrt{d\mu_\pi})=S_\lambda^{univ}W(f)
 \end{split}
 \]
Since $W$ is an isometry, it follows that $W$ intertwines $\{t_\lambda\}_{\lambda\in \Lambda}$ and $\{S_\lambda^{univ}|_{\mathcal L^2(\mu_\pi)}\}_{\lambda\in\Lambda}$, as claimed.
 \end{proof}

\section{Purely atomic representations of $C^*(\Lambda)$}
\label{sec:atomic_repn}

In this section, we define purely atomic representations of $C^*(\Lambda)$ in terms of the projection valued measure being purely atomic (c.f. Definition 4.1 of \cite{dutkay-jorgensen-atomic}).

 Nearly all of the properties and characterizations of purely atomic representations which were established in the Cuntz algebra setting \cite{dutkay-jorgensen-atomic} transfer smoothly to the setting of higher-rank graphs.  Given the fundamental structural differences between Cuntz algebras and $k$-graph algebras (for example, the latter need not be purely infinite or simple, nor is their $K$-theory known in general), this was surprising to the authors, and suggests that purely atomic representations might also be fruitfully applied to other classes of $C^*$-algebras which can be described via generators and relations, such as Cuntz-Pimsner algebras or the $C^*$-algebras of topological higher-rank graphs.

 \begin{defn} 
 \label{defatomic}
(c.f. Definition 4.1 of \cite{dutkay-jorgensen-atomic}.) 
Let $\Lambda$ be a row-finite $k$-graph with no sources.
A representation $\{ t_\lambda\}_{\lambda \in \Lambda}$ of  $C^*(\Lambda)$ on a Hilbert space ${\mathcal H}$ is called \emph{purely atomic}  if there exists a Borel subset $\Omega\subset\Lambda^{\infty}$ such that the projection valued measure $P$ defined on the Borel sets of $\Lambda^{\infty}$ as in Definition \ref{conj-palle-proj-valued-measure-gen-case} satisfies
\begin{enumerate}
 \item[(a)] $P(\Lambda^{\infty}\backslash \Omega)\;\;=\; 0_{\mathcal H},$
 \item[(b)] $P(\{\omega\})\not=0_{\mathcal H}$ for all $\omega\in \Omega$,
 \item[(c)] $\bigoplus_{\omega\in \Omega}P(\{\omega\})=\;\text{Id}_{\mathcal H},$
 \end{enumerate}
 where the sum on the left-hand side of (c) converges in the strong operator topology.
 \end{defn}
 Thus, a representation of $C^*(\Lambda)$ is purely atomic if the corresponding projective-valued measure is purely atomic on ${\mathcal B}o(\Lambda^{\infty}).$

\begin{example}
Consider the $2$-graph $\Lambda$ with  $1$-skeleton
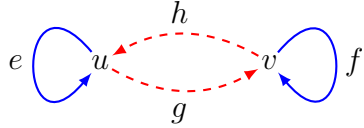
\[
\begin{tikzpicture}[scale=1.5]
 \node[inner sep=0.5pt, circle] (u) at (0,0) {$u$};
    \node[inner sep=0.5pt, circle] (v) at (1.5,0) {$v$};
    \draw[-latex, thick, blue] (v) edge [out=50, in=-50, loop, min distance=30, looseness=2.5] (v);
    \draw[-latex, thick, blue] (u) edge [out=130, in=230, loop, min distance=30, looseness=2.5] (u);
\draw[-latex, thick, red, dashed] (v) edge [out=150, in=30] (u);
\draw[-latex, thick, red, dashed] (u) edge [out=-30, in=210] (v);
\node at (-0.75, 0) {\color{black} $e$}; 
\node at (0.7, 0.45) {\color{black} $h$};
\node at (0.7, -0.45) {\color{black} $g$};
\node at (2.25, 0) {\color{black} $f$};
\end{tikzpicture}
\]
and factorization rules given by $eh=hf$ and $fg=ge$. With these factorization rules, there is only one infinite path $x\in u\Lambda^\infty$, namely
\[
x=ehfgehfg\dots = hfgehfge\dots = hgeehgee\dots.
\]
Similarly, one can see that there is only one infinite path $y\in v\Lambda^\infty$, namely
\[
y=fgehfgeh\dots = gehfgehf \dots = ghffghff \dots .
\]
Since the infinite path space consists of two elements, any nontrivial representation of $C^*(\Lambda)$ must be purely atomic. 
\end{example}

We now define the notion of an {\it orbit} of an element $\omega$ in the infinite path space $\Lambda^{\infty},$ motivated by the groupoid characterization of $C^*(\Lambda).$ Recall from Definition 2.7 of \cite{KP} that for any row-finite source-free $k$-graph $\Lambda$, we have a groupoid 
\[ \mathcal G_\Lambda =\{ (x, m-n, y) \in \Lambda^\infty \times \Z^k \times \Lambda^\infty: \sigma^m(x) = \sigma^n(y)\}\]
such that $C^*(\Lambda) \cong C^*(\mathcal G_\Lambda)$.  The groupoid $\mathcal G_\Lambda$ is \'etale \cite[Proposition 2.8]{KP} and amenable \cite[Theorem 5.5]{KP}. (Also see \cite{FMY} for groupoid models of general $k$-graphs and their $C^*$-algebras).
 \begin{defn} 
 \label{deforbit}
 Let $\Lambda$ be a row-finite $k$-graph with no sources, and let $\Lambda^{\infty}$ be the infinite path space of $\Lambda.$
 For any $\omega\in \Lambda^{\infty},$ we set 
 $$\text{Orbit}(\omega)\;=\{\gamma\in\Lambda^{\infty}:(\gamma,n,\omega)\in {\mathcal G}_{\Lambda}\;\text{for some}\;n\in \mathbb Z^k\},$$
 i.e. $\gamma\in \text{Orbit}(\omega)$ if and only if there exist $m,\;\ell\in\mathbb N^k$ such that $\sigma^m(\gamma)\;=\sigma^{\ell}(\omega).$
 \end{defn}
 We note that $\text{Orbit}(\omega)$ is invariant under $\sigma^n$ and $(\sigma^n)^{-1},$ for all $n\in \mathbb N^k.$ 
 {Note also that each orbit is a   Borel set, being a countable union of points (which are countable intersections of cylinder sets).}
 
 The same arguments as we used for Proposition \ref{prop-atomic-basic-equns} give us the following result.

\begin{prop}
\label{prop:atomic-2}  Let $\Lambda$ be a row-finite $k$-graph with no sources, and let $\{t_\lambda\}_{\lambda \in \Lambda}$ generate a purely atomic representation of $C^*(\Lambda).$  Let $P$ be the associated projection valued measure on the Borel subsets of the infinite path space $\Lambda^{\infty}.$ Then we have the following.
\begin{itemize}
\item[(a)] For $\lambda\in\Lambda$ and $\omega\in Z(s(\lambda))\subset\; \Lambda^{\infty}$, we have 
\[
t_{\lambda}P(\{\omega\})t_{\lambda}^*\;=\;P(\{\lambda\omega\}),\]
and for $n\in \mathbb N^k$, we have 
\[
t_{\omega(0,n)}^*P(\{\omega\})t_{\omega(0,n)}\;=\;P(\{\sigma^n(\omega)\}).
\]
\item[(b)] For $\eta\in\Lambda^n$ and $\omega\in \Lambda^\infty$ with $\eta\not=\omega(0,n)$, we have
\[
t_{\eta}^*P(\{\omega\})t_{\eta}=0.
\]
\end{itemize}
\end{prop}
\begin{proof}
These statements follow from taking the limit in the strong operator topology on a nested sequence of cylinder sets decreasing to $\{\omega\}$ and using Proposition~\ref{prop-atomic-basic-equns}(a).
\end{proof}

 The following Corollary is an immediate consequence of Proposition \ref{prop:atomic-2} and our observations above.
 
 \begin{cor} Let $\Lambda$ be a row-finite, source-free $k$-graph and let $\{ t_\lambda\}_{\lambda \in \Lambda}$ be a representation of $C^*(\Lambda)$. If  $Orbit(\omega) = Orbit(\gamma)$ then $P(\{\omega\}) = 0$ iff $ P(\{\gamma\}) = 0$.
 
 Moreover, for any purely atomic representation $\{ t_\lambda\}_{\lambda \in \Lambda}$, if \[\Omega = \{ \omega \in \Lambda^\infty: P(\{\omega\}) \not= 0\},\]
  then we can decompose $\Omega$ as a disjoint union of orbits: 
 $
 \Omega\;=\;\bigsqcup_{\omega}\text{Orbit}(\omega).$  In particular, 
 \[\bigoplus_{\omega \in \Omega} P(\text{Orbit}(\omega))= \text{Id}_{\mathcal H}.\]
\label{prop:orbits-of-atoms-are-nonzero} 

 \label{proporbitsatomic}
 \end{cor}

\begin{example}
\label{exampleatomic}
(c.f \cite{KP}, Proposition 2.11) Recall that for a row-finite, source-free $k$-graph $\Lambda$, the infinite path representation of $C^*(\Lambda)$ first given by A. Kumjian and D. Pask via the partial isometries $\{S_{\lambda}: \lambda\in \Lambda\}$ on the non-separable Hilbert space $\ell^2(\Lambda^{\infty})$ with orthonormal basis $\{h_\omega: \omega\in \Lambda^{\infty}\}$ is given by 
$$S_{\lambda}(h_{\omega})\;=\;\delta_{s(\lambda), r(\omega)}h_{\lambda\omega},\;\;\text{and}\;\;S_{\lambda}^*h_{\omega}=\delta_{\lambda, \omega(0,d(\lambda))}h_{\sigma^{d(\lambda)}(\omega)}.$$
Then one can check that this representation is purely atomic, and we have  for all $\omega\in \Lambda^{\infty},$
$$P(\{\omega\})=\lim_{n} S_{ \omega(0,n)}S_{ \omega(0,n)}^*\;=\;P_{{\mathcal M}_{\omega}}.$$
Here the limit is taken in the strong operator topology using the partially ordered set $\mathbb N^k,$ and ${\mathcal M}_{\omega}$ is the one-dimensional subspace of  $\ell^2(\Lambda^{\infty})$  spanned by $h_{\omega}.$  This is a standard example to keep in mind when considering both purely atomic representations and the permutative representations which we discuss in Section \ref{sec:permutative_repn} below.
\end{example}

We now show that any intertwiner of purely atomic represetations of a $k$-graph algebra intertwines the associated projection valued measures.  This is the analog of  Proposition 2.10 of  Dutkay,  Haussermann and  Jorgensen in  \cite{dutkay-jorgensen-atomic} for the Cuntz algebra case.
\begin{prop}
\label{propintertwine}  Let $\Lambda$ be a row-finite $k$-graph with no sources, and let  $\{ t_\lambda\}_{\lambda \in \Lambda}$ and  $\{ \tilde t_\lambda\}_{\lambda \in \Lambda}$ be two purely atomic representations of $C^{\ast}(\Lambda)$ on the Hilbert spaces ${\mathcal H}$ and ${\mathcal H}',$ respectively.  Suppose that $X:{\mathcal H}\to {\mathcal H}'$ is an intertwining operator for these representations, so that 
\[
\tilde{t_\lambda}X\;=\;X t_{\lambda}\;\;\text{and}\;\;(\tilde t_\lambda)^*X=Xt_{\lambda}^*\quad\text{for all}\;\; \lambda\in \Lambda.
\]
Let $P$ and $\tilde{P}$ be the associated projection valued measures on ${\mathcal B}(\Lambda^{\infty}).$  Then for every $\omega\in \Lambda^{\infty},$
$$\tilde{P}(\{\omega\})X\;=\;X P(\{\omega\}).$$ 
If $X$ is a unitary operator so that the representations are unitarily equivalent, then the supports of $P$ and $\tilde{P}$ are the same.
\end{prop}
\begin{proof}
Since $X$ intertwines the representations, we see that for every $\lambda\in \Lambda,$
$$\tilde{P}(Z(\lambda))X=\tilde t_\lambda(\tilde t_\lambda)^*X\;=\;Xt_{\lambda}t_{\lambda}^*\;=\;XP(Z(\lambda)).$$
Therefore $X$ intertwines the projection valued measures on all Borel subsets in $\Lambda^{\infty},$ including the point sets. It follows that if $X$ is unitary, then for every $\omega\in \Lambda^{\infty},$ we have 
$$\tilde{P}(\{\omega\})=XP(\{\omega\})X^*,$$ so that $\text{dim}P(\{\omega\})=\text{dim}\tilde{P}(\{\omega\})$ for all $\omega\in \Lambda^{\infty},$ and  hence $\text{supp}(P)=\text{supp}(\tilde{P}).$
\end{proof}
We now derive some straightforward consequences of Proposition \ref{proporbitsatomic}.

\begin{prop}
Suppose that an irreducible representation $\{t_\lambda\}_{\lambda \in \Lambda}$  of $C^*(\Lambda)$ has an atom $\omega$.  Then $\{t_\lambda\}_{\lambda \in \Lambda}$ is purely atomic and the associated projection valued measure is supported on Orbit$(\omega)$.  
\label{prop-atomic-1}
\end{prop}
\begin{proof}
 We begin by observing that for any $x \in \Lambda^\infty$, 
 \[ P(\{x\}) t_\lambda = \begin{cases} 0, & x \not\in Z(\lambda) \\
 t_\lambda P(\{\sigma^{d(\lambda)}(x)\}), & x \in Z(\lambda).
 \end{cases}\]
  This follows from writing $P(\{x\}) = \lim \{ t_\eta t_\eta^*: x \in Z(\eta)\}$ and observing that if $d(\eta) \geq d(\lambda)$, then 
  \[ t_\eta^* t_\lambda = \sum_{(\rho, \nu) \in \Lambda^{\min}(\lambda, \eta)} t_\rho t_\nu^* = \begin{cases} t_\rho, & \eta = \lambda \rho \\ 
  0, & \text{ else.}
  \end{cases}\]
  Therefore, if $x \not\in Z(\lambda)$ then we can find $Z(\eta) \ni x$ with $d(\eta) \geq d(\lambda)$ and such that $\eta$ does not extend $\lambda$.
  Consequently, if we set 
  \[ P := \sum_{x\not\in Orbit(\omega)} P(\{x\}),\] then $ P t_\lambda   = t_\lambda \sum_{x \in Z(\lambda) \backslash Orbit(\omega)} P(\{ \sigma^{d(\lambda)}(x)\}).$
Note that $P$ is a projection, since $P(\{x\}) P(\{y\}) = \delta_{x,y}$.  
  
  On the other hand, 
  \[t_\lambda P = t_\lambda \sum_{y \not\in Orbit(\omega): r(y) = s(\lambda)} P(\{y\}).\]
  Since $\{y \not\in Orbit(\omega): r(y) = s(\lambda)\} = \{ \sigma^{d(\lambda)}(x): x \in Z(\lambda) \backslash Orbit(\omega)\}$, we have 
  \[ t_\lambda P = P t_\lambda\]
  for any $\lambda \in \Lambda$.  
  Our assumption that $\{t_\lambda\}_{\lambda}$ is irreducible now implies that $P$ must be a multiple of the identity.  However, $P < 1$ since $\omega$ is an atom, so we must have $P =0$.
  
 Proposition \ref{prop:orbits-of-atoms-are-nonzero} now implies that $P(\{x\}) \not= 0$ for every $x \in Orbit(\omega)$, completing the proof that $\{t_\lambda\}_{\lambda \in \Lambda}$ is purely atomic.
\end{proof}

%
%
%

What follows is an analog of Corollary 4.8 of \cite{dutkay-jorgensen-atomic}.
\begin{thm}
\label{thm-atomic-repres}
For a row-finite, source-free $k$-graph $\Lambda$, let $\{t_\lambda\}_{\lambda \in \Lambda}, \{\tilde{t}_{\lambda}\}_{\lambda \in \Lambda}$ generate purely atomic representations of $C^*(\Lambda)$ on the same Hilbert space ${\mathcal H}$, with associated projection valued measures $P,\tilde{P}$.    Then the two representations are unitarily equivalent if and only if the following conditions are satisfied:
\begin{enumerate}
\item[(a)] $\text{supp}(P)\;=\;\text{supp}(\tilde{P})\;=\; :\Omega.$
\item[(b)] For every $x\in \Omega,\; \text{dim}[\text{Range}(P(\{x\}))]\;=\;\text{dim}[\text{Range}(\tilde{P}(\{x\}) )]$  
\end{enumerate}
\end{thm}
\begin{proof}
Suppose that the purely atomic representations $\{t_\lambda\}_{\lambda \in \Lambda}, \{t'_{\lambda}\}_{\lambda \in \Lambda}$ on the same Hilbert space ${\mathcal H}$ are unitarily equivalent.  Proposition \ref{propintertwine} then implies that $P, \tilde P$ have the same support, and moreover that the intertwining unitary takes $P(\{\omega\})$ to $\tilde P(\{\omega\})$ for every $\omega \in \Omega$.


Now, suppose that conditions $(a)$ and $(b)$ hold; we will show that the representations $\{t_\lambda\}_{\lambda \in \Lambda}$ and  $\{\tilde{t}_{\lambda}\}_{\lambda \in \Lambda}$ of $C^*(\Lambda)$ are unitarily equivalent.

Without loss of generality, we suppose that our representations are irreducible, so that $\Omega$  is just one orbit, $\Omega=\text{Orbit}(\omega)$.  Since
$$\text{dim}[\text{Range}(P(\{\omega\}))]\;=\;\text{dim}[\text{Range}(\tilde{P}(\{\omega\}))]$$
{by hypothesis,}
 there is a unitary isomorphism $U_{\omega}: \text{Range}(P(\{\omega\}))\;\to\;\text{Range}(\tilde{P}(\{\omega\})),$ since Hilbert spaces of the same dimension are isomorphic. 
 For every $\gamma\in \Omega=\text{Orbit}(\omega),$
we now construct a unitary $U_{\gamma}: \text{Range}(P(\{\gamma\})) \to \text{Range}(\tilde{P}(\{\gamma\}))$  as follows.
If $\gamma \in \text{Orbit}(\omega)\subset \Lambda^{\infty}$ satisfies $\gamma = a \sigma^j(\omega) $ 
 for some $a \in \Lambda$, we would like to define 
$U_{\gamma}: \text{Range}(P(\{\gamma\})) \to \text{Range}(\tilde{P}(\{\gamma\}))$
 by
\begin{equation}\label{eq:unitary}
U_{\gamma}:=\; \tilde{t}_a\tilde{t}_{\omega(0,j)}^*U_{\omega}t_{\omega(0,j)}t_a^*.
\end{equation}
We must check that $U_{\gamma}$ is well-defined and unitary, and {that \[U := \bigoplus_{\gamma \in Orbit(\omega)} U_\gamma\]} intertwines the representations.

To see that $U_{\gamma}$ is well-defined, suppose that $\gamma = a \sigma^j(\omega) = a' \sigma^{j'}(\omega)$. Fix $\xi\in {\mathcal H}$ and $\varepsilon>0.$
Since the projections $P(Z(\omega(0,n))$ tend to $P(\{\omega\})$ in the strong operator topology, it is possible to find $N_1\in \mathbb N$ (depending on $\xi\in{\mathcal H}$ and $\varepsilon>0$) such that, if we write ${\bf 1} = (1, \ldots 1) \in \N^k$, then  $ N_1 {\bf 1}  \geq j, j'$ and 
whenever $N\geq N_1,$
\begin{equation}\|P(\{\omega\})t_{\omega(0,j)}t_a^*(\xi)-P(Z(\omega(0,N {\bf 1}))t_{\omega(0,j)}t_a^*(\xi)\| <\varepsilon.\label{eq:zeroth}
\end{equation}

Write $A\;=\;a\omega(j, N {\bf 1}).$ Note that $\gamma(0,N{\bf 1})=A$ and 
$$t_{\omega(0, N {\bf 1})}t_{\omega(0, N{\bf 1})}^*t_{\omega(0,j)}t_a^* 
=\;t_{\omega(0, N{\bf 1})}t_{\omega(j, N {\bf 1})}^*t_a^*=t_{\omega(0, N {\bf 1})}t_A^*.$$
Consequently, \eqref{eq:zeroth} implies that for $N\geq N_1,$ we have 
\begin{equation}\| P(\{\omega\}) t_{\omega(0,j)} t_a^*(\xi) - t_{\omega(0, N{\bf 1})} t_A^*(\xi)\| 
 <\varepsilon\label{eq:first}
\end{equation}
and
\begin{align*}
\| U_{\omega} t_{\omega(0, N{\bf 1})}t_A^*(\xi)\;& -\;U_{\omega}P(\{\omega\})t_{\omega(0,j)}t_a^*(\xi)\|  =\| U_{\omega} t_{\omega(0, N{\bf 1})}t_{\omega(j, N{\bf 1})}^*t_a^*(\xi) \;-\;U_{\omega}P(\{\omega\})t_{\omega(0,j)}t_a^*(\xi)\|\\
& =\|U_{\omega}t_{\omega(0, N{\bf 1})}t_{\omega(0, N {\bf 1})}^*t_{\omega(0,j)}t_a^*(\xi)\;-\;U_{\omega}P(\{\omega\}))t_{\omega(0,j)}t_a^*(\xi)\| \\
&=\|U_{\omega}P(Z(\omega(0,N  {\bf 1})))t_{\omega(0,j)}t_a^*(\xi)-U_{\omega}P(\{\omega\}))t_{\omega(0,j)}t_a^*(\xi)\|<\varepsilon.\end{align*}
In the same way, we can find $N_2$ large enough so that for the same $\xi\in {\mathcal H}$ and the same $\varepsilon>0,$ for all 
$N\geq N_2,$
$$\|\tilde{P}(\{\omega\})U_{\omega}t_{\omega(0,j)}t_a^*(\xi)-\tilde{P}(Z(\omega(0,N{\bf 1})))U_{\omega}t_{\omega(0,j)}t_a^*(\xi)\|<\varepsilon.$$
Then, since $\tilde{t}_a$ and $\tilde{t}_{\omega(0,j)}$ are partial isometries,
$$\|\tilde{t}_a\tilde{t}_{\omega(0,j)}^* \tilde{P}(\{\omega\})U_{\omega}t_{\omega(0,j)}t_a^*(\xi)-\tilde{t}_a\tilde{t}_{\omega(0,j)}^*\tilde{P}(Z(\omega(0,N{\bf 1})))U_{\omega}t_{\omega(0,j)}t_a^*(\xi)\|<\varepsilon.$$
We now write, for $N\geq N_2,$
\begin{equation*}\begin{split}
\tilde{t}_a\tilde{t}_{\omega(0,j)}^*\tilde{P}(Z(\omega(0,N{\bf 1})))U_{\omega}t_{\omega(0,j)}t_a^*\; &=\;\tilde{t}_a\tilde{t}_{\omega(0,j)}^*\tilde{t}_{\omega(0,N{\bf 1})}\tilde{t}_{\omega(0,N{\bf 1})}^*U_{\omega}t_{\omega(0,j)}t_a^*\\
&\;=\;\tilde{t}_a\tilde{t}_{\omega(j,N{\bf 1})}\tilde{t}_{\omega(0,N{\bf 1})}^*U_{\omega}t_{\omega(0,j)}t_a^*\\
&\;=\;\tilde{t}_A\tilde{t}_{\omega(0,N{\bf 1})}^*U_{\omega}t_{\omega(0,j)}t_a^*.
\end{split}\end{equation*}
Therefore for $N\geq 
N_2$ we have
\begin{equation}\|\tilde{t}_a\tilde{t}_{\omega(0,j)}^* \tilde{P}(\{\omega\})U_{\omega}t_{\omega(0,j)}t_a^*(\xi)-\tilde{t}_A\tilde{t}_{\omega(0,N{\bf 1})}^* U_{\omega}t_{\omega(0,j)}t_a^*(\xi)\|<\varepsilon.\label{eq:tech-1}
\end{equation}

Since $U_\omega P(\{\omega\}) = U_\omega$ and $\tilde{P}(\{\omega\}) U_\omega = U_\omega$, Equations \eqref{eq:tech-1} and \eqref{eq:first} combine to give 
%
$$ \|\tilde{t}_a\tilde{t}_{\omega(0,j)}^* U_{\omega}t_{\omega(0,j)}t_a^*(\xi)-\tilde{t}_A\tilde{t}_{\omega(0,N{\bf 1})}^* U_{\omega} t_{\omega(0, N\cdot {\bf 1})}t_A^*(\xi)\|<2\varepsilon$$
whenever $N\geq \text{max}\{N_1,N_2\}.$

In the same way, for the same $\varepsilon > 0$ and $\xi \in \mathcal H$, we can find $M'\in \N$ (depending on $a',\; j',$ and $\xi\in {\mathcal H}$) such that {for any $N' \geq M'$,} setting 
$A'=a'\omega(j', N' {\bf 1}),$ we have  $\gamma(0,N')=A'$ and 
$$\|\tilde{t}_{a'}\tilde{t}_{\omega(0,j')}^*
 U_{\omega}t_{\omega(0,j')}t_{a'}^*(\xi)-\tilde{t}_{A'}\tilde{t}_{\omega(0,N'{\bf 1})}^* U_{\omega} t_{\omega(0, N'\cdot {\bf 1})}t_{A'}^*(\xi)\|<2\varepsilon.$$
{Choosing  $N = N' \geq \max \{ M', N_1, N_2\}$ implies that $A = \gamma(0, N {\bf 1}) = A'$.  Thus, by the triangle inequality,}
\begin{equation*}
\|\tilde{t}_a\tilde{t}_{\omega(0,j)}^* 
U_{\omega}t_{\omega(0,j)}t_a^*(\xi)-\tilde{t}_{a'}\tilde{t}_{\omega(0,j')}^* 
U_{\omega}t_{\omega(0,j')}t_{a'}^*(\xi)\|<4\varepsilon.\end{equation*}
Since $\varepsilon$ {and $\xi$} were arbitrary, it follows that  if $\gamma = a \sigma^j(\omega) = a' \sigma^{j'}(\omega)$,
\[\tilde{t}_a \tilde{t}_{\omega(0, j)}^* U_{\omega}t_{\omega(0,j)}t_a^*=\tilde{t}_{a'}\tilde{t}_{\omega(0,j')}^* 
U_{\omega}t_{\omega(0,j')}t_{a'}^*.\]
Thus, the operator $U_{\gamma} : \text{Range}(P(\{\gamma\})) \to \text{Range}(\tilde{P}(\{\gamma\}))$ of Equation \eqref{eq:unitary}
is well-defined.

Now we want to show that $U_{\gamma}$ is unitary. Since $U_\omega$ is a unitary and hence $U_\omega^* U_\omega = P(\{\omega\})$, using the facts that $ \tilde P(\{\omega\}) U_\omega = U_\omega$ and $\tilde P(\{\omega\}) \tilde P(Z(\omega(0,j))) = \tilde P(\{\omega\})$, one easily computes that $U_\gamma ^* U_\gamma = t_a t_{\omega(0,j)}^* P(\{\omega\}) t_{\omega(0,j)} t_a^*$.

We now note that for $\gamma= a \sigma^j(\omega),$ 
$t_a^*$ takes $\text{Range}P(\{\gamma\})$ to $ \text{Range}P(\{\sigma^j(\omega)\})$
and 
$t_{\omega(0,j)}$ takes $\text{Range}P(\{\sigma^j(\omega)\})$ to $\text{Range}P(\{\omega\}).$ {Recalling that $U_\gamma = U_\gamma P(\{\gamma\})$, we deduce that 
\begin{align*} U_\gamma^* U_\gamma &= t_a t_{\omega(0,j)}^* t_{\omega(0,j)} t_a^* P(\{\gamma\}) = t_a t_a^* P(\{\gamma\}) \\
&= P(Z(a)) P(\{\gamma\}) = P(\{\gamma\}).
\end{align*}
}
Similarly, one can show that
$$U_{\gamma}U_{\gamma}^*\;=\;\tilde{P}(\{\gamma\}),$$ which implies that $U_\gamma$ is unitary from its domain to its range.

To show that $U= \bigoplus_{\gamma \in Orbit(\omega)} U_\gamma$ intertwines the representations, we must establish that for $\lambda\in\Lambda$ and $\omega\in \Lambda^\infty$ with $s(\lambda)=r(\omega)$,
$$\tilde{t}_{\lambda}U_{\omega}\;=\;U_{\lambda\omega}t_{\lambda}.$$
By our construction of $U_{\lambda \omega}$, if $s(\lambda)\;=\; r(\omega),$
\[
U_{\lambda\omega}t_{\lambda}\;=\;\tilde{t}_{\lambda}U_{\omega}t_{\lambda}^*t_{\lambda}.
\]
Then using the fact that $t_\lambda$ is an isometry {and that $U_\omega = U_\omega P(\{\omega\}) = U_\omega P(\{\omega\}) P(Z(\omega(0,n)))$ for any $n \in \N^k$,} we obtain
\[\begin{split}
\tilde{t}_{\lambda} U_{\omega}t_{\lambda}^*t_{\lambda}&=
\tilde{t}_\lambda U_\omega P(Z(s(\lambda)))=\tilde{t}_\lambda U_\omega P(Z(r(\omega)))=\tilde{t}_\lambda U_\omega.
\end{split}\]


Therefore we see that 
$$U= \bigoplus_{\gamma\in\;\text{Orbit}(\omega)}U_{\gamma}:\;\bigoplus_{\gamma\in\;\text{Orbit}(\omega)}\text{Range}P(\{\gamma\})\;\to\; \;\bigoplus_{\gamma\in\;\text{Orbit}(\omega)}\text{Range}\tilde{P}(\{\gamma\})$$
is a unitary operator that intertwines the representations $\{t_{\lambda}: \lambda\in \Lambda\}$ and $\{\tilde{t}_{\lambda}: \lambda\in \Lambda\}.$
\end{proof}

Recall from Definition \ref{def:disjoint} that two representations $\pi, \pi'$ of a $C^*$-algebra $A$ are  \emph{disjoint} if  no nonzero
subrepresentation of $\pi$ is unitarily equivalent to a subrepresentation of $\pi'$.

The following Proposition (the analogue of Proposition 4.5 in \cite{dutkay-jorgensen-atomic}) can now be derived from the previous result.
\begin{prop}\label{prop:irr-rank1}
For a row-finite, source-free $k$-graph $\Lambda,$  let $\{t_\lambda\}_{\lambda \in \Lambda}, \{\tilde{t}_{\lambda}\}_{\lambda \in \Lambda}$ generate purely atomic representations of $C^*(\Lambda)$, with associated projection valued measures $P, \tilde{P}$.  Suppose that $P, \tilde P$ are supported on Orbit$(\gamma)$ and Orbit$(\omega)$ respectively, for some $\gamma, \omega \in \Lambda^\infty$.  
\begin{enumerate}
\item[(a)] If Orbit$(\gamma) \not= \text{Orbit}(\omega)$ then the representations are disjoint.
\item[(b)] If Orbit$(\gamma) = \text{Orbit}(\omega)$, then every  operator $Y$ which intertwines the representations $\{t_\lambda\}_{\lambda\in\Lambda}, \{\tilde{t}_\lambda\}_{\lambda\in\Lambda}$ can be reconstructed from an operator $X: \Ran P(\omega) \to \Ran \tilde{P}(\omega)$ via
\[
Y=\tilde{P}(\{\omega\})XP(\{\omega\}): \operatorname{Range} P(\{\omega\}) \to \operatorname{Range} \tilde{P}(\{\omega\}).
\]
\item[(c)] The representation  $\{t_\lambda\}_{\lambda\in\Lambda}$ is irreducible if and only if the dimension of $\Ran P(\{\omega\})$ is 1.
\end{enumerate}
\end{prop}
\begin{proof}
{To see (a), note that under our hypotheses,  neither $\{t_\lambda\}_\lambda$ nor $\{\tilde t_\lambda\}_\lambda$ has nontrivial subrepresentations; therefore Theorem \ref{thm-atomic-repres} implies (a).}
For (b),   given that $\text{Orbit}(\gamma)=\text{Orbit}(\omega)$, we can write $\gamma = a \sigma^j(\omega)$ for some $j\in \N^k, \ a \in \Lambda$.  Given an operator $X: \Ran P(\omega) \to \Ran \tilde{P}(\omega)$, we would like to define an operator $Y$ which intertwines the representations by setting 
\[ Y|_{\Ran P(\{\gamma\})} := \tilde{t}_a (\tilde{t}_{\omega(0, j)})^* X t_{\omega(0, j)} t_a^*.\]
The essence of the proof consists in showing that $Y$ is well defined -- that is, independent of our choice of $j, a$.  The idea is essentially the same as the content of Theorem \ref{thm-atomic-repres} in the case where $X$ is unitary, so we omit the details.  

{For (c),} let  ${\mathcal H}=\text{Range}\,P(\{\omega\}),$ and $\tilde{\mathcal H}= \text{Range}\,\tilde{P}(\{\omega\}).$  Recall from Theorem \ref{thm-atomic-repres} above that if ${\mathcal H}$ and $\tilde{\mathcal H}$ have the same dimension,  any unitary $U_{\omega}\in {\mathcal U}({\mathcal H},\tilde{\mathcal H})$ can be used to construct an intertwiner 
$$\bigoplus_{\gamma\in\;\text{Orbit}(\omega)}U_{\gamma}:\;\bigoplus_{\gamma\in\;\text{Orbit}(\omega)}\text{Range}\,P(\{\gamma\})\;\to\; \bigoplus_{\gamma\in\;\text{Orbit}(\omega)}\text{Range}\,\tilde{P}(\{\gamma\}).$$
In  the same way, if $T_{\omega}$ is a finite linear combination of unitary elements in $B({\mathcal H})=B(\text{Range}\,P(\{\omega\})),$ defining for $\gamma=a\sigma^j(\omega)$ the bounded operator 
$$T_{\gamma}=\;t_a t_{\omega(0,j)}^*T_{\omega}t_{\omega(0,j)}t_a^*,$$
Theorem \ref{thm-atomic-repres} shows us that $T_{\gamma}$ is well-defined and that 
$$\bigoplus_{\gamma\in\;\text{Orbit}(\omega)}T_{\gamma}:\;\bigoplus_{\gamma\in\;\text{Orbit}(\omega)}\text{Range}\,P(\{\gamma\})\;\to\; \bigoplus_{\gamma\in\;\text{Orbit}(\omega)}\text{Range}\,P(\{\gamma\})$$
intertwines the representation  $\{t_\lambda\}_{\lambda\in\Lambda}$ with itself.


   Consequently, the set of intertwiners between  $\{t_\lambda\}_{\lambda\in\Lambda}$ and itself contains the span of the unitary elements in 
$B({\mathcal H})=B(\text{Range}\,P(\{\omega\})).$  
However, by Russo-Dye's Theorem, the closure of the span of the unitary elements in $B({\mathcal H})$ is exactly $B({\mathcal H}).$   By Schur's Lemma, our representation is irreducible if and only if the self-intertwiners of   $\{t_\lambda\}_{\lambda\in\Lambda}$ consist solely of scalar multiples of the identity.  But by our preceding construction, we see that this happens if and only if the dimension of ${\mathcal H}=\text{Range}\,P(\{\omega\})$ is equal to $1,$ as desired.
\end{proof}

\subsection{Relation between monic and purely atomic representations}
\label{sec:relation_monic_atomic}

A version of the following result, for the Cuntz algebras $\mathcal O_N$, was established in Theorem 3.15 of \cite{dutkay-jorgensen-monic}. 

\begin{thm}
Let $\Lambda$ be a finite $k$-graph with no sources.
Let $\{t_\lambda:\lambda\in \Lambda\}$ be a purely atomic representation of $C^*(\Lambda)$ on a separable Hilbert space $\mathcal{H}$. Suppose that $t_\lambda t^*_\lambda\ne 0$ for all $\lambda\in\Lambda$. Then the representation is monic if and only if for every atom $x\in \Lambda^\infty$, $P(\{x\})$ is one-dimensional. Moreover, in this case the associated measure $\mu$ arising from the monic representation is atomic.
\label{thm:atomic-1D}
\end{thm}

\begin{rmk}
Since $L^2(\Lambda^\infty, \mu)$ is separable for any measure $\mu$ associated to a monic representation, in the setting of Theorem \ref{thm:atomic-1D}
we conclude that the set of atoms for $\mu$ must be countable.
\end{rmk}

\begin{proof}
Suppose that the given purely atomic representation $\{t_\lambda:\lambda\in \Lambda\}$ on $\mathcal{H}$ is monic, with cyclic vector $\xi$ for $\{ t_\lambda t_\lambda^*\}_{\lambda \in \Lambda}$. Then by Theorem~\ref{thm-characterization-monic-repres} we can assume that $\mathcal{H}$ is of the form $L^2(\Lambda^{\infty},\mu)$, where the measure $\mu$ is given by the projection valued measure $P$ determined by the representation, i.e. $\mu(Z(\lambda))=\langle P(Z(\lambda)) \xi, \xi \rangle = \| t_\lambda^* \xi  \|^2$ for $\lambda\in\Lambda$.
Since $\{t_\lambda\}_\lambda$ is purely atomic, $\mu(\{\omega\}) = \|P(\{\omega\})\xi\|^2$ is nonzero iff $\omega \in \Omega$.  In other words, the atoms of $\mu$ are precisely the atoms of $P$.

To show that $P(\{\omega\})$ is always a rank-one projection for an atom $\omega$, we argue by contradiction.  Suppose that there exists $\omega \in \Lambda^\infty$ and a strict subprojection $Q_\omega \leq P(\{\omega\})$ with $Q_\omega \not= P(\{\omega\})$.  For any $\gamma = a \sigma^j(\omega) \in \text{Orbit}(\omega)$, write 
\[ Q_\gamma = t_a t_{\omega(0,j)}^* Q t_{\omega(0,j)} t_a^*,\] and set $Q = \bigoplus_{\gamma \in \text{Orbit}(\omega)} Q_\gamma$.

The fact that the projections $P(\{\gamma\})$ are mutually orthogonal  implies that $Q$ is indeed a sum of orthogonal projections.  Moreover, Proposition \ref{prop:atomic-2} implies that
each summand $Q_\gamma$ is a strict subprojection of $P(\{\gamma\})$.

We will show that $t_\eta Q  = Qt_\eta$ for all $\eta \in \Lambda$.
Since $\{t_\lambda\}_{\lambda \in \Lambda}$ is monic by assumption,   Theorem \ref{thm:ergodic} and Theorem \ref{thm-characterization-monic-repres} will then imply that $Q$ must be a multiplication operator, which contradicts the fact that each $Q_\gamma$ is a strict subprojection of $P(\{\gamma\})$.

Fix $ \eta \in \Lambda$ and $\gamma = a \sigma^j(\omega)$.  As in the proof of Proposition \ref{prop-atomic-basic-equns}, 
\[ Q_\gamma  t_\eta = t_a t_{\omega(0,j)}^*  Q_\omega t_{\omega(0,j)} t_a^* t_\eta =\sum_{(\rho, \zeta) \in \Lambda^{\min}(a, \eta)} t_a t_{\omega(0,j)}^*  Q_\omega t_{\omega(0,j)} t_\rho t_\zeta^*.\]
By Proposition \ref{prop:atomic-2}, $t_{\omega(0,j)}^* Q_\omega t_{\omega(0,j)} = Q_{\sigma^j(\omega)} = Q_{\sigma^j(\omega)} P(Z(\omega(j, j +d(\rho)))),$ and $t_{\omega(j, j+d(\rho))}^* t_\rho =0$ unless $\rho = \omega(j, j+d(\rho))$.  Thus, the sum collapses to (at most) a single term:  Writing $m = d(\rho) = d(a) \vee d(\eta) - d(a)$, 
\[ Q_\gamma t_\eta = \begin{cases}
t_a Q_{\sigma^j(\omega)} t_{\omega(j, j+m)} t_\zeta^*,& \eta \zeta = a \omega(j, j+m)\\
0 ,& \eta \zeta \not= a \omega(j, j+m)
\end{cases}\]
Now, using the fact that $Q_{\sigma^j(\omega)} = t_{\omega(j, j+m)} t_{\omega(j,j+m)}^* Q_{\sigma^j(\omega)}$, we obtain that if $Q_\gamma t_\eta \not= 0$, 
\[ Q_\gamma t_\eta = t_{\eta \zeta} t_{\omega(j,j+m)}^* Q_{\sigma^j(\omega)} t_{\omega(j,j+m)} t_\zeta^* = t_\eta Q_{\zeta \sigma^{m+j}(\omega)}.\]
For each fixed $\eta$, the map 
\[a \sigma^j(\omega) \mapsto\zeta \sigma^{m+j}(\omega), \quad \text{ where }a \omega(j,m+j) =\eta \zeta,\]
is a bijection from $\{\gamma \in \text{Orbit}(\omega): Q_\gamma t_\eta \not= 0\}$ to $\{ \tilde \gamma \in \text{Orbit}(\omega): t_\eta Q_{\tilde \gamma} \not= 0\}$. (Surjectivity follows by observing that, given $\eta \in \Lambda$ and $\tilde\gamma = \zeta \sigma^{q}(\omega)$ with $s(\eta) = r(\zeta)$, we can take $a = \eta \zeta, j = q$ to construct the preimage $\gamma$ of $\tilde \gamma$.)

  It now follows that, as claimed,
\[ Q t_\eta = t_\eta Q.\]

Conversely, suppose that $\{t_\lambda:\lambda\in \Lambda\}$ is a purely atomic representation of $C^*(\Lambda)$ on a separable Hilbert space $\mathcal{H}$ such that for every atom $x\in \Lambda^{\infty},\;P(\{x\})\H$ is one-dimensional. Let $\Omega\subset \Lambda^{\infty}$ be the support of the associated projection valued measure $P$ on $\Lambda^{\infty}.$ Since $\mathcal{H}$ is separable and since $P(\{x\})\H$ is orthogonal to $P(\{y\})\H$ for $x,\;y\in \Omega,\; x\not=y,$ we must have that  $\Omega$ is countable; let us enumerate $\Omega\;=\;\{\omega_n\}_{n=1}^{\infty}.$ Then
$$\sum_{n=1}^{\infty}P(\{\omega_n\}) = Id_{\mathcal{H}},$$
where the convergence is in the strong operator topology.
 For each $n\in\mathbb N,$ choose a unit vector $e_n\in P(\{\omega_n\})\H,$ and note that $\text{span}\{e_n\}\;=\;P(\{\omega_n\})\H.$
Define $\xi\in\H$ by
 $$\xi=\sum_{n=1}^{\infty}\frac{e_n}{2^n}.$$
 We note that $P(\{\omega_n\})(\xi)\;=\;\frac{e_n}{2^n}.$
 It follows that for each $n\in\mathbb N,$
 $$e_n\in\;\overline{\text{span}}\{t_{\lambda}t_{\lambda}^*(\xi)=P(Z(\lambda)(\xi):\;\lambda\in \Lambda\}.$$
This is due to the fact that for each $n\in\mathbb N,$ 
 $$\lim_{j\to \infty}t_{\omega_n(0,j)}t_{\omega_n(0,j)}^*(\xi)\;=\;\lim_{j\to \infty}P(Z(\omega(0,j)))(\xi)$$
 $$=\;P(\{\omega_n\})(\xi)\;=\;\frac{e_n}{2^n}.$$
 
Therefore $\xi$ is a cyclic vector for $\{t_{\lambda}t_{\lambda}^*: \;\lambda\in \Lambda\},$ so that this representation is monic.

\end{proof}

\section{Permutative representations of $C^*(\Lambda)$}
\label{sec:permutative_repn}


Here we  study permutative   representations of $C^*(\Lambda)$.  These are similar, but not precisely equivalent, to the atomic representations of single-vertex $k$-graphs studied by Davidson, Power and Yang in \cite{dav-pow-yan-atomic}; see that paper and the references therein for more details.

\subsection{Definition of permutative representations}

 \begin{defn} 
 \label{defpermutative}
(c.f. Definition 4.9 of \cite{dutkay-jorgensen-atomic}.) 
 Let $\Lambda$ be a row-finite $k$-graph with no sources. 
A representation $\{t_\lambda\}_{\lambda \in \Lambda}$ of $C^*(\Lambda)$ on a Hilbert space ${\mathcal H}$ is called {\it permutative} if ${\mathcal H}$ has an orthonormal basis $\{e_i: i\in I\}$ for some index set $I$ such that for each $\lambda\in \Lambda$ there are subsets $J_{\lambda}$ and $K_{\lambda}$ of $I$ and a bijection 
 $\tilde{\sigma}_{\lambda}:J_{\lambda}\to K_{\lambda}$ satisfying
\begin{enumerate}
 \item[(a)] For each $n\in\mathbb N^k,\;\cup_{\lambda\in \Lambda^n}J_{\lambda}\;=\;\cup_{\lambda\in \Lambda^n}K_{\lambda}=I;$
\item[(b)] For each $\lambda\in \Lambda$ and $\nu\in s(\lambda)
\Lambda,$ we have $K_{\nu}\subset J_{\lambda}$ and $\tilde{\sigma}_{\lambda}\circ \tilde{\sigma}_{\nu}=\tilde{\sigma}_{\lambda\nu}$. (This implies $J_{\lambda\nu}=J_{\nu}$ whenever $s(\lambda)=r(\nu)$).
 \item[(c)] $t_{\lambda}(e_i)\;=\;e_{\tilde{\sigma}_{\lambda}(i)}$ for $i\in J_{\lambda},$ and $t_{\lambda}(e_i)=0,$ for $i\notin J_{\lambda}.$
 \item[(d)] $t_{\lambda}^{\ast}(e_{\tilde{\sigma}_{\lambda}(i)})=e_{i}$ for $i\in J_{\lambda},$ and $t_\lambda^*(e_j) = 0$ for $j \in K_{\lambda'}$, if $\lambda \not= \lambda'$ but $d(\lambda) = d(\lambda')$.
 \end{enumerate}

 \end{defn}

\begin{rmk}
\label{rmk:permutative-vs-SBFS}
Alternatively, one can view a  permutative representation as a $\Lambda$-semibranching representation arising from a countable discrete measure space $(X, m)$, where $m(x) = 1 \ \forall \ x \in X$.  (The coding maps $\tilde{\sigma}^n$ of the $\Lambda$-semibranching function system are defined in Proposition \ref{prop:encoding-map}
below.) In particular, the faithful separable representation of Theorem \ref{prop:faithful-repn} is also a permutative representation.
\end{rmk}

We first prove:
\begin{lemma}  
 Let $\Lambda$ be a row-finite $k$-graph with no sources.
Let $\{ t_\lambda\}_{\lambda \in \Lambda}$ be a permutative representation of $C^*(\Lambda)$ on a Hilbert space ${\mathcal H},$  and let $\{J_{\lambda}\}_{\lambda\in\Lambda}$ and $\{K_{\lambda}\}_{\lambda\in\Lambda}$ be as in Definition~\ref{defpermutative}.  Then for $n\in\mathbb N^k,$ if $\lambda,\;\lambda'\in \Lambda^n$ and $\lambda\not=\lambda',$ then we have 
$K_{\lambda}\cap K_{\lambda'}=\emptyset.$
\label{lem:permutative}

\end{lemma}
\begin{proof} We recall if $\lambda,\lambda'\in\Lambda^n$ then 
$$t_{\lambda'}^{\ast}t_{\lambda}=\delta_{\lambda',\lambda}t_{s(\lambda)}.$$
So, for $\lambda,\lambda'\in\Lambda^n$ with $\lambda\not=\lambda',$ if there exists $j\in K_{\lambda}\cap K_{\lambda'},$ we could find $i\in J_{\lambda}$ with $\tilde{\sigma}_{\lambda'}(i)=j,$ and $k\in J_{\lambda'}$ with $\tilde{\sigma}_{\lambda}(k)=j.$
But then by definition of permutative representation, we would have 
$$t_{\lambda'}^{\ast}t_{\lambda}(e_i)\;=\;t_{\lambda'}^{\ast}(e_{\tilde{\sigma}_{\lambda}(i)})=t_{\lambda'}^{\ast}(e_j)\;=\;t_{\lambda'}^{\ast}(e_{\tilde{\sigma}_{\lambda}}(k))=e_k\not=0.$$
But this contradicts the fact that $t_{\lambda'}^{\ast}t_{\lambda}=0$ for $\lambda\ne \lambda'$, so we must have $K_{\lambda}\cap K_{\lambda'}=\emptyset.$
\end{proof}
 We define the {\it  encoding map} $E$ from the index set $I$ into $\Lambda^{\infty}$
by 
\begin{equation}\label{eq:encoding}
E(i)((0,n))=\lambda,\;\text{where}\;\lambda\;\text{is the unique element of}\;\Lambda^n\;\text{such that}\; i \;\in\:K_{\lambda}.
\end{equation}
To see that $E(i) \in \Lambda^\infty$ is well defined, we must check that if $m \geq n$ and $\lambda = E(i)((0,n))$, so that $i \in K_\lambda,$  then $E(i)((0,m)) = \lambda \nu$ for some $\nu \in \Lambda$.  Thus, 
suppose that $m\geq n\in \mathbb N^k.$
Write
$E(i)((0,m))=\mu.$  Then there exists a unique $j'\in J_{\mu}$ with $\tilde{\sigma}_{\mu}(j')=i\in K_{\mu}.$

Write $\mu = \lambda' \nu'$ where $d(\lambda') = n$.  We will show that $\lambda' = \lambda$.
Since $J_{\mu}=J_{\lambda'\nu'}\;=J_{\nu'}$ and 
$$\tilde{\sigma}_{\lambda'\nu'}\;=\tilde{\sigma}_{\lambda'}\circ \tilde{\sigma}_{\nu'},$$
we obtain $i = \tilde{\sigma}_{\lambda'} ( \tilde{\sigma}_{\nu'}(j')) \in K_{\lambda'}$.  Lemma \ref{lem:permutative} now implies that $\lambda = \lambda'$.

\begin{prop}\label{prop:encoding-map}
Let $\Lambda$ be a row-finite $k$-graph with no sources and let $I$ be an index set associated to a permutative representation  $\{ t_\lambda\}_{\lambda \in \Lambda}$  of $C^*(\Lambda)$.  Suppose that $E:I\to \Lambda^{\infty}$ is the encoding map given in \eqref{eq:encoding}. Then we have the following.
\begin{itemize}
\item[(a)]For each $i\in I$ and $n\in\mathbb N^k$, there is a unique $\lambda\in \Lambda^n$ and $i_n\in\;J_{\lambda}\subset I$ such that  $\tilde{\sigma}_{\lambda}(i_n)=i.$  Writing $\tilde{\sigma}^n(i)=i_n,$  we have $\tilde{\sigma}_{\lambda}\circ \tilde{\sigma}^n(i)=i$ for all $i\in K_{\lambda},$ and $\tilde{\sigma}^n\circ \tilde{\sigma}_{\lambda}(i)=i$ for all $i\in J_{\lambda}.$  

\item[(b)] The map $E:I\to\Lambda^{\infty}$ defined above satisfies 
\[
\sigma_{\lambda}(E(i))\;=\;E(\tilde{\sigma}_{\lambda}(i))\quad \text{for $i\in J_{\lambda}$},\quad\text{and}\quad \sigma^n(E(i))\;=\;E(\tilde{\sigma}^n(i))\quad \text{for $i\in I$}.
\]
\end{itemize}
\end{prop}
\begin{proof} For (a), notice that for $n\in \mathbb N^k,$ we have $I=\bigsqcup_{\lambda\in \Lambda^n}K_{\lambda},$ 
so that fixing $i\in I,$  there is a unique $\lambda\in \Lambda^n$ such that  $i\in K_{\lambda}.$ Since $\tilde{\sigma}_{\lambda}$ is a bijection from $J_{\lambda}$ to $K_{\lambda},$ there is a unique $i_n\in J_{\lambda}$ such that $\tilde{\sigma}_{\lambda}(i_n)=i$. 

Also, by definition of $\tilde{\sigma}^n,$ we have that  $\tilde{\sigma}_{\lambda}\circ \tilde{\sigma}^n(i)\;=\;\tilde{\sigma}_{\lambda}(i_n)=i$ for $i\in K_{\lambda},$ and  similarly 
$\tilde{\sigma}^n\circ \tilde{\sigma}_{\lambda}(i)=i,$ for all $i\in J_{\lambda}.$

For (b), recall $ \omega\in \Lambda^{\infty}$ is in the domain of $\sigma_{\lambda}$ if and only if $r(\omega)=s(\lambda).$  So recalling that $E(i)((0,0))$ is the unique $v\in \Lambda^0$ such that $i\in J_v,$ we will have $\sigma_{\lambda}(E(i))$ is defined if and only if $s(\lambda)=E(i)((0,0)),$ i.e $i\in J_{s(\lambda)}=J_{r(E(i))}.$ 

Recall $\tilde{\sigma}_{\lambda}(i)$ is defined only when $i\in J_{\lambda},$  and that Condition (b) of  Definition \ref{defpermutative} implies that  $J_{\lambda}=J_{s(\lambda)}\subset I.$  Also, by definition, if $i\in J_{\lambda},$ we have $\tilde{\sigma}_{\lambda}(i)\in K_{\lambda},$ and then by definition of the map $E,$ we obtain 
$$E( \tilde{\sigma}_{\lambda}(i))((0,d(\lambda)))=\lambda.$$
  On the other hand, notice that Condition (b) of Definition \ref{defpermutative} forces $K_v = J_v$ and $\tilde{\sigma}_v = id$ for all $v \in \Lambda^0$.
 Therefore, the fact that $J_\lambda = J_{s(\lambda)}$ implies that $r(E(i)) = s(\lambda)$.
Consequently, 
$$\sigma_{\lambda}( E(i))((0,d(\lambda)))=\;\lambda.$$
 We thus can see if $n'\in \mathbb N^k$ and $n'\leq d(\lambda),$
$$\sigma_{\lambda}(E(i))((0,n'))=E(\tilde{\sigma}_{\lambda}(i))((0,n'))=\lambda(0, n').$$

Now suppose $m\in\mathbb N^k,\;m\not=n.$  Let $\ell=m\vee n$, the coordinatewise maximum of $m$ and $n$.	
Then $n\leq \ell$ and $m\leq \ell.$  Find $n',\;m'\in \mathbb N^k$ such that $n+n'=m+m'=\ell.$ 
Write $\eta :=\sigma_{\lambda}(E(i))((0,\ell))=\eta;$
by the factorization property, we can find $\lambda',\gamma, \gamma'\in \Lambda$ with $d(\lambda')=n',\;d(\gamma)=m,$ and $d(\gamma')=\; m'$ such that 
$\eta\;=\;\lambda\lambda'\;=\;\gamma\gamma'.$
Then
$$\eta = \sigma_{\lambda}(E(i))((0,\ell))=\sigma_{\lambda}(E(i))((0,n+n'))= \lambda \, E(i)((0, n')).$$
It moreover follows that $E(i)((0, n')) =\lambda'$.

Our earlier argument now implies that $\sigma_{\lambda}(E(i))((0,m))\;=\;\gamma$.
Now, we observe that

$$E(\tilde{\sigma}_{\lambda}(i))((0,\ell))=E( \tilde{\sigma}_{\lambda}(i))((0,n+n'))=\lambda \lambda'=\eta,$$
since an argument similar to the one we used to show that $E(i)$ is well defined will tell us 
that $\eta$ is the unique element of $\Lambda^{\ell}$ such that $\tilde{\sigma}_{\lambda}(i)\in K_{\eta}=K_{\lambda\lambda'}.$
For, we note that finding $i'$ in $J_{\lambda'}$ such that $\tilde{\sigma}_{\lambda'}(i')=i\in K_{\lambda},$ we then obtain from our initial properties of these maps:
$$\tilde{\sigma}_{\lambda}(i)\;=\;\tilde{\sigma}_{\lambda}\circ \tilde{\sigma}_{\lambda'}(i')=\tilde{\sigma}_{\lambda \lambda'}(i')=\;\tilde{\sigma}_{\eta}(i'),$$ 
so that 
$$( \tilde{\sigma}_{\lambda}(i))((0,\ell))\;=\; E( \tilde{\sigma}_{\eta}(i'))((0,\ell))\;=\;\eta.$$
Since $\eta=\gamma \gamma'$ with $d(\gamma)=m,$ we also have  
$$E(\tilde{\sigma}_{\lambda}(i))((0,m))\;=\; E( \tilde{\sigma}_{\eta}(i'))((0,m))\;=\gamma.$$
It follows that for all $m\in \mathbb N^k,$
$$\sigma_{\lambda}(E(i))((0,m))\;=\;\gamma\;=\;E( \tilde{\sigma}_{\lambda}(i))((0,m)),$$
proving the first equality of (b).

From this, we will deduce the second equality.
Let $i\in I,\; n\in \mathbb N^k,$ and suppose that $E(i)((0,n))=\lambda\in \Lambda^n.$
Then setting $i_n=\tilde{\sigma}^n(i),$ we have $i_n\in J_{\lambda},\;i=\; \tilde{\sigma}_{\lambda}\circ \tilde{\sigma}^n(i)\;=\;\tilde{\sigma}_{\lambda}(i_n)\;\in K_{\lambda}.$ Moreover, the first equality of (b) gives
$$\sigma_{\lambda}( E(i_n))\;=\;E(\tilde{\sigma}_{\lambda}(i_n)).$$
We now apply $\sigma^n$ to both sides of this equation to obtain:
$$\sigma^n\circ\;\sigma_{\lambda}\circ E(i_n)\;=\;\sigma^n\circ E\circ \tilde{\sigma}_{\lambda}(i_n),$$
so that 
$E(i_n)\;=\;\sigma^n\circ E\circ \tilde{\sigma}_{\lambda}(i_n).$
But this implies 
\[E\circ \tilde{\sigma}^n(i)\;=\;\sigma^n\circ E(i). \qedhere \]
\end{proof}
When $E$ is injective, we obtain the following corollary.
\begin{cor}
Let $\Lambda$ be a row-finite $k$-graph with no sources.
 Let $\{ S_\lambda\}$ be a permutative representation of $C^*(\Lambda)$ on a Hilbert space ${\mathcal H},$ and let $\{e_i:i\in I\}$ be the ``permuted" basis for ${\mathcal H}.$  Let $E:I\to \Lambda^{\infty}$ be the encoding map of Equation \eqref{eq:encoding}.  Then if $E$ is one-to-one, the set $I$ can be identified with a subset of infinite paths $\Omega := E(I)$ in $\Lambda^{\infty}$ and the maps $\{\tilde{\sigma}_{\lambda}:\lambda\in\Lambda\}$ and $\{\tilde{\sigma}^n\;:n\in\mathbb N^k\}$ can be identified with the corresponding shifts and coding maps on the subset $E(I)=:\Omega$ of $\Lambda^{\infty}.$
\end{cor}
\begin{proof}
Since $E$ is a bijection from $I$ onto $\Omega\subset \Lambda^{\infty},$ the map $E^{-1}:\Omega\to I$ is well-defined, and we obtain 
from Proposition~\ref{prop:encoding-map}
\[
E^{-1}\circ \sigma_{\lambda}\circ E(i)\;=\;\tilde{\sigma}_{\lambda}(i)\quad\text{for $i\in J_{\lambda}$},\quad\text{and}\quad E^{-1}\circ\sigma^n\circ E(i)\;=\; \tilde{\sigma}^n(i)\quad \text{for $i\in I$.} \qedhere
\]
\end{proof}

\subsection{Decomposition of permutative representations}

The following is an extension of Theorem 4.13 of \cite{dutkay-jorgensen-atomic} from the case of the Cuntz algebras $\mathcal O_N$ to the much broader setting of higher-rank graph $C^*$-algebras.
 
\begin{thm}
\label{them-decomp-perm-repres}
Let $\Lambda$ be a row-finite $k$-graph with no sources and let $\{t_\lambda:\lambda\in \Lambda\}$ be a purely atomic representation of $C^*(\Lambda)$ on a Hilbert space $\mathcal{H}$. If the representation
is supported on an {orbit of aperiodic paths}, then the representation is permutative. Moreover the representation can be
decomposed into a direct sum of permutative representations with injective encoding maps.
\end{thm}
\begin{proof}
Let $P$ be the projection valued measure associated to the representation $\{S_\lambda:\lambda\in \Lambda\}$  and suppose it is supported on $\Omega\subset \Lambda^\infty,$ which by our earlier results can be decomposed into orbits corresponding to a decomposition of the original representation.  So let us assume that our set $\Omega$ is equal to a single orbit of the aperiodic path $\omega\in \Omega\subset \Lambda^{\infty}.$ As in the proof of Theorem 4.13 of \cite{dutkay-jorgensen-atomic}, let $\{e_{\omega,\ell}\}_{\ell\in {\mathcal J}}$ be an orthonormal basis for $P(\{\omega\})\mathcal{H}$ for an index set $\mathcal{J}$. 

We know that every point in the orbit of $\omega$ is of the form $a\sigma^j(\omega),$ for some finite path $a$ and an element $j\in\mathbb N^k.$ Moreover, since $\omega$ is an aperiodic path, this decomposition is unique. We define an orthonormal basis on $P(\{a\sigma^j(\omega)\})\mathcal{H}$ by $\{S_aS^*_{\omega(0,j)}e_{\omega,\ell}:=e_{a\sigma^j(\omega),\ell}\}_{\ell\in {\mathcal J}}.$ The results of our previous sections show that this is indeed an orthonormal basis for $P(\{a\sigma^j(\omega)\})\mathcal{H}.$ Since 
$$\Omega=\cup_{\gamma\in \Omega}\{\gamma\}\;=\;\bigcup_{a\in \Lambda: s(a)=r(\sigma^j(\omega))}\bigcup_{j\in \mathbb N^k}\{a\sigma^j(\omega)\},$$
we have 
$$\text{Id}_{\mathcal{H}}\;=\;P(\Omega)=\bigoplus_{a\in \Lambda: s(a)=r(\sigma^j(\omega))}\sum_{j\in \mathbb N^k}P(\{a\sigma^j(\omega)\}),$$
where the sum converges in the strong operator topology.
Therefore, an orthonormal basis for ${\mathcal H}$ is given by 
$$\bigcup_{a\in \Lambda: s(a)=r(\sigma^j(\omega))}\bigcup_{j\in \mathbb N^k}\{S_aS_{\omega(0,j)}^*e_{\omega,\ell}\}_{\ell\in {\mathcal J}}$$
$$=\;\bigcup_{a\in \Lambda: s(a)=r(\sigma^j(\omega))}\bigcup_{j\in \mathbb N^k}\{e_{a\sigma^j(\omega),\ell}\}_{\ell\in {\mathcal J}}.$$  
As in the case for ordinary Cuntz algebras, we can check that setting
$${\mathcal H}_\ell\;=\;\overline{\text{span}}\bigcup_{j\in \mathbb N^k}\{S_aS_{\omega(0,j)}^*e_{\omega,\ell}=e_{a\sigma^j(\omega),\ell}\}\;=\;\{e_{\gamma,\ell}:\;\gamma\in \Omega\},$$
each ${\mathcal H}_\ell$ is an invariant subspace for the representation, and
$${\mathcal H}\;=\;\bigoplus_{\ell \in {\mathcal J}}{\mathcal H}_\ell.$$
Thus, in this case, our index set for the orthonormal basis for ${\mathcal H}$ is given by:
$$I:=\;\{(a\sigma^j(\omega)=\gamma,\ell):\; \gamma\in \Omega,\; \ell\in {\mathcal J}\}.$$
Returning to our notation in Definition 11.1, if $a\in \Lambda,$ 
then 
$$J_a\;=\;\{(\gamma,\ell):\ell\in {\mathcal J},\;\gamma\in \Omega\;|\;s(a)=r(\gamma)\},$$
and 
$$K_a\;=\;\{(\gamma,\ell):\ell\in {\mathcal J},\;\gamma\in \Omega,\;\gamma(0,d(a))=a\}.$$

The encoding map from 
$$I=\;\{(a\sigma^j(\omega)=\gamma,\ell):\; \gamma\in \Omega,\; \ell\in {\mathcal J}\}.$$ is evidently given by
$$E((\gamma,\ell))\;=\;\gamma.$$
The maps 
$$\tilde{\sigma}_a:J_a\to K_a$$
are given by 
$$\tilde{\sigma}_a((\gamma,\ell))\;=\;(a\gamma,\ell),$$ and the coding maps 
$$\tilde{\sigma}^n(\gamma,\ell)\;=\;(\sigma^n(\gamma),\ell).$$
One calculates that the desired conditions of Definition 11.1 hold, and one can verify that taking the afore-mentioned orthonormal basis for the subspace ${\mathcal H}_{\ell},$ and restricting the encoding map $E$ to that basis, the encoding map $E$ will be injective. 

\end{proof}

\end{document}